\documentclass[11pt]{article}
 \usepackage{amsmath, amsthm, amssymb, bbm, setspace,bigints}
 \usepackage{mathrsfs}
  \usepackage{bm}
 \usepackage[margin=1 in]{geometry}
\usepackage{caption}
\usepackage[ruled,linesnumbered]{algorithm2e}
\usepackage{subcaption}
\usepackage[toc,page]{appendix}     
\usepackage{pdfpages}
\usepackage{epstopdf}
\usepackage{booktabs}
\usepackage{tabularx, multirow}
\usepackage{verbatim}
\usepackage[square,sort,comma,numbers]{natbib}
\usepackage{empheq}
\usepackage{accents}
\usepackage{bm}
\usepackage{authblk}
\usepackage{hyperref}
\usepackage[normalem]{ulem}
\newcommand*{\B}[1]{\ifmmode\bm{#1}\else\textbf{#1}\fi}

\newtheorem{theorem}{Theorem}[section]
\newtheorem{lemma}[theorem]{Lemma}
\newtheorem{corollary}[theorem]{Corollary}
\newcounter{assumptionA}

\stepcounter{assumptionA} 

\theoremstyle{remark}
\newtheorem{definition}[theorem]{Definition}

\newtheorem{assumption}{Assumption}[assumptionA]
\newcommand{\dd}{{\rm d}}
\newcommand{\argmin}{\mathop{\rm argmin~}}
\newcommand{\argmax}{\mathop{\rm argmax~}}
\newcommand{\matnorm}[1]{{\left\vert\kern-0.25ex\left\vert\kern-0.25ex\left\vert #1 
    \right\vert\kern-0.25ex\right\vert\kern-0.25ex\right\vert}_{\rm op}}
\newcommand{\ri}{(\textrm{i})}
\newcommand{\rii}{(\textrm{ii})}
\newcommand{\riii}{(\textrm{iii})}
\newcommand{\wht}{\widehat}
\newcommand{\wt}{\widetilde}
\newcommand{\mx}{\mbox}
\newcommand{\KL}{D_{\rm KL}}
\newcommand{\id}{\textrm{Id}}

\def\ms{\mathscr}
\def\mb{\mathbb}
\def\m{\mathcal}

\newcommand{\PX}{\ms P_2^r(\mb R^d)}

\title{\vspace{-2em} Mean-Field Variational Inference via Wasserstein\\
Gradient Flow}

\author{Rentian Yao}
\author{Yun Yang}
\affil{Department of Statistics, University of Illinois at Urbana-Champaign \authorcr Email: \{rentian2, yy84\}@illinois.edu}
\date{\vspace{-2em}}

\begin{document}
\maketitle

\begin{abstract}
Variational inference, such as the mean-field (MF) approximation, requires certain conjugacy structures for efficient computation. These can impose unnecessary restrictions on the viable prior distribution family and further constraints on the variational approximation family. In this work, we introduce a general computational framework to implement MF variational inference for Bayesian models, with or without latent variables, using the Wasserstein gradient flow (WGF), a modern mathematical technique for realizing a gradient flow over the space of probability measures. Theoretically, we analyze the algorithmic convergence of the proposed approaches, providing an explicit expression for the contraction factor. We also strengthen existing results on MF variational posterior concentration from a polynomial to an exponential contraction, by utilizing the fixed point equation of the time-discretized WGF. Computationally, we propose a new constraint-free function approximation method using neural networks to numerically realize our algorithm. This method is shown to be more precise and efficient than traditional particle approximation methods based on Langevin dynamics.
\end{abstract}

\vspace{-1em}
\tableofcontents

\section{Introduction}
One of the core problems of modern Bayesian inference is to compute the posterior distribution, a joint probability measure over unknown quantities, such as model parameters and unobserved latent variables, obtained by combining data information with prior knowledge in a principled manner.
Modern statistics often rely on complex models for which the posterior distribution is analytically intractable 
and requires approximate computation.
As a common alternative strategy to conventional Markov chain Monte Carlo (MCMC) sampling approach for approximating the posterior, variational inference (VI,~\cite{bishop2006pattern}), or variational Bayes~\cite{fox2012tutorial}, finds the closest member in a user specified class of analytically tractable distributions, referred to as the variational (distribution)
family, to approximate the target posterior. Although MCMC is asymptotically exact, VI is usually orders of magnitude faster \cite{blei2017variational, salimans2015markov} since it turns the sampling or integration into an optimization problem.  VI has successfully demonstrated its power in a wide variety of applications, including clustering~\cite{blei2006variational, corduneanuvariational}, semi-supervised learning~\cite{kingma2014semi}, neural-network training~\cite{anderson1987mean, opper1997mean}, and probabilistic modeling~\cite{jordan1999introduction, blei2003latent}. Among various approximating schemes, the mean-field (MF) approximation, which originates from statistical mechanics and uses the approximating family consisting of all fully factorized density functions over (blocks of) the unknown quantities, is the most widely used and representative instance of VI that is conceptually simple yet practically powerful.

On the downside, VI still requires certain conditional conjugacy structure to facilitate efficient computation (c.f.~Section~\ref{sec:MFVI}), in the same spirit as the requirement of a closed-form E-step in the expectation-maximization (EM,~\cite{dempster1977maximum}) algorithm, a famous iterative method 
for parameter estimation in statistical models involving unobserved latent variables.
Such a requirement unfortunately may: 1.~add restrictions to the viable prior distribution family, limiting the applicability of VI; 2.~call for specifically designed tricks for the implementation, making the VI methodology less generic and user-friendly; 3.~need impose further constraints on the variational family, leading to increased approximation error.  
For example, when implementing Bayesian Gaussian mixture models for clustering, although independent Gaussian priors of cluster centers meet the aforementioned conditional conjugacy property, it is sensible to instead employ a class of repulsive priors~\cite{xie2020bayesian} to encourage the well-separatedness of cluster centers and reduce the potential redundancy of components. Unfortunately, the complicated dependence structure introduced by the repulsive prior destroys the conditional conjugacy, making the standard coordinate ascent variational inference (CAVI,~\cite{bishop2006pattern}) algorithm for implementing the MF approximation inapplicable (see Section~\ref{subsec: RGMM} for further details).
Another example is Bayesian logit model~\cite{jaakkola1997variational}. Due to the lack of conditional conjugacy,~\cite{jaakkola2000bayesian} proposes to use a tangent transformation motivated by convex duality to make the variational approximation computationally tractable. For the mixed multinomial logit model,~\cite{braun2010variational} derives a  variational procedure based on the multivariate delta method for moments, which again requires specialized treatments and lacks generality.

In this paper, we propose a new computational framework for MF variational inference based on Wasserstein gradient flow, that is, running a ``gradient descent'' over the Wasserstein space, the space of all probability distributions with finite second moments endowed with the $2$-Wasserstein metric $W_2$~\citep{ambrosio2008gradient}.
Comparing to existing approaches, our approach does not impose any extra restrictions on the MF variational family, and can be applied to Bayesian models without any structural constraint on the prior and data likelihood function. 

\subsection{Related work}
There are a number of studies aiming at building generic VI procedures for dealing with non-conjugate models while maintaining the computational tractability. For example,~\cite{wang2012variational} develops two generic methods, Laplace variational inference and delta method variational inference, for a class of non-conjugate models with certain constraints (more precisely, partly-conjugate models), by enforcing the variational family for the model parameters in the MF approximation to be the (multivariate) location-scale Gaussian family.
Automatic differentiation variational inference (ADVI)~\citep{kucukelbir2017automatic, wingate2013automated} provides an automatic scheme that derives an iterative algorithm for implementing the variational inference based on automatic differentiation and stochastic gradient ascent; but the performance of ADVI heavily depends on the parametrization of the variational family, and little theory has been developed to analyze its algorithmic convergence. In a related thread,~\cite{ranganath2014black} proposes black box variational inference (BBVI) based on stochastic optimization, which is shown to have exponential convergence up to the noise level of stochastic gradient. However, both ADVI and BBVI only apply to parametric variational families that are finite-dimensional, which may unnecessarily impose additional constraints on top of the MF approximation --- the constituting components of the density product in the MF family should be restriction-free and may not be characterized by a finite number of parameters.

The notion of Wasserstein gradient flow is first introduced in the influential work of~\cite{jordan1998variational}. The authors reveal an appealing connection between: 1.~the dynamics of a gradient flux, or steepest descent, for minimizing the free energy with respect to the Wasserstein metric; 2.~a special class of partial differential equations (PDE), called the Fokker-Planck equation~\cite{gardiner1985handbook,risken1996fokker}, which describes the evolution of the probability density for the position of a particle whose motion is
described by a corresponding Ito stochastic differential equation (SDE). Specifically, \cite{jordan1998variational} constructs a discrete and iterative variational scheme, also called the Jordan-Kinderlehrer-Otto (JKO) scheme, which extends from the Euclidean gradient descent and whose solutions (weakly) converge to the solution of the Fokker-Planck equation with the gradient of a potential as the drift term. Later,~\cite{otto2001geometry} extends this connection to the porous medium equation, and points out the resemblance between the Wasserstein space and an ``infinite-dimenisonal'' Riemannian manifold. A comprehensive development of gradient flows in a general metric space, including the Wasserstein space as a representative application, is provided in the monograph~\cite{ambrosio2008gradient}. 
The deep connection between Wasserstein gradient flows and a rich class of PDE (SDE) builds a bridge between geometric analysis, optimal transport, control theory and partial differential equations; and also motivates a class of particle based methods~\cite{carrillo2019blob,carrillo2022primal,frogner2020approximate} for numerically solving PDE (SDE). It is worth highlighting that the development of Wasserstein gradient flows
heavily relies on recent techniques from modern optimal transport theory~\cite{brenier1991polar,monge1781memoire,villani2003topics, villani2009optimal}.

Some recent works also apply gradient flow over the space of probability measures to faciliate the computation of Bayesian statistics. For example,~\cite{trillos2020bayesian} considers sampling from the posterior distribution based on gradient flows in a different context, by treating the posterior distribution as the minimizer of functionals with certain forms; and they propose to use the gradient flow to guide the choice of proposals for MCMC methods. 
While we are preparing the manuscript, we learnt that a concurrent work~\cite{lambert2022variational} also study the application of Wasserstein gradient flow to the computation of variational inference. Unlike our work, \cite{lambert2022variational} focuses on Gaussian variational inference, where the target posterior (without latent variables) is approximated by the closest member in the Gaussian (local-scale) distribution family. Since the Gaussian distribution family is a parametric family, their gradient flow is defined on the Bures-Wasserstein space of Gaussian measures and is intrinsically finite-dimensional. \cite{lambert2022variational} proves the exponential convergence of a time-discretized version of the evolutionary ODE on the mean vector and covariance matrix, under the assumption that the target posterior distribution is strictly log-concave. In contrast, the mean-field (MF) variational approximation considered in our work involves an infinite-dimensional family, and our Bayesian latent variable model framework accommodates latent variables that are of discrete types. Moreover, we also study the large-sample statistical properties of the MF approximation, utilizing the fixed-point equation of our proposed time-discretized Wasserstein gradient flow.

\subsection{Contribution summary}
The main contribution of this paper is to propose a mean-field Wasserstein gradient flow (MF-WGF) algorithm for implementing the MF variational inference and to build a general theoretical framework for analyzing its statistical and algorithmic convergence for a generic class of Bayesian models (under the frequentist perspective).

Methodology-wise, by viewing the KL divergence as an objective functional over the space of all factorized probability measures, we develop a minimization scheme for implementing the MF approximation based on a time-discretized WGF. For Bayesian models without latent variables, the proposed algorithm is a distributional version of parallel coordinate proximal descent for updating the constituting components in the MF approximation. For Bayesian latent variable models, the proposed algorithm resembles a distributional version of the classical Expectation–Maximization algorithm, consisting of an E-step of updating the latent variable variational distribution and an M-step of conducting steepest descent over the variational distribution of model parameters; the developed algorithm can also be viewed as an extension of the general Majorize-Minimization (MM) principal to minimizing a functional over the space of probability measures. 

Theoretically, since a Wasserstein gradient flow extends the usual Euclidean gradient flow, we analogously define the notion of (local) ``convexity'' and ``smoothness'' for a generic functional in the Wasserstein space, under which (local) exponential convergence towards the optimum of the functional can be proved.
To prove and quantify the algorithmic convergence, we illustrate how the ``convexity'' and ``smoothness'' of the objective functional in VI, which is the Kullback–Leibler divergence to the target posterior distribution, translate into conditions of the statistical model. 
As a result, we explicitly determine the algorithmic contraction rate in terms of various problem characteristics such as step size, sample size, smoothness of the likelihood function, missing data Fisher information, and observed data Fisher information. As an intermediate result in our proof, we show that the MF approximation to the posterior distribution inherits the consistency and contraction of the latter (Theorems~\ref{thm: MFVI_posterior_convergence_rate} and~\ref{thm: posterior_convergence_rate}); our result of a squared-exponential (or sub-Gaussian) type deviation bound on the MF approximation is stronger than most existing results that only implies a polynomially decay bound. In addition, unlike many previous works relying on case-by-case analysis~\citep{hall2011theory,hall2011asymptotic,ormerod2012gaussian,titterington2006convergence,westling2015establishing,bickel2013asymptotic,zhang2020theoretical} or applying some information inequality that relates the variational objective functional value to certain risk function evaluating the estimation error~\cite{alquier2020concentration,pati2018statistical,yang2020alpha,zhang2020convergence}, our proof is general and based on identifying and analyzing the fixed point of the iterative scheme in MF-WGF. Our proof strategy offers a somewhat more direct insight explaining why MF approximation leads to consistent estimation, and can be potentially useful for investigating statistical properties of other approximation schemes beyond the mean-field. 

Computation-wise, we discuss and compare two concrete numerical methods for realizing the JKO scheme. 
The first method is a Langevin SDE-based particle method for approximately realizing the JKO scheme, which is commonly used in the literature. However, according to our numerical experiments and discussion, the SDE approach suffers from a systematic error that remains undiminished even with more iterations and number of particles due to a long term bias term.
This motivates us to propose an alternative method based on function approximation (FA) using neural networks. As we illustrate, the FA approach is unbiased, meaning that its output precisely solves the JKO scheme. Consequently, the unique fixed point of the iterative process from FA precisely yields the MF approximation solution; and there is no long term systematic bias arising from using a finite step size.
We also highlight that different from the previous work on functional approximation such as~\cite{mokrov2021large}, our function approximation approach is based on an unconstrained formulation (c.f.~Theorem~\ref{thm: FA_approx}) without the need of restricting the transport map into a gradient vector field of a convex potential.
This property allows flexible choices of numerical methods for solving the corresponding optimization problem, and significantly enhances the convergence speed and overall performance of the algorithm.

\subsection{Organization}
The remainder of this paper is organized as follows.  Section~\ref{sec: Background and examples} provides some preliminary results and the problem formulation. Specifically, we start with some background introduction to optimal transport theory and Wasserstein gradient flows; then we provide some new theoretical results about contraction properties of a discretized Wasserstein gradient flow, called the one-step minimization movement or the JKO scheme, with an explicit contraction rate; lastly, we discuss the connection between Wasserstein gradient flows and mean-field variational inference, and formulate the problem to be addressed in this work. 
In Section \ref{sec:MF-WGF}, we first provide a general computational framework for mean-field inference via alternating minimization, and then propose a new algorithm based on the discretized Wasserstein gradient flow. Section \ref{sec: main results} presents our main theoretical result about the statistical concentration of the mean-field approximation and the algorithmic contraction of the proposed algorithm.
In Section~\ref{sec:numeric_method}, we introduce and compare two numerical methods, particle approximation via SDE and function approximation method, for implementing the JKO scheme.
In Section \ref{sec:thm_app}, we apply our theoretical results to two representative examples, namely, the Gaussian mixture model and the mixture of regression model; we also conduct some numerical experiments to compliment the theoretical findings.
All proofs and other technical details are postponed to a supplementary material, which includes all the appendices.

\subsection{Notation}
We use $\ms P(\mb R^d)$ to denote the space of all probability measures on $\mb R^d$, and use $\ms P_2(\mb R^d)$ to denote the subset of $\ms P(\mb R^d)$ composed of all measures with finite second-order moment, i.e.
\begin{displaymath}
    \ms P_2(\mb R^d) = \Big\{\mu\in\ms P(\mb R^d): \int_{\mb R^d}\|x\|^2\ \dd\mu(x)<\infty\Big\}.
\end{displaymath}
Let $\ms P_2^r(\mb R^d)$ denote the space of all probability measures in $\ms P_2(\mb R^d)$ that admit a density function relative to the Lebesgue measure of $\mb R^d$. For any measure $\mu$ on $\mb R^d$ and map $T:\, \mb R^d\to \mb R^d$, the pushforward measure $\nu = T_\#\mu$ is defined as the unique measure on $\mb R^d$ such that $\nu(A) = \mu\big( T^{-1}(A)\big)$ holds for any measurable set $A$ on $\mb R^d$. We use $\KL(p\,\|\,q)$ to denote the KL divergence between two probability measures $p$, $q\in\PX$. Depending on the context, we may use upper letters to denote probability measures, and lower letters to denote their probability density functions. 
For any $\alpha\in[1, \infty)$, let $\psi_{\alpha}:\,\mb R_+\to \mb R_+$ be the function defined by $\psi_{\alpha} (x) = \exp(x^{\alpha})-1$.  We use the notation $\| \xi \|_{\psi_\alpha}=\inf \big\{ C>0: \,\mathbb{E}[ \,\psi_{\alpha}( | \xi | /C)\,] \leq 1\big\}$ to denote the $\alpha$-th order \emph{Orlicz norm} of a real-valued random variable $\xi$ (see Appendix~\ref{Appendix: conc} for a brief review). We also use $\m L(\xi)$ to denote the law (distribution) of random variable $\xi$. We use $\matnorm{M}=\sup_{v\in\mb S^{n-1}} \|Mv\|$ to denote the matrix operator norm of a matrix $\bm M\in\mb R^{m\times n}$, where $\mb S^{n-1}$ is the $(n-1)$-dimensional unit sphere. We use $\id$ to denote the identity map.

\section{Preliminary Results and Problem Formulation} \label{sec: Background and examples}
In this section, we first briefly review some concepts and basic results from optimal transport theory. After that, we discuss the notion of Wasserstein gradient flow and its discrete-time version, and present some new results about the contraction of one-step discretized Wasserstein gradient flow, which will be useful in our later analysis of alternating minimization for solving mean-field variational inference. Finally, we setup the Bayesian framework, review the mean-field inference, and formulate the problem to be addressed in this paper. Further details and techniques, such as subdifferential calculus in the Wasserstein space for analyzing the optimization landscape of functionals of probability measures and its connection with the usual Gateaux derivative (a.k.a.~first variation), are deferred to Appendix~\ref{app:background}.

\subsection{Optimal transport and Wasserstein space}\label{sec: WGF}
The Wasserstein space $\mb W_2(\mb R^d)=\big(\ms P_2(\mb R^d), W_2\big)$ is the separable metric space that endows $\ms P_2(\mb R^d)$ with the $2$-Wasserstein metric $W_2$~\citep{ambrosio2008gradient}. In particular, the 2-Wasserstein distance between two distributions $\mu$ and $\nu$ in $\ms P_2(\mb R^d)$ is defined as
\begin{equation}
    \label{eqn:kantorovich_problem}
    W_2^2(\mu, \nu) := \inf_{\gamma\in\Pi(\mu,\nu)} \Big\{ \int_{\mb R^d \times \mb R^d} \|x - y\|^2 \; \dd \gamma(x,y) \Big\},\qquad\mbox{(KP)}
\end{equation}
where $\Pi(\mu,\nu)$ consists of all possible distributions over $\mb R^d\times\mb R^d$ with marginals $\mu$ and $\nu$, and any $\gamma\in \Pi(\mu,\nu)$ is called a coupling between $\mu$ and $\nu$. It can be proved (Section 5 of~\cite{santambrogio2015optimal}) that $W_2$ is indeed a metric on $\ms P_2(\mb R^d)$ and satisfies the triangle inequality; moreover, convergence with respect to $W_2$ is equivalent to the usual weak convergence of probability measures plus convergence of second moments. If one of the distributions, say $\mu$, is absolutely continuous with respect to the Lebesgue measure of $\mb R^d$, or $\mu\in\ms P_2^r(\mb R^d)$, then the optimal coupling $\gamma^\ast =(\id,\, T^\ast)_\# \mu$ is unique (Theorem 1.22,~\cite{santambrogio2015optimal}) and supported on the graph of a map $T^\ast:\,\mb R^d\to\mb R^d$, called the optimal transport map from $\mu$ to $\nu$; see Appendix~\ref{app:OPTmap} for further properties of this optimal transport map.

\subsection{Wasserstein gradient flow}
\label{section: gradient_flow_in_probability_space}
Consider the problem of minimizing a functional $\m F:\ms P_2(\mb R^d) \to \mb R$ in the Wasserstein space $\mb W_2(\mb R^d)$ via ``steepest descent''. A direct generalization of the ODE formulation of the Euclidean gradient flow (c.f.~Appendix~\ref{app:EGF}) is to define a time-dependent measure $\rho_t\in\PX$ satisfying $\partial_t\rho_t = -\nabla_{W_2}\m F(\rho_t)$ for $t>0$, with some initialization $\rho_0\in\PX$. Here $\nabla_{W_2}\m F$ stands for some proper notion of gradient, or steepest (ascent) direction, of $\m F$ with respect to the $W_2$ metric in $\mb W_2(\mb R^d)$. To formally define and prove the well-posedness of this Wasserstein gradient flow (WGF), one can first consider a minimization movement scheme, also called the Jordan-Kinderlehrer-Otto (JKO) scheme~\cite{jordan1998variational} (see Figure~\ref{fig: JKO vs GF} for an illustration),
\begin{equation}\label{eqn: JKO_scheme}
    \rho_{k+1}^\tau = \argmin_{\rho\in\ms P_2(\mb R^n)} \m F(\rho) + \frac{1}{2\tau}W_2^2(\rho_k^\tau, \rho) \quad\mx{for $k\geq 0$}; \qquad\mx{(JKO)}
\end{equation}
and then show by using a generalised version of Arzel\`{a}–Ascoli theorem that after suitable interpolation, the solution of this JKO scheme admits a limit as step size $\tau\to 0_+$; finally this limit as an absolutely continuous curve (proved by a priori estimate) in $\mb W_2(\mb R^d)$ is defined as the Wasserstein gradient flow for minimizing $\m F$ starting from $\rho_0$. Details of a complete proof in the more general setting of gradient flows in metric spaces can be found in Chapter 3 of~\cite{ambrosio2008gradient}. A proof in the case of Fokker-Planck equation as the Wasserstein gradient flow of the KL functional~\eqref{eqn: functional_of_interest} (c.f.~Section~\ref{sec:KL_flow}) can be found in~\cite{jordan1998variational} or Chapter~8.3 of~\cite{santambrogio2015optimal}. 
\begin{figure}
    \centering
    \includegraphics[scale=0.6]{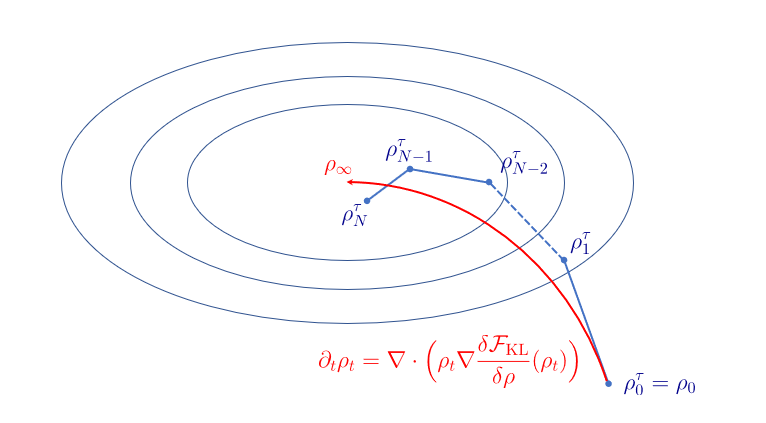}
    \caption{Illustration of a Wasserstein gradient flow (WGF, red curve) and its time-discretization via JKO-scheme (blue curve) with step size $\tau$. For a functional $\m F_{\rm KL}$ that is strictly convex along generalized geodesics, WGF converges to its global minimizer $\rho_\infty$ exponentially fast. JKO-scheme discretizes the WGF, has the same limiting (or stationary) point $\rho_\infty$ as WGF, and weakly converges to WGF as $\tau\to 0$.}
    \label{fig: JKO vs GF}
\end{figure}

Under the above perspective, it can be shown that the WGF for minimizing $\m F$ (by taking the limit of JKO scheme as $\tau\to 0_+$) can be characterized by the following partial differential equation (PDE), also called continuity equation,
\begin{align}\label{eqn:cont_equ}
    \partial_t \rho_t = - \nabla\cdot (\rho_t v_t), \quad \mx{with } v_t = -\nabla \frac{\delta \m F}{\delta\rho}(\rho_t),
    \quad \mx{for }t>0,
\end{align}
where $v_t:\, \mb R^d\to \mb R^d$ is the flow velocity vector field at time $t$, corresponding to the location-dependent steepest descent direction (i.e.~negative subdifferential) given in Lemma~\ref{lem:FV_SD}. 
Here, we have abused the notation by using $\rho_t\in\PX$ to denote both the probability measure and its density function. Similarly, in the rest of the paper, we will use the notation $\rho$ for a generic regular probability measure in $\PX$ and its density. PDE~\eqref{eqn:cont_equ} provides the Eulerian description of the WGF for functional $\m F$, and motivates one numerical method for implementing WGF via functional approximation (Section~\ref{sec:numeric_method}). 

In the continuity equation~\eqref{eqn:cont_equ}, we can view the time derivative $\partial_t\rho_t$ as the accumulation of probability mass, and intepret $\nabla\cdot (\rho_t v_t)$ as the ``Wasserstein gradient'' $\nabla_{W_2}\m F$, where $\rho_t v_t$ is the flux and the divergence term $\nabla\cdot (\rho_t v_t)$ represents the difference in flow in versus flow out. 
Note that the continuity equation~\eqref{eqn:cont_equ} may be interpreted as the equation governing the evolution of the density $\{\rho_t = (Y_t)_\#\rho_0:\, t > 0\}$ of a family of particles initially distributed according to
$\rho_0$, and each of which follows the flow $\{Y_t:\,t>0\}$. Here, the map $Y_t:\,\mb R^d\to\mb R^d$ is defined through $Y_t(x)=y_x(t)$ where, for any $x\in\mb R^d$, $\{y_x(t):\, t\geq 0\}$ is the solution to the following ODE,
\begin{align}\label{eqn:flow_ODE}
    \dot{x}_t=v_t(x_t), \quad\mx{for }t>0, \quad\mx{with }x_0=x,
\end{align}
where $v_t$ specifies the (steepest descent) direction of particles in the gradient flow. This ODE corresponds to a Lagrangian description of the WGF that characterizes the state of each individual ``particle" at each time, rather than counting the number of ``particles" sharing the same state (e.g., location and velocity), and motivates another numerical method for implementing WGF via particle approximation (Section~\ref{sec:numeric_method}).

\subsection{Contraction of one-step minimization movement}\label{sec:con_one_step}
The following functional $\m F_{\tau,\mu}:\, \ms P_2^r(\mb R^d)\to (-\infty,\infty]$ defined as
\begin{align}\label{Eqn:Phi_fun}
\m F_{\tau,\mu}(\nu) =  \m F(\nu) + \frac{1}{2\tau} W_2^2(\nu,\mu),
\end{align}
has been used in defining the minimization movement scheme~\eqref{eqn: JKO_scheme} for minimizing $\m F$ on $\mb W_2(\mb R^d)$.
We assume that for some $\tau_\ast>0$, $\m F_{\tau,\mu}$ admits at least a minimum point $\mu_\tau$, for all $\tau\in(0,\tau_\ast)$ and $\mu\in\ms P_2^r(\mb R^d)$. The map $\mu\mapsto \mu_\tau$ can be seen as a generalization from the usual Euclidean space to $\ms P_2^r(\mb R^d)$ of the proximal operator associated with functional $\tau \m F$, where the Euclidean distance is replaced by the Wasserstein distance. 


We can use $\mu_\tau$ or $\m F_{\tau,\mu}$ to define the one-step discretization of the Wasserstein gradient flow, which can then be used for both formally defining the gradient flow (as in Section~\ref{section: gradient_flow_in_probability_space}) and providing a numeric scheme for approximating the gradient flow.
Such a one-step discretization will also serve as the building block of the proposed MF-WGF with $\m F$ being the KL divergence to the target posterior (c.f.~Section~\ref{sec:alter_min}). In the rest of this subsection, we provide a theoretical analysis of the one-step minimization movement of minimizing $\m F_{\tau,\mu}$. This technical result will be useful in analyzing the convergence of the proposed MF-WGF method later. 

Convexity plays an important role in proving convergence and deriving explicit convergence rates of gradient flows in Euclidean space. To extend the notion of convexity to the Wasserstein space, one approach is to consider convexity along generalized geodesics. This requires the target functional $\mathcal{F}$ to exhibit convexity along certain interpolating curve that connects any pair of probability measures in $\mathbb{W}_2(\mathbb{R}^d)$. For a formal definition and additional properties, please refer to Appendix~\ref{app:convexity}. For any $\lambda>0$, we say that $\mathcal{F}$ is $\lambda$-convex along generalized geodesics if it is $\lambda$-convex along any generalized geodesic in the usual sense (as a univariate function under the constant speed parameterization of the curve). Using this notion, we have the following theorem about the contraction of one-step minimization movement.
Its proof is left to Appendix~\ref{app:proof_thm: implicit_WGF_conv}, which utilizes a key Lemma~\ref{Lem:stong_conx_generalized} to derive a contraction with an explicit contraction factor. 

\begin{theorem}\label{thm: implicit_WGF_conv}
Let $\m F:\, \ms P_2^r(\mb R^d)\to (-\infty,\infty]$ be $\lambda$-convex along generalized geodesics. Then for any $\mu$, $\pi\in \ms P_2^r(\mb R^d)$, 
\begin{align*}
(1+\tau\lambda )\, W_2^2(\mu_\tau,\pi) \leq  W_2^2(\mu,\pi) - 2\tau\big[\m F(\mu_\tau) - \m F(\pi)\big] - W_2^2(\mu_\tau,\mu),
\end{align*}
where
\begin{displaymath}
\mu_\tau = \argmin_{\rho\in\ms P_2^r(\mb R^d)} \m F(\rho) + \frac{1}{2\tau}W_2^2(\mu, \rho).
\end{displaymath}
In particular, if $\pi^\ast$ is any minimizer of $\m F$, then
\begin{align*}
W_2^2(\mu_\tau,\pi^\ast) \leq (1+\tau\lambda)^{-1} \, W_2^2(\mu,\pi^\ast), \quad \forall\mu\in\ms P_2^r(\mb R^d).
\end{align*}
\end{theorem}
As a direct consequence of the theorem, the time-discretized Wasserstein gradient flow for minimizing a $\lambda$-convex (along generalized geodesics) functional $\m F$ obtained by repeatedly applying the one-step minimization movement achieves an exponential convergence to the unique global minimizer of $\m F$, with contraction factor $(1+\tau \lambda)^{-1}\in(0,1)$ for any step size $\tau>0$. Note that this convergence behavior is similar to the implicit Euler scheme for minimizing a $\lambda$-convex function on $\mb R^d$, while the explicit Euler scheme is convergent only when $\tau$ is smaller than some threshold inverse proportional to the largest eigenvalue of the Hessian $\nabla^2 F$, indicating the robustness and stability of implicit schemes.

\subsection{KL divergence functional}\label{sec:KL_flow}
In this paper, we are interested in functionals over $\PX$ of the following form, due to the close connection with the KL divergence as the optimization objective in VI,
\begin{align}\label{eqn: functional_of_interest}
    \m F_{\rm KL}(\rho) = \underbrace{\int_{\mb R^d} V(x)\,\dd \rho(x)}_{\small \mx{potential energy $\m V(\rho)$}} +\ \  \underbrace{\int_{\mb R^d}\log\rho(x)\,\dd \rho(x)}_{\small \mx{entropy $\m E(\rho)$}}.
\end{align}
The KL functional $\m F_{\rm KL}$ consists of an entropy functional $\rho\mapsto \int\log\rho\,\dd\rho$ and a potential energy functional $\rho\mapsto \int V\dd\rho$, where $V:\,\mb R^d\to\mb R$ is the potential (function). 

When specialized to the KL functional $\m F_{\rm KL}$, the continuity equation~\eqref{eqn:cont_equ} for characterizing its Wasserstein gradient flow becomes the famous Fokker-Planck equation
\begin{equation}\label{eqn: Fokker_Planck_equation}
    \frac{\partial\rho_t}{\partial t} - \Delta\rho_t - \nabla\cdot(\rho_t\nabla V)=0,
\end{equation}
since the first variation $\frac{\delta \m F_{\rm KL}}{\delta\rho} = V + \log \rho +C$ (first variation is defined up to a constant) and $v_t=-\nabla V - \nabla \log \rho_t$.
It is well known that the solution $\rho_t$ to the Fokker-Planck equation also corresponds to the law of Langevin stochastic differential equation (SDE)
\begin{equation}\label{eqn: Langevin_dynamics}
\dd X_t = -\nabla V(X_t)\,\dd t + \sqrt{2}\,\dd W_t, \quad X_0 \sim \rho_0.
\end{equation}
This connection will motivate one of our discretizing schemes for realizing the Wasserstein gradient flow for $\m F_{\rm KL}$ (c.f.~Section~\ref{sec:particle_approx}).

It turns out that the entropy $\m E$ is convex along generalized geodesics (Lemma~\ref{lemma:entropy} in Appendix~\ref{app:convexity}) and the potential energy $\m V$ is 
$\lambda$-convex along generalized geodesics if the corresponding potential function $V$ is a $\lambda$-convex function over $\mb R^d$
(Lemma~\ref{lemma:potential}  in Appendix~\ref{app:convexity}).
Therefore, using Theorem~\ref{thm: implicit_WGF_conv}, we obtain the following corollary characterizing the contraction property of one-step movement minimization for minimizing the KL functional $\m F_{\rm KL}$.

\begin{corollary}\label{coro:one_step_KL}
If potential $V:\mb R^d\to\mb R$ is a $\lambda$-convex function over $\mb R^d$, then for any $\mu$, $\pi\in \ms P_2^r(\mb R^d)$, 
\begin{align*}
(1+\tau\lambda )\, W_2^2(\mu_\tau,\pi) \leq  W_2^2(\mu,\pi) - 2\tau\big[\m F_{\rm KL}(\mu_\tau) - \m F_{\rm KL}(\pi)\big] - W_2^2(\mu_\tau,\mu),
\end{align*}
where
\begin{displaymath}
\mu_\tau = \argmin_{\rho\in\ms P_2^r(\mb R^d)} \m F_{\rm KL}(\rho) + \frac{1}{2\tau}W_2^2(\mu, \rho).
\end{displaymath}
In particular, if $\pi^\ast(x)\propto e^{-V(x)}$ for $x\in\mb R^d$, then
\begin{align*}
W_2^2(\mu_\tau,\pi^\ast) \leq (1+\tau\lambda)^{-1} \, W_2^2(\mu,\pi^\ast), \quad \forall\mu\in\ms P_2^r(\mb R^d).
\end{align*}
\end{corollary}

\subsection{Mean-field variational inference}
\label{sec:MFVI}
A generic probabilistic model consists of a collection of observed variables $X\in\m X$ and a collection of hidden variables $Z\in\m Z$, where $Z$ may contain model parameters and latent variables as its components in a Bayesian setting. The goal is to (approximately) learn the posterior distribution $p(Z\,|\,X)=p(X,\,Z)/p(X)$ of the hidden variables given the observed ones. In a typical problem setting, the joint distribution $p(X,Z)$ is only known up to a constant; therefore, the exact computation of $p(Z\,|\,X)$ is intractable due to the high-dimensional integral involved in computing the normalization constant.

A generic variational inference (VI) approaches this task by turning the integration problem into an optimization one as below,
\begin{align}\label{Eqn:VI_for}
    \widehat q =\argmin_{q\in\Gamma} \KL\big(q\,\|\,p(\,\cdot\,|\,X)\big),
\end{align}
where $\Gamma$ is an user-specified distribution family over the hidden variable space $\m Z$, called the variational family. In another word, VI uses a closest member (relative to KL divergence) in the variational family $\Gamma$ to approximate the target posterior. The KL divergence is used as the discrepancy measure for two reasons: 1.~it can be computed up to a constant without the knowledge of the normalization constant in the posterior; 2.~it captures the information geometry in the statistical models. 


MF inference is a special case of VI when the variational family $\Gamma_{\rm MF}$ is composed of all factorized $q$ with the following form,
\begin{displaymath}
    \qquad\qquad q(z) = q_1(z_1)\, q_2(z_2)\,\cdots\, q_m(z_m), \quad \mx{for }z=(z_1,\,z_2,\,\ldots,\,z_m) \in \m Z=\m Z_1\times\cdots\times\m Z_m,
\end{displaymath}
where each component (block) $z_j$ of $z$ may contain more than one variables. 
In general, to alleviate the bias incurred by ignoring the dependence among the blocks $\{Z_j\}_{j=1}^m$, it is preferable to use a reduced number of blocks while maintaining the computational tractability of solving problem~\eqref{Eqn:VI_for}. In this work, we consider two model settings: Bayesian models with and without latent variables.



\smallskip
\noindent {\bf Bayesian models without latent variables.} In this setting, the hidden variables $Z$ solely consist of model parameters $\theta \in \Theta \subset \mathbb{R}^d$ and we consider mean-field approximation over (blocks of) components of $\theta$. We consider a standard model setting where the observations $X = X^n = \{X_1, \cdots, X_n\}$ are i.i.d.~given $\theta$. We denote the prior and posterior distributions of $\theta$ as $\pi_\theta$ and $\pi_n$, respectively, where
\begin{align}
    \pi_n(\theta) = \frac{\pi_\theta(\theta)\prod_{i=1}^n p(X_i\,|\,\theta)}{\int_\Theta\pi_\theta(\theta)\prod_{i=1}^np(X_i\,|\,\theta)\,\dd\theta}, 
    \quad\mx{for }\theta\in\Theta.
\end{align}
We further divide the parameter space into $m$ blocks, i.e., $\Theta = \bigotimes_{j=1}^m \Theta_j$, where $\Theta_j \subset \mathbb{R}^{d_j}$ and $d_1 + \cdots + d_m = d$. The corresponding MF approximation to $\pi_n$ to be  studied is 
\begin{align}\label{eqn: blockwise_optimize}
\wht q_\theta = \bigotimes_{j=1}^d \wht q_j \in\argmin_{q = \otimes_{j=1}^d q_j}\KL(q\,\|\,\pi_n).
\end{align}
Our theoretical result in Section~\ref{sec:Analysis_MF} demonstrates that point estimators obtained from above MF approximation achieve the same rate of convergence in estimation error as those obtained from the full posterior $\pi_n$, under the frequentist perspective that assumes $X^n$ to be generated from a true underlying data generating model indexed by a true parameter $\theta^\ast$.

\smallskip
\noindent {\bf Bayesian models with latent variables.} In this setting, we have observed variable $X=X^n$ as before but the hidden variable $Z$ now includes model parameter $\theta\in\Theta\subset\mb R^d$ and a collection of latent variables $Z^n = \{Z_1,\cdots, Z_n\}\in \m Z^n$, such that $(X_i, Z_i)_{i=1}^n$ are i.i.d.~given $\theta$. For simplicity, we assume the latent variables to be discrete, such as the latent class (cluster) indicators in Gaussian mixture models. Let $\pi_\theta$ denote the prior distribution defined on parameter space $\Theta$. To maintain a minimal number of blocks for maximally reducing the potential bias, we consider the following (two-block) mean-field approximation over the parameter block $\theta$ and latent variables block $Z^n$,
\begin{equation}\label{eqn: mean_field_optimization_problem}
(\widehat{q}_\theta, \,\widehat{q}_{Z^n}) = \argmin_{q_\theta\in \ms P(\Theta),\, q_{Z^n}\in\ms P(\m Z^n)}D_{KL}(q_{\theta}\otimes q_{Z^n}\,\|\,\pi_n),
\end{equation}
where in this case $\pi_n$ denotes the joint posterior distribution of $(\theta, Z^n)$, given by
\begin{equation}\label{eqn: meanfield_VI}
\pi_n(\theta, z^n) = \frac{\pi_{\theta}(\theta)\prod_{i=1}^np(X_i, z_i\,|\,\theta)}{\sum_{z^n\in\m Z^n}\int_{\Theta}\pi_{\theta}(\theta)\prod_{i=1}^np(X_i, z_i\,|\,\theta)\,\dd\theta},
\quad\mx{for }\theta\in\Theta \ \ \mx{and} \ \ z^n\in\m Z^n.
\end{equation}
It is also possible to consider a full mean-field approximation by also factorizing $q_{\theta}$ over blocks of $\theta$. However, this scheme may introduce additional complications without providing further insights due to its overlap with the first setting without latent variables.
A similar theoretical result in Section~\ref{sec:Analysis_MF} demonstrates the statistical optimality of point estimation using MF approximation~\eqref{eqn: mean_field_optimization_problem}.

\smallskip
\noindent {\bf Computation via coordinate ascent variational inference.} Alternating minimization is a natural and commonly used algorithm for optimizing over quantities taking a product form as in MF inference. 
The idea of alternative minimizing is to optimize over one component of $q$ at a time while fixing the others. Consider the generic MF approximation~\eqref{Eqn:VI_for} and let $q_{-j}(z_{-j})=\prod_{s\neq j} q_{s}(z_s)$ denote the joint distribution of $z_{-j}$, all components in $z$ except for $z_j$. 
When optimizing over the $j^{\text {th }}$ component $q_j$, one may
explicitly solve the optimizer $q_{j}^{*}:\,=\argmax_{q_j} \KL\big(q_j\otimes q_{-j}\,\|\,p(\cdot\,|\,X)\big)$ as
\begin{equation}\label{cavi}
q_{j}^{*}(z_j)\propto \exp \Big\{\int_{\m Z_{-j}} \log p(z_j,z_{-j}, \,X) \,\dd q_{-j}(z_{-j})\Big\},\quad\mx{for }z_j\in\m Z_j.
\end{equation}
However, to make the computation of $q_{j}^{*}$ tractable, one requires certain conditional conjugacy structures so that the integral inside the exponent can be explicitly calculated and the normalization constant of $q_{j}^{*}$ can be identified.
To avoid overly aggressive moves that may lead to non-convergence of the algorithm, it may be necessary to introduce a partial step size into the above update if $q_j^\ast$ can be recognized as a member of some parametric family~\cite{bhattacharya2023convergence}, leading to the so-called coordinate ascent variational inference (CAVI) algorithm~\cite{bishop2006pattern}. 

\smallskip
\noindent {\bf Goal of this work.} The main problem to be addressed in this work is to design a new class of computational algorithms for solving the above optimization problems for MF variational inference based on Wasserstein gradient flow while adapting the idea of alternating minimization, and to study their theoretical properties. Since Wasserstein gradient flow directly operates over the space of probability measures, the new method does not need
impose any extra restrictions on the MF variational family (which may unnecessarily increase the approximation error), and can be applied to Bayesian models without any structural constraint on the prior and data likelihood function. Moreover, a step size tuning parameter is naturally incorporated to prevent overly aggressive moves which can cause the algorithm to diverge.

\section{Mean-Field Variational Inference via Wasserstein Gradient Flow}
\label{sec:MF-WGF}
In this section, we propose a generic computational framework of MF variational inference for models with and without latent variables by alternating minimization and coordinate ascent in the Wasserstein space via repeatedly applying a one-step discretized Wasserstein gradient flow to components in the MF approximation. 

\subsection{Bayesian models without latent variables}\label{sec: models_without_latent}

Recall that the mean-field variational family $\Gamma = \{q = \bigotimes_{j=1}^m q_j: q_j\in\ms P_2(\Theta_j)\}$ is the set of all factorized distributions over $m$ blocks of parameter $\theta$. We use the shorthand $q_{-j}^{(k)} = \bigotimes_{l\neq j} q_l^{(k)}$ to denote the joint variational distribution of $\theta_{-j}$, the parameter vector $\theta$ without its $j$-th block $\theta_j$, in the $k$-th iteration. A standard algorithm for solving optimization problems involving multiple variables is alternating minimization. For technical convenience, we consider a parallel (simultaneous) update scheme for implementing the alternative minimization framework~\eqref{cavi} to motivate our proposed method, which takes the following form under the current model setting,
\begin{align}\label{eqn:AMinimization}
    q_j^{(k+1)} = \argmin_{q_j}\KL(q_j\otimes q_{-j}^{(k)}\,\|\,\pi_n)\quad\mx{for }j\in[m]\ \ \mx{and}\ \ k = 0, 1, \cdots.
\end{align}
Alternative minimization for solving MF can diverge due to its overly aggressive moves~\cite{bhattacharya2023convergence}. A common solution to avoid divergence when optimizing a multivariate function in Euclidean space is to use a one-step gradient descent, rather than fully minimizing the target function. In light of this, we propose replacing the update of $q_j$ by solving~\eqref{eqn:AMinimization} with a one-step discretized Wasserstein gradient flow for the functional $\KL(q_j\otimes q_{-j}^{(k)}\,\|\,\pi_n)$. This leads to a new computational framework for implementing the mean-field approximation~\eqref{eqn: blockwise_optimize} for Bayesian models without latent variables, which we call \emph{mean-field Wasserstein gradient flow} (MF-WGF), by iteratively solving $m$ sub-problems associated with discretized WGF in each iteration, which can be formulated as
\begin{align}\label{eqn: JKO_update_para}
    q_j^{(k+1)} \in \argmin_{q_j} \KL(q_j\otimes q_{-j}^{(k)}\,\|\,\pi_n) + \frac{1}{2\tau} W_2^2(q_j, q_j^{(k)})\quad\mx{for } j\in[m] \ \ \mx{and}\ \ k = 0, 1, \cdots.
\end{align}
The iterative updating formula~\eqref{eqn: JKO_update_para} can be treated as the coordinate proximal descent algorithm in the Wasserstein space for minimizing the multi-input functional $\KL(q_1\otimes\cdots\otimes q_m\,\|\,\pi_n)$. 
Here, we consider the parallel scheme which allows us to compute $q_j^{(k+1)}$ for different $j$ parallelly, making the algorithm computationally efficient for large $m$.

Another appealing feature of MF-WGF is that the time discretization via the minimization movement scheme does not introduce any bias---the MF solution $\wht q_\theta$ in~\eqref{eqn: blockwise_optimize} is the (unique) fixed point of the corresponding iterative procedure, as shown by our theoretical results in Section~\ref{sec: main results}.  Furthermore, the iterative procedure has exponential convergence to this solution. It is straightforward to show that $\wht q_\theta = \bigotimes_{j=1}^m \wht q_j$, as a fixed point to MF-WGF, satisfies the distributional equations
\begin{align}\label{eqn: dist_equation_para}
    \wht q_j(\theta_j) = \frac{\exp\big\{\int_{\Theta_{-j}}\log\pi_\theta(\theta) + \sum_{k=1}^n\log p(X_k\,|\,\theta)\,\dd\wht q_{-j}(\theta_{-j})\big\}}{\int_{\Theta_j}\exp\big\{\int_{\Theta_{-j}}\log\pi_\theta(\theta) + \sum_{k=1}^n\log p(X_k\,|\,\theta)\,\dd\wht q_{-j}(\theta_{-j})\big\}\,\dd\theta_j}, \quad\mx{for }j\in[m],
\end{align}
which can be proved by applying the first order optimality condition to~\eqref{eqn: JKO_update_para} in terms of the first variation as described in Section~\ref{sec:subdiff}. Later, we will use this fixed point equation to show the concentration of MF approximation $\wht q_\theta$ towards the true parameter $\theta^\ast$ (c.f.~Theorem~\ref{thm: MFVI_posterior_convergence_rate}, also see Section~\ref{sec:Analysis_MF} for a sketched proof). 
Unlike Bayesian latent variable models, we do not need this concentration property to prove the linear convergence of $\wht q^{(k)}$ towards $\wht q_\theta$ in the sense of $W_2$ metric, as stated in Theorem~\ref{thm: para_update_whp} in the next section.

\subsection{Bayesian latent variable models}\label{sec:alter_min}
Due to the conditional independence among discrete latent variables $Z_1,\ldots,Z_n$ given $\theta$ and $X^n$, it is easy to verify that any minimizer $\wht q_{Z^n}$ of optimization problem~\eqref{eqn: mean_field_optimization_problem} also factorizes as $\wht q_{Z^n} = \bigotimes_{i=1}^n\wht q_{Z_i}$. 
Consequently, the alternating minimization framework~\eqref{cavi} for solving mean-field optimization~\eqref{eqn: mean_field_optimization_problem} under this model setting can be formulated as: for iteration $k=0,1,\ldots$,
\begin{align}
\mx{\bf Latent variable update:} \quad q_{Z_i}^{(k+1)} &=  \Phi(q^{(k)}_\theta, X_i), \quad i=1,2,\ldots,n,    \quad\mx{with}    \notag\\
  \quad  \Phi(q_\theta, X_i)(z) &= \frac{ \exp\big\{\mb E_{q_\theta} \log p(z\,|\, X_i,\theta)\big\}}   { \sum_{z\in\m Z}\exp\big\{ \mb E_{q_\theta} \log p(z\,|\, X_i,\theta)\big\}},\ \ z\in\m Z; \label{Eqn:Phi_def}\\
\mx{\bf Parameter update:}\qquad \quad q_\theta^{(k+1)} &=\argmin_{q_\theta} V_n\big( q_\theta \,\big| \,q_\theta^{(k)}\big),\quad\mx{with}  \notag\\
\mx{(sample energy functional)} &\quad V _n(q_\theta\,|\,q_\theta') :\,= n \,\mb E_{q_\theta} \big[U_n(\theta;\, q'_\theta) \big]+ \KL(q_\theta\,||\,\pi_\theta),   \quad\mbox{and} \notag\\
\mx{(sample potential function)} & \quad   U_n(\theta,\, q'_\theta) :\, = -\frac{1}{n}\sum_{i=1}^n\sum_{z\in\m Z} \log p(X_i,z\,|\,\theta) \, \Phi(q_\theta', X_i)(z),\notag
\end{align}
where $\Phi:\, \ms P^r(\Theta) \times \m X \mapsto \ms P(\m Z)$ denotes the map that turns a $(q_\theta, x)$ pair to a probability measure (pmf) over $\m Z$. Since $\m Z$ is discrete, the latent variable update can be easily performed using the closed form formula. The $V_n$ functional above resembles the $Q$ function computed in the E-step of a generic EM algorithm, and is equivalent up to a constant to the KL divergence $D_{KL}\big(q_{\theta}\otimes q^{(k+1)}_{Z^n}\,\big\|\,\pi_n\big)$.

Writing the updating formula for $q_\theta$ via minimizing the KL divergence functional $V_n$ is more convenient for the design of algorithms and theoretical analysis. 
Since $Z_i$'s are discrete, 
in the preceding alternating minimization algorithm, updating $q_{Z^n}$ with a given $q_\theta$ amounts to solving $\min_{q_{Z^n}} D_{KL}(q_{\theta}\otimes q_{Z^n}\,\|\,\pi_n)$, which admits a closed form expression with tractable normalization. 
However, the step of updating $q_\theta$ by solving for
the exact minimizer of $V_n(\cdot\,|\,q_\theta^{(k)})$ may not be computationally tractable unless some conditional conjugacy condition is satisfied. Instead, we view $V_n(\cdot\,|\,q_\theta^{(k)})$ as the KL divergence functional over $\ms P^r_2(\Theta)$ and propose to update $q_\theta$ via its associated one-step discretized Wasserstein gradient flow. Note that in the situation where $Z^n$ is continuous, we may also apply a one-step discreteized Wasserstein gradient flow to update $q_{Z^n}$ rather than exactly minimize $D_{KL}(q_{\theta}\otimes q_{Z^n}\,\|\,\pi_n)$ over $q_{Z^b}$; the resulting algorithm then becomes coordinate descent over the space of all factorized probability distributions. We leave the formal study of this case to future work.

\begin{figure}[t]
    \centering
    \includegraphics[scale=0.7]{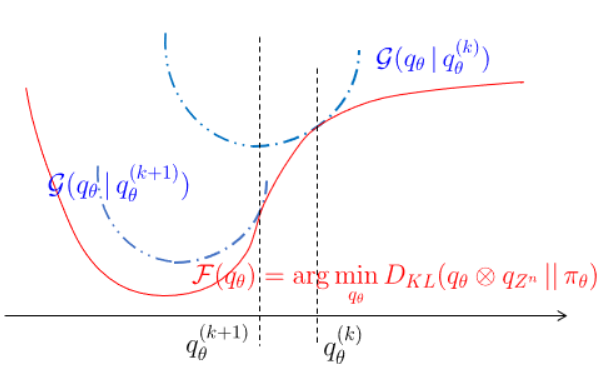}
    \caption{Mean-field Wasserstein gradient flow (MF-WGF) as an extension of the Majorize-Minimization (MM) algorithm~\cite{lange2016mm} for minimizing (profile-KL) functional $\m F=\min_{q_{Z^n}}D_{KL}\big(q_{\theta}\otimes q_{Z^n}\,\big\|\,\pi_n\big)$ over the space of all probability measures on parameter space $\Theta$. Here, $\m G(q_\theta\,|\,q_\theta'):\,=V_n(q_\theta\,|\,q_\theta') + \frac{1}{2\tau}W_2^2(q_\theta,\, q_\theta')$ majorizes $\m F$.}
    \label{fig: WFMGF as MM}
\end{figure}

The perspective of viewing the parameter update step as minimizing a time-dependent KL divergence functional over $\ms P_2(\Theta)$ leads to our new computational framework of MF-WGF for Bayesian latent variable models. More precisely, MF-WGF involves iteratively cycling through the following two steps for iteration $k=0,1,\ldots$:

\emph{{\bf Step 1} (Local latent variable):}  For $i=1,\ldots,n$, compute $q^{(k+1)}_{Z_i}$ based on the updating formula~\eqref{Eqn:Phi_def};

\emph{{\bf Step 2} (Global model parameter):} Compute the energy functional  $V_n(q_\theta\,|\,q_\theta^{(k)})$ using the most recent $q^{(k+1)}_{Z^n}=\bigotimes_{i=1}^n q^{(k+1)}_{Z_i}$, and
update $q_\theta$ via the one-step minimization movement scheme~\eqref{eqn: JKO_scheme} with objective functional $V_n(q_\theta\,|\,q_\theta^{(k)})$,
    \begin{equation}\label{eqn: JKO_update_qtheta}
        q_\theta^{(k+1)} = \argmin_{q_\theta}V_n(q_\theta\,|\,q_\theta^{(k)}) + \frac{1}{2\tau}W_2^2(q_\theta,\, q_\theta^{(k)}).
    \end{equation}

\noindent The two steps of MF-WGF resemble a distributional version of the E-step and the M-step respectively in the classical EM algorithm for dealing with missing data problems. One can also view MF-WGF as an Majorize-Minimization (MM) algorithm~\cite{lange2016mm} for distributional optimization (see Figure~\ref{fig: WFMGF as MM} for an illustration) where $\m G(q_\theta\,|\,q_\theta'):\,=V_n(q_\theta\,|\,q_\theta') + \frac{1}{2\tau}W_2^2(q_\theta,\, q_\theta')$ serves as the majorized version of the (profile) objective functional $q_\theta\mapsto \min_{q_{Z^n}}D_{KL}\big(q_{\theta}\otimes q_{Z^n}\,\big\|\,\pi_n\big)$, where $q_{Z^n}$ has been profiled out since the minimizing over $q_{Z^n}$ admits a closed form solution as in Step 1 of MF-WGF.

Similar to Bayesian models without latent variables, in Section~\ref{sec: main results} we show that the MF solution $(\wht q_\theta,\, \wht q_{Z^n})$ in~\eqref{eqn: mean_field_optimization_problem} is a unique fixed point of the corresponding iterative procedure in a constant radius $W_2$-neighhorhood around the solution, and the iterative procedure has exponential convergence to this solution given it is initialized in this neighborhood. It is straightforward to show that $\wht q_\theta$, as a fixed point to MF-WGF, satisfies
\begin{align}\label{eqn: distributional_equation}
    \mu(\theta) = \frac{1}{Z_n(\mu)}\,\pi_\theta(\theta)\, e^{-n\,U_n(\theta;\, \mu)}, \quad\mx{with } Z_n(\mu) = \int_\Theta\pi_\theta(\theta)\, e^{-n\,U_n(\theta;\, \mu)}\,\dd\theta,
\end{align}
which can be proved by applying the first order optimality condition to~\eqref{eqn: JKO_update_qtheta} in terms of the first variation (see Appendix~\ref{sec:fix_point_proof} for further details). This fixed point equation is helpful to show the concentration of MF approximation $\wht q_\theta$ towards the true parameter $\theta^\ast$ (c.f.~Theorem~\ref{thm: posterior_convergence_rate}, also see Section~\ref{sec:Analysis_MF} for a sketched proof). Heuristically, when $n$ is large, $\wht q_\theta$ is expected to concentrate around the point mass measure $\delta_{\theta^\ast}$ at $\theta^\ast$, so that we can roughly approximate $\wht q_\theta$ by the right hand side of~\eqref{eqn: distributional_equation} with $\mu$ being replaced by $\delta_{\theta^\ast}$; then the convergence follows by the fact that $\theta^\ast$ approximately minimizes the potential $U_n(\theta;\,\delta_{\theta^\ast})$.


\section{Theoretical Results}\label{sec: main results}
In this subsection, we present two main theoretical results of this work: concentration of the MF approximation $\wht q_\theta$ to the true parameter $\theta^\ast$, and the convergence of the proposed MF-WGF algorithm. In the next section, we will apply the theoretical results to three representative examples by verifying the assumptions. All proofs are deferred to the Appendices in the supplement of the paper.

\subsection{Analysis of mean-field approximation}\label{sec:Analysis_MF}
We adopt the frequentist perspective by assuming that data $X^n$ are generated from a data generating model indexed by a true parameter $\theta^\ast$. Before presenting the formal result, we make the following assumptions, most are standard for proving concentration of Bayesian posteriors~\cite{ghosal2000convergence,shen2001rates} and their MF counterpart~\cite{pati2018statistical,yang2020alpha,alquier2020concentration,zhang2020convergence}.

\begin{assumption}[test condition]\label{assump: test condition}
For some constants $c_1, c_2 > 0$ and any $\varepsilon > c_1\sqrt{\log n/n}$, there is a test function $\phi_n$, such that
\begin{align*}
\mb E_{\theta^\ast}[\phi_n] \leq e^{-c_2 n\varepsilon^2},\quad \sup_{\theta:\exists j\in[m], s.t.\|\theta_j-\theta_j^\ast\| > \varepsilon} \mb E_\theta[1 - \phi_n] \leq e^{-c_2n\varepsilon^2}.
\end{align*}
\end{assumption}

\smallskip
\noindent In a typical parametric setting, the existence of such a test can be proved by decomposing $\{\theta_j: \|\theta_j-\theta_j^\ast\| > \varepsilon\}$ into a countable union of annuluses. Each annuluses can be covered by a finite number of balls, within each ball the likelihood ratio type test can be employed \cite{birge1979theoreme,le2012asymptotic}. When $m=1$ in Bayesian latent variable models, this is just the standard test condition discussed in \cite[][Section 7]{ghosal2000convergence}.

\begin{assumption}[prior thickness]\label{assump: prior condition}
There is a measure $\widetilde Q = \otimes_{j=1}^m\widetilde Q_j$, subsets $\widetilde \Theta_j\subset \Theta_j$ for $j\in[m]$, and positive constants $c_3$ and $c_4$, such that for any $\theta\in\widetilde\Theta \coloneqq \otimes_{j=1}^m\widetilde \Theta_j$ we have
\begin{align*}
&\,\KL\big(p(\cdot\,|\,\theta^\ast)\,\|\, p(\cdot\,|\,\theta)\big) \leq c_4 \varepsilon_n^2, 
\quad \int_{\m X}\Big(\log\frac{p(x\,|\,\theta^\ast)}{p(x\,|\,\theta)}\Big)^2 p(x\,|\,\theta^\ast)\,\dd x \leq c_4\varepsilon_n^2,\\
&\, \log\frac{\dd\widetilde Q}{\dd \Pi_\theta}\leq c_4n\varepsilon_n^2 \quad \mbox{and} \quad \log\widetilde Q(\widetilde \Theta) = \sum_{j=1}^m\log\widetilde Q_j(\widetilde \Theta_j) \geq - c_3 n\varepsilon_n^2,
\end{align*}
where $\Pi_\theta$ denotes the prior distribution, and $\varepsilon_n = M\sqrt{\log n/n}$ for some $M > 1$.
\end{assumption}

\smallskip
\noindent The case of $m=1$ reduces to the Bayesian posterior without MF approximation. Under $m=1$, this assumption is implied by the standard prior thickness assumption~\cite{ghosal2000convergence} by taking $\widetilde Q = \Pi_\theta$. When $m > 1$, this assumption requires the existence of a fully factorized probability measure $\widetilde Q$ in the MF family that is close to the prior distribution and puts enough mass around the ground truth $\theta^\ast$.

\vspace{0.5em}
For Bayesian latent variable models, we need an additional assumption on the conditional likelihood function of latent variable $Z$ given $\theta$ and observation $X$.

\begin{assumption}[local bound of KL divergence]\label{assump: qudratic growth of KL}
The marginal distribution $p(x\,|\,\theta)$ of observation $X$ under $\theta$ and the conditional distribution $p(z\,|\,x,\theta)$ of latent variable $Z$ given $X=x$ and $\theta$ satisfy
\begin{align}    
D_{KL}\big(p(\cdot\,|\,x,\theta^\ast)\,\|\,p(\cdot\,|\,x,\theta)\big) &\leq G(x)\,\varepsilon_n^2,\quad\forall\,\theta\in\widetilde\Theta,
\end{align}
where $\widetilde \Theta$ is the local neighborhood of $\theta^\ast$ defined in Assumption~\ref{assump: prior condition}, and $G(X)$ is a sub-exponential random variable with parameters $\sigma_4$ under $p(\cdot\,|\,\theta^\ast)$, i.e. $\mb E_{\theta^\ast}\big[\exp\big\{\sigma_4^{-1}|G(X)|\big\}\big]\leq 2$. 
\end{assumption}

\smallskip
\noindent This assumption is a mild condition, which is implied by a quadratic growth of KL divergence for $\theta\in\widetilde\Theta$, i.e.~$\KL(p(\cdot\,|\,x, \theta^\ast)\,\|\,p(x, \theta)) \lesssim G(X)\|\theta-\theta^\ast\|^2$. This quadratic growth property holds if the logarithms of both distributions (density or mass function) are twice differentiable with controlled Hessians. For simple presentation, we adopt the assumption that $G(X)$ is sub-exponential to derive a high probability upper bound of $n^{-1}\sum_{i=1}^n G(X_i)$. This sub-exponential assumption on $G(X)$ can be generalized to $G(X)$ having a finite Orlicz-norm. See Appendix~\ref{Appendix: conc} for the definition of the Orlicz norm of a random variable and further details.

\smallskip
\noindent {\bf Bayesian models without latent variables.}
Recall that $\widehat{q}_\theta$ is the solution of the mean-field optimization problem~\eqref{eqn: blockwise_optimize}, which should satisfy the following optimality condition (see Lemma~\ref{lem: functional_Wn} in Appendix~\ref{app: concentration_MFVI}),
\begin{align}\label{eqn: MFVI_FOC_again}
\wht q_j(\theta_j) = \frac{\exp\big\{\int_{\Theta_{-j}}\log\pi_\theta(\theta) + \sum_{k=1}^n\log p(X_k\,|\,\theta)\,\dd\wht q_{-j}(\theta_{-j})\big\}}{\int_{\Theta_j}\exp\big\{\int_{\Theta_{-j}}\log\pi_\theta(\theta) + \sum_{k=1}^n\log p(X_k\,|\,\theta)\,\dd\wht q_{-j}(\theta_{-j})\big\}\,\dd\theta_j}, \quad\mx{for }j\in[m].
\end{align}
A major challenge in our analysis of the MF solution $\wht q_\theta$ is how to deal with the normalization constant (denominator) in the preceding display, which depends on $\wht q_\theta$ and complicates the analysis. To address this issue, we rewrite equation~\eqref{eqn: MFVI_FOC_again} by adding a $\theta_j$-independent term to both nominator and denominator, 
\begin{align*}
\wht q_j(\theta_j) = \frac{\exp\big\{\int_{\Theta_{-j}}\log\frac{\pi_\theta(\theta)}{\widetilde Q_j(\theta_j)\wht q_{-j}(\theta_{-j})} + \sum_{i=1}^n\log\frac{p(X_i\,|\,\theta)}{p(X_i\,|\,\theta^\ast)}\,\dd\wht q_{-j}\big\}\widetilde Q_j(\theta_j)}{\int_{\Theta_j}\exp\big\{\int_{\Theta_{-j}}\log\frac{\pi_\theta(\theta)}{\widetilde Q_j(\theta_j)\wht q_{-j}(\theta_{-j})} + \sum_{i=1}^n\log\frac{p(X_i\,|\,\theta)}{p(X_i\,|\,\theta^\ast)}\,\dd\wht q_{-j}\big\}\,\dd\widetilde Q_j}.
\end{align*}
To prove the concentration of $\wht q_j$ around the parameter $\theta_j^\ast$, we can proceed as the usual steps for proving posterior concentration (e.g.~\cite{ghosal2000convergence,shen2001rates}) by proving an upper and a lower bound to the numerator and the denominator respectively. Applying a union bound then yields the concentration of $\wht q_\theta$ around $\theta^\ast$ due to the factorization structure of $\wht q_\theta$.

For the lower bound to the denominator in the preceding display, denoted as $D_j$, we utilize the following equivalent expression,
\begin{align}
   & \log D_j = -\widetilde W_n(\wht q_\theta) = - \min_{q_\theta=\otimes_{j=1}^m q_j}\widetilde  W_n(q_\theta), \quad\mx{with }\\
   \widetilde W_n(q_\theta) = \int_\Theta &\sum_{i=1}^n\log\frac{p(X_i\,|\,\theta^\ast)}{p(X_i\,|\,\theta)}\,\dd q_1(\theta_1)\cdots\dd q_m(\theta_m) + \KL(q_1\otimes\cdots\otimes q_m\,\|\,\pi_\theta).\notag
\end{align}
Here, functional $\widetilde W_n$ is, up to a $q_\theta$-independent constant, the same as the objective functional $q_\theta = \bigotimes_{j=1}^m q_j \mapsto D_{KL}\big(q_{\theta}\,\big\|\,\pi_n\big)$. Thus, $\wht q_\theta$ minimizes $\widetilde W_n$.
Since $\wht q_\theta$ is expected to be concentrated around $\theta^\ast$, we may use $-\widetilde W_n(q_\theta)$ with some carefully constructed $q_\theta$ (the $\widetilde Q$ from Assumption~\ref{assump: prior condition}) suitably concentrated around $\theta^\ast$ (e.g.~a uniform distribution supported on a small neighborhood around $\theta^\ast$) for providing a lower bound to $\log D_j$.

For the upper bound to the numerator, the first term can be controlled by directly applying Jensen's inequality; the second term $\sum_{i=1}^n\log\frac{p(X_i\,|\,\theta)}{p(X_i\,|\,\theta^\ast)}$, which is roughly negative $n$ times $\KL\big[\,p(\cdot\,|\,\theta^\ast)\,\big\|\,p(\cdot\,|\,\theta)\big]$ since $\{X_i\}_{i=1}^n$ are marginally i.i.d.~from $p(\cdot\,|\,\theta^\ast)$ under the frequentist perspective. To formally bound the numerator, or more precisely, the integral of numerator over set $\Theta_\varepsilon=\{\|\theta-\theta^\ast\|\geq \varepsilon\}$ for suitably large $\varepsilon>0$, we use the commonly adopted test condition (i.e.,~Assumption~\ref{assump: test condition}) for uniformly controlling the log-likelihood ratio process over $\Theta_\varepsilon$.

The following theorem shows the concentration of $\widehat{q}_\theta$ by characterizing the tail probability of being away from the true parameter $\theta^\ast$. Recall that we are adopting a frequentist perspective, where the randomness in all high probability bound is coming from the randomness in the samples $X_1, \cdots, X_n$ that are generated under a true parameter $\theta^\ast$.

\begin{theorem}[Exponential posterior concentration without latent variables]\label{thm: MFVI_posterior_convergence_rate}
Under Assumptions~\ref{assump: test condition} and~\ref{assump: prior condition}, if the sample size satisfies $c_4n^{c_2 M^2} \geq 3m M^2\log n$, then for any $M\geq 1$, the MF variational approximation $\wht Q_\theta$ to the posterior distribution of $\theta$ satisfies the following with probability at least $1 - \frac{2c_4}{n\varepsilon_n^2} = 1 - \frac{2c_4}{M^2\log n}$,
\begin{align}
&\wht{Q}_\theta\big(\exists\,j\in[m]\,\,s.t.\,\,\|\theta_j - \theta_j^\ast\| > \varepsilon\big) \leq e^{-c_2n\varepsilon^2/2},\notag\\
&\qquad\qquad\qquad \mx{for all}\quad\varepsilon > M\Big(3 + c_1 + \frac{2c_4 + c_3 + 1}{c_2}\Big)\sqrt{\frac{\log n}{n}}. \label{eqn: constraint_epsilon}
\end{align}
\end{theorem}

\smallskip
\noindent {\bf Bayesian latent variable models.}
Let $\m Z =\{1, 2, \ldots ,K\}$ be the support of each discrete latent variable.
In this case, $\widehat{q}_\theta$ is the solution of the mean-field optimization problem~\eqref{eqn: mean_field_optimization_problem}, satisfying (see Lemma~\ref{lem:functional_Vn} in Appendix~\ref{app:proof_posterior_con} and Appendix~\ref{sec:fix_point_proof}),
\begin{align}\label{MF_solution_form}
    \widehat{q}_\theta(\theta)  = \frac{1}{\wht Z_n}\,\pi_\theta(\theta)\,e^{-n\,U_n(\theta,\, \widehat{q}_\theta)}, \quad\mx{with }\wht Z_n = \int_\Theta\pi_\theta(\theta)\,e^{-n\,U_n(\theta,\, \widehat{q}_\theta)}\,\dd\theta,
\end{align}
where the (sample) potential function $U_n:\,\Theta\times \ms P^r(\Theta) \to \mb R$ is
\begin{align}
    U_n(\theta,\,q_\theta) &= -\frac{1}{n}\sum_{i=1}^n\sum_{z=1}^K\log p(X_i, z\,|\,\theta)\,\Phi(q_\theta,X_i)(z), \quad\mx{where} \label{eq:U_n_exp}\\
     \Phi(\widehat{q}_\theta,&\, X_i)(z):\,=\widehat{q}_{Z_i}(z) = \frac{\exp\big\{\mb E_{\widehat{q}_\theta}\big[\log p(X_i, z\,|\,\theta)\big]\big\}}{\sum_{k=1}^K\exp\big\{\mb E_{\widehat{q}_\theta}\big[\log p(X_i, k\,|\,\theta)\big]\big\}},\quad z\in[K]. \notag
\end{align}

Similar to the analysis of Bayesian models without latent variables, we rewrite equation~\eqref{MF_solution_form} by adding a $\theta$-independent term to both nominator and denominator, 
\begin{align*}
    \widehat{q}_\theta(\theta)= \frac{\exp\big\{-\sum_{i=1}^n\sum_{z=1}^K\Phi(\widehat{q}_\theta, X_i)(z)\log\frac{\Phi(\widehat{q}_\theta, X_i)(z)}{p(z\,|\,X_i,\theta)} - \sum_{i=1}^n\log\frac{p(X_i\,|\,\theta^\ast)}{p(X_i\,|\,\theta)}\big\}\,\pi_\theta(\theta)}{\int_{\Theta}\exp\big\{-\sum_{i=1}^n\sum_{z=1}^K\Phi(\widehat{q}_\theta, X_i)(z)\log\frac{\Phi(\widehat{q}_\theta, X_i)(z)}{p(z\,|\,X_i,\theta)} - \sum_{i=1}^n\log\frac{p(X_i\,|\,\theta^\ast)}{p(X_i\,|\,\theta)}\big\}\,\dd\pi_\theta(\theta)}.
\end{align*}
For the lower bound to the denominator, denoted as $D_n$, we can use the following equivalent expression, which is more convenient to analyze,
\begin{align}
   & \log D_n = - W_n(\wht q_\theta) = - \min_{q_\theta} W_n(q_\theta), \quad\mx{with } W_n(q_\theta) = \label{eqn:profile_KL}\\
     \int_\Theta \bigg\{\sum_{i=1}^n  &\, \sum_{z=1}^K  \Phi(q_\theta, X_i)(z) \log\frac{\Phi(q_\theta, X_i)(z)}{p(X_i,z\,|\,\theta)} + \sum_{i=1}^n\log\frac{p(X_i\,|\,\theta^\ast)}{p(X_i\,|\,\theta)}\bigg\}\,\dd q_\theta(\theta) + \KL(q_\theta\,\|\,\pi_\theta).\notag
\end{align}
Here, functional $W_n$ is, up to a $q_\theta$-independent constant, the same as the (profile) objective functional $q_\theta\mapsto \min_{q_{Z^n}}D_{KL}\big(q_{\theta}\otimes q_{Z^n}\,\big\|\,\pi_n\big)$ after $q_{Z^n}$ being maxed out or replaced by $\Phi(q_\theta, X_i)$; so $\wht q_\theta$ minimizes $W_n$.
Again, we may use $-W_n(q_\theta)$ with some carefully constructed $q_\theta$ suitably concentrated around $\theta^\ast$ (e.g.~a uniform distribution supported on a small neighborhood around $\theta^\ast$) for providing a lower bound to $\log D_n$.

For the upper bound to the numerator, the second term can be treated in the same way as in Bayesian models without latent variables. For the first term in the exponent, just note that $$-\sum_{z=1}^K\Phi(\widehat{q}_\theta, X_i)(z)\log\frac{\Phi(\widehat{q}_\theta, X_i)(z)}{p(z\,|\,X_i,\theta)}$$
is the negative KL divergence between two discrete measures, and therefore is non-positive. 

The following theorem shows the concentration of $\widehat{q}_\theta$ by characterizing the tail probability of being away from the true parameter $\theta^\ast$ in Bayesian latent variable models. 

\begin{theorem}[Exponential posterior concentration with latent variables]\label{thm: posterior_convergence_rate}
Under Assumptions \ref{assump: test condition}, \ref{assump: prior condition}, and \ref{assump: qudratic growth of KL}, if the sample size satisfies $6M\log n\leq \min\{n^{c_2M}, e^{4n\sigma_4^{-1}}\}$, then for any $M\geq 1$, the MF variational approximation $\wht Q_\theta$ to the marginal posterior of $\theta$ satisfies the following with probability at least $1-\frac{2c_4}{n\varepsilon_n^2} = 1 - \frac{2c_4}{M^2\log n}$,
\begin{align}
    &\wht Q_\theta \big(\|\theta - \theta^\ast\| > \varepsilon\big) \leq e^{-{c_2n\varepsilon^2}/2}, \label{eqn: concentration_target_measure}\\
 &\qquad\qquad \mx{for all}\quad \varepsilon \geq M\Big(3+c_1 + \frac{\mb E[G(X)] + c_3 + c_4 + 2}{c_2}\Big)\,\sqrt{\frac{\log n}{n}}.\label{cond: r_in_posterior_convergence_rate}
\end{align}
\end{theorem}


Concentration properties of variational inference have been studied in the recent literature under different criteria. \cite{alquier2020concentration} and~\cite{ yang2020alpha} consider a variant of the usual variational inference, called the $\alpha$-variational inference,
obtained by raising the likelihood to a fractional power $\alpha\in(0,1]$ to facilitate the theoretical analysis. They prove upper bounds for the variational Bayes risk, defined as the expected R\'{e}nyi divergence with respect to their $\alpha$-fractional variational posterior. When $\alpha$ is strictly small than one, they only need a prior concentration assumption (similar to our Assumption~\ref{assump: prior condition}); however, under some mild conditions their risk function behaves like the second moment $\mb E_{\wht Q_{\theta}}\big[\|\theta-\theta^\ast\|^2\big]$, which is much weaker than our sub-Gaussian type tail result. 
For the usual variational inference (or $\alpha$-variational inference with $\alpha=1$), \cite{pati2018statistical,yang2020alpha,zhang2020convergence} proves high probability upper bounds to some similar variational Bayes risks that scales as $\mb E_{\wht Q_{\theta}}\big[\|\theta-\theta^\ast\|^2\big]$ under similar test conditions (as our Assumption~\ref{assump: test condition}) and a stronger version of the prior concentration assumption. Their proofs avoid assuming the compactness of parameter space by considering a sequence of sieve sets. \cite{han2019statistical} proves a similar sub-Gaussian concentration result as ours; their proof is based on a perturbation analysis specifically tailored to the MF approximation and does not seem easily generalizable to other variatioal or hybrid schemes.


Our proof technique is very different from existing proofs of the variational posterior concentration in the literature, most of which are based on applying the varitional characterization of KL divergence, $\KL(p\,||\,q) = \sup_{h}\big\{\int h p -\log(\int e^h q)\big\}$. We instead view the variational posterior $\wht Q_\theta$ as a point in the Wasserstein space $\mb W_2(\mb R^d)$ that minimizes a KL divergence functional, and uses its first order optimality condition (or equivalently, its stationarity to the time-discrete WGF) to show the concentration via subdiffential calculus in $\mb W_2(\mb R^d)$. This general perspective might be useful in extending the developed proof technique to other approximation scheme beyond MF, such as many recent generative model based variational inference procedures. 
Our obtained rate of convergence is also nearly optimal in the parametric setting; and the assumptions we made needed are standard in Bayesian asymptotics literature, and appears to be among the weakest in the context of variational inference. In addition, our proof techniques can be straightforwardly extended to non-parametric settings where the optimal rate of convergence is slower than root-$n$.


\subsection{Analysis of MF-WGF algorithm}\label{sec:MF-WGF_analysis}
We separately analyze the convergence of MF-WGF for Bayesian models with/without latent variables.

\smallskip
\noindent {\bf Bayesian models without latent variables.}
Recall that in the $k$-th iteration, the MF-WGF algorithm updates the joint variational distribution $\bigotimes_{j=1}^m q_j^{(k)}$ into $\bigotimes_{j=1}^m q_j^{(k+1)}$ with
\begin{align*}
q_j^{(k+1)} = \argmin_{q_j} n\,\mb E_{q_j\otimes q_{-j}^{(k)}} \big[U_n\big] + \KL\big(q_j\otimes q_{-j}^{(k)}\,\big\|\,\pi_\theta\big) + \frac{1}{2\tau}W_2^2(q_j, q_j^{(k)}), \quad j\in[m],
\end{align*}
where $U_n(\theta) = -\frac{1}{n}\sum_{i=1}^n \log p(X_i\,|\,\theta)$ is the sample potential function. The corresponding population-level potential function is
\begin{align*}
U(\theta) = -\int_{\mb R^d}\log p(x\,|\,\theta)\,p(\dd x\,|\,\theta^\ast).
\end{align*}
We expect that $U_n$ and $U$ are uniformly close enough when the sample size $n$ is sufficiently large (e.g.~Theorem 1 in \cite{mei2018landscape}).

Before formally presenting our result, we begin by introducing several assumptions that are commonly used to prove exponential convergence of iterative algorithms in optimization literature.

\stepcounter{assumptionA}
\begin{assumption}[strong convexity of population-level potential]\label{assump: MFVI_SC}
There exists $\lambda > 0$ such that $U$ is $\lambda$-strongly convex, i.e.
\begin{align*}
U\big((1-t)\theta + t\theta'\big) \leq (1-t)\,U(\theta) + t\,U(\theta') - \frac{\lambda}{2}\,t(1-t)\|\theta - \theta'\|^2
\end{align*}
for all $t\in[0, 1]$ and $\theta, \theta'\in\Theta$. Moreover, the parameter space $\Theta = \bigotimes_{j=1}^m\Theta_j\subset\mb R^d$ is convex and contained a ball centered at the origin with radius $R$. 
\end{assumption}
\noindent This strong convexity assumption guarantees that $\KL(\cdot\,\|\,\pi_n)$ is strongly convex along generalized geodesics on $\ms P_2^r(\Theta_1)\times\cdots\times \ms P_2^r(\Theta_m)\subset\ms P_2^r(\Theta)$. 
It is possible to relax this global strong convexity to a local strong convexity  within a small but constant-radius neighborhood around the true parameter $\theta^\ast$, which is always true for regular models with non-singular Fisher information matrix. One simple strategy is to assume that both the prior and the initialization distribution of the algorithm are supported within this neighborhood. In practice, one can construct this initialization distribution by identifying a reasonably good initial point estimate of $\theta^\ast$, for example, using simple and fast methods such as the method of moments; and also modify the prior by restricting it onto a constant neighborhood around the estimate. Due to the flexibility in selecting the prior distribution for MF-WGF, such a modification will have a minimal impact on the implementation. A second technical strategy is to further impose a dissipative condition (see~\cite{raginsky2017non} for definition). A dissipative condition is commonly made to guarantee the long term stability of sampling algorithms such as Langevin dynamics~\cite{raginsky2017non} as it causes most probability mass absorbed into a constant neighborhood of $\theta^\ast$ after a number of iterations; then the behavior of the algorithm inside this neighborhood is driven by the local convexity of the potential. Since Theorem~\ref{thm: MFVI_posterior_convergence_rate} tells that there is at most $O\big(e^{-cn}\big)$ probability mass of $\wht q_\theta$ outside this neighborhood, we expect our current analysis to be valid up to an extra  $O\big(e^{-cn}\big)$ remainder term.
Due to the significant complexity of our current proof with global convexity, we will leave a systematic study on such an extension in a separate work.

\begin{assumption}[smoothness of population-level potential]\label{assump: L-smooth}
There exists $L > 0$ such that $U$ is $L$-smooth, which is defined by
\begin{align}\label{eqn: L-smooth}
U\big((1-t)\theta + t\theta'\big) \geq (1-t)\,U(\theta) + t\,U(\theta') - \frac{L}{2}\,t(1-t)\|\theta - \theta'\|^2
\end{align}
for all $t\in[0, 1]$ and $\theta, \theta'\in\Theta$.
\end{assumption}
\noindent A smoothness condition on the objective function is usually necessary to prove exponential convergence of a coordinate descent-type optimization algorithms~\cite{wright2015coordinate, wright2022optimization}. When $U$ is twice differentiable, the above Assumption~\ref{assump: L-smooth} is equivalent to $\matnorm{\nabla^2 U(\theta)}\leq L$ for all $\theta\in\Theta$. In our proof, we can slightly relax this condition since we only require $\|\nabla_j U(\theta_j, \theta_{-j}) - \nabla_j U(\theta_j, \theta_{-j}')\| \leq L\|\theta_{-j} - \theta_{-j}'\|$, where $\nabla_j$ is the gradient with respect to the $j$th component of $U$. This inequality is equivalent to $\matnorm{\nabla_j\nabla_{-j} U(\theta)}\leq L$ when $U$ is twice differentiable and is weaker than Assumption~\ref{assump: L-smooth}. 

\smallskip
To show that the sample-level potential $U_n$ is uniformly close to its population version $U$ and inherits the convexity and the smoothness of $U$ (see the proof of Theorem~\ref{thm: para_update_whp} in Appendix~\ref{app: mainthm_nolatent}), we need the following assumption which characterizes the continuity and the sub-exponential tail of the (higher-order) derivatives of the log-likelihood functions. 
\begin{assumption}[regularity of log-likelihood function]\label{assump: MFVI_reg}
The log-likelihood function $\log p(x\,|\,\theta)$ is twice differentiable with respect to $\theta\in\Theta$. Let $X$ denote a sample generated from the true distribution $p(\cdot\,|\,\theta^\ast)$. Then, the following regularity assumptions hold.
\begin{enumerate}
\item[1.] The Lipschitz constant (relative to the matrix operator norm) of the log-likelihood Hessian
\begin{align*}
    J(X) := \sup_{\theta\neq\theta'} \frac{\matnorm{\nabla^2\log p(X\,|\,\theta) - \nabla^2\log p(X\,|\,\theta')}}{\|\theta-\theta'\|}
\end{align*}
satisfies $\mb E_{\theta^\ast}[J(X)] < J_\ast$ for some finite $J_\ast$.
\item[2.]
For any $v\in B_{\mb R^d}(0, 1)$ and $\theta\in\Theta$, $\big\langle v, \nabla^2\log p(X\,|\,\theta)v\big\rangle$ is sub-exponential with parameter $\sigma_5$. In particular, a sufficient condition for this to be true is $\matnorm{\nabla^2\log p(X\,|\,\theta)}$ being sub-exponential for any $\theta\in \Theta$.
\end{enumerate}
\end{assumption}
\noindent The first part of this assumption can be checked by controlling the third order derivatives of $\log p(X\,|\,\theta)$ with respect to $\theta$ when it is sufficiently smooth; the second part can be verified by directly calculating the first order Orlicz norm of $v^T\nabla^2\log p(X\,|\,\theta)\,v$, and can also be extended to a bounded $\psi_\alpha$ (Orlicz) norm for some $\alpha > 0$.

\begin{theorem}[MF-WGF without latent variables]\label{thm: para_update_whp}
Suppose Assumptions~\ref{assump: MFVI_SC}--\ref{assump: MFVI_reg} hold, and $\log\pi_\theta$ is twice differentiable. Then there is a universal constant $C>0$, such that for any fixed $\eta\in(0, 1)$, the following inequality holds with probability at least $1 - \eta$,
\begin{align*}
    W_2^2(q^{(k)}, \wht q) \leq \big(1 + 2\tau\lambda_{lb} -  L_{ub}^2\tau^2m\big)^{-k}W_2^2(q^{(0)}, \wht q),\quad k \geq 1,
\end{align*}
when $n \geq Cd\log d\cdot\max\big\{\log J_\ast / \log d, \log(R\sigma_5/\eta), 1\big\}$ and the step size $\tau$ satisfies
\begin{align*}
    \tau^{-1} \geq \sqrt{m}L_{ub}, \quad\mx{and}\quad 1 + 2\tau\lambda_{lb} -  L_{ub}^2\tau^2m \geq 0,
\end{align*}
where
\begin{align*}
\lambda_{lb} &\coloneqq n{\lambda} - \lambda_{\max}(\nabla^2\log\pi_\theta) - \sigma_5^2\sqrt{\frac{Cd\log n}{n}\cdot\max\Big\{\frac{\log J_\ast}{\log d}, \log\frac{R\sigma_5}{\eta}, 1\Big\}},\\
L_{ub} &\coloneqq n L - \lambda_{\min}(\nabla^2\log\pi_\theta) + \sigma_5^2\sqrt{\frac{Cd\log n}{n}\cdot\max\Big\{\frac{\log J_\ast}{\log d}, \log\frac{R\sigma_5}{\eta}, 1\Big\}},
\end{align*}
and $\lambda_{\max}(\nabla^2\log\pi_\theta)$ and $\lambda_{\min}(\nabla^2\log\pi_\theta)$ are the largest and the smallest eigenvalues of the Hessian matrix $\nabla^2\log\pi_\theta$ in $\Theta$ respectively.
In particular, if we take $\tau = \lambda_{lb}/( L_{ub}^2m)$, then
\begin{align}\label{eqn: MFVI_convergence_rate}
    W_2^2(q^{(k)}, \wht q) \leq \Big(1 + \frac{\lambda_{lb}^2}{ L_{ub}^2m}\Big)^{-k}W_2^2(q^{(0)}, \wht q), \quad k \geq 1.
\end{align}
\end{theorem}

\noindent The proof of this theorem does not require utilizing the concentration property on $\wht q_\theta$ as stated in Theorem~\ref{thm: MFVI_posterior_convergence_rate}, and the exponential convergence is solely driven by the convexity of population level potential $U$. However, when an effective potential $U_n$ varies across iterations, which is the case in MF-WGF for Bayesian latent variable models, the concentration property becomes essential to manage the fluctuation of $U_n$.

 Equation~\eqref{eqn: MFVI_convergence_rate} implies that $O\big(\frac{m L_{ub}^2}{\lambda_{lb}^2}\log\big(\frac{1}{\varepsilon}\big)\big) = O\big(\frac{m L^2}{\lambda^2}\log\big(\frac{1}{\varepsilon}\big)\big)$ iterations are sufficient for the algorithm to achieve an accuracy of $\varepsilon\in(0,1)$ in computing $\wht q$. This iteration complexity matches a typical iteration complexity of coordinate gradient descent for minimizing a strongly convex and smooth function in the Euclidean space  when taking the step size in the order of $O\big(\frac{\lambda}{L^2 m}\big)$ as we do (see, e.g.,~Theorem 6.3 of~\cite{wright2022optimization}). However, analyzing the coordinate proximal gradient descent in a product Wasserstein space presents some unique challenges.

In the Euclidean space, $\|(x^{(k+1)} - x^{(k)}) - (\wht x - x^{(k)})\| = \|x^{(k+1)} - \wht x\|$ holds for any arbitrary points $\wht x$, $x^{(k)}$, and $x^{(k+1)}$. However, we only have $\big\|T_{q_j^{(k)}}^{q_j^{(k+1)}} - T_{q_j^{(k)}}^{\wht q_j}\big\|^2_{L^2(q_j^{(k)};\Theta_j)} \geq W_2^2(q_j^{(k+1)}, \wht q_j)$ due to the positive curvature of the Wasserstein space. This difference indicates that we have to evaluate the change of KL divergence along the generalized geodesics (see Appendix~\ref{app:convexity} for a definition) connecting $q_j^{(k+1)}$ and $\wht q_j$ with the base measure $q_j^{(k)}$, rather than the geodesics connecting $q_j^{(k+1)}$ and $\wht q_j$.

In the proof, the convexity of the potential function $U(\theta)$ plays two roles: (1) allowing us to apply Theorem~\ref{Lem:stong_conx_generalized} to control the subdifferential; 
 (2) deriving a quadratic growth property of the KL-divergence functional, i.e.
\[
\KL(q_1\otimes \cdots\otimes q_m\,\|\,\pi_n) - \KL(\wht q_1\otimes\cdots\otimes \wht q_m\,\|\,\pi_n) \gtrsim W_2^2(q, \wht q_\theta), \quad \forall\,q_j\in\ms P_2^r(\Theta_j),\,\,j\in[m].
\]
$L$-smoothness of $U(\theta)$ helps guarantee that the KL divergence between $q^{(k)}$ and $\pi_n$ is decreasing when the step size is small. The detailed proof is postponed to Appendix~\ref{app: mainthm_nolatent}.

\medskip
\noindent {\bf Bayesian latent variable models.}
Recall that with the presence of latent variables, the MF-WGF algorithm can be summarized by the following iterative updating rule: for $k=0,1,\ldots$,
\begin{equation*}
        q_\theta^{(k+1)} = \argmin_{q_\theta}V_n(q_\theta\,|\,q_\theta^{(k)}) + \frac{1}{2\tau}W_2^2(q_\theta,\, q_\theta^{(k)}),
\end{equation*}
where for any $q_\theta'\in\ms P(\theta)$, the (sample) energy (or KL divergence) functional $V_n(\cdot\,|\,q_\theta')$ is defined as
\begin{align*}
    V _n(q_\theta\,|\,q_\theta') :\,= n \,\mb E_{q_\theta} \big[U_n(\theta;\, q'_\theta) \big]+ \KL(q_\theta\,||\,\pi_\theta),
\end{align*}
and $U_n(\cdot\,;\,q'_\theta)$ is the (sample) potential function given in~\eqref{eq:U_n_exp}. The corresponding population version of the potential is
\begin{equation}\label{eqn:pop_U}
    U(\theta;\, q_\theta') = -\int_{\mb R^d}\bigg\{\sum_{z=1}^K\log p(x,z\,|\,\theta)\,\Phi(q_\theta', x)(z)\bigg\}\,p(\dd x\,|\,\theta^\ast).
\end{equation}

The main difficulty in analyzing this MF-WGF algorithm is that the energy functional $V_n(\,\cdot\,|\,q_\theta^{(k)})$ determining $q_\theta^{(k+1)}$ also depends on the previous iterate $q_\theta^{(k)}$. With a time-independent energy functional, whose global minimizer denoted as $\pi^\ast$, we may directly apply Theorem~\ref{thm: implicit_WGF_conv} with $\pi = \pi^\ast$ to prove the contraction of the one-step discrete WGF towards $\pi^\ast$. However, by directly applying Theorem~\ref{thm: implicit_WGF_conv} with $\pi$ therein being the minimizer of $V_n(\,\cdot\,|\,q_\theta^{(k)})$, we can only prove the one-step contraction of MF-WGF towards this minimizer, which changes over iteration count $k$ and is generally different from the target $\wht q_\theta$.

To overcome this difficulty in the convergence analysis, we may introduce an accompanied population-level MF-WGF, defined as
\begin{align}\label{eqn:pop_MF-WGF}
     \wt q_\theta^{(k+1)} = \argmin_{q_\theta}V(q_\theta\,|\,\delta_{\theta^\ast}) + \frac{1}{2\tau}W_2^2(q_\theta,\, \wt q_\theta^{(k)}),
\end{align}
obtained by replacing $q_\theta^{(k)}$ in $V_n(\cdot\,|\,q_\theta^{(k)})$ by the point mass measure $\delta_{\theta^\ast}$ at $\theta^\ast$, and the sample energy functional by its population counterpart
\begin{align}
     V(q_\theta\,|\,q_\theta') :\,= n \,\mb E_{q_\theta} \big[U(\theta;\, q'_\theta) \big]+ \KL(q_\theta\,||\,\pi_\theta).
\end{align}
Since the energy function $V(\cdot\,|\,\delta_{\theta^\ast})$ in the population-level MF-WGF is time-independent, we can apply Theorem~\ref{thm: implicit_WGF_conv} or Corollary~\ref{coro:one_step_KL} to prove its contraction. Moreover,
since $\wht q_\theta$ is expected to be concentrated around $\theta^\ast$, we may expect the trajectory of the sample-level MF-WGF to be close to that of the population-level MF-WGF under the same initialization.

Based on this discussion, a natural strategy to prove the convergence of the MF-WGF algorithm can be divided into two main steps: 1.~control the difference between the sample-level iterates $\big\{q_\theta^{(k)}:\,k\geq 0\big\}$ and the population-level iterates $\big\{\wt q_\theta^{(k)}:\,k\geq 0\big\}$; 2.~analyze the convergence of population-level MF-WGF~\eqref{eqn:pop_MF-WGF}. Our actual proof is slightly different from the above heuristics. Specifically, to simplify the proof, we do not explicitly bound the difference between the sample-level and population-level iterates, but use some population level quantities, such as potential $U$ and energy functional $V$, to substitute their sample versions and properly control the resulting extra error terms in analyzing the sample iterates (see Appendix~\ref{app:proof_main_theorem} for further details). In particular, the freedom of choosing an arbitrary $\pi$ in Theorem~\ref{thm: implicit_WGF_conv} allows us to directly apply the theorem to analyze the sample-level MF-WGF by taking $\pi=\wht q_\theta$; however, some careful perturbation analysis will be required for the proof to go through.

The following assumptions are needed to formally prove the convergence of MF-WGF.

\stepcounter{assumptionA}
\begin{assumption}[strong convexity of population-level potential]\label{assump: strong_convexity}
There exists some constant $r > 0$, such that for any $\mu\in B_{\mb W_2}(\delta_{\theta^\ast}, r) :\,=\big\{\mu\in \PX:\, W_2(\mu,\,\delta_{\theta^\ast})\leq r\big\}$, function $U(\cdot\,;\, \mu):\Theta\to\mb R$ is $\lambda$-strongly convex, i.e.
\begin{equation}
    U\big(\,(1-t)\,\theta + t\,\theta'\,;\, \mu\big) \leq (1-t)\,U(\theta,\mu) + t\,U(\theta',\mu) - \frac{\lambda}{2}\,t(1-t)\,\|\theta - \theta'\|^2
\end{equation}
for all $t\in[0,1]$ and $\theta,\theta'\in\Theta$. Moreover, parameter space $\Theta\subset\mb R^d$ is convex and contained a ball centered at the origin with radius $R$. 
\end{assumption}
\noindent  Notice that this assumption only requires $U(\cdot\,;\,\mu)$ to be strongly convex when $\mu$ is close to $\delta_{\theta^\ast}$. We will verify this assumption for the two applications considered in Section~\ref{sec:thm_app}. When the initialization $q_\theta^{(0)}$ is close enough to the $\delta_{\theta^\ast}$, it can be proved by induction that any later iterates $q^{(k)}_\theta$ will stay in the same $W_2$ neighborhood. Therefore, we do not need $U(\cdot\,;\,\mu)$ to be strongly convex for all $\mu\in\ms P_2^r(\Theta)$. This assumption plays a similar role as Assumption~\ref{assump: MFVI_SC} for models without latent variable. Similarly, we expect that the strong convexity of $U(\theta\,;\,\mu)$ with respect to parameter $\theta$ can be relaxed to a local strong convexity within a neighborhood of $\theta^\ast$ with sufficiently small constant radius for all $\mu\in B_{\mb W_2}(0, r)$.
Here, we simply assume the global strictly convexity of $U$ in the current analysis to avoid these technicalities without affecting the convey of our main proof ideas.  


To show that the sample-level potential $U_n$ is uniformly close to its population version $U$ and inherits the convexity property of $U$ (see Lemma~\ref{lem: ULLN_of_potential_function} in Appendix~\ref{app:proof_main_theorem}), we need the following assumption characterizing continuity and sub-Gaussianity of the (higher-order) derivatives of the log-likelihood functions with the latent variable and the observed data. 

\begin{assumption}[regularity of log-likelihood function]\label{assump: continuity_of_Hessian}
The log-conditional-likelihood function $\log p(z\,|\,x,\theta)$ of the latent variable $Z$ and the log-marginal-likelihood function $\log p(x\,|\,\theta)$ of the observation $X$ are twice differentiable with respect to $\theta$ for all $z\in[K]$ and $x\in \mb R^d$. Let $X$ denote a sample from the true underlying data generating distribution $p\,(\,\cdot\,|\,\theta^\ast)$, then the following properties hold.
\begin{enumerate}
    \item For $i=1, 2$, the random variable $S_i(X) := \sum_{k=1}^K\big\|\nabla\log p(k\,|\,X,\theta^\ast)\big\|_2^i$
has finite expectation; and $S_2(X)$ is sub-exponential with parameter $\sigma_3 < \infty$, i.e. $\mb E\exp\{\sigma_3^{-1}|S_2(X)|\} \leq 2.$

\item If we denote the Lipschitz constant of the log-likelihood Hessian by
\begin{equation*}
J_k(X) := \sup_{\theta\neq\theta'}\frac{\matnorm{\nabla^2\log p(X,k\,|\,\theta) - \nabla^2\log p(X,k\,|\,\theta')}}{\|\theta-\theta'\|},
\end{equation*}
then there exist some finite constant $J_\ast$ such that $\sum_{k=1}^K \mb E_{\theta^\ast}[J_k(X)] \leq J_\ast$.

\item For any $v\in B_{\mb R^d}(0, 1)$ and $\theta\in\Theta$,
$\sum_{k=1}^K p(k\,|\,X,\theta^\ast)\cdot \big\langle v, \nabla^2\log p(X,k\,|\,\theta)\,v\big\rangle$
is sub-exponential with parameter $\sigma_1$. In particular, a sufficient condition for this to hold is $\matnorm{\nabla^2\log p(X, k\,|\,\theta)}$ being sub-exponential for any $\theta\in \Theta$ and $k\in[K]$.

\item $\lambda(X):\,= \sup_{\theta\in\Theta, k\in[K]}\matnorm{\nabla^2\log p(k\,|\,X,\theta)}$ is sub-exponential with parameter $\sigma_2$.
\end{enumerate}
\end{assumption}

\noindent 
Parts 3 and 4 of the assumption can also be extended to a bounded $\psi_\alpha$ (Orlicz) norm for some $\alpha > 0$. 
Now we present our main theoretical result on the convergence of MF-WGF for Bayesian latent variable models.
\begin{theorem}[MFVI with latent variables]\label{thm: main_theorem}
Suppose Assumptions~\ref{assump: test condition}--\ref{assump: qudratic growth of KL} and~\ref{assump: strong_convexity}--\ref{assump: continuity_of_Hessian} hold, and $\log\pi_\theta$ is twice differentiable. Let $\gamma$ denote the operator norm of the missing data Fisher information matrix $I_S(\theta^\ast)$, i.e.~$\gamma = \matnorm{I_S(\theta^\ast)}$ where
\begin{align*}
    I_S(\theta^\ast) = \int_{\mb R^d} \sum_{z=1}^K p(z\,|\,x,\theta^\ast)\big[\nabla\log p(z\,|\,x,\theta^\ast)\big]\big[\nabla\log p(z\,|\,x,\theta^\ast)\big]^T p(x\,|\,\theta^\ast)\,\dd x,
\end{align*}
and recall that $r$ is the radius of the $W_2$-ball in Assumption \ref{assump: strong_convexity}. Assume $\kappa := \frac{\lambda}{\gamma} > 2$, and define
\begin{align*}
    R_W := \min\bigg\{\sqrt{\frac{\lambda(\kappa-2)}{32A(\kappa+3)}}, \frac{\lambda(\kappa-2)}{16C(\kappa+3)}, \frac{\lambda(\kappa-2)}{8B(\kappa+3)}, \frac{r}{3}\bigg\},
\end{align*}
where explicit expressions of the constants $A$, $B$ and $C$ are provided in the proof. If the initial distribution $q_\theta^{(0)}$ satisfies
\begin{displaymath}
W_2(q_\theta^{(0)},\,\delta_{\theta^\ast}) = \sqrt{\mb E_{q_{\theta}^{(0)}}\big[\|\theta-\theta^\ast\|^2\big]}  \leq R_W,
\end{displaymath}
and the sample size $n$ is large enough (explicit lower bound of $n$ provided in the proof),
then the $k$-th iterate $q_\theta^{(k)}$ satisfies
\begin{align*}
    W_2^2(\,q_\theta^{(k)}, \widehat{q}_\theta) \leq \bigg(1 - \frac{(\kappa-2)(3\kappa+2) - \frac{2(3\kappa+2)}{n\gamma}\lambda_{\max}(\nabla^2\log\pi_\theta)}{(4\kappa^2+\kappa-2)- \frac{2(3\kappa+2)}{n\gamma}\lambda_{\max}(\nabla^2\log\pi_\theta)+ \frac{2(3\kappa+2)}{n\tau\gamma}}\bigg)^k\, W_2^2(q_\theta^{(0)}, \widehat{q}_\theta).
\end{align*}
with probability at least $1 - \frac{2}{\log n} - ne^{-\frac{\sqrt{n}\lambda(\kappa-2)}{4d\sigma_1(3\kappa+2)}} - 2e^{3d-\frac{cn\sigma_3\gamma(\kappa-2)}{4(3\kappa+2)}} - 2e^{-cn^{1/6}\sigma_2^{-1}}-4e^{-cn^{1/6}\sigma_3^{-1}}$ for some universal constant $c > 0$. Again, $\lambda_{\max}(\nabla^2\log\pi_\theta)$ is the largest eigenvalue of the Hessian matrix $\nabla^2\log\pi_\theta$ in $\Theta$. This means MF-WGF algorithm has exponential convergence towards the MF approximation $\wht q_\theta$.
\end{theorem}

The contraction factor provided in the theorem decreases as the step size $\tau$ increases, with limit $1 - \frac{(\kappa-2)(3\kappa+2)}{4\kappa^2+\kappa - 2}$ as $\tau\to\infty$. As we argued in Section \ref{sec:particle_approx}, JKO-scheme can be viewed as an implicit scheme in Wasserstein space. In the Euclidean setting, an implicit Euler scheme converges without any restriction on the step size. Similarly, we do not need any restriction on $\tau$ in our theory, and the existence of the solution of the optimization problem~\eqref{eqn: JKO_update_qtheta} for any $\tau>0$ is proved in \cite[Proposition 8.5,][]{santambrogio2015optimal}. However, in practice, we need the $k$-th step size to satisfy $\tau_k\leq\frac{1}{2L_k+1}$ to guarantee the convergence of discretized Langevin SDE scheme. Here $L_k$ denotes the Lipschitz constant of $n\nabla U_n(\cdot\,;\, q_\theta^{(k-1)}) - \nabla\log\pi_\theta$ (see Lemma \ref{lem: onestep_err} and its proof for more details). The extra factor $n$ is due to the leading multiplicative factor $n$ in the definition of $V_n$. As a consequence, the theoretical upper bound requirement of step size is of order $O(n^{-1}|\!|\!|\nabla^2 U_n(\cdot\,;\,q_\theta^{(k-1)})|\!|\!|_{\rm op}^{-1})$, which matches the typical requirement of step sizes in gradient descents for empirical risk minimization.

The EM algorithm can be seen as a specific instance of our approach when the distributions in the MF family are further restricted to point mass measures. Thus, it is not surprising that our algorithm can only guarantee local convergence as the EM algorithm. Here, our Assumption~\ref{assump: strong_convexity} does not directly impose local convexity on the population version of the negative log-likelihood. Instead, we focus on the population version of a distributional counterpart of the standard $Q$-function in the EM algorithm, which is defined in~\eqref{eqn:pop_U}.  In other words, our assumption allows the negative log-conditional likelihood function of observed data $x$ given each latent variable value $z$ to be non-convex in $\theta$, as long as their weighted average remains convex; this assumption also does not require the conditional posterior of the parameter $\theta$ given latent variables $Z^n$ to be log-concave. A similar local convexity assumption is made in the recent refined analysis of the EM algorithm by~\cite{balakrishnan2017statistical}.   In addition, although our algorithm requires initializing in a neighborhood of the solution, the neighborhood radius from our theory is a \emph{constant}, independent of sample size $n$, as opposed to a radius decreasing in $n$. This means that any initial estimator that is consistent can lead to a good initialization for our algorithm.
While it may be feasible to relax this local convexity assumption, the primary focus of this work is not to enhance the existing convergence analysis of the EM algorithm but to show that many of the desirable properties associated with point estimators in frequentist literature can be extended to Bayesian cases. A primary message we wish to convey is that the computational framework of the Wasserstein gradient flow aligns well with existing analyses in conventional optimization (over Euclidean space) literature. This framework can leverage the inherent ``convexity" structure to guarantee the convergence of certain MF algorithms. In contrast, it is unclear whether the traditional CAVI algorithm for MF implementation (despite its limitation of requiring conditional conjugacy) can benefit from convexity, given that it can be interpreted as a gradient flow with respect to the KL divergence rather than the Wasserstein metric.

Our result requires an informative initialization which exists as long as $\kappa > 2$. This lack of global convergence is due to the following two reasons: (1) the strong convexity~\eqref{assump: strong_convexity} is only required to hold in a neighborhood of $\delta_{\theta^\ast}$; (2) MF-WGF is an EM-type algorithm---it is known that, even in the Euclidean case, the EM algorithm converges to the true parameter with high probability when the initialization is good enough, but may converge to bad local optima with an uninformative initialization~\cite{balakrishnan2017statistical}.
In practice, in order to choose a $q_\theta^{(0)}$ to satisfy the initial condition, we can run a simple and fast algorithm to get a consistent estimator of the parameter in order to construct a good initialization before applying MF-WGF. For example, in clustering problems, we may apply the EM algorithm or the K-means method to derive pilot estimates of all parameters; then, we may add independent noises with constant order variance to the previous estimates to generate i.i.d. particles inducing an initialization $q_\theta^{(0)}$.

Our proof of the theorem is based on an induction argument, by repeatedly applying Theorem~\ref{thm: implicit_WGF_conv} to analyze the evolution of one-step discretized WGF~\eqref{eqn: JKO_update_qtheta} for minimizing the energy functional $V(q_\theta\,|\,q_\theta^{(k)})$ whose form changes over the iteration count $k$. When sample size $n$ is sufficiently large, the prior tend to have diminishing impact on the algorithm. If $\lambda \gg \gamma$ is also satisfied, then the derived algorithmic contraction rate is roughly of order $\m O(\gamma/\lambda)$. Interestingly, $\gamma$ reflects the amount of missing data information (by viewing latent variables as missing data), since recall that $\gamma$ is defined as the operator norm of the missing data Fisher information $I_S(\theta^\ast)$; while $\lambda$ corresponds to the complete data information, since it provides a lower bound to the complete data Fisher information $I_C(\theta^\ast)$ as the Hessian matrix of potential $U(\,\cdot\,;\,\delta_{\theta^\ast})$ at the point mass measure at $\theta^\ast$.
In comparison, the algorithmic contraction rate of the classical EM algorithm has a local contraction rate bounded by the largest eigenvalue of $\big[I_C(\theta^\ast)\big]^{-1}I_S(\theta^\ast)$~\cite{dempster1977maximum}; and is consistent with the derived contraction rate of our MF-WGF algorithm viewed as a distributional extension of the EM.

By drawing an analogue from the local contraction rate of the EM algorithm, we believe that by incorporating some local geometric structures into the algorithm and our theoretical analysis, the current technical assumption $\lambda > 2\gamma$ can also be weakened to $I_C(\theta^\ast) \succeq a \,I_S(\theta^\ast)$ for all $\mu\in B_{\mb W_2}(0, r)$ and any constant $a>1$. 
For example, we may use a weighted Euclidean norm, defined through $\|x - y\|_I^2 = (x-y)^T\big[I_C(\theta^\ast)\big]^{-1}(x-y)$, to substitute the isotropic Euclidean norm $\|x-y\|$ when defining the $W_2$ distance~\eqref{eqn:kantorovich_problem} and the one-step minimization movement scheme~\eqref{eqn: JKO_update_qtheta}. 
With this substitution, we may define the strongly convexity coefficient of $U(\cdot\,; \mu)$ to be with respect to the $\|\cdot\|_I$ metric in the theoretical analysis, so that the key matrix $\big[I_C(\theta^\ast)\big]^{-1}I_S(\theta^\ast)$ will naturally appear when analyzing the contraction of the discrete gradient flow using Theorem~\ref{thm: implicit_WGF_conv}. We leave a formal methodological and theoretical investigation about this improvement as a future direction.

The block MF approximation to posteriors in Bayesian latent variable models becomes accurate when the dependence between the parameter $\theta$ and latent variables $Z^n$ is weak; or more formally, when the missing data Fisher information matrix $I_S(\theta^\ast)$ is small, such that the latent variable distributions are not sensitive to perturbations or changes in the parameter $\theta$. In fact, it is proved in~\cite{han2019statistical} that under this block MF, the marginal variational distribution $\widehat Q_\theta$ of the parameter $\theta$ approaches $N\big(\theta^{\rm MLE}, (n I_C(\theta^\ast))^{-1}\big)$ as the sample size $n$ approaches infinity, where $\theta^{\rm MLE}$ denotes the maximum likelihood estimator of $\theta$, $I_C(\theta^\ast)=I_S(\theta^\ast) + I(\theta^\ast)$ is the complete data Fisher information and $I(\theta^\ast)$ denotes the (marginal) Fisher information matrix. In comparison, the classical Bernstein von-Mises theorem shows that the exact marginal posterior distribution of $\theta$ is close to $N\big(\theta^{\rm MLE}, (n I(\theta^\ast))^{-1}\big)$. Therefore, the block MF provides a good approximation to the target posterior distribution if and only if $I_S(\theta^\ast)$ is small. As an interesting implication, our Theorem~\ref{thm: main_theorem} on the convergence of MFVI also suggests that the computational efficiency of MFVI improves as the statistical difficulty of approximating the joint posterior via MFVI decreases.


\section{Computation}\label{sec:numeric_method}
Note that both updating formulas~\eqref{eqn: JKO_update_para} and~\eqref{eqn: JKO_update_qtheta} require solving the JKO scheme~\eqref{eqn: JKO_scheme} when specializing $\m F$ to be the KL-divergence type functional $\m F_{\rm KL}$, i.e.
\begin{align}\label{eqn: KL_JKO}
    \rho_{k+1}^\tau = \argmin_{\rho\in\ms P_2^r} \underbrace{\int V\,\dd\rho + \int \rho\log\rho}_{\m F_{\rm KL}(\rho)} + \frac{1}{2\tau}W_2^2(\rho, \rho_k^\tau).
\end{align}
In this section, we will consider and compare two numerical methods for numerically solving~\eqref{eqn: KL_JKO}:  particle approximation via SDE/diffusion and function approximation (FA) approach based on neural networks.



\paragraph{SDE approach.}
Recall that the JKO scheme~\eqref{eqn: KL_JKO} for KL divergence is an implicit scheme for discretizing the Fokker--Planck equation~\eqref{eqn: Fokker_Planck_equation}, which is also known as the WGF of $\m F_{\rm KL}$. According to Section~\ref{sec:KL_flow}, the WGF of $\m F_{\rm KL}$ starting from $\rho_0$ is the evolution of the following Langevin stochastic differential equation,
\begin{align}\label{eqn: Langevin_again}
\dd X_t = -\nabla V(X_t)\,\dd t + \sqrt{2}\,\dd W_t, \quad X_0\sim\rho_0.
\end{align}
This connection between SGD and WGF motivates one to discretize the WGF by discretizing its corresponding SDE, and approximate the solution $\rho_{k+1}^\tau$ of the JKO scheme~\eqref{eqn: KL_JKO} by the evolution of the discretized SDE through the empirical measure of particles which satisfy the following updating formula,
\begin{align}\label{eqn: discrete_SDE}
X^{(k+1)}_b - X_b^{(k)} = -\nabla V(X_b^{(k)})\tau + \sqrt{2\tau}\eta_b^{(k)}, \quad b\in[B]
\end{align}
where $\big\{X_b^{(k)}: b\in[B]\big\}$ are $B$ samples generated from $\rho_k^\tau$, and $\eta_b^{(k)}$ are i.i.d. samples generated from $\m N(0, I)$. This recursive equation is the discretized representation of the SDE~\eqref{eqn: Langevin_again} for approximating~\eqref{eqn: KL_JKO}, and $\rho_{k+1}^\tau$ can be approximated by the empirical distribution of $\big\{X_b^{(k+1)}: b\in[B]\big\}$. See Appendix~\ref{sec:particle_approx} for further discussion about particle approximation and a numerical error analysis of its implementation via SDE.

\paragraph{FA approach.}
The function approximation method converts the JKO scheme~\eqref{eqn: KL_JKO} into an optimization problem over the function space. Note that finding the solution $\rho_{k+1}^\tau$ of~\eqref{eqn: KL_JKO} is equivalent to finding a transport map $T$ such that $T_\#\rho_k^\tau$ minimizes~\eqref{eqn: KL_JKO}. To be precise, we present the following theorem.
\begin{theorem}[JKO scheme via function approximation]\label{thm: FA_approx}
If $\rho_k^\tau\in\ms P_2^r$, and 
\begin{align}\label{eqn: func_approx_JKO}
T_k^\tau = \argmin_T \int V\circ T\,\dd\rho_k^\tau - \int \log\lvert\det \nabla T\rvert\,\dd\rho_k^\tau + \frac{1}{2\tau} \int\|T - \id\|^2\,\dd\rho_k^\tau,
\end{align}
then $\rho_{k+1}^\tau \coloneqq (T_k^\tau)_\#\rho_k^\tau$ minimizes~\eqref{eqn: KL_JKO}.
\end{theorem}

We want to highlight a key property that the optimization problem~\eqref{eqn: func_approx_JKO} is unconstrained, although the last term $\int\|T - \id\|^2\,\dd\rho_k^\tau$ corresponds to $W_2^2(\rho, \rho_k^\tau)=\min_{T,\,\mbox{\scriptsize st}\, T_\# \rho = \rho^k_\tau} \mb E_{\rho}\big[\|X-T(X)\|^2\big]$, and requires the optimal transport map $T_k^\tau$ from $\rho_k^\tau$ to $\rho_{k+1}^\tau$ to be the gradient of a convex function according to Brenier's Theorem~\cite{brenier1991polar} (see Appendix \ref{app:OPTmap} for more details). Most existing methods in the literature for numerically solving the JKO scheme, such as~\cite{mokrov2021large}, require solving a constrained optimization problem by restricting $T = \nabla\phi$ to the gradient of a convex function $\phi$, where the convexity is imposed by using an input-convex neural network (ICNN)~\cite{amos2017input}. However, although ICNN is known to provide universal approximation to convex functions~\cite{chen2018optimal}, it is not clear whether its gradient also provides universal approximation to the gradients of convex functions. Moreover, based on our empirical observations, the inclusion of the convexity constraint tends to make the optimization problem particularly difficult to solve due to numerous local minima, extremely slow convergence and high sensitivity to tuning.
On the contrary, Theorem~\ref{thm: FA_approx} shows that solving the unconstrained optimization problem is equivalent to solving the JKO scheme, and even restricting $T$ to be a gradient vector field is not necessary. The intuition is that, if a solution $\widetilde T_k^\tau$ to problem~\eqref{eqn: func_approx_JKO} is not the optimal transport map $T_k^\tau$ from $\rho_k^\tau$ to $\rho_{k+1}^\tau$, then changing $T_k^\tau$ to $\widetilde T_k^\tau$ in the objective function~\eqref{eqn: func_approx_JKO} will strictly decrease the last transport cost term while keeping the rest terms unchanged. This contradicts to the optimality of $\widetilde T_k^\tau$. A formal proof is deferred to Appendix~\ref{app: proof_FA}.

In practice, the optimization problem in Theorem~\ref{thm: FA_approx} over the function space can be solved by using function approximation methods, for example, based on (deep) neural networks. If we use $T_k^\tau$ to denote the transport map computed in the $k$-th iteration and choose an initial distribution $\rho_0^\tau$ that is easy to sample from, then we can approximate the objective functional in~\eqref{eqn: func_approx_JKO} up to arbitrary accuracy by Monte Carlo approximation via sampling from $\rho_k^\tau = \big[T_{k-1}^\tau\circ T_{k-1}^\tau\circ\cdots\circ T_0^\tau\big]_\# \rho_0^\tau$.
Concretely, suppose $\{X_b^{(k)}: b\in[B]\}$ are $B$ samples drawn from $\rho_k^\tau$ using the transport maps. We can compute the optimal transport map $T_k^\tau$ from $\rho_k^\tau$ to $\rho_{k+1}^\tau$ by solving
\begin{align*}
T_k^\tau = \argmin_T \frac{1}{B}\sum_{b=1}^B \Big[V\circ T(X_b^{(k)}) - \log\big\lvert\det\nabla T(X_b^{(k)})\big\rvert + \frac{1}{2\tau}\big\|X_b^{(k)} - T(X_b^{(k)})\big\|^2\Big].
\end{align*}

\paragraph{FA versus SDE.} 
We recommend FA over SDE due to two major deficiencies arising in the SDE approach. 

First, the SDE approach might introduce a systematic error that remains undiminished even with more iterations and number of particles.
Specifically, it is known in the literature (e.g.,~\cite{cheng2018sharp,chewi2021optimal}) that applying a time-discretized SDE plus particle approximation to numerically compute a gradient flow over the space of all distributions suffers from two source of errors. 
One is the space and/or time discretization error due to a finite step size $\tau$ and a finite number $B$ of particles; and another is the long term bias due to the mismatch between the limiting distribution of the Markov chain induced by the time-discretized SDE and the limiting distribution of the continuous time SDE. As a consequence, to attain an accuracy of $\varepsilon\in(0,1)$ in the $W_2$ distance, both space and time complexities are $\mathcal O(\varepsilon^{-2})$ up to logarithmic factors. The second error resulting from the limiting bias, can be mitigated by incorporating a Metropolis-Hastings correction step. This gives rise to the Metropolis-adjusted Langevin algorithm (MALA, see e.g., \cite{roberts1996exponential}). However, MALA is computationally much more expensive and it is not clear whether such a correction would be beneficial in our context.
In comparison, each step of the FA approach is unbiased, meaning that any fixed point of the FA iterative formula~\eqref{eqn: KL_JKO} precisely gives a critical point to the target functional $\mathcal F_{\rm KL}$. As a consequence, unlike the SDE approach, the numerical error from earlier iterations of the FA approach will not accumulate provided the dynamics converge exponentially, which is the case in our scenario. Indeed, we observe this numerical issue with the SDE in our numerical experiments (e.g., refer to Figures~\ref{fig: BayesianLinearRegression},~\ref{fig: GMM} and~\ref{fig: MR} in Section~\ref{sec:thm_app}), where the optimization error from the SDE approach initially reduces but becomes unimprovable with increased iterations, stemming from the accumulated error and the limiting bias. In comparison, the optimization error in the FA approach continues to decline exponentially and shows a steeper decline in the logarithmic scale, suggesting a smaller contraction factor. 

Second, the SDE approach corresponds to forward scheme that necessitates an upper bound on the step size $\tau$ to avoid divergence. In comparison, the FA approach is an implicit scheme that does not diverge for any $\tau$, provided that the associated optimization problem~\eqref{eqn: func_approx_JKO} is solved effectively.
A typical upper bound on $\tau$ for SDE is proportional to the inverse smoothness parameter $\sup_{\theta\in\Theta}\|\nabla U_n(\theta)\|_{\rm Lip}^{-1}$. This restricts the use of larger step sizes for problems involving a fluctuating sample potential $U_n$, which, in turn, necessitates more iterations for SDE to converge to a reasonably good estimate (see Figures~\ref{fig: BayesianLinearRegression},~\ref{fig: GMM} and~\ref{fig: MR}).


\section{Examples}\label{sec:thm_app}
In this section, we apply our theoretical results to three representative Bayesian models and discuss their consequences. We also conduct some numerical studies to compliment the theoretical predictions.

\subsection{Bayesian linear regression}

\begin{figure}[htbp]
	\centering
	\subfloat[Numerical error of $\theta$.]{\includegraphics[width=.45\columnwidth]{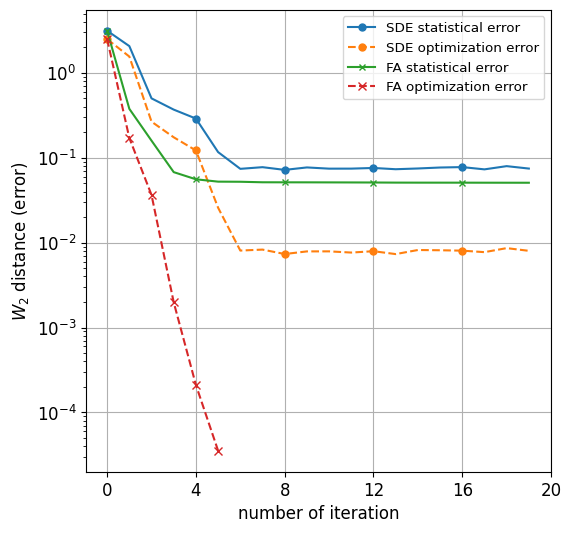}\label{fig: SDE vs FA, BLR_theta}}\hspace{5pt}
    \subfloat[Numerical error of $\alpha = \beta^{-2}$.]{\includegraphics[width=.45\columnwidth]{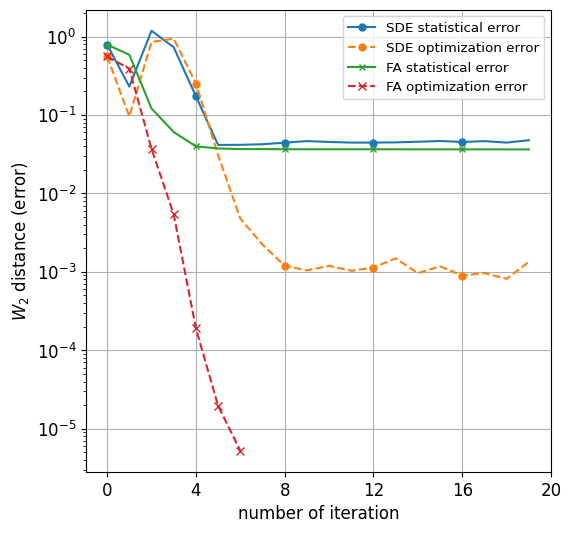}\label{fig: SDE vs FA, BLR_alpha}}\\
    \subfloat[Numerical error for different sample size.]{\includegraphics[width=.45\columnwidth]{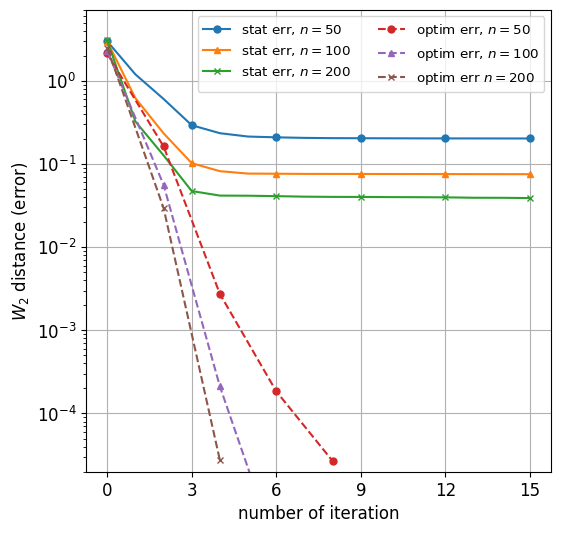}\label{fig: diff_sample_size, BLR}}\hspace{5pt}
    \subfloat[Contour plot.]{\includegraphics[width=.445\columnwidth]{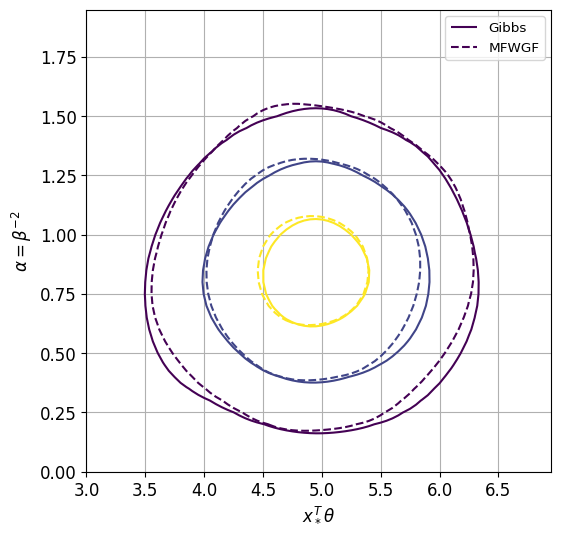}\label{fig: MF-WGF vs Gibbs, BLR}}
	\caption{Numerical results in the Bayesian linear regression example with sample size $n=100$, $\theta^\ast = (1, -2, 3)$, and $\beta^\ast = 1$.
    (a) and (b) Comparison of the numerical errors obtained by using the FA approach and the SDE approach of $\alpha = \beta^{-2}$ and $\theta$. Both approaches have similar statistical errors, but different from the FA approach, the optimization error in the SDE approach converges to the approximation error after several iterations. A smaller step size leads to a smaller error when using the SDE approach; as a trade-off, it takes more iterations to converge. 
    (c) Comparison of the numerical errors of $\theta$ obtained by using the FA approach with different sample sizes. When the sample size gets larger, the statistical error gets smaller, and the contraction rate does not change too much for sufficiently large sample size.
    (d) Comparison of contours from the joint posterior distribution of $(x_\ast^T\theta, \alpha)$ with $x_\ast = (-2, 1, 3)$ from Gibbs sampling versus their MF approximation output from MF-WGF. The MF approximation computed via MF-WGF is quite close to the true posterior.
    }\label{fig: BayesianLinearRegression}
\end{figure}

We consider Bayesian linear regression models as a representative example for Bayesian models without latent variables and verify the assumptions in Theorem \ref{thm: para_update_whp}. We consider a random design case where $n$ i.i.d.~pairs $(X_i, y_i)$ are sampled from
\begin{align*}
    y_i = \theta^TX_i + \varepsilon_i, \qquad X_i\overset{iid}{\sim}N(0, I_d) \qquad\mx{and} \qquad \varepsilon_i\stackrel{\text{i.i.d.}}{\sim}\m N(0, \beta^2).
\end{align*} 
In this example, we assume both the coefficient $\theta$ and the variance $\beta^2$ are unknown parameters with the prior distribution $\pi(\theta, \beta^2)$. More specifically, we have
\begin{align*}
    y_i\,|\,X_i, \theta, \beta^2 \sim \m N(\theta^TX_i, \beta^2), \quad X_i\stackrel{\text{i.i.d.}}{\sim} \m N(0, \Sigma), \quad \mbox{and}\quad (\theta, \beta^2)\sim \pi(\theta, \beta^2),
\end{align*}
where the covariance matrix $\Sigma$ is positive definite. In the traditional setting, due to computational tractability, the prior of $\alpha := \beta^{-2}$ is usually Gamma distribution and the conditional prior distribution of $\theta\,|\,\beta^{-2}$ is chosen as a normal distribution. Here, we directly choose a uniform prior for $\theta$ and $\beta^2$, but our method can be easily implemented for all prior distributions that are absolutely continuous with respect to the Lebesgue measure.

\begin{corollary}\label{coro: BLR}
Let $\Theta_\alpha$ and $\Theta_\theta$ be the parameter spaces of $\alpha$ and $\theta$ respectively. Assume $0 < \alpha_{lb} < \alpha < \alpha_{ub}$ for all $\alpha\in\Theta_\alpha$, and
\begin{align}\label{eqn: para_space, BLR}
\sup_{\theta\in\Theta_\theta}\|\theta - \theta^\ast\| =: R_\theta < \sqrt{\frac{\lambda_1}{2\alpha_{ub}\lambda_d^2}}.
\end{align}
If $\lambda_1 I_d\preceq \Sigma \preceq \lambda_d I_d$, then we have
\begin{align*}
W_2^2(q_\theta^{(k)}\otimes q_\alpha^{(k)}, \wht q_{\theta}\otimes\wht q_\alpha) \leq \Big(1 + \frac{\lambda_{lb}^2}{ L_{ub}^2m}\Big)^{-k} W_2^2(q_\theta^{(0)}\otimes q_\alpha^{(0)}, \wht q_{\theta}\otimes\wht q_\alpha),
\end{align*}
where
\begin{align*}
    \lambda_{lb} &=  \frac{n(\frac{\lambda_1}{2\alpha_{ub}} - \lambda_d^2R_\theta^2)}{\max\{\alpha_{ub}\lambda_1, \frac{1}{2\alpha_{lb}^2}\} + \lambda_d R_\theta} - \lambda_M(\nabla^2\log\pi_\theta) \\
    &\qquad\qquad\qquad\qquad\qquad- \sigma_5^2\sqrt{\frac{Cd\log n}{n}\cdot\max\Big\{\frac{\log\big(2d + \frac{\alpha_{ub}}{\alpha_{lb}^4}\big)}{\log d}, \log\frac{R_\theta\sigma_5}{\eta}, 1\Big\}}\\
    L_{ub} &= n\Big(\max\Big\{\alpha_{ub}\lambda_d, \frac{1}{2\alpha_{lb}^2}\Big\} + \lambda_dR_\theta\Big) - \lambda_m(\nabla^2\log\pi_\theta)\\
    &\qquad\qquad\qquad\qquad\qquad+ \sigma_5^2\sqrt{\frac{Cd\log n}{n}\cdot\max\Big\{\frac{\log\big(2d + \frac{\alpha_{ub}}{\alpha_{lb}^4}\big)}{\log d}, \log\frac{R_\theta\sigma_5}{\eta}, 1\Big\}}.
\end{align*}

\end{corollary}

Figure~\ref{fig: BayesianLinearRegression} summarizes the numerical results to support our theory in Bayesian models without latent variables. In this experiment, we consider the regression model with true parameters $\theta^\ast = (1, -2, 3)$ and $\beta^\ast = 1$. We choose the sample size $n=100$ and use $B=1000$ to approximate the posterior distribution. When applying the SDE approach, the particles which are used to approximate the distribution of $\alpha$ may go beyond the origin and become negative due to the unboundedness of Gaussian noise. To address this issue, we choose a threshold $\epsilon = 0.1$. At the end of each iteration, we add a projection step $\alpha_{b, \rm{proj}}^{(t)} = \alpha^{(t)}_b1\{\alpha^{(t)}_b > \epsilon\} + \epsilon1\{\alpha^{(t)}_b\leq\epsilon\}$ for all $b\in[B]$. We choose $\tau = 0.01$ for the SDE approach since it is the largest step size for SDE without incuring divergence, and use $\tau=1$ for the FA approach.

Figure~\ref{fig: SDE vs FA, BLR_theta} presents the statistical error $W_2^2(\delta_{\theta^\ast}, q_\theta^{(k)})$ and the optimization error $W_2^2(\wht q_\theta, q_\theta^{(k)})$ for the linear coefficient $\theta$.
Figure~\ref{fig: SDE vs FA, BLR_alpha} shows the statistical error $W_2^2(\delta_{\alpha^\ast}, q_\alpha^{(k)})$ and the optimization error $W_2^2(\wht q_\alpha, q_\alpha^{(k)})$ for the inverse of noise $\alpha=\beta^{-2}$. The increase of the statistical and optimization error in the first several iterates in the SDE approach is due to the existence of the projection step. As we can see, in both figures, the optimization error in the FA approach indicated by red dashed curves keeps decaying exponentially fast as predicted by our theory. However, in the SDE approach, the optimization error indicated by the orange dashed lines will finally be dominated by the approximation error and converge to quite large values compared with the FA approach. Choosing a smaller step size in the SDE approach can help decrease the approximation error. As a sacrifice, the algorithm takes more iterations to converge. In comparison, the statistical error indicated by solid curves has exponential decay at some initial period and then stabilizes in both approaches, which indicating the dominance of statistical error over optimization error in the later period.

Figure~\ref{fig: diff_sample_size, BLR} studies the the effect of the sample size on the contraction rate and the statistical error when $\tau = 1$. In the plot, we can see that the statistical error indicated by the solid lines decreases when the sample size gets larger, which is consistent to the common knowledge in statistics. As for the contraction rate of the optimization error indicated by the dashed lines, increasing the sample size is helpful to get a smaller contraction rate when the sample size is small (compare $n=50$ with $n=100$); however, once a sufficient sample size has been acquired, further increase in sample size may result in little improvement (compare $n=100$ with $n=200$).

Figure~\ref{fig: MF-WGF vs Gibbs, BLR} shows the contour plot of the joint posterior distribution of $(\alpha, x_\ast^T\theta)$ with $x_\ast = (-2, 1, 3)$, computed by Gibbs sampling, versus their MF approximation output by MF-WGF. From the plot, we see that the MF approximation is close to the true joint posterior distribution, meaning that prediction and its associated uncertainty quantification using MF tends to be accurate at $x_\ast$.

\subsection{Repulsive Gaussian mixture model}\label{subsec: RGMM}
In this example, we consider the Gaussian mixture model (GMM) as a simplest latent variable model that are widely used for clustering.
We focus on the following (isotropic) Gaussian mixture model (GMM) with $K$ components in $\mb R^d$,
\begin{displaymath}
p(x\,|\,m) = \sum_{k=1}^K w_i\,\m N(x\,|\,m_i,\, \beta^2I_d),
\end{displaymath}
where the common covariance matrix is $\beta^2$ times the identity matrix $I_d$, $w = (w_1,\cdots, w_K)\in\mb R^K$ are the nonnegative mixing weight parameters satisfying $\sum_{k=1}^K w_k =1$, and cluster centers $m = (m_1, \cdots, m_K)\in\mb R^{d\times K}$ are the primary parameters of interest. For theoretical convenience, we assume both nuisance parameters $\beta^2$ and $w$ to be known; in our numerical results to follow, we treat $w$ as unknown as well and use MF-WGF to approximate the joint posterior of $(w,m)$. Our theory can also cover the model with unknown $w$ and $\beta^2$ as long as the corresponding parameter spaces are convex, compact, and bounded away from zero.

As a common practice to simplify the likelihood computation and facilitate the interpretation, we introduce a latent variable $Z\in[K]:= \{1, \cdots, K\}$ to indicate which underlying mixture component an observation $X$ from GMM belongs to. Under this data augmentation, the full Bayesian latent variable model can be formulated as
\begin{align*}
\big[X_i\,\big|\,m, \,Z_i=k\big] &\ \sim\  \m N(m_k, \,\beta^2I_d),\quad
Z_i \ \stackrel{\textrm{i.i.d.}}{\sim} \ \textrm{Multi}([K], w), \quad\mbox{and}\quad 
m \ \sim\  \pi_m,
\end{align*}
where $\pi_m$ denotes the prior distribution over $m$.
We apply a block MF approximation to the joint posterior distribution $\pi_n(m, \,Z^n)$ over parameter $m\in\mb R^{d\times K}$ and latent variables $Z^n=\{Z_1,\ldots,Z_n\}$, by using variational distributions of the form $q_{m,\,Z^n}=q_m\otimes q_{Z^n}$, to maximally preserve the dependence structure while retaining the computational tractability. 

In the literature, there is a class of priors $\pi_m$, called repulsive priors \cite{petralia2012repulsive,xie2020bayesian}, that are preferable to use than independent priors over $\{m_k\}_{k=1}^K$. Let $d_{\text{min}} = \min_{1\leq i<j\leq K}\|m_i - m_j\|$ denote the minimum distance between cluster centers. A typical repulsive prior takes the form of
\begin{equation}\label{eqn: repulsive_prior}
\pi_m \ \propto\  g(m;\, g_0)\cdot\prod_{k=1}^K\m N(m_k\,|\,0, \sigma^2 I_d),
\end{equation}
which modifies the independent priors with a repulsive function $g(m; g_0) = \frac{d_{\text{min}}}{d_{\text{min}} + g_0}$ that encourages the well-separatedness of cluster centers and reduces the potential redundancy of components. The complicated dependence structure introduced in the repulsive prior destroys the conditional conjugacy, making the standard coordinate ascent variational inference (CAVI,~\cite{bishop2006pattern}) algorithm for finding the best MF approximation $q_{m,\,Z^n}$ inapplicable. In comparison, the proposed MF-WGF can be easily applied in a straightforward manner. We want to emphasize that while we presented the repulsive prior as an illustrative example, our method is flexible and can be applied to various other priors without demanding additional restrictive conditions, such as the conditional conjugacy condition required by the widely-used CAVI algorithm.
The following corollary proves the exponential convergence of MF-WGF and characterizes the explicit dependence of various problem characteristics on the contraction rate.



\begin{figure}[htbp]
	\centering
	\subfloat[$\beta=3$]{\includegraphics[width=.45\columnwidth]{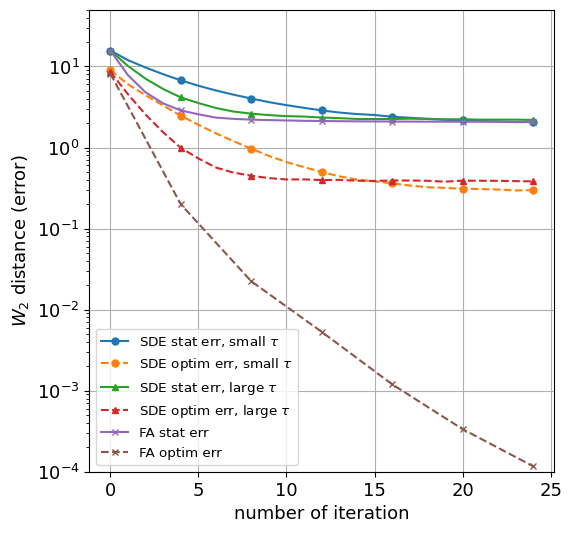}\label{fig: FA vs SDE, largenoise, GMM}}\hspace{5pt}
	\subfloat[$\beta=2$]{\includegraphics[width=.45\columnwidth]{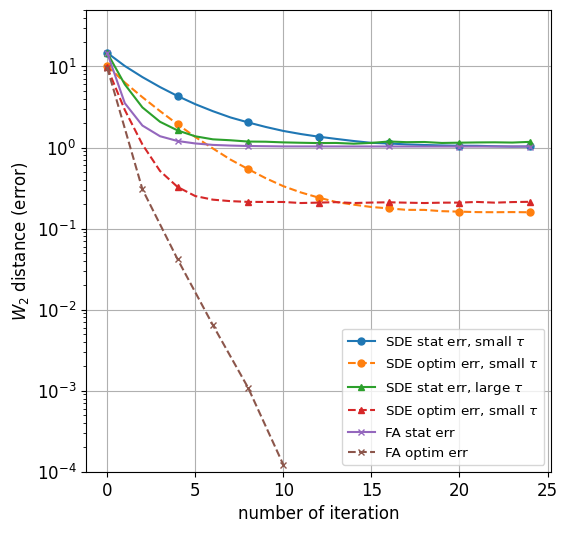}\label{fig: FA vs SDE, smallnoise, GMM}}\\
	\subfloat[different sample size when $\beta=2$]{\includegraphics[width=.45\columnwidth]{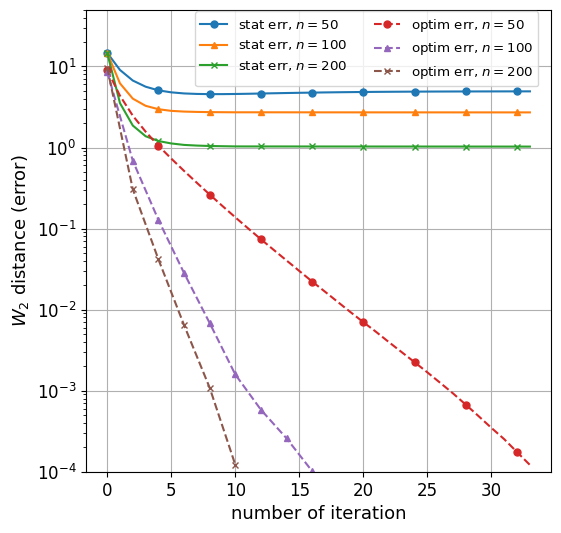}\label{fig: diff_sample_size, GMM}}\hspace{5pt}
	\subfloat[Contour plot.]{\includegraphics[width=.445\columnwidth]{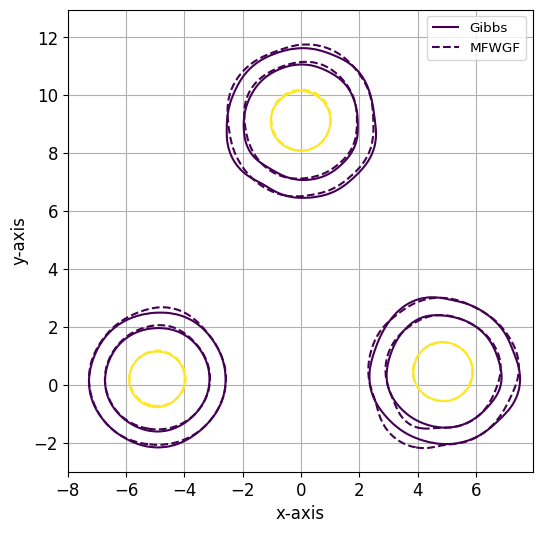}\label{fig: MF-WGF vs Gibbs, GMM}}
	\caption{Numerical results of repulsive GMM with sample size $n=200$ and repulsive parameter $g_0=1$. True centers are $m_1^\ast = (5, 0)$, $m_2^\ast = (0, 5\sqrt{3})$, and $m_3^\ast = (-5, 0)$ with weights $w^\ast = (0.27, 0.27, 0.46)$.
    (a) and (b): Comparison of the numerical errors obtained by using the FA approach and the SDE approach under different noise levels are shown in the figure. Statistical errors $W_2^2(q_\theta^{(t)}, \delta_{\theta^\ast})$ of both approaches converge to similar values. However, the optimization error $W_2^2(q_\theta^{(t)}, \wht q_{\theta})$ in the SDE approach is dominated by the approximation error after several iterations. A smaller step size leads to a smaller error when using the SDE approach; as a trade-off, it takes more iterations to converge. 
    (c) Comparison of the numerical errors obtained by using the FA approach with different sample sizes. Larger sample sizes yields smaller statistical error and faster convergence rate.
    (d) Comparison of contours from the marginal distributions of the cluster centers $m_1$, $m_2$ and $m_3$ from Gibbs sampling versus their MF approximation output from MF-WGF. The MF approximation computed via MF-WGF is quite close to the distribution computed via Gibbs sampling.}
    \label{fig: GMM}
\end{figure}


\begin{corollary}\label{coro: GMM}
Let $\pi_m$ be the prior of centers $m = (m_1, \cdots, m_K)$ satisfying Assumption \ref{assump: prior condition} and the parameter space $\Theta\subset B_{\mb R^{Kd}}(0, R)$. Assume the signal-to-noise ratio (SNR) $\kappa_{\rm SNR} = \frac{d_{\min}}{\beta} > C$ for some constant $C = C(w, K) > 0$. Then, there exists constants $R_W = R_W(m^\ast, w, K, \beta)$ and $N = N(m^\ast, \beta, R_W, w, \pi_m, d)$ such that when $n > N$ and $W_2(\delta_{m^\ast}, q_m^{(0)}) \leq R_W$ (see (\ref{eqn: GMM_RW}) for explicit expressions), we have that
\begin{align*}
    W_2^2(\wht q_m, q_m^{(k)}) \leq \rho^k\,W_2^2(\wht q_m, q_m^{(0)}),\quad\forall\, k\in\mb N
\end{align*}
holds with probability at least $1-\frac{3}{\log n}$. Here, the contraction factor $\rho$ takes the form of
\begin{align*}
    1 - \frac{(\zeta-2)(3\zeta+2)}{4\zeta^2 + \zeta - 2}\quad \mx{with} \ \ \zeta=\frac{w_{\min}^2}{6K}\cdot\frac{\exp\{\kappa_{\rm SNR}^2/256\}}{2+\kappa_{\rm SNR}^2},
\end{align*}
as $n\to\infty$, which monotonically decreases to $\frac{1}{4}$ as $\kappa_{\rm SNR}\to\infty$.
\end{corollary}

Here, we want to make several remarks: 1.~the repulsive prior~\eqref{eqn: repulsive_prior} satisfies Assumption \ref{assump: prior condition}, which then implies the exponential convergence of $q_m^{(k)}$ to $\wht q_m$ in the $W_2$ metric (the proof can be found in Appendix \ref{app:pf_example}); 2.~in practice, the compactness assumption on the parameter space is usually not necessary. Moreover, our algorithm can be straightforwardly extended to the setting where both cluster centers $m$ and weights $w$ are unknown as in the numerical studies shown below; 
3.~the lower bound of $\kappa_{\rm SNR}$ is not tight and can be improved. From our numerical experiments, a much smaller $\kappa_{\rm SNR}$ value is sufficient to ensure convergence to the global minimum. However, some lower bound on the 
$\kappa_{\rm SNR}$ is necessary to ensure the exponential convergence of the algorithm with a constant factor of contraction rate. With a low SNR, the model falls into the singular regime, and the EM algorithm (as well as our algorithm) may converge extremely slowly; see, for example,~\cite{dwivedi2020singularity}. This slow-convergence is natural since in the singular regime, the parameter itself becomes statistically non-identifiable (due to a near singular Fisher information matrix) and cannot be accurately estimated. The same remark also applies to the next mixture of regression example.

Figure~\ref{fig: GMM} summarizes some numerical results to complement the theoretical predictions. In this experiment, we consider GMM with three classes centered at $m_1 = (5, 0)$, $m_2 = (0, 5\sqrt{3})$, and $m_3 = (-5, 0)$ with weights $w_1 = w_2 = 0.27$ and $w_3 = 0.46$. We choose the repulsive prior~\eqref{eqn: repulsive_prior} with $g_0 = 1$ and $\sigma^2=10$. We let the sample size $n=200$. For the SDE approach under both noise settings ($\beta=2$ and $\beta=3$), we choose the largest step size while trying not to increase the statistical error significantly. We construct the initialization by applying the $K$-means clustering to obtain an initial estimates of $m$. We use 1000 particles in the simulation in order to estimate the optimization (numerical) error $W_2^2(\wht q_\theta, q_\theta^{(k)})$. Far less particles will be needed for conducting accurate inference on the model parameters. As indicated by Corollary~\ref{coro: GMM}, we define $d/\beta$ as the signal to noise ratio (SNR) that characterizes the algorithmic convergence, and vary it in the simulation by tuning $\beta$.
Since we are using the log-scale for the vertical axis, straight lines means our considered squared $W_2$ distance, either the optimization error $W_2^2(\widehat{q}_\theta, q_\theta^{(k)})$ or the statistical error $W_2^2(\delta_{\theta^\ast}, q_\theta^{(k)})$, decays exponentially fast in the number of iterations, with the slope corresponding to the logarithm of the contraction rate.

In Figure~\ref{fig: FA vs SDE, largenoise, GMM}, we choose $\tau = 0.03$ and $\tau=0.015$ for the SDE approach and $\tau = 0.08$ for the FA approach. Similar to the phenomenon in Figure~\ref{fig: BayesianLinearRegression}, in the FA approach, the optimization error indicated by dashed curves keeps decaying exponentially fast as predicted by our theory; in the SDE approach, the optimization error will finally be dominated by the approximation error. Choosing a smaller step size can help decrease the approximation error but makes the algorithm take more iterations to converge. In comparison, the statistical error indicated by solid curves has exponential decay at some initial period and then stabilizes in both approaches, which indicating the dominance of statistical error over optimization error in the later period. 

In Figure~\ref{fig: FA vs SDE, smallnoise, GMM}, we choose $\tau = 0.025$ and $\tau = 0.01$ for the SDE approach and $\tau = 0.1$ for the FA approach. The same trend of curves as in Figure~\ref{fig: FA vs SDE, largenoise, GMM} is observed as well. Comparing with Figure~\ref{fig: FA vs SDE, largenoise, GMM}, we can see that a higher SNR (smaller $\beta$) corresponds to a faster decay, i.e.~smaller contraction rate. SNR also encodes the statistical hardness of the problem: a higher SNR corresponds to a higher statistical error as the stabilized values of solid curves in the plots.

In Figure~\ref{fig: diff_sample_size, GMM}, we consider the effect of the sample size on the contraction rate and the statistical error when $\tau=0.1$. In the plot, we can see that the statistical error indicated by the solid lines decreases when the sample size gets larger, which is consistent to the common knowledge in statistics. As for the contraction rate of the optimization error indicated by the dashed lines, increasing the sample size is helpful to get a smaller contraction rate when the sample size is small (compare $n=50$ with $n=100$); however, once a sufficient sample size has been acquired, further increases in sample size may result in little improvement (compare $n=100$ with $n=200$).

Figure~\ref{fig: MF-WGF vs Gibbs, GMM} shows the contour plots of the true posterior distribution of cluster centers (solid curves) versus their MF approximation output by MF-WGF (dashed curves), which are pretty close to each other.

\subsection{Mixture of regression}\label{subsec: FMNR}
\begin{figure}[htbp]
	\centering
	\subfloat[$\beta=4$]{\includegraphics[width=.45\columnwidth]{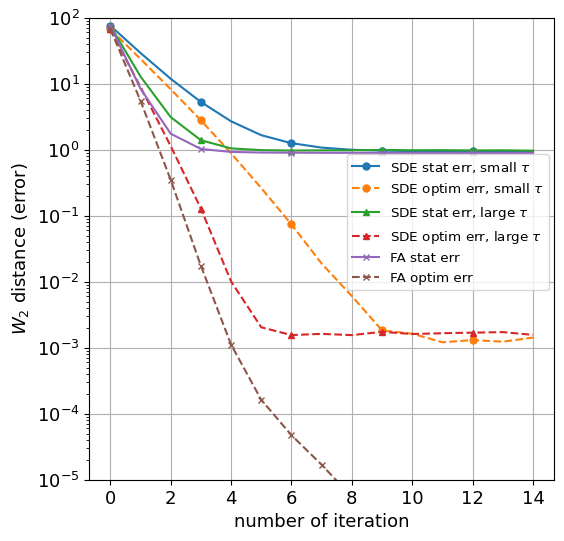}\label{fig: FA vs SDE, largenoise, MR}}\hspace{5pt}
	\subfloat[$\beta=2$]{\includegraphics[width=.45\columnwidth]{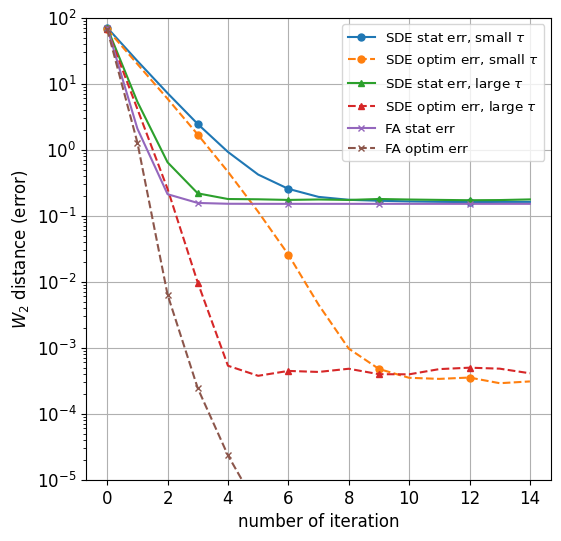}\label{fig: FA vs SDE, smallnoise, MR}}\\
	\subfloat[different sample size when $\beta=2$]{\includegraphics[width=.45\columnwidth]{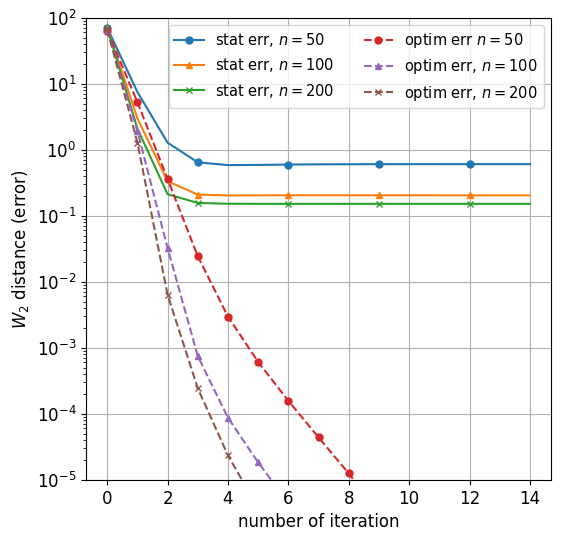}\label{fig: diff_sample_size, MR}}\hspace{5pt}
	\subfloat[Contour plot.]{\includegraphics[width=.445\columnwidth]{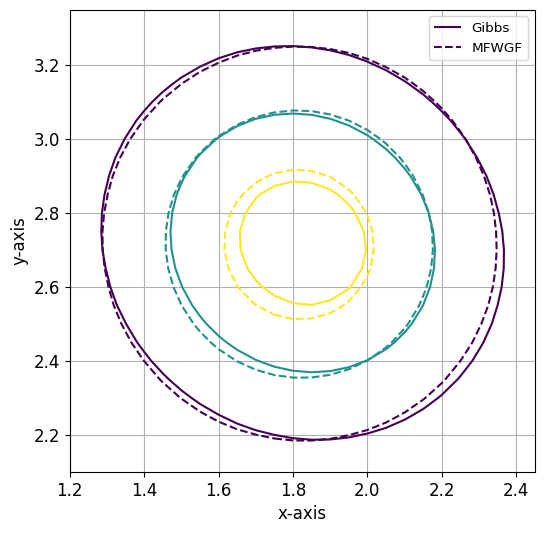}\label{fig: MF-WGF vs Gibbs, MR}}
	\caption{Numerical results in the mixture of regression example with sample size $n=200$, $\theta^\ast = (2, 3)$.
    (a) and (b): Comparison of the numerical errors obtained by using the FA approach and the SDE approach under different noise levels are shown in the figure. Both approaches have similar statistical errors $W_2^2(\wht q_\theta, \delta_{\theta^\ast})$, but different from the FA approach, the optimization error $W_2^2(q_\theta^{(t)}, \wht q_\theta)$ in the SDE approach converges to the approximation error after several iterations. A smaller step size leads to a smaller error when using the SDE approach; as a trade-off, it takes more iterations to converge. 
    (c) Comparison of the numerical errors obtained by using the FA approach with different sample sizes. When the sample size gets larger, the statistical error gets smaller.
    (d) Comparison of contours from the posterior distribution of the coefficient $\theta$ from Gibbs sampling versus their MF approximation output from MF-WGF. The MF approximation computed via MF-WGF is quite close to the true posterior.}
    \label{fig: MR}
\end{figure}

We consider as the third illustrative example a finite mixture regression model (FMRM,~\cite{sung2004gaussian,viele2002modeling}), where can be viewed as an extension of the GMM by including covariates in the mixture formulation. In the standard (random-design) linear regression model, we observe $n$ i.i.d.~pairs $(y_i, X_i)_{i=1}^n$ from
\begin{align*}
    y_i = X_i^T\theta + \varepsilon_i,\qquad \varepsilon_i\stackrel{\textrm{i.i.d.}}{\sim}\m N(0, \,\beta^2)
\end{align*}
where $X_i\in\mb R^d$ denotes the $i$th covariant vector, $y_i$ is the $i$th response variable, $\theta\in\mb R^d$ is the unknown regression coefficient vector parameter of interest, and the Gaussian noise $\varepsilon_i$ is independent of $(X_i, Y_i)$. In our theoretical analysis, we assume $X_i$ to be sampled from $\m N(0, I_d)$. By introducing clustering structures on the conditional distribution of $Y_i$ given $X_i$, we reach the FMRM. Concretely, 
we focus on the simple case with two equally weighted symmetric clusters, where each cluster is determined by its own regression coefficient vector in $\mb R^d$, as
\begin{align*}
    \big[y_i\,\big|\,X_i,\theta,Z_i\big] \ \sim\  \m N(Z_iX_i^T\theta, \beta^2), \ \ \ Z_i \ \stackrel{\textrm{i.i.d.}}{\sim}\  \textrm{Unif}\{1, -1\},\ \ \ 
    X_i \ \stackrel{\textrm{i.i.d.}}{\sim}\ \m N(0, I_d),\ \ \ \mbox{and} \ \ \
    \theta \ \sim\ \pi_\theta,
\end{align*}
where $\pi_\theta$ denotes the prior of $\theta$.
In this model, we are interested in approximating the posterior distribution of $\theta$. For the sake of parameter identifiability, we assume the first non-zero component of $\theta$ is positive ($\theta$ and $-\theta$ correspond to the same model). We also consider the block MF approximation by using $q_{\theta,\,Z^n} =q_\theta \otimes q_{Z^n}$ to approximate the joint posterior of $(\theta,\,Z^n)$, where $Z^n=\{Z_1,\ldots,Z_n\}$. The following corollary provides the algorithmic convergence of MF-WGF algorithm for computing the MF solution $\wht q_\theta$.

\begin{corollary}\label{coro: MR}
Let $\pi_\theta$ be any prior satisfying Assumption \ref{assump: prior condition}. If the SNR of the problem, defined as $\kappa_{\rm SNR}=\frac{\|\theta^\ast\|}{\beta}$, is sufficiently large so that the $\zeta$ to be defined below satisfies $\zeta>2$, then there exists constant $N = N(\theta^\ast, \beta, d, \pi_\theta)$, such that as long as the initialization satisfies
\begin{align*}
    &W_2(\delta_{\theta^\ast}, q_\theta^{(0)}) \leq R_W\\ &= \frac{C'\beta^{-2}(\zeta-2)}{\Big[K^2\big(d\beta^{-3}(\kappa_{\rm SNR}^2+1)^{3/2}+1\big) + K^3\big(d\beta^{-6}(\kappa_{\rm SNR}^2+1)^3+1\big)\Big]\,(\zeta+3)}
\end{align*}
for some constants $C'>0$, $\zeta$ and $n>N$, we have that
\begin{align*}
    W_2^2(\wht q_\theta, q_\theta^{(k)}) \leq \rho^kW_2^2(\wht q_\theta, q_\theta^{(0)})\quad\forall\, k\in\mb N
\end{align*}
holds with probability at least $1-\frac{3}{\log n}$. The contraction factor $\rho$ takes the form of
\begin{align*}
    1 - \frac{(\zeta-2)(3\zeta+2)}{4\zeta^2+\zeta - 2} \quad\mx{with} \ \ \zeta=\frac{(16+\kappa_{\rm SNR}^2)^{1/4}}{2174},
\end{align*}
as $n\to\infty$, which is decreasing in $\kappa_{\rm SNR}$.
\end{corollary}
The proof of the above corollary is postponed to Appendix~\ref{app: MR_proof}. Directly solving $\zeta > 2$ with $\zeta$ defined above provides a very loose bound of $\kappa_{\rm SNR} > 1.8\times 10^7$. In fact, the lower bound requirement can be substantially improved to a positive constant less than $10$ by numerically calculating an analytically intractable constant in our proof. More details are referred to the end of Appendix~\ref{app: MR_proof}.

Figure~\ref{fig: MR} shows the simulation results for implementing the mixture of regression model via MF-WGF. We set $\theta^\ast = (2, 3)$ in the data generative model and generate $n=200$ i.i.d. samples. For $\beta = 4$, we choose $\tau = 0.055$ and $\tau = 0.035$ for the SDE approach and $\tau = 0.2$ for the FA approach. For $\beta = 2$, we choose $\tau = 0.015$ and $\tau = 0.009$ for the SDE approach and $\tau = 0.2$ for the FA approach. We vary the SNR value $\|\theta^\ast\|/\beta$ by changing the noise variance $\beta^2$. 
Similar to the GMM example, we observe nearly straight lines for the numerical error $\log W_2^2(\widehat{q}_\theta, q_\theta^{(k)})$ versus the iteration count, indicating the exponential convergence of the algorithm. Moreover, a higher SNR corresponds to a smaller contraction rate as predicted by our theory. The statistical error $W_2^2(\delta_{\theta^\ast}, q_\theta^{(k)})$, indicated by solid curves in the plot, is at first dominated by the optimization error, but afterwards dominates the latter and stabilizes. 
Figure~\ref{fig: FA vs SDE, largenoise, MR} and Figure~\ref{fig: FA vs SDE, smallnoise, MR} compares the FA approach with the SDE approach and the numerical error under different noise levels. Figure~\ref{fig: diff_sample_size, MR} studies the affect of the sample size on the contraction rate and the statistical error when $\tau = 0.2$. Figure~\ref{fig: MF-WGF vs Gibbs, MR} shows the contour plots of the true posterior distribution of $\theta$ and its MF approximation.


\section{Discussion}
In this paper, we have proposed a general computational framework for realizing the mean-field variational approximation to Bayesian posteriors via Wasserstein gradient flows. We also applied the developed methods and theory to three concrete examples, linear regression model for Bayesian models without latent variables, and Gaussian mixture model and mixture of regression model for Bayesian latent variable models. Our analysis implies the exponential convergence of the algorithm given a good initialization.  

We also expect the development of this paper can be extended to other varitional approximation schemes, and the theoretical results to hold under weaker assumptions. For instance, as we briefly remarked in Section~\ref{sec:MF-WGF_analysis}, we may relax the global convexity condition on $U$ relative to the parameter $\theta$ into a local one, and the condition on the ``condition number'' $\kappa$ in Theorem~\ref{thm: main_theorem} into a weaker one. It is also possible to use a pre-conditioned Wasserstein distance with cost function as a weighted Euclidean norm square in constructing the discrete-time Wasserstein gradient flow, so that the local geometric structure can be captured while updating the variational distribution. This variant can be viewed as the generalization of the usual quasi-Newton's method to the Wasserstein space, which may enjoy a faster rate of algorithmic convergence. We leave all these threads into future directions.


\bibliographystyle{plain}
\bibliography{bibliography}


\newpage
\appendix
\begin{center}
{\bf\Large Supplementary Materials: Appendix}
\end{center}
\numberwithin{equation}{section}





\section{Background on optimal transport and Wasserstein gradient flow}\label{app:background}

\subsection{Optimal transport map}\label{app:OPTmap}
The optimization problem defining the $W_2$ metric is called the Kantorovich formulation of the optimal transport problem (KP) with quadratic cost. 
Problem (KP) always admits a solution~(Theorem 1.7,~\cite{santambrogio2015optimal}), called an optimal transport plan, which is not unique in general. Therefore, we cannot replace the inf in (KP) above by min. We use the notation $\Pi_o(\mu,\nu)$ to denote the set of all optimal transport plans. If one of the distributions, say $\mu$, is absolutely continuous with respect to the Lebesgue measure of $\mb R^d$, or $\mu\in\ms P_2^r(\mb R^d)$, then the optimal transport plan is unique and takes the form of $(\id,\, T^\ast)_\# \mu$ (Theorem 1.22,~\cite{santambrogio2015optimal}), implying $\nu= T^\ast_\#\mu$ (also see, e.g.~\cite{brenier1991polar,mccann1995existence}). Any map $T^\ast$ pushforwarding $\mu$ to $\nu$ such that $(\id,\, T^\ast)_\# \mu$ solves (KP) is called an optimal transport map from $\mu$ to $\nu$. In particular, if $\mu\in\ms P_2^r(\mb R^d)$, then the following Monge formulation of optimal transport problem (MP) with quadratic cost admits a unique solution as $T^\ast$,
\begin{align*}
   \inf_{T} \int_{\mb R^d} \|x - T(x)\|^2 \;\dd \mu(x), \quad \mbox{s.t.}\quad T_\#\mu=\nu. \qquad\mbox{(MP)}
\end{align*}
Moreover, the solution can be uniquely written as the gradient of a convex function $u^\ast$, i.e.~$T^\ast=\nabla u^\ast$.
It is worth noting that problem (MP) may admit no solution if $\mu$ contains singular components (e.g.~see Section 1.4 of~\cite{santambrogio2015optimal}), although the inf is always well-defined. If $\nu$ also belongs to $\ms P_2^r(\mb R^d)$, then $T_{\mu}^\nu$ is invertible and there exists another convex function $v^\ast$ such that $[\nabla v^\ast]_\# \nu= \mu$, where $v^\ast$ is the convex conjugate to $u^\ast$, i.e.~$v^\ast(x) = \sup_y\{\langle x,y\rangle - u^\ast(y)\}$.

To indicate the dependence on $\mu,\nu\in \ms P_2^r(\mb R^d)$, we will use the notation $T_{\mu}^\nu$ to denote the unique optimal transport map from $\mu$ to $\nu$; then $T_\mu^\nu = (T_\nu^\mu)^{-1}$. 
It is well-known that Wasserstein space $\mb W_2(\mb R^d)$ is a geodesic space with non-negative curvature in the Alexandrov sense~\citep{Lott2008}. For any $\mu_0, \mu_1\in\ms P_2(\mb R^d)$ and $\gamma\in \Pi_o(\mu_0, \mu_1)$, the (constant-speed) geodesic connecting $\mu_0$ and $\mu_1$ is $\mu_t = (\pi_t)_{\#}\gamma$ for $t\in[0,1]$, where $\pi_t = (1-t)\pi^0 + t\pi^1$. Here $\pi^0, \pi^1: \mb R^{2d}\to \mb R^d$ are the projection maps defined by $\pi^0(x, y) = x$ and $\pi^1(x, y) = y$. 
In particular, when $\mu_0,\mu_1\in \ms P_2^r(\mb R^d)$, the geodesic is uniquely given by $\mu_t= \big[(1-t) \,\id + t T_{\mu_0}^{\mu_1}\big]_\# \mu_0$ for $t\in[0,1]$.

\subsection{Subdifferential calculus in $\mb W_2(\mb R^d)$}\label{sec:subdiff}
Let $\m F:\, \ms P_2^r(\mb R^d)\to (-\infty,\infty]$ be a proper and lower semicontinuous functional on $\ms P_2^r(\mb R^d)$. 
We say that $\xi\in L_2(\mu;\mb R^d)$ belongs to the Fr\'{e}chet subdifferential $\partial \m F(\mu)$ of $\m F$ at $\mu$ if for any $\nu\in \ms P_2^r(\mb R^d)$,
\begin{align}\label{eqn: def_subdifferential}
\m F(\nu) \geq \m F(\mu) + \int_{\mb R^d} \langle \xi(x), T_\mu^\nu(x) - x\rangle\, \dd \mu(x) + o\big(W_2(\mu,\nu)\big),\ \ \mx{as}\ W_2(\mu,\nu)\to 0.
\end{align}
In addition, if $\xi\in \partial \m F(\mu)$ also satisfies the following for any (transport) map $T:\,\mb R^d\to\mb R^d$, 
\begin{align*}
\m F(T_\# \mu) \geq \m F(\mu) + \int_{\mb R^d} \langle \xi(x), T(x) - x\rangle\, \dd \mu(x) + o\big(\|T- \id\|_{L^2(\mu;\mb R^d)}\big),\ \ \mx{as }\|T- \id\|_{L^2(\mu;\mb R^d)}\to 0,
\end{align*}
then $\xi$ is called a strong subdifferential of $\m F$ at $\mu$. For the Wasserstein space, the ``differentiablity'' of a functional $\m F$ can be an overly stringent property, but subdifferential exists under fairly mild conditions (Chapter 10,~\cite{ambrosio2008gradient}).
Similar to the Euclidean case, the notion of subdifferential generalizes ``gradient'' and is useful in characterizing local minima of functional $\m F$ and rigorously defining a gradient flow. For example, the first order optimality condition based on subdifferential is as follows: if $\mu$ is a local minimum of $\m F$ over $\PX$, then $0$ (the zero map from $\mb R^d$ to $\mb R^d$) belongs to $\partial \m F(\mu)$. 

\smallskip
\noindent \emph{Connection with first variation.}
Another common technique for finding maxima and minima of generic functionals is based on the first variation (or Gateaux derivative) from the field of calculus of variations. When specialized to all regular measures from $\PX$ in the Wasserstein space, a map $\frac{\delta\m F}{\delta\mu}(\mu):\, \mb R^d \to \mb R$ is called the first variation of a functional $\m F : \ms P^r(\mb R^d) \to \mb R$ at $\mu\in\ms P^r(\mb R^d)$, if
\begin{displaymath}
\frac{\dd}{\dd\varepsilon} \m F(\mu +\varepsilon\chi)\bigg|_{\varepsilon=0} = \int_{\mb R^d}\frac{\delta\m F}{\delta\mu}(\mu)\,\dd\chi
\end{displaymath}
for any perturbation $\chi = \tilde{\mu} - \mu$ with $\tilde{\mu} \in \ms P^r(\mb R^d)$. 
$\frac{\delta\m F}{\delta\mu}$ can be regarded as an infinite-dimensional gradient in $\PX$. A first order optimality condition of functional $\m F$ in $\PX$ using the first variation can be stated as follows: if $\mu^\ast\in\ms P_r^2(\mb R^d)$ is a local minimum of $\m F$ and $\frac{\delta\m F}{\delta\mu}(\mu^\ast)$ is a measurable map, then $\frac{\delta\m F}{\delta\mu}(\mu^\ast)$ attains its essential infimum a.e.~on $\{\mu^\ast > 0\}$ (that is, it is a constant $\mu^\ast$-almost everywhere,~See Proposition 7.20 in \cite{santambrogio2015optimal} for a proof).
The following lemma provides a connection between first variation and subdifferential, whose proof is provided in Appendix~\ref{sec:proof_lem:FV_SD}.

\begin{lemma}\label{lem:FV_SD}
If $\mu\in\PX$ satisfies $\m F(\mu)<\infty$ and $\xi$ is a strong subdifferential of $\m F$ at $\mu$, then
\begin{align*}
    \xi(x) = \nabla \frac{\delta \m F}{\delta \mu}(\mu)(x) \quad \mx{for $\mu$-a.e.}\  x\in\mb R^d.
\end{align*}
Conversely, if $\m F$ is Fr\'{e}chet differentiable at $\mu$ relative to the $W_2$ metric (which implies the Gateaux differentiability with the same derivative), that is,
$$\m F(\nu) \geq \m F(\mu)+  \int_{\mb R^d}\frac{\delta\m F}{\delta\mu}(\mu)\,\dd(\nu-\mu) +o\big(W_2(\mu,\nu)\big)\quad \mx{as } W_2(\mu,\nu)\to 0,
$$
then $\xi(x) = \nabla \frac{\delta \m F}{\delta \mu}(\mu)(x)$ is a subdifferential of $F$ at $\mu$.
\end{lemma}

\subsection{Euclidean gradient flow}\label{app:EGF}
To provide motivation for defining the gradient flow in the Wasserstein space $\mathbb{W}_2(\mathbb{R}^d)$ as discussed in Section~\ref{section: gradient_flow_in_probability_space}, let us draw a comparison with an equivalent definition of the gradient flow in the Euclidean space $\mathbb{R}^d$ that can be extended to a general metric space.
Let $F: \mb R^d\to\mb R$ be a smooth function and $x^0\in\mb R^d$ a point. 
Intuitively, a gradient flow initialized at $x_0$ is an evolution starting from $x^0$ and always moving in the direction where $F$ decreases the most (a.k.a.~steepest descent) and thus gradually minimizing $F$. Rigorously, it is the solution of the following ordinary differential equation (ODE)
\begin{align}\label{eqn:Euclidean_GF}
    \dot{x}_t = -\nabla F(x_t), \quad\mx{for }t>0, \quad \mx{with }x_0=x^0,
\end{align}
where the negative gradient $-\nabla F(x)$ gives the steepest direction towards a (local) minimizer.
This is a standard Cauchy problem which has a unique solution if $\nabla F$ is Lipschitz continuous.
In particular, when $F$ is strictly convex, the gradient flow has exponential convergence to the unique global minimizer. Unfortunately, such a definition of gradient flow via ODE will encounter a number of obstacles when adapted to $\mb W_2(\mb R^d)$. There is an equivalent perspective that is relatively easy to generalize and also provides a numeric scheme for practically approximating the Euclidean gradient flow. 
More precisely, one can define or view the continuous-time gradient flow~\eqref{eqn:Euclidean_GF} as the weak convergence limit of the following iterative variational scheme, also called minimization movements, as the step size $\tau$ tends to zero,
\begin{align}\label{eqn: gradient_flow_in_Rn}
    x_{k+1}^{\tau} \in \argmin_{x\in\mb R^n} F(x) + \frac{1}{2\tau}\|x_k^\tau - x\|^2, \quad\mx{for $k\geq 0$} \quad \mx{with }x_0^\tau = x^0.
\end{align}
Under some mild conditions on $F$, such as eigenvalues of Hessian $\nabla^2 F$ bounded from below, above minimization problem admits a unique solution for all sufficiently small $\tau$.
It is also worth noting that the first order optimality condition for solving above minimization problem is exactly the discrete-time implicit Euler scheme for the ODE~\eqref{eqn:Euclidean_GF}.
For any fixed time horizon $T$, by letting $\tau = \frac{T}{n}$ and $\tilde{x}^{(n)}_t = x^\tau_{[t/\tau]}$, the piecewise constant interpolation $\tilde{x}^{(n)}_t$ uniformly converges to the unique solution to the ODE~\eqref{eqn: gradient_flow_in_Rn} on $[0, T]$ as $n\to\infty$ given $\nabla F$ is Lipschitz continuous.

\subsection{Convexity along generalized geodesics}\label{app:convexity}
Convexity has a particularly prominent role in proving the convergence and deriving an explicit convergence rate of gradient flows in the Euclidean space.  To analyze the optimization landscape of minimizing a proper and lower semicontinuous functional $\m F:\, \ms P_2^r(\mb R^d)\to (-\infty,\infty]$ on $\ms P_2^r(\mb R^d)$, it would be helpful to properly extend the notion of convexity to the Wasserstein space $\mb W_2(\mb R^d)$. 
In the Euclidean space, for every $x_0, x_1, z\in\mb R^d$, we have
\begin{align*}
    t\,\|x_1-z\|^2 + (1-t)\,\|x_0-z\|^2 - \|x_t - z\|^2 \ = \  t(1-t)\,\|x_0-x_1\|^2 \ \geq\ 0,
\end{align*}
where $x_t=tx_0+(1-t)x_1$ for $t\in[0,1]$ is the (constant-speed) geodesic connecting $x_0$ and $x_1$.
This indicates that the squared distance $\|\cdot \,- \,z\,\|^2$\, is convex (along geodesics) for all $z\in\mb R^d$, which is essential to study the basic regularity properties of gradient flows in $\mb R^d$ since geodesics are the (locally) shortest paths that locally interpolates the gradient flow. 
However, the geodesic in a general length space~\cite{burago2001course,mccann1997convexity}, such as $\mb W_2(\mb R^d)$, does not have this property; and we need to define the convexity of a functional along different interpolating curves, along which the $W_2^2(\cdot,\,\mu)$ exhibits a nicer behavior. 
This motivates the following definition of convexity along generalized geodesics in $\mb W_2(\mb R^d)$ (Chapter 9 of~\cite{ambrosio2008gradient}).

Let $\pi^1$, $\pi^2$, $\pi^3$ be the projections onto the first, second and third coordinate in $\big(\mb R^d)^3$, respectively, and $\pi_t^{2\to 3} = (1-t)\pi^2+t\pi^3$ for $t\in[0,1]$. Let $\Pi(\mu_1,\mu_2,\mu_3)$ denote the space of all joint distributions (couplings) over $\big(\mb R^d\big)^3$ with marginals $\mu_1,\mu_2$ and $\mu_3\in\ms P_2(\mb R^d)$.

\begin{definition}[$\lambda$-convexity along generalized geodesics]
A generalized geodesic joining $\mu^2$ to $\mu^3$ (with base $\mu^1$) is a curve in $\PX$ of the type
\begin{align*}
\mu_t^{2\to 3} = (\pi_t^{2\to 3})_\# \bm \mu, \quad t\in[0,1],
\end{align*}
where $\bm \mu\in \Pi(\mu^1,\mu^2,\mu^3)$, $\pi^{1,2}_\# \bm \mu \in \Pi_o(\mu^1,\mu^2)$ and $\pi^{1,3}_\# \in\Pi_o(\mu^1,\mu^2)$.
A functional $\m F:\, \ms P_2^r(\mb R^d)\to (-\infty,\infty]$ is said to be $\lambda$-convex along generalized geodesics in $\PX$ if for any $\mu^1$, $\mu^2$, $\mu^3\in \PX$ there exists a generalized geodesic $\mu_t^{2\to 3}$ induced by a plan $\bm \mu\in \Pi(\mu^1,\mu^2,\mu^3)$ such that
\begin{align*}
\m F(\mu_t^{2\to 3}) \leq (1-t)\, \m F(\mu^2) + t \,\m F(\mu^3) -\frac{\lambda}{2}\, t(1-t)\, W_{\bm \mu}^2(\mu^2,\mu^3),\quad\forall t\in[0,1],
\end{align*}
where $W_{\bm \mu}^2(\mu^2,\mu^3):\,=\int_{(\mb R^d)^3}\|x_3-x_2\|^2\,\dd\bm \mu(x_1,x_2,x_3) \geq W_2^2(\mu^2,\mu^3)$.
\end{definition}

If a functional $\m F$ is $\lambda$-convex along all generalized geodesics in $\PX$ whose starting point $\mu^2=\mu^1$ is the same as the base point $\mu^1$, then $\m F$ is said to be $\lambda$-convex along geodesics (also called displacement convexity,~\cite{mccann1997convexity}). Therefore, the convexity along generalized geodesics is stronger than that along geodesics. 
As another remark, in dimensions greater than one, $\frac{1}{2}W_2^2(\cdot, \mu)$ is not $1$-convex along geodesics (in fact, it satisfies the opposite inequality); however, it is $1$-convex along all generalized geodesics with base point $\mu^1=\mu$. This property is important to define and study the convergence of Wasserstein gradient flows.

\smallskip
\noindent \emph{Connection with subdifferentiability.}
Convexity along (generalized) geodesics strengthens the local notion of Fr\'{e}chet subdifferentiability described in Section~\ref{sec:subdiff} into a global one, similar to the subdifferential for convex functions in $\mb R^d$.
\begin{lemma} [Section~10.1.1~in~\cite{ambrosio2008gradient}]
Suppose $\m F:\, \ms P_2^r(\mb R^d)\to (-\infty,\infty]$ is $\lambda$-convex along geodesics. Then a vector $\mb\xi \in L^2(\mu;\mb R^d)$ belongs to the Fr\'{e}chet subdifferential of $\m F$ as $\mu$ if and only if
\begin{align*}
\m F(\nu) \geq \m F(\mu) + \int_{\mb R^d} \langle \xi(x), T_\mu^\nu - x\rangle\, \dd \mu(x) + \frac{\lambda}{2} W_2^2(\mu,\nu) \quad \forall\nu\in \PX.
\end{align*}
In particular, if $\xi_i\in \partial \m F(\mu_i)$, $i=1,2$, and $T=T_{\mu_1}^{\mu_2}$ is the optimal transport map, then
\begin{align*}
\int_{\mb R^d} \big\langle \xi_2\big(T(x)\big) -\xi_1(x), T(x) - x\big\rangle \, \dd\mu_1(x) \geq \lambda W_2^2(\mu_1,\mu_2).
\end{align*}
\end{lemma}

If the stronger property of convexity along generalized geodesics holds, then we can replace the optimal transport map $T_\mu^\nu$ in one direction of the above lemma through any transport map $\bm  t^{\mu^3}_{\mu^1} \circ T_{\mu^2}^{\mu^1}$ that interpolates $\mu^2$ and $\mu^3$ via any intermediate probability measure $\mu^1$, as in the following lemma. Due to this additional flexibility of choosing the interpolating curve, such a lemma will play a crucial role when analyzing the convergence of a discretized Wasserstein gradient flow (c.f.~Theorem~\ref{thm: implicit_WGF_conv}), and a proof is provided in Appendix~\ref{sec:proof_Lem:stong_conx_generalized}.

\begin{lemma}\label{Lem:stong_conx_generalized}
Suppose $\m F:\, \ms P_2^r(\mb R^d)\to (-\infty,\infty]$ is $\lambda$-convex along generalized geodesics, and $\mu^1$, $\mu^2$, $\mu^3\in\PX$. If $\xi$ is a strong Fr\'{e}chet subdifferential of $\m F$ at $\mu^2$, then 
\begin{align*}
\m F(\mu^3) \geq  \m F(\mu^2) + \int_{\mb R^d} \big\langle \xi\big(T_{\mu^1}^{\mu^2}(x)\big), T_{\mu^1}^{\mu^3}(x) - T_{\mu^1}^{\mu^2}(x)\big\rangle\, \dd \mu^1(x) + \frac{\lambda}{2} W_2^2(\mu^2,\mu^3).
\end{align*}
\end{lemma}

The following results provide conditions under which the two constituting functionals of the KL divergence functional $\m F_{\rm KL}$ are convex along generalized geodesics.
\begin{lemma}[Entropy functional, Proposition 9.3.9~in~\cite{ambrosio2008gradient}]\label{lemma:entropy}
Consider the entropy functional $\m E:\PX\to (-\infty, \infty]$, $\m E(\rho) = \int \log \rho(x)\,\dd \rho(x)$ for all $\rho\in \PX$. Then $\ms F$ is convex along generalized geodesics in $\PX$.
\end{lemma}

Recall that the $\lambda$-convexity of a function $V:\mb R^d\to\mb R$ over $\mb R^d$ means for all $x_1,x_2\in\mb R^d$,
\begin{align*}
V\big((1-t) x_1+tx_2\big) \leq (1-t)V(x_1) + tV(x_2) -\frac{\lambda}{2}\,t(1-t)\, \|x_1-x_2\|^2,\ \forall t\in[0,1].
\end{align*}

\begin{lemma}[Potential energy functional, Proposition 9.3.2~in~\cite{ambrosio2008gradient}]\label{lemma:potential}
If potential $V:\mb R^d\to\mb R$ is a $\lambda$-convex function over $\mb R^d$, then the potential energy functional $\m V:\PX\to \mb R$, $\m V(\mu) = \int_\mb R^d V(x)\,\dd\mu(x)$ for $\mu\in\PX$, is $\lambda$-convex along generalized geodesics.
\end{lemma}

\section{Proofs related to subdifferential calculus and gradient flows in Wasserstein space}
In this appendix, we collect all proofs of the results appearing in Section~\ref{sec: Background and examples} of the main paper about subdifferential calculus and gradient flows in the Wasserstein space $\mb W_2(\mb R^d)$.

\subsection{Proof of Lemma \ref{lem:FV_SD}}\label{sec:proof_lem:FV_SD}
The first part follows from Lemma 10.4.1 of~\cite{ambrosio2008gradient}.
To prove the second part, we use the Fr\'{e}chet differentiability property of functional $\m F$ to obtain
\begin{align*}
    \m F(\nu)\geq&\, \m F(\mu)+  \int_{\mb R^d}\frac{\delta\m F}{\delta\mu}(\mu)\,\dd(\nu-\mu) +o\big(W_2(\mu,\nu)\big)\\
    =&\,\m F(\mu) + \int_{\mb R^d} \frac{\delta\m F}{\delta\mu}(\mu)(x)\, \dd \nu(x) - \int_{\mb R^d} \frac{\delta\m F}{\delta\mu}(\mu)(x) \,\dd \mu(x) + o\big(W_2(\mu,\nu)\big)\\
    \overset{\ri}{=} &\, \m F(\mu) + \int_{\mb R^d} \Big[\frac{\delta\m F}{\delta\mu}(\mu)\big(T_{\mu}^\nu (x)\big) -  \frac{\delta\m F}{\delta\mu}(\mu)(x)\Big] \, \dd \mu(x) + o\big(W_2(\mu,\nu)\big)\\
    \overset{\rii}{=} &\, \m F(\mu) + \int_{\mb R^d} \Big\langle \nabla \frac{\delta\m F}{\delta\mu}(\mu)(x),\, T_{\mu}^\nu (x) - x\Big\rangle \, \dd \mu(x) + o\big(W_2(\mu,\nu)\big) 
\end{align*}
where step (i) follows by the change of variable $x\mapsto T_{\mu}^\nu(x)$ to the first integral, and step (ii) follows by applying the Taylor expansion to $\frac{\delta\m F}{\delta\mu}(\mu)(x)$ and using the fact that 
$$\mb E_{\mu}\big[\|T_{\mu}^\nu (x) - x\|\big] \leq \sqrt{\mb E_{\mu}\big[\|T_{\mu}^\nu (x) - x\|^2\big]}  = W^2_2(\mu,\nu),
$$
since $T_\mu^\nu$ is the optimal transport map from $\mu$ to $\nu$. Therefore, $\nabla \frac{\delta\m F}{\delta\mu}(\mu)$ is a subdifferential of $\m F$ at $\mu$.

\subsection{Proof of Lemma~\ref{Lem:stong_conx_generalized}}\label{sec:proof_Lem:stong_conx_generalized}
By the $\lambda$-convexity along generalized geodesics, we have
\begin{align*}
\frac{\m F(\mu_t^{2\to 3}) - \m F(\mu^2)}{t} \leq  \m F(\mu^3) - \m F(\mu^2) - \frac{\lambda}{2}\, (1-t) W_{\bm \mu}^2(\mu^2,\mu^3),\quad\forall t\in(0,1].
\end{align*}
By using the identity $\mu_t^{2\to 3} = \big((1-t) T_{\mu^1}^{\mu^2}+t T_{\mu_1}^{\mu^3}\big)_\# \mu^1 = \big((1-t) \,\id+t \,T_{\mu_1}^{\mu^3}\circ T_{\mu^2}^{\mu^1} \big)_\# \mu^2$ and that fact that $\xi$ is a strong sub-differential of $\m F$ at $\mu^2$, we obtain
\begin{align*}
\m F(\mu_t^{2\to 3}) - \m F(\mu^2)&  \geq t\int_{\mb R^d} \big\langle \xi(x), T_{\mu_1}^{\mu^3}\big(T_{\mu^2}^{\mu^1}(x)\big)-x\big\rangle\,\dd \mu^2(x) + o(t)\\
& = t\int_{\mb R^d} \big\langle \xi\big(T_{\mu^1}^{\mu^2}(x)\big), T_{\mu_1}^{\mu^3}(x)-t_{\mu^1}^{\mu^2}(x)\big\rangle\,\dd \mu^1(x) + o(t), \quad\mx{as }t\to 0_+.
\end{align*}
By combining the previous two displays and taking the limit as $t\to 0_+$, we obtain
\begin{align*}
\int_{\mb R^d} \big\langle \xi\big(T_{\mu^1}^{\mu^2}(x)\big), T_{\mu_1}^{\mu^3}(x)-t_{\mu^1}^{\mu^2}(x)\big\rangle\,\dd \mu^1(x) &\leq \m F(\mu^3) - \m F(\mu^2) - \frac{\lambda}{2}\, W_{\bm \mu}^2(\mu^2,\mu^3)\\
&\leq \m F(\mu^3) - \m F(\mu^2) - \frac{\lambda}{2}\, W_2^2(\mu^2,\mu^3),
\end{align*}
which implies the claimed inequality.

\subsection{Proof of Theorem~\ref{thm: implicit_WGF_conv}}\label{app:proof_thm: implicit_WGF_conv}
We first collect some results of subdifferential calculus that are will be used in the proof. For a proper and lower semicontinuous functional $\m F: \ms P_2^r(\mb R^d)\to (-\infty, \infty]$, define functional $\m F_{\tau, \mu}$ as
\begin{displaymath}
\m F_{\tau,\mu}(\nu) = \m F(\nu) + \frac{1}{2\tau}W_2^2(\nu, \mu).
\end{displaymath}
Assume that for some $\tau_\ast > 0$, $\m F_{\tau,\mu}$ admits at least a minimum point $\mu_\tau$, for all $\tau\in(0, \tau_\ast)$ and $\mu\in\ms P_2^r(\mb R^d)$. The map $\mu\mapsto \mu_\tau$ can be seen as a generalization from the usual Euclidean space to $\ms P_2^r(\mb R^d)$ of the proximal operator associated with functional $\m F$ with step $\tau$, where the Euclidean distance is replaced by the Wasserstein distance.
\begin{lemma}[lemma 10.1.2 in \cite{ambrosio2008gradient}]\label{lem: appendB1}
Let $\mu_\tau$ be a minimum of $\m F_{\tau,\mu}$, then
\begin{displaymath}
\frac{t_{\mu_\tau}^\mu - \id}{\tau} \in \partial\m F(\mu_\tau)
\end{displaymath}
is a strong subdifferential of $\m F$ at $\mu_\tau$.
\end{lemma}

Now let us return to the proof of Theorem \ref{thm: implicit_WGF_conv}.
Let $\xi = \frac{1}{\tau}(t_{\mu_\tau}^\mu - \id) \in\partial\m F(\mu_\tau)$ denote the strong subdifferential of $\m F$ at $\mu_\tau$ implied by Lemma \ref{lem: appendB1}, so that $t_{\mu_\tau}^\mu = \id + \tau\xi$ is the optimal transport map from $\mu_\tau$ to $\mu$. Notice that $(t_\mu^{\mu_\tau}, t_\mu^\pi)_\#\mu$ forms a coupling between $\mu_\tau$ and $\pi$. Consequently, we have
\begin{align*}
W_2^2(\mu_\tau, \pi) &\leq \|t_\mu^\pi - t_\mu^{\mu_\tau}\|^2_{L^2(\mu; \mb R^d)} = \|t_\mu^\pi - \id + \id - t_\mu^{\mu_\tau}\|_{L^2(\mu; \mb R^d)}^2\\
&= \|t_\mu^\pi - \id\|_{L^2(\mu; \mb R^d)}^2 + \|t_\mu^{\mu_\tau} - \id\|_{L^2(\mu; \mb R^d)}^2 - 2\big\langle t_{\mu}^\pi - \id, t_\mu^{\mu_\tau} - \id\big\rangle_{L^2(\mu; \mb R^d)}\\
&= \|t_\mu^\pi - \id\|_{L^2(\mu; \mb R^d)}^2 - \|t_\mu^{\mu_\tau} - \id\|_{L^2(\mu; \mb R^d)}^2 + 2\big\langle t_\mu^{\mu^\tau} - t_\mu^\pi, t_\mu^{\mu_\tau} - \id\big\rangle_{L^2(\mu;\mb R^d)}.
\end{align*}
Using the identity that $t_\mu^{\mu_\tau}\circ t_{\mu_\tau}^\mu = \id$ along with $t_{\mu_\tau}^\mu = \id + \tau\xi$, we obtain $t_\mu^{\mu_\tau} = \id - \tau\xi\circ t_{\mu_\tau}^\mu$. Therefore, the last term in the preceding display can be expressed as
\begin{align*}
2\big\langle t_\mu^{\mu^\tau} - t_\mu^\pi, \,t_\mu^{\mu_\tau} - \id\big\rangle_{L^2(\mu;\mb R^d)}
&= 2\tau\int_{\mb R^d}\big\langle\xi\circ t_\mu^{\mu_\tau}, t_\mu^\pi - t_\mu^{\mu_\tau}\big\rangle\,\dd\mu\\
&\leq 2\big(\m F(\pi) - \m F(\mu_\tau)\big) - \tau\lambda W_2^2(\mu_\tau, \pi),
\end{align*}
where in the last step we applied Lemma~\ref{Lem:stong_conx_generalized} with $\mu^1 = \mu, \mu^2 = \mu_\tau$ and $\mu^3 = \pi$. Combining the two preceding  displays and using the identities $\|t_\mu^\pi - \id\|_{L^2(\mu, \mb R^d)}^2 = W_2^2(\mu, \pi)$ and $\|t_\mu^{\mu_\tau} - \id\|_{L^2(\mu, \mb R^d)}^2 = W_2^2(\mu, \mu_\tau)$ yields the claimed inequality.

\section{Proof of main theorems}\label{sec: proof2}
In this appendix, we provide proofs of the results from Section~\ref{sec: main results} of the main paper about the concentration of MF approximation and the geometric convergence of MF-WGF algorithm.

\subsection{Proof of Theorem~\ref{thm: MFVI_posterior_convergence_rate}}\label{app: concentration_MFVI}
We start with the following lemma which characterizes the variational approximation $(\wht q_1, \cdots, \wht q_m)$. The proof will be provided in Section~\ref{app: proof_C.1} in Appendix~\ref{app:tech_results}.
\begin{lemma}\label{lem: functional_Wn}
The variational approximation $(\wht q_1, \cdots, \wht q_m)$ is a minimizer of the functional
\begin{align*}
\widetilde W_n(\rho_1, \cdots, \rho_m) = \int_\Theta \sum_{i=1}^n\log\frac{p(X_i\,|\,\theta^\ast)}{p(X_i\,|\,\theta)}\,\dd\rho_1(\theta_1)\cdots\dd\rho_m(\theta_m) + \KL(\rho_1\otimes\cdots\otimes\rho_m\,\|\,\Pi_\theta).
\end{align*}
In addition, we have
\begin{align}\label{eqn: FOC_qj}
\wht q_j(\theta_j) = \frac{\exp\big\{\int_{\Theta_{-j}}\log\pi_\theta(\theta) + \sum_{i=1}^n \log p(X_i\,|\,\theta)\,\dd\wht q_{-j}\big\}}{\int_{\Theta_j}\exp\big\{\int_{\Theta_{-j}}\log\pi_\theta(\theta) + \sum_{i=1}^n \log p(X_i\,|\,\theta)\,\dd\wht q_{-j}\big\}\,\dd\theta_j}
\end{align}
\end{lemma}

Let $B_j(\varepsilon) = \{\theta_j\in\Theta_j: \|\theta_j - \theta_j^\ast\| \leq \varepsilon\}$ be the ball with radius $\varepsilon$ centered at $\theta_j^\ast$ in $\Theta_j$. We have
\begin{align*}
&\quad\,\wht{Q}_\theta\big(\|\theta_j - \theta_j^\ast\| \leq \varepsilon, \quad \forall\,j\in[m]\big) = \prod_{j=1}^m \wht Q_j\big(\theta_j\in B_j(\varepsilon)\big)\\
&\stackrel{\ri}{=} \prod_{j=1}^m \frac{\int_{B_j(\varepsilon)}\exp\big\{\int_{\Theta_{-j}}\log\pi_\theta(\theta) + \sum_{i=1}^n \log p(X_i\,|\,\theta)\,\dd\wht q_{-j}\big\}\,\dd\theta_j}{\int_{\Theta_j}\exp\big\{\int_{\Theta_{-j}}\log\pi_\theta(\theta) + \sum_{i=1}^n \log p(X_i\,|\,\theta)\,\dd\wht q_{-j}\big\}\,\dd\theta_j}\\
&\stackrel{\rii}{=} \prod_{j=1}^m\frac{\int_{B_j(\varepsilon)}\exp\big\{\int_{\Theta_{-j}}\log\frac{\pi_\theta(\theta)}{\widetilde Q_j(\theta_j)\wht q_{-j}(\theta_{-j})} + \sum_{i=1}^n\log\frac{p(X_i\,|\,\theta)}{p(X_i\,|\,\theta^\ast)}\,\dd\wht q_{-j}\big\}\,\dd\widetilde Q_j}{\int_{\Theta_j}\exp\big\{\int_{\Theta_{-j}}\log\frac{\pi_\theta(\theta)}{\widetilde Q_j(\theta_j)\wht q_{-j}(\theta_{-j})} + \sum_{i=1}^n\log\frac{p(X_i\,|\,\theta)}{p(X_i\,|\,\theta^\ast)}\,\dd\wht q_{-j}\big\}\,\dd\widetilde Q_j}.
\end{align*}
Here, (i) is by the characterization~\eqref{eqn: FOC_qj} of $\wht q_\theta$; (ii) is by adding a same $\theta_j$-independent term to the exponents of the numerator and the denominator. Therefore, we only need to find an upper bound of 
\begin{align*}
\frac{\widetilde N_j(\varepsilon)}{\widetilde D_j} \coloneqq \frac{\int_{B_j(\varepsilon)^c}\exp\big\{\int_{\Theta_{-j}}\log\frac{\pi_\theta(\theta)}{\widetilde Q_j(\theta_j)\wht q_{-j}(\theta_{-j})} + \sum_{i=1}^n\log\frac{p(X_i\,|\,\theta)}{p(X_i\,|\,\theta^\ast)}\,\dd\wht q_{-j}\big\}\,\dd\widetilde Q_j}{\int_{\Theta_j}\exp\big\{\int_{\Theta_{-j}}\log\frac{\pi_\theta(\theta)}{\widetilde Q_j(\theta_j)\wht q_{-j}(\theta_{-j})} + \sum_{i=1}^n\log\frac{p(X_i\,|\,\theta)}{p(X_i\,|\,\theta^\ast)}\,\dd\wht q_{-j}\big\}\,\dd\widetilde Q_j},
\end{align*}
since we have
\begin{align}\label{eqn: MFVI_tail_prob}
\wht{Q}_\theta\big(\exists\,j\in[m]\,\,s.t.\,\,\|\theta_j - \theta_j^\ast\| > \varepsilon\big) = 1 - \prod_{j=1}^m\Big[1- \frac{\widetilde N_j(\varepsilon)}{\widetilde D_j}\Big] \leq \sum_{j=1}^m \frac{\widetilde N_j(\varepsilon)}{\widetilde D_j}.
\end{align}
Define the event
\begin{align*}
\widetilde{\m A}_n = \Big\{\frac{1}{\widetilde Q(\widetilde\Theta)}\int_{\widetilde\Theta}\sum_{i=1}^n\log\frac{p(X_i\,|\,\theta)}{p(X_i\,|\,\theta^\ast)}\,\dd\widetilde Q \leq -( c_4 + 1)n\varepsilon_n^2\Big\}.
\end{align*}
Then, we have the decomposition
\begin{align}
\mb E_{\theta^\ast}\Big[\frac{\widetilde N_j(\varepsilon)}{\widetilde D_j}\Big] 
&= \mb E_{\theta^\ast}\Big[ \phi_n \frac{\widetilde N_j(\varepsilon)}{\widetilde D_j}\Big] + \mb E_{\theta^\ast}\Big[(1- \phi_n)1_{\widetilde{\m A}_n}\frac{\widetilde N_j(\varepsilon)}{\widetilde D_j}\Big] + \mb E_{\theta^\ast}\Big[(1- \phi_n)1_{\widetilde{\m A}_n^c}\frac{\widetilde N_j(\varepsilon)}{\widetilde D_j}\Big]\nonumber\\
&\leq \mb E_{\theta^\ast}[ \phi_n] + \mb P_{\theta^\ast}\big(\widetilde{\m A}_n\big) + \mb E_{\theta^\ast}\Big[(1- \phi_n)1_{\widetilde{\m A}_n^c}\frac{\widetilde N_j(\varepsilon)}{\widetilde D_j}\Big],\label{eqn: MFVI_decomp}
\end{align}
where $\phi_n$ is the test function in Assumption~\ref{assump: test condition}.

To bound the denominator $\widetilde D_j$, note that
\begin{align*}
\widetilde D_j
&= \int_{\Theta_j} \exp\Big\{\int_{\Theta_{-j}}\log\frac{\pi_\theta(\theta)}{\widetilde Q_j(\theta_j)\wht q_{-j}(\theta_{-j})} + \sum_{i=1}^n \log\frac{p(X_i\,|\,\theta)}{p(X_i\,|\,\theta^\ast)}\,\dd\wht q_{-j}(\theta_{-j})\Big\}\,\dd\widetilde Q_j(\theta_j)\\
&= \int_{\Theta_j}\exp\Big\{\int_{\Theta_{-j}}\log \pi_\theta(\theta) + \sum_{i=1}^n \log p(X_i\,|\,\theta)\,\dd\wht q_{-j}(\theta_{-j})\Big\}\,\dd\theta_j\\
&\qquad\qquad\qquad\qquad\cdot\exp\Big\{\int_{\Theta_{-j}}-\log\wht q_{-j}(\theta_{-j}) - \sum_{i=1}^n \log p(X_i\,|\,\theta^\ast)\,\dd\wht q_{-j}(\theta_{-j})\Big\}\\
&\stackrel{\ri}{=} \wht q_j(\theta_j)^{-1} \exp\Big\{\int_{\Theta_{-j}}\log\pi_\theta(\theta) + \sum_{i=1}^n\log p(X_i\,|\,\theta)\,\dd\wht q_{-j}(\theta_{-j})\Big\}\\
&\qquad\qquad\qquad\qquad\cdot\exp\Big\{\int_{\Theta_{-j}}-\log\wht q_{-j}(\theta_{-j}) - \sum_{i=1}^n \log p(X_i\,|\,\theta^\ast)\,\dd\wht q_{-j}(\theta_{-j})\Big\}\\
&= \exp\Big\{\int_{\Theta_{-j}} \log\frac{\pi_\theta(\theta)}{\wht q_\theta(\theta)} + \sum_{i=1}^n \log\frac{p(X_i\,|\,\theta)}{p(X_i\,|\,\theta^\ast)}\,\dd\wht q_{-j}(\theta_{-j})\Big\}\\
&\stackrel{\rii}{=} \exp\Big\{\int_{\Theta} \log\frac{\pi_\theta(\theta)}{\wht q_\theta(\theta)} + \sum_{i=1}^n \log\frac{p(X_i\,|\,\theta)}{p(X_i\,|\,\theta^\ast)}\,\dd\wht q(\theta)\Big\}\\
&= \exp\{- \widetilde W_n(\wht q_1, \cdots, \wht q_m)\}.
\end{align*}
Here, (i) is due to the characterization~\eqref{eqn: FOC_qj}; (ii) is because the quantity inside the exponent is not a function of $\theta_j$. By Lemma~\ref{lem: functional_Wn}, we know $(\wht q_1, \cdots, \wht q_m) = \argmin\widetilde W_n(\rho_1, \cdots, \rho_m)$, which implies that
\begin{align*}
\widetilde W_n(\wht q_1, \cdots, \wht q_m) 
&\leq \widetilde W_n\Big(\frac{\widetilde Q_1 1_{\widetilde\Theta_1}}{\widetilde Q_1(\widetilde \Theta_1)}, \cdots, \frac{\widetilde Q_m 1_{\widetilde\Theta_m}}{\widetilde Q_m(\widetilde \Theta_m)}\Big)\\
&=\frac{1}{\widetilde Q(\widetilde\Theta)}\int_{\widetilde \Theta}\sum_{i=1}^n\log\frac{p(X_i\,|\,\theta^\ast)}{p(X_i\,|\,\theta)}\,\dd\widetilde Q(\theta) + \KL\Big(\frac{\widetilde Q_1 1_{\widetilde\Theta_1}}{\widetilde Q_1(\widetilde \Theta_1)}\otimes\cdots\otimes\frac{\widetilde Q_j 1_{\widetilde\Theta_j}}{\widetilde Q_j(\widetilde \Theta_j)}\,\Big\|\,\Pi_\theta\Big)\\
&=\frac{1}{\widetilde Q(\widetilde\Theta)}\int_{\widetilde \Theta}\sum_{i=1}^n\log\frac{p(X_i\,|\,\theta^\ast)}{p(X_i\,|\,\theta)}\,\dd\widetilde Q(\theta) + \frac{1}{\widetilde Q(\widetilde\Theta)}\int_{\widetilde\Theta}\log\frac{\widetilde Q(\theta)}{\pi_\theta(\theta)}\,\dd\widetilde Q - \log\widetilde Q(\widetilde\Theta).
\end{align*}
By Assumption~\ref{assump: prior condition}, we know
\begin{align*}
\widetilde W_n(\wht q_1, \cdots, \wht q_m)
&\leq \frac{1}{\widetilde Q(\widetilde\Theta)}\int_{\widetilde \Theta}\sum_{i=1}^n\log\frac{p(X_i\,|\,\theta^\ast)}{p(X_i\,|\,\theta)}\,\dd\widetilde Q(\theta) +  c_4n\varepsilon_n^2 +  c_3n\varepsilon_n^2.
\end{align*}
When $\widetilde{\m A}_n^c$ holds, we further have $\widetilde W_n(\wht q_1, \cdots, \wht q_m) \leq (2 c_4 +  c_3 + 1)n\varepsilon_n^2$. Therefore, we have
\begin{align*}
\mb E_{\theta^\ast}\Big[(1- \phi_n)1_{\widetilde{\m A}_n^c}\frac{\widetilde N_j(\varepsilon)}{\widetilde D_j}\Big] \leq e^{(2 c_4 +  c_3 + 1)n\varepsilon_n^2}\mb E_{\theta^\ast}\big[(1- \phi_n)\widetilde N_j(\varepsilon)\big].
\end{align*}

To bound $\mb E_{\theta^\ast}[(1- \phi_n)\widetilde N_j(\varepsilon)]$, note that
\begin{align*}
&\quad\,\mb E_{\theta^\ast}\big[(1- \phi_n)\widetilde N_j(\varepsilon)\big]\\
&= \int_{\m X^n}(1- \phi_n)\prod_{i=1}^n p(x_i\,|\,\theta^\ast)\int_{B_j(\varepsilon)^c}\exp\Big\{\int_{\Theta_{-j}}\log\frac{\pi_\theta(\theta)}{\widetilde Q_j(\theta_j)\wht q_{-j}(\theta_{-j})}+\sum_{i=1}^n\log\frac{p(x_i\,|\,\theta)}{p(x_i\,|\,\theta^\ast)}\,\dd\wht q_{-j}\Big\}\,\dd\widetilde Q_j\,\dd x_1\cdots\dd x_n\\
&= \int_{B_j(\varepsilon)^c}e^{\int_{\Theta_{-j}}\log\frac{\pi_\theta(\theta)}{\widetilde Q_j(\theta_j)\wht q_{-j}(\theta_{-j})}\,\dd\wht q_{-j}}\cdot\int_{\m X^n}(1- \phi_n)e^{\int_{\Theta_{-j}}\sum_{i=1}^n\log\frac{p(x_i\,|\,\theta)}{p(x_i\,|\,\theta^\ast)}\,\dd\wht q_{-j}}\prod_{i=1}^n p(x_i\,|\,\theta^\ast)\,\dd x_1\cdots\dd x_n \,\dd\widetilde Q_j\\
&=\int_{B_j(\varepsilon)^c}e^{\int_{\Theta_{-j}}\log\frac{\pi_\theta(\theta)}{\widetilde Q_j(\theta_j)\wht q_{-j}(\theta_{-j})}\,\dd\wht q_{-j}}\cdot\int_{\m X_n}(1- \phi_n)e^{\int_{\Theta_{-j}}\sum_{i=1}^n\log p(x_i\,|\,\theta)\,\dd\wht q_{-j}}\,\dd x_1\cdots\dd x_n\,\dd\widetilde Q_j.
\end{align*}
Due to the convexity of the exponential function, we have
\begin{align*}
e^{\int_{\Theta_{-j}}\log\frac{\pi_\theta(\theta)}{\widetilde Q_j(\theta_j)\wht q_{-j}(\theta_{-j})}\,\dd\wht q_{-j}} \leq \int_{\Theta_{-j}} \frac{\pi_\theta(\theta)}{\widetilde Q_j(\theta_j)\wht q_{-j}(\theta_{-j})}\,\dd\wht q_{-j} = \int_{\Theta_{-j}}\frac{\pi_\theta(\theta)}{\widetilde Q_j(\theta_j)}\,\dd\theta_{-j},
\end{align*}
and
\begin{align*}
e^{\int_{\Theta_{-j}}\sum_{i=1}^n\log p(x_i\,|\,\theta)\,\dd\wht q_{-j}} \leq \int_{\Theta_{-j}} \prod_{i=1}^n p(x_i\,|\,\theta)\,\dd\wht q_{-j},
\end{align*}
by applying Jensen's inequality. Thus,
\begin{align*}
\mb E_{\theta^\ast}\big[(1- \phi_n)\widetilde N_j(\varepsilon)\big] 
&\leq \int_{B_j(\varepsilon)^c}\bigg(\int_{\Theta_{-j}}\frac{\pi_\theta(\theta)}{\widetilde Q_j(\theta_j)}\,\dd\theta_{-j}\bigg)\bigg(\int_{\m X^n}(1- \phi_n)\int_{\Theta_{-j}}\prod_{i=1}^n p(x_i\,|\,\theta)\,\dd\wht q_{-j}\,\dd x_1\cdots\dd x_n\bigg)\dd\widetilde Q_j\\
&=\int_{B_j(\varepsilon)^c}\bigg(\int_{\Theta_{-j}}\frac{\pi_\theta(\theta)}{\widetilde Q_j(\theta_j)}\,\dd\theta_{-j}\bigg)\bigg(\int_{\Theta_{-j}}\mb E_{\theta}[1- \phi_n]\,\dd\wht q_{-j}\bigg)\dd\widetilde Q_j.
\end{align*}
By Assumption~\ref{assump: test condition}, we know $\mb E_\theta[1- \phi_n] \leq e^{- c_2n\varepsilon^2}$ when $\theta_j\in B_j(\varepsilon)^c$. This implies
\begin{align*}
\mb E_{\theta^\ast}\big[(1- \phi_n)\widetilde N_j(\varepsilon)\big] 
&\leq e^{- c_2n\varepsilon^2}\int_{B_j(\varepsilon)^c}\!\int_{\Theta_{-j}} \frac{\pi_\theta(\theta)}{\widetilde Q_j(\theta_j)}\,\dd\theta_{-j}\dd\widetilde Q_j(\theta_j) \leq e^{- c_2n\varepsilon^2}.
\end{align*}
Combining all pieces above yields
\begin{align*}
\mb E_{\theta^\ast}\Big[(1- \phi_n)1_{\widetilde{\m A}_n^c}\frac{\widetilde N_j(\varepsilon)}{\widetilde D_j}\Big] \leq e^{(2 c_4 +  c_3+1)n\varepsilon_n^2 -  c_2n\varepsilon^2}.
\end{align*}
By Assumption~\ref{assump: test condition} again, we have $\mb E_{\theta^\ast}[ \phi_n] \leq e^{- c_2n\varepsilon^2}$. Thus, we have
\begin{align}\label{eqn: good_event_uppbd}
\mb E_{\theta^\ast}\Big[1_{\widetilde{\m A}_n^c}\frac{\widetilde N_j(\varepsilon)}{\widetilde D_j}\Big] \leq 2e^{(2 c_4 +  c_3 + 1)n\varepsilon_n^2 - c_2n\varepsilon^2}
\end{align}

Finally, we will use the preceding arguments, the Markov inequality, and a simple union bound to prove the theorem, that is, the following inequality holds with high probability
\begin{align*}
\wht{Q}_\theta\big(\exists\,j\in[m]\,\,s.t.\,\,\|\theta_j - \theta_j^\ast\| > \varepsilon\big) \leq e^{- c_2n\varepsilon^2/2}
\end{align*}
when $\varepsilon$ is sufficiently large. By Markov's inequality, for any $k\geq 3$ we have
\begin{align*}
&\quad\,\mb P_{\theta^\ast}\Big(\widetilde{\m A}_n^c\cap\Big\{\wht Q_\theta\big(\exists\,j\in[m]\,\,s.t.\,\,\|\theta_j-\theta_j\|^\ast > k\varepsilon_n\big) > e^{- c_2n(k+1)^2\varepsilon_n^2/2}\Big\}\Big)\\
&=\mb P_{\theta^\ast}\Big(1_{\widetilde{\m A}_n^c}\cdot\wht Q_\theta\big(\exists\,j\in[m]\,\,s.t.\,\,\|\theta_j-\theta_j\|^\ast > k\varepsilon_n\big) > e^{- c_2n(k+1)^2\varepsilon_n^2/2}\Big)\\
&\stackrel{\ri}{\leq} e^{ c_2n(k+1)^2\varepsilon_n^2/2}\mb E_{\theta^\ast}\Big[1_{\widetilde{\m A}_n^c}\sum_{j=1}^m\frac{\widetilde N_j(k\varepsilon_n)}{\widetilde D_j}\Big]\\
&\stackrel{\rii}{\leq} e^{ c_2 n(k+1)^2\varepsilon_n^2/2}\cdot 2me^{(2 c_4 +  c_3 + 1)n\varepsilon_n^2 - c_2nk^2\varepsilon_n^2}\\
&\leq 2m^{-(k-5/2) c_2n\varepsilon_n^2}\cdot e^{(2 c_4+ c_3+1)n\varepsilon_n^2}
\end{align*}
Here, (i) is by Markov's inequality and~\eqref{eqn: MFVI_tail_prob}; (ii) is due to~\eqref{eqn: good_event_uppbd}. To derive the union bound, for any fixed $\varepsilon$, define the event
\begin{align*}
\widetilde{\m G}_\varepsilon = \Big\{\wht Q_\theta\big(\exists\,j\in[m]\,\, s.t. \,\,\|\theta_j-\theta_j^\ast\| > \varepsilon\big) > e^{- c_2n\varepsilon^2 / 2}\Big\}.
\end{align*}
The following relationship helps cover all $\widetilde{\m G}_\varepsilon$ with only those $\varepsilon$ as an integer multiple of $\varepsilon_n$,
\begin{align*}
\bigcup_{k\varepsilon_n\leq\varepsilon\leq(k+1)\varepsilon_n} \widetilde{\m G}_\varepsilon \subset \Big\{\wht Q_\theta\big(\exists\,j\in[m]\,\, s.t. \,\,\|\theta_j-\theta_j^\ast\| > k\varepsilon_n\big) > e^{- c_2n(k+1)^2\varepsilon_n^2 / 2}\Big\}.
\end{align*}
Let $N = \big\lceil (3+ c_1 + \frac{2 c_4 +  c_3 + 1}{ c_2})\big\rceil$. By applying a union bound, we obtain that
\begin{align*}
&\quad\,\mb P_{\theta^\ast}\Big(\widetilde{\m A}_n^c\cap \Big\{\wht Q_\theta\big(\exists\,j\in[m]\,\,s.t.\,\,\|\theta_j-\theta_j^\ast\| > \varepsilon\big) > e^{- c_2n\varepsilon^2/2},\quad\mx{for some $\varepsilon$ satisfies~\eqref{eqn: constraint_epsilon}}\Big\}\Big)\\
&\leq \sum_{k\geq N}\mb P_{\theta^\ast}\Big(\widetilde{\m A}_n^c\cap \Big\{\wht Q_\theta\big(\exists\,j\in[m]\,\, s.t. \,\,\|\theta_j-\theta_j^\ast\| > k\varepsilon_n\big) > e^{- c_2n(k+1)^2\varepsilon_n^2 / 2}\Big\}\Big)\\
&\leq \sum_{k\geq N} 2m^{-(k-5/2) c_2n\varepsilon_n^2}\cdot e^{(2 c_4+ c_3+1)n\varepsilon_n^2} = \frac{2me^{-(N-5/2) c_2n\varepsilon_n^2 + (2 c_4+ c_3 + 1)n\varepsilon_n^2}}{1 - e^{- c_2n\varepsilon_n^2}}\\
&\leq 3me^{- c_2n\varepsilon_n^2/2}
\end{align*}
for all $n\geq 3^{1/ c_2}$. To bound $\mb P_{\theta^\ast}(\widetilde{\m A}_n)$, we need the following lemma, whose proof is deferred to Section~\ref{app: proof_C.2} in Appendix~\ref{app:tech_results}.
\begin{lemma}\label{lem: prob_antilde}
Under Assumption~\ref{assump: prior condition}, we have
\begin{align*}
\mb P_{\theta^\ast}(\widetilde{\m A}_n) \leq \frac{ c_4}{n\varepsilon_n^2}.
\end{align*}
\end{lemma}
\noindent Combining all pieces above yields
\begin{align*}
\wht Q_\theta\big(\exists\,j\in[m]\,\,s.t.\,\,\|\theta_j-\theta_j^\ast\| > \varepsilon\big) \leq e^{- c_2n\varepsilon^2/2}
\end{align*}
holds for all $\varepsilon$ satisfies~\eqref{eqn: constraint_epsilon} with probability at least
\begin{align*}
1 - 3me^{- c_2n\varepsilon_n^2/2} - \frac{ c_4}{n\varepsilon_n^2} \geq 1 - \frac{2 c_4}{n\varepsilon_n^2}.
\end{align*}

\subsection{Proof of Theorem \ref{thm: posterior_convergence_rate}}\label{app:proof_posterior_con}
It is straightforward verify that the mean-field inference objective functional $\KL(q_\theta\otimes q_{Z^n}\,\|\,\pi_n)$ is equivalent to the following functional up to some constant independent of $(q_\theta,\,q_{Z^n})$,
\begin{align*}
   \int_{\Theta}\sum_{i=1}^n\sum_{z=1}^K&\log \frac{q_{Z_i}(z)}{p(z\,|\,X_i,\theta)}q_{Z_i}(z)\,q_\theta(\dd\theta) + \int_\Theta\sum_{i=1}^n\log\frac{p(X_i\,|\,\theta^\ast)}{p(X_i\,|\,\theta)}\,q_\theta(\dd\theta) + D_{KL}(q_\theta\,\|\,\pi_\theta). 
\end{align*}
The following lemma provides a characterization to the MF approximation $(\wht q_\theta,\wht q_{Z^n})$ as the minimizer of this functional.

\begin{lemma}\label{lem:functional_Vn}
Consider the following functional 
\begin{align*}
   W_n(\rho, F_x) 
   := \int_{\Theta}\sum_{i=1}^n\sum_{z=1}^K&\log \frac{F_{X_i}(z)}{p(z\,|\,X_i,\theta)}F_{X_i}(z)\,\rho(\dd\theta)\\ 
   &\quad+ \int_\Theta\sum_{i=1}^n\log\frac{p(X_i\,|\,\theta^\ast)}{p(X_i\,|\,\theta)}\,\rho(\dd\theta) + D_{KL}(\rho\,\|\,\pi_\theta), 
\end{align*}
where $\rho$ ranges over all distributions on $\Theta$ and $F_x$ ranges all measurable functions of $(x,z)$ such that for each $x\in\mb R^d$, $F_x(\cdot)$ is a probability mass function over $\{1,2,\ldots,K\}$. Then $\big(\wht q_\theta,\, \Phi(\wht q_\theta, \,x)\big)$ is a minimizer to this functional, where recall that $\Phi(q_\theta, x)$ is defined in~\eqref{Eqn:Phi_def} as
\begin{align*}
    \Phi(q_\theta, x)(z) &= \frac{ \exp\big\{\mb E_{q_\theta} \log p(z\,|\, x,\theta)\big\}}   { \sum_{z\in\m Z}\exp\big\{ \mb E_{q_\theta} \log p(z\,|\, x,\theta)\big\}},\ \ z\in[K].
\end{align*}
In addition, we have $\wht q_{Z_i} = \Phi(\wht q_\theta, X_i)$ for $i\in[n]$, and $\wht q_\theta$ satisfies 
\begin{align}\label{eqn:optimality_eqn}
    \wht q_\theta(\theta) =  \frac{\pi_\theta(\theta)\, e^{-n U_n(\theta,\wht q_\theta)}}{\int_{\Theta}\pi_\theta(\theta)\, e^{-n U_n(\theta,\wht q_\theta)}\,\dd\theta} \quad \mbox{for all} \quad \theta\in\Theta, 
\end{align}
where $U_n$ is the sample potential function defined in~\eqref{eq:U_n_exp}.
\end{lemma}
\noindent A proof to this lemma is deferred to Section~\ref{sec:proof_lem:functional_Vn} in Appendix~\ref{app:tech_results}.

Now let us return to the proof of the theorem. By using the characterization~\eqref{eqn:optimality_eqn} of $\wht q_\theta$ from Lemma~\ref{lem:functional_Vn}, we can express the tail variational posterior probability of $B_\varepsilon^c = \{\theta\in\Theta: \|\theta - \theta^\ast\| > \varepsilon\}$ for any fixed $\varepsilon\geq \varepsilon_n$ as
\begin{eqnarray*}
&\quad\,&\wht Q_\theta(\|\theta - \theta^\ast\| > \varepsilon\,|\,X^n) = \frac{\int_{B_\varepsilon^c}\exp\{-nU_n(\theta, \widehat{q}_\theta)\}\,\dd\pi_\theta}{\int_{\Theta}\exp\{-nU_n(\theta, \widehat{q}_\theta)\}\,\dd\pi_\theta}\\
&\stackrel{\ri}{=}& \frac{\int_{B_\varepsilon^c}\exp\{\sum_{i=1}^n\sum_{z=1}^K \log p(X_i, z\,|\,\theta)\Phi(\widehat{q}_\theta, X_i)(z)\}\,\dd\pi_\theta}{\int_{\Theta}\exp\{\sum_{i=1}^n\sum_{z=1}^K \log p(X_i, z\,|\,\theta)\Phi(\widehat{q}_\theta, X_i)(z)\}\,\dd\pi_\theta}\\
&\stackrel{\rii}{=}& \frac{\int_{B_\varepsilon^c}\exp\{-\sum_{i=1}^n\sum_{z=1}^K\Phi(\widehat{q}_\theta, X_i)(z)\log\frac{\Phi(\widehat{q}_\theta, X_i)(z)}{p(z\,|\,X_i,\theta)} - \sum_{i=1}^n\log\frac{p(X_i\,|\,\theta^\ast)}{p(X_i\,|\,\theta)}\}\,\dd\pi_\theta}{\int_{\Theta}\exp\{-\sum_{i=1}^n\sum_{z=1}^K\Phi(\widehat{q}_\theta, X_i)(z)\log\frac{\Phi(\widehat{q}_\theta, X_i)(z)}{p(z\,|\,X_i,\theta)} - \sum_{i=1}^n\log\frac{p(X_i\,|\,\theta^\ast)}{p(X_i\,|\,\theta)}\}\,\dd\pi_\theta}\\
&=:&\frac{N_n(\varepsilon)}{D_n},
\end{eqnarray*}
where step (i) is due to the definition of $U_n$, and step (ii) is obtained by adding the same $\theta$-independent term to the exponents of the numerator and the denominator. 
Let $\ell_n(\theta) = \sum_{i=1}^n\log p(X_i\,|\,\theta)$ denote the log-likelihood function, and 
\begin{displaymath}
\m A_n = \bigg\{\frac{1}{\Pi(\widetilde\Theta)}\int_{\widetilde\Theta}\big[\ell_n(\theta) - \ell_n(\theta^\ast)\big]\,\dd\pi_\theta \leq -(c_4+1)n\varepsilon_n^2\bigg\}
\end{displaymath}
be the event that a weighted average of the log-likelihood function over the neighborhood $\widetilde\Theta$ (defined in Assumption~\ref{assump: prior condition}) around $\theta^\ast$ is small, which will be used for bounding the numerator $D_n$ from below. 

By Assumption \ref{assump: test condition}, there exists a set of test functions $\phi_n$ s.t. 
\begin{align*}
    \mb E_{\theta^\ast}\big[\phi_n\big] \leq e^{-c_2n\varepsilon^2}, \quad \mbox{and}\quad \sup_{\|\theta-\theta^\ast\|\geq \varepsilon}\mb E_\theta\big[1-\phi_n\big] \leq e^{-c_2n\varepsilon^2}.
\end{align*}
Based on $\phi_n$ and event $\m A_n$, we have the decomposition
\begin{eqnarray}
    \mb E_{\theta^\ast}\bigg[\frac{N_n(\varepsilon)}{D_n}\bigg]
    &=& \mb E_{\theta^\ast}\bigg[\phi_n\frac{N_n(\varepsilon)}{D_n}\bigg] + \mb E_{\theta^\ast}\bigg[(1-\phi_n)\,1_{\m A_n}\frac{N_n(\varepsilon)}{D_n}\bigg] + \mb E_{\theta^\ast}\bigg[(1-\phi_n)\,1_{\m A_n^c}\frac{N_n(\varepsilon)}{D_n}\bigg]\nonumber\\
    &\leq& \mb E_{\theta^\ast}[\phi_n] + \mb P_{\theta^\ast}(\m A_n) + \mb E_{\theta^\ast}\bigg[(1-\phi_n)1_{\m A_n^c}\frac{N_n(\varepsilon)}{D_n}\bigg],\label{inequal: decompositoin_of_prob}
\end{eqnarray}
since by definition, $N_n(\varepsilon)/D_n =\wht Q_\theta(\|\theta - \theta^\ast\| > \varepsilon\,|\,X^n) \in[0,1]$.

Recall the definition of $W_n(\rho, F_x)$ in Lemma~\ref{lem:functional_Vn}. Firstly, we shall prove $D_n = \exp\{-W_n(\widehat{q}_\theta, \Phi(\widehat{q}_\theta, x))\}$. In fact, we have the following series of identities
\begin{eqnarray*}
D_n
&=& \exp\bigg\{-\sum_{i=1}^n\sum_{z=1}^K\Phi(\widehat{q}_\theta, X_i)(z)\log\Phi(\widehat{q}_\theta, X_i)(z) - \sum_{i=1}^n\log p(X_i\,|\,\theta^\ast)\bigg\}\cdot\int_\Theta e^{-nU_n(\theta, \widehat{q}_\theta)}\,\dd\pi_\theta\\
&=& \exp\bigg\{-\sum_{i=1}^n\sum_{z=1}^K\Phi(\widehat{q}_\theta, X_i)(z)\log\Phi(\widehat{q}_\theta, X_i)(z) - \sum_{i=1}^n\log p(X_i\,|\,\theta^\ast)\bigg\}\cdot\frac{\pi_\theta e^{-nU_n(\theta, \widehat{q}_\theta)}}{\widehat{q}_\theta}\\
&=& \exp\bigg\{-\sum_{i=1}^n\sum_{z=1}^K\Phi(\widehat{q}_\theta, X_i)(z)\log\Phi(\widehat{q}_\theta, X_i)(z) - \sum_{i=1}^n\log p(X_i\,|\,\theta^\ast) -nU_n(\theta, \widehat{q}_\theta) + \log\frac{\pi_\theta}{\widehat{q}_\theta}\bigg\}\\
&\stackrel{\ri}{=}& \exp\bigg\{\int_\Theta\bigg[-\sum_{i=1}^n\sum_{z=1}^K\Phi(\widehat{q}_\theta, X_i)(z)\log\Phi(\widehat{q}_\theta, X_i)(z) - \sum_{i=1}^n\log p(X_i\,|\,\theta^\ast) -nU_n(\theta, \widehat{q}_\theta) + \log\frac{\pi_\theta}{\widehat{q}_\theta}\bigg]\,\dd\widehat{q}_\theta\bigg\}\\
&=& \exp\big\{-W_n\big(\widehat{q}_\theta, \Phi(\widehat{q}_\theta, x)\big)\big\}.
\end{eqnarray*}
Here, (i) is because the quantity inside the exponent of $D_n$ is not a function of $\theta$, and thus the integration with respect to the probability measure $\wht q_{\theta}$ will not change its value. 
Notice that by Lemma~\ref{lem:functional_Vn}, $\big(\widehat{q}_\theta, \Phi(\widehat{q}_\theta, x)\big)$ is the minimizer of functional $W_n$, therefore, we can apply this optimality to obtain that for another feasible pair
\begin{equation*}
    \rho(\theta) = \frac{\pi_\theta(\theta)}{\Pi(\widetilde\Theta)}1_{\widetilde\Theta} \quad \mbox{and}\quad F_x(z) = p(z\,|\,x,\theta^\ast),
\end{equation*}
it always holds that
\begin{eqnarray*}
&\quad\,&
W_n\big(\widehat{q}_\theta, \Phi(\widehat{q}_\theta, x)\big) \leq W_n(\rho, F_x)\\
&=& \frac{1}{\Pi(\widetilde\Theta)}\int_{\widetilde\Theta}\sum_{i=1}^n\sum_{z=1}^K\log\frac{p(z\,|\,X_i,\theta^\ast)}{p(z\,|\,X_i,\theta)}p(z\,|\,X_i,\theta^\ast)\,\dd\pi_\theta + \frac{1}{\Pi(\widetilde\Theta)}\int_{\widetilde\Theta}\sum_{i=1}^n\log\frac{p(X_i\,|\,\theta^\ast)}{p(X_i\,|\,\theta)}\,\dd\pi_\theta - \log\Pi(\widetilde\Theta)\\
&=& \frac{1}{\Pi(\widetilde\Theta)}\int_{\widetilde\Theta}\sum_{i=1}^n D_{KL}\big[p(\cdot\,|\,X_i,\theta^\ast)\, \big\|\,p(\cdot\,|\,X_i,\theta)\big]\,\dd\pi_\theta + \frac{1}{\Pi(\widetilde\Theta)}\int_{\widetilde\Theta}\sum_{i=1}^n\log\frac{p(X_i\,|\,\theta^\ast)}{p(X_i\,|\,\theta)}\,\dd\pi_\theta -\log\Pi(\widetilde\Theta)\\
&\leq& \frac{1}{\Pi(\widetilde\Theta)}\int_{\widetilde\Theta}\sum_{i=1}^n G(X_i)\,\varepsilon_n^2\,\dd\pi_\theta + \frac{1}{\Pi(\widetilde\Theta)}\int_{\widetilde\Theta}\sum_{i=1}^n\log\frac{p(X_i\,|\,\theta^\ast)}{p(X_i\,|\,\theta)}\,\dd\pi_\theta -\log\Pi(\widetilde\Theta)\\
&\leq& \varepsilon_n^2\sum_{i=1}^n G(X_i) + \frac{1}{\Pi(\widetilde\Theta)}\int_{\widetilde\Theta} \big[\ell_n(\theta^\ast) - \ell_n(\theta)\big]\,\dd\pi_\theta + c_3\,n\varepsilon_n^2,
\end{eqnarray*}
where in the last two inequalities, we used Assumptions~\ref{assump: prior condition} and~\ref{assump: qudratic growth of KL} so that $-\log\Pi(\widetilde\Theta) \leq c_3n\varepsilon_n^2$ and $\KL(p(\cdot\,|\,X_i, \theta^\ast)\,\|\,p(\cdot\,|\,X_i, \theta)) \leq G(X_i)\varepsilon_n^2$ for all $\theta\in \widetilde\Theta$. 

Let $\m C_n$ be the event defined by
\begin{equation*}
    \m C_n  := \bigg\{\frac{1}{n}\sum_{i=1}^n \big(G(X_i) - \mb E[G(X)]\big) > 1\bigg\}.
\end{equation*}
Since $G(X)$ is sub-exponential with parameters $\sigma_4$ by Assumption~\ref{assump: qudratic growth of KL}, we may apply Lemma \ref{lem: tail_iidsum} from Appendix~\ref{Appendix: conc} with $\alpha = 1$ to obtain $\mb P(\m C_n) \leq e^{-A_1n^\frac{1}{2}\sigma_4^{-1}}$. Under the complimentary event $\m C_n^c$, we then have
\begin{equation*}
    W_n\big(\widehat{q}_\theta, \Phi(\widehat{q}_\theta, x)\big) \leq \big(\mb E[G(X)]+ 1 + c_3\big)\,n\varepsilon_n^2 + \frac{1}{\Pi(\widetilde\Theta)}\int_{\widetilde\Theta}\big[\ell_n(\theta^\ast) - \ell_n(\theta)\big]\,\dd\pi_\theta.
\end{equation*}
By putting pieces together, we obtain the following lower bound to $D_n$ on the event $\m A_n$,
\begin{eqnarray}
D_n &=& \exp\big\{-W_n\big(\widehat{q}_\theta, \Phi(\widehat{q}_\theta, x)\big)\big\} \nonumber\\
&\geq& \exp\bigg\{-\big(\mb E[G(X)]+ 1 + c_3\big)\,n\varepsilon_n^2 - \frac{1}{\Pi(\widetilde\Theta)}\int_{\widetilde\Theta}\big[\ell_n(\theta^\ast) - \ell_n(\theta)\big]\,\dd\pi_\theta\bigg\} \notag\\
&\geq & \exp\big\{-\big(\mb E[G(X)]+ 2 + c_3 + c_4\big)\,n\varepsilon_n^2\big\}. \label{inequal: bound_of_Dn}
\end{eqnarray}

For the numerator $N_n(\varepsilon)$, by noticing that
\begin{displaymath}
\sum_{z=1}^K\Phi(\widehat{q}_\theta, X_i)(z)\log\frac{\Phi(\widehat{q}_\theta, X_i)(z)}{p(z\,|\,X_i,\theta)} = D_{KL}\big[\Phi(\widehat{q}_\theta, X_i)(\cdot)\,\big\|\,p(\cdot\,|\,X_i,\theta)\big] \geq 0,
\end{displaymath}
we obtain
\begin{eqnarray}
N_n(\varepsilon)
&=& \int_{B_\varepsilon^c}\exp\bigg\{-\sum_{i=1}^n\sum_{z=1}^K\Phi(\widehat{q}_\theta, X_i)(z)\log\frac{\Phi(\widehat{q}_\theta, X_i)(z)}{p(z\,|\,X_i,\theta)} - \sum_{i=1}^n\log\frac{p(X_i\,|\,\theta^\ast)}{p(X_i\,|\,\theta)}\bigg\}\,\dd\pi_\theta\nonumber\\
&\leq & \int_{B_\varepsilon^c}\exp\bigg\{-\sum_{i=1}^n\log \frac{p(X_i\,|\,\theta^\ast)}{p(X_i\,|\,\theta)}\bigg\}\,\dd\pi_\theta = \int_{B_\varepsilon^c}e^{\ell_n(\theta) - \ell_n(\theta^\ast)}\,\dd\pi_\theta. \label{inequal: bound_of_Nn}
\end{eqnarray}
Therefore, under event $\m C_n^c$, we can apply the lower bound of $D_n$ and the upper bound of $N_n(\varepsilon)$ together to control the third term in decomposition~\eqref{inequal: decompositoin_of_prob} as
\begin{align*}
&\quad\,\mb E_{\theta^\ast}\bigg[(1-\phi_n)\,1_{\m A_n^c}\,\frac{N_n(\varepsilon)}{D_n}\bigg]\\
&\leq \mb E_{\theta^\ast}\bigg[(1-\phi_n)\,1_{\m A_n^c}\,\frac{\int_{B_\varepsilon^c}e^{\ell_n(\theta) - \ell_n(\theta^\ast)}\,\dd\pi_\theta}{\exp\big\{-\big(\mb E[G(X)] + 2 + c_3 + c_4\big)\,n\varepsilon_n^2\big\}}\bigg]\\
&= \exp\big\{\big(\mb E[G(X)] + c_3 + c_4 + 2\big)\,n\varepsilon_n^2\big\} \cdot\mb E_{\theta^\ast}\bigg[(1-\phi_n)\,1_{\m A_n^c}\int_{B_\varepsilon^c}e^{\ell_n(\theta)-\ell_n(\theta^\ast)}\,\dd\pi_\theta\bigg]\\
&\leq \exp\big\{(\mb EG(X) + c_3 + c_4 + 2)n\varepsilon_n^2\big\} \cdot\mb E_{\theta^\ast}\bigg[(1-\phi_n)\int_{B_\varepsilon^c}e^{\ell_n(\theta)-\ell_n(\theta^\ast)}\,\dd\pi_\theta\bigg]\\
&\stackrel{\ri}{=} \exp\big\{(\mb EG(X) + c_3 + c_4 + 2)n\varepsilon_n^2\big\} \cdot \int_{B_\varepsilon^c}\mb E_{\theta}[1-\phi_n]\,\dd\pi_\theta.
\end{align*}
Here, step (i) is due to the following series of identities,
\begin{align*}
&\quad\,\mb E_{\theta^\ast}\bigg[(1-\phi_n)\int_{B_r^c}e^{\ell_n(\theta)-\ell_n(\theta^\ast)}\,\dd\pi_\theta\bigg]\\
&= \int_{\m X^n}\int_{B_r^c} \big(1-\phi_n(X^n)\big)\cdot \exp\bigg\{\sum_{i=1}^n\log\frac{p(X_i\,|\,\theta)}{p(X_i\,|\,\theta^\ast)}\bigg\}\cdot \prod_{i=1}^n p(X_i\,|\,\theta^\ast)\,\dd\pi_\theta\,\dd X^n\\
&= \int_{\m X^n}\int_{B_r^c}\big(1-\phi_n(X^n)\big)\cdot \prod_{i=1}^np(X_i\,|\,\theta)\,\dd\pi_\theta\,\dd X^n = \int_{B_r^c}\mb E_\theta[1-\phi_n]\,\dd\pi_\theta,
\end{align*}
where recall that $X^n = (X_1, \cdots, X_n)\in\m X^n$ denotes the $n$ observations. 

To use the decomposition~\eqref{inequal: decompositoin_of_prob} and Markov inequality for bounding $\wht Q_\theta(\|\theta - \theta^\ast\| > \varepsilon\,|\,X^n)$, it remains to bound $\mb P_{\theta^\ast}(\m A_n)$, which is standard in the Bayesian asymptotics literature and the result is summarized in the lemma below. For completeness, we also provide its proof in Section~\ref{sec:proof_lemma:A_n_bound}.

\begin{lemma}\label{lemma:A_n_bound}
If Assumption~\ref{assump: prior condition} holds, then for any $M>1$, the probability of event $\m A_n$ can be bounded from above as
\begin{align*}
    \mb P_{\theta^\ast}(\m A_n) \leq \frac{c_4}{n\varepsilon_n^2}.
\end{align*}
\end{lemma}
By putting all pieces together, we obtain that for any $\varepsilon>0$, there exists some test function $\phi_n$, such that
\begin{equation}\label{eqn:key_decomp}
    \begin{aligned}
    \mb E_{\theta^\ast}\Big[\frac{N_n(\varepsilon)}{D_n}\cdot 1_{\m A_n^c}\Big] &\leq \mb E_{\theta^\ast}[\phi_n] + \mb E_{\theta_\ast}\Big[(1-\phi_n)1_{\m A_n^c}\frac{N_n(\varepsilon)}{D_n}\Big]\\
    &\leq e^{-c_2n\varepsilon^2} + e^{(\mb E[G(X)] + c_3 + c_4 + 2)\,n\varepsilon_n^2}\cdot\int_{B_\varepsilon^c}e^{-c_2n\varepsilon^2}\,\dd\pi_\theta\\
    &\leq 2e^{-c_2n\varepsilon^2}\cdot e^{(\mb E[G(X)] + c_3 + c_4 + 2)\, n\varepsilon_n^2}.
\end{aligned}
\end{equation}

Finally, we will utilize the preceding display, the Markov inequality and a simple union bound argument to prove the claimed result in the theorem, that is, the following holds with high probability,
\begin{align*}
    \wht Q_\theta\big(\|\theta-\theta^\ast\| > \varepsilon \big)\leq e^{-c_2n\varepsilon^2/2}
\end{align*}
for all sufficiently large $\varepsilon$ satisfying the lower bound requirement~\eqref{cond: r_in_posterior_convergence_rate} in the theorem. Concretely, for each positive integer $j$, let $\phi_{n,j}$ be the test function such that
\begin{align*}
    \mb E_{\theta^\ast}\big[\phi_{n,j}\big]\leq e^{-c_2nj^2\varepsilon_n^2} \quad\mbox{and}\quad \sup_{\|\theta-\theta^\ast\|>j\varepsilon_n}\mb E_\theta\big[1-\phi_{n,j}\big] \leq e^{-c_2nj^2\varepsilon_n^2}.
\end{align*}
For any $j\geq 3$, by applying the Markov inequality and inequality~\eqref{eqn:key_decomp}, we have
\begin{align*}
&\quad\,\mb P_{\theta^\ast}\bigg(\m A_n^c \cap\Big\{\wht Q_\theta\big(\|\theta-\theta^\ast\| > j\varepsilon_n\,|\,X^n\big) > e^{-c_2n(j+1)^2\varepsilon_n^2/2}\Big\}\bigg)\\
&= \mb P_{\theta^\ast}\bigg(1_{\m A_n^c}\cdot\wht Q_\theta\big(\|\theta-\theta^\ast\| > j\varepsilon_n\,|\,X^n\big) > e^{-{c_2}n(j+1)^2\varepsilon_n^2/2}\bigg)\\
&\leq e^{c_2n(j+1)^2\varepsilon_n^2/2} \cdot \mb E_{\theta^\ast}\bigg[\frac{N_n(j\varepsilon_n)}{D_n}\cdot 1_{\m A_n^c}\bigg]\\
&\leq 2 e^{c_2n(j+1)^2\varepsilon_n^2/2-c_2nj^2\varepsilon_n^2}\cdot e^{(\mb E[G(X)]+c_3+c_4+2)\,n\varepsilon_n^2}\\
&\leq 2e^{-(j-5/2)\,c_2n\varepsilon_n^2}\cdot e^{(\mb E[G(X)]+c_3+c_4+2)\,n\varepsilon_n^2}.
\end{align*}

Let 
$\m G_\varepsilon := \Big\{\wht Q_\theta\big(\|\theta-\theta^\ast\| > \varepsilon\,|\,X^n\big) > e^{-c_2n\varepsilon^2/2}\Big\}$ denote the desired event in the theorem for a fixed $\varepsilon$. Then we have the following relationship that can be used to cover all $\m G_\varepsilon$ using only those $\varepsilon$ as an integer multiple of $\varepsilon_n$, 
\begin{align*}
    \bigcup_{j\varepsilon_n\leq \varepsilon\leq (j+1)\varepsilon_n}\m G_\varepsilon \subset \Big\{\wht Q_\theta\big(\|\theta-\theta^\ast\| > j\varepsilon_n\big) > e^{-c_2n(j+1)^2\varepsilon_n^2/2}\Big\}.
\end{align*}
Let $N := \big\lceil (\mb E[G(X)] + c_3 + c_4 + 2)/c_2 + 3\big\rceil$, where $\lceil x\rceil$ denotes the smallest integer that is greater than $x$. Then, by applying a union bound and the preceding display, we obtain that
\begin{align*}
    &\quad\,\mb P_{\theta^\ast}\bigg(\m A_n^c\cap\Big\{\wht Q_\theta\big(\|\theta-\theta^\ast\| > \varepsilon\big) > e^{-c_2n\varepsilon^2/2}, \ \ \mbox{for some $\varepsilon$ satisfies (\ref{cond: r_in_posterior_convergence_rate})\Big\}} \bigg)\\
    &\leq \sum_{j\geq N} \mb P_{\theta^\ast}\bigg(\m A_n^c\cap\Big\{\bigcup_{j\varepsilon_n\leq \varepsilon\leq (j+1)\varepsilon_n}\m G_\varepsilon\Big\}\bigg)\\
    &\leq \sum_{j\geq N} \mb P_{\theta^\ast}\bigg(\m A_n^c\cap\Big\{\wht Q_\theta\big(\|\theta-\theta^\ast\| > j\varepsilon_n\big) > e^{-c_2n(j+1)^2\varepsilon_n^2/2}\Big\}\bigg)\\
    &\leq \sum_{j\geq N} 2e^{-(j-5/2)\,c_2n\varepsilon_n^2}\cdot e^{(\mb E[G(X)]+c_3+c_4+2)\,n\varepsilon_n^2}\\
    &= \frac{2e^{-(N-5/2)c_2n\varepsilon_n^2}}{1-e^{-c_2n\varepsilon_n^2}}\cdot e^{(\mb E[G(X)]+c_3+c_4+2)n\varepsilon_n^2}
    \leq 3e^{-c_2n\varepsilon_n^2/2}
\end{align*}
for all $n\geq 3^{1/c_2}$. As a consequence, we proved that
\begin{align*}
    \wht Q_\theta\big(\|\theta-\theta^\ast\| > \varepsilon \big)\leq e^{-c_2n\varepsilon^2/2}
\end{align*}
holds for all $\varepsilon$ satisfying~\eqref{cond: r_in_posterior_convergence_rate} with probability at least
\begin{align*}
    &\quad\,1 - 3e^{-c_2n\varepsilon_n^2/2} - \mb P_{\theta^\ast}(\m A_n) - \mb P_{\theta^\ast}(\m C_n)\\
    &\geq 1 - 3e^{-c_2n\varepsilon_n^2/2} - \frac{c_4}{n\varepsilon_n^2} - e^{-A_1\sqrt{n}\sigma_4^{-1}} 
    \geq 1 - \frac{2c_4}{n\varepsilon_n^2}.
\end{align*}

\subsection{Proof of Theorem \ref{thm: para_update_whp}}\label{app: mainthm_nolatent}
By Theorem 1 in \cite{mei2018landscape}, we know
\begin{align*}
    \sup_{\theta\in\Theta}\matnorm{\frac{1}{n}\sum_{i=1}^n\nabla^2\log p(X_i\,|\,\theta) - \mb E_{\theta^\ast}\nabla^2\log p(X\,|\,\theta)} \leq \sigma_5^2 \sqrt{\frac{Cd\log n}{n}\cdot \max\Big\{\frac{\log\tilde J_\ast}{\log d}, \log\frac{R\sigma_5}{\eta}, 1\Big\}}
\end{align*}
with probability at least $1-\eta$. This implies
\begin{align*}
\tilde\lambda_n I_d\preceq -\sum_{i=1}^n \nabla^2\log p(X_i\,|\,\theta) - \nabla^2\log\pi_\theta(\theta) \preceq \tilde L_n I_d,
\end{align*}
where
\begin{align*}
    \tilde\lambda_n &\geq n\tilde{\lambda} - \lambda_M(\nabla^2\log\pi_\theta) - \sigma_5^2\sqrt{\frac{Cd\log n}{n}\cdot\max\Big\{\frac{\log\tilde J_\ast}{\log d}, \log\frac{R\sigma_5}{\eta}, 1\Big\}} =: \tilde\lambda_{lb}\\
    \tilde L_n &\leq n\tilde L - \lambda_m(\nabla^2\log\pi_\theta) + \sigma_5^2\sqrt{\frac{Cd\log n}{n}\cdot\max\Big\{\frac{\log\tilde J_\ast}{\log d}, \log\frac{R\sigma_5}{\eta}, 1\Big\}} =: \tilde L_{ub},
\end{align*}
where $\lambda_m(\nabla^2\log\pi_\theta)$ and $\lambda_M(\nabla^2\log \pi_\theta)$ are the the uniform lower bound and upper bound of eigenvalues of $\nabla^2\log\pi_\theta$. Directly applying the following Lemma \ref{lem: para_update} yields the result.
\begin{lemma}[MFVI without latent variables]\label{lem: para_update}
Assume $\widetilde U(\theta) = -\sum_{i=1}^n\log p(X_i\,|\,\theta)-\log\pi_\theta(\theta)$ is twice differentiable, $\tilde\lambda_n$-strongly convex, and $\tilde L_n$-smooth, i.e. $0 \preceq \tilde\lambda_n I_d\leq \nabla^2\widetilde U\preceq \tilde L_n I_d$. If the step size $\tau < \frac{1}{\sqrt{m}\tilde L_n}$, then $q^{(t)}$ derived by (\ref{eqn: JKO_update_para}) satisfies 
\begin{align*}
    W_2^2(q^{(t)}, \wht q_\theta) \leq \Big(1+2\tau\tilde\lambda_n - \tilde L_n^2\tau^2m\Big)^{-t}W_2^2(q^{(0)}, \wht q).
\end{align*}
Moreover, if $\tau < \frac{2\tilde\lambda_n}{\tilde L_n^2m}$, this inequality implies $q^{(t)}$ converges to $\wht q$ exponentiall fast w.r.t. $W_2$ distance.
\end{lemma}

\subsection{Proof of Theorem~\ref{thm: main_theorem}}\label{app:proof_main_theorem}

Before the proof, we first list explicit expressions to some constants appearing in the theorem. Specifically, the lower bound condition on the sample size $n$ is
\begin{align*}
    n&\geq \max\bigg\{\Big(\frac{(1+c_2^{-1})^2\big(\mb EG(X) + c_3 + c_1 + 5\big)^2 + 1}{\min\{1, \big(\frac{\eta\gamma}{2D}\big)^2, \big(\frac{R_W}{2}\big)^2\}}\Big)^{2}, \max\Big\{e, \frac{2}{c_2}\Big\}^{\frac{8}{c_2(1+c_2^{-1})^{2}(\mb EG(X) + c_1 + c_3 + 5)^2}}\\
    &\qquad\qquad \frac{324RdJ_\ast}{\sigma_1}, \Big(C\sigma_2\log\frac{6}{\eta}\Big)^6, \Big(C\sigma_2\log\frac{6}{\eta}\Big)^6, e^{e^2}, e^{4\max\{\sigma_4A_1^{-1}, 2c_2^{-1}\}^2}
    \bigg\},
\end{align*}
and constants $A$, $B$, and $C$ can be taken as
\begin{align*}
A &= (3K^2+2K)\big(\mb E_{\theta^\ast}[S_1(X)^3] + 1\big) + \frac{27K^3+24K^2+4K}{16}\big(\mb E_{\theta^\ast}[\lambda(X)^3] + 1\big)+ \frac{K}{2}\big(\mb E_{\theta^\ast}[\lambda(X)^3]+1\big)^{\frac{2}{3}}\\
B &= 
\frac{21K^2+32K+12}{8}\big(\mb E_{\theta^\ast}[S_1(X)^3] + 1\big) + \frac{27K^3+42K^2+16K}{8}\big(\mb E_{\theta^\ast}[\lambda(X)^3] + 1\big)\\
&\qquad\qquad\qquad\quad\qquad\qquad\qquad\qquad +
\frac{1}{2}\big(\mb E_{\theta^\ast}[S_1(X)^3] + 1\big)^{\frac{2}{3}} + \frac{2K+1}{2}\big(\mb E_{\theta^\ast}[\lambda(X)^3] + 1\big)^{\frac{2}{3}}\\
C &= \frac{2K^2+K}{2}\big(\mb E_{\theta^\ast}[\lambda(X)^3] + 1\big)^{\frac{2}{3}} + \frac{K}{2}\big(\mb E_{\theta^\ast}[S_1(X)^3] + 1\Big)^{\frac{2}{3}}\\
D &= K(2K+3) \big(\mb E_{\theta^\ast}[\lambda(X)^3] + 1\big)^{\frac{2}{3}} + 
(2K+1)\big(\mb E_{\theta^\ast}[S_1(X)^3] + 1\big)^{\frac{2}{3}} + \frac{K+2}{2}\big(\mb E_{\theta^\ast} S_2(X)^2 + 1\big).
\end{align*}

\smallskip

Now let us proceed to the proof of the theorem. Recall that the MF-WGF algorithm can be summarized by the following iterative updating rule: for $k=0,1,\ldots$,
\begin{equation*}
        q_\theta^{(k+1)} = \argmin_{q_\theta}V_n(q_\theta\,|\,q_\theta^{(k)}) + \frac{1}{2\tau}W_2^2(q_\theta,\, q_\theta^{(k)})
\end{equation*}
where given any $q_\theta'\in\ms P(\theta)$, the (sample) energy (or KL divergence) functional $V_n(\cdot\,|\,q_\theta')$ is given by 
\begin{align*}
    V _n(q_\theta\,|\,q_\theta') :\,= n \,\mb E_{q_\theta} \big[U_n(\theta;\, q'_\theta) \big]+ \KL(q_\theta\,||\,\pi_\theta),
\end{align*}
and $U_n(\cdot\,;\,q'_\theta)$ is the (sample) potential function given in~\eqref{eq:U_n_exp}. 

The main difficulty in analyzing the MF-WGF algorithm is that the energy functional $V_n(\,\cdot\,|\,q_\theta^{(k)})$ determining $q_\theta^{(k+1)}$ also depends on the previous iterate $q_\theta^{(k)}$. With a time-independent energy functional, whose global minimizer denoted as $\pi^\ast$, we may directly apply Theorem~\ref{thm: implicit_WGF_conv} with $\pi = \pi^\ast$ to prove the contraction of the one-step discrete WGF towards $\pi^\ast$. However, by directly applying Theorem~\ref{thm: implicit_WGF_conv} with $\pi$ therein being the minimizer of $V_n(\,\cdot\,|\,q_\theta^{(k)})$, we can only prove the one-step contraction of MF-WGF towards this minimizer, which changes over iteration count $k$ and is generally different from the target $\wht q_\theta$. Fortunately, the freedom of choosing an arbitrary $\pi$ in Theorem~\ref{thm: implicit_WGF_conv} allows us to directly apply the theorem to analyze the sample-level MF-WGF by taking $\pi=\wht q_\theta$; however, some careful perturbation analysis will be required to show that such a replacement will only incur negligible extra error, as we will elaborate in the proof below.

In order to apply Theorem~\ref{thm: implicit_WGF_conv} to energy functional $V_n(\,\cdot\,|\,q_\theta^{(k)})$, we first need to show its convexity along generalized geodesics, which according to Lemmas~\ref{lemma:entropy} and~\ref{lemma:potential} (also see Corollary~\ref{coro:one_step_KL}), boils down to the verification of the strong convexity of sample potential function $U_n(\cdot,\, q_\theta^{(k)})$. To prove this, we know that by assumption~\ref{assump: strong_convexity}, the population level potential $U(\cdot, \mu)$ is $\lambda$-strongly convex for all $\mu$ such that $W_2(\mu, \delta_\theta^\ast) \leq r$. 

The following lemma provides an explicit error bound to the (Hessian of) sample potential function $U_n$. It allows $U_n$ to inherit the strong-convexity of $U$ given the error is strictly small than the $\lambda$ in Assumption~\ref{assump: strong_convexity}, and plays a crucial role in showing the closeness between the population and sample versions of MF-WGF.

\begin{lemma}\label{lem: ULLN_of_potential_function}
Under Assumption~\ref{assump: continuity_of_Hessian}, for any $\eta\in(0,1)$, if the sample size satisfies
\begin{displaymath}
n \geq \max\bigg\{\, 6, \ \frac{324RdJ_\ast}{\sigma_1},\ \Big(C\sigma_2\log\frac{6}{\eta}\Big)^6,\  \Big(C\sigma_3\log\frac{6}{\eta}\Big)^6\, \bigg\},
\end{displaymath}
for some constant $C$ (explicit expression provided in the proof) and a radius parameter $r_n$ satisfies
\begin{align*}
    r_n \leq \frac{1}{K}\Bigg(\sqrt{\frac{4d\sigma_1\log\frac{n}{\eta}}{9\big[\big(\mb E_{\theta^\ast}[\lambda(X)^3] + 1\big)^{2/3} + \big(\mb E_{\theta^\ast}[S_1(X)^3]+1\big)^{2/3}\big]}}\ -\, 1\,\Bigg),
\end{align*}
then the following inequality holds with probability at least $1 - \eta$,
\begin{equation}\label{eqn:uniform_Hessian}
\sup_{\substack{\theta\in\Theta\\ \mu: \,W_2(\mu, \delta_{\theta^\ast})\leq r_n}} \matnorm{\nabla^2U_n(\theta;\, \mu) - \nabla^2U(\theta;\, \mu)}\leq \frac{2d\sigma_1\log\frac{n}{\eta}}{\sqrt{n}}.
\end{equation}
\end{lemma}

By Lemma \ref{lem: ULLN_of_potential_function}, we also know that for any $0 < \eta < 1$, if
\begin{displaymath}
n \geq \max\bigg\{6, \frac{324RdJ_\ast}{\sigma_1},  \Big(\frac{\eta\lambda}{2d\sigma_1}\Big)^2\bigg\},
\end{displaymath}
then
\begin{equation}\label{eqn: convexity_sample_potential_func}
\sup_{\substack{\theta\in\Theta\\ \mu:W_2(\mu, \delta_{\theta^\ast})\leq r}} \matnorm{\nabla^2U_n(\theta;\,\mu) - \nabla^2U(\theta;\,\mu)} \leq \eta\lambda
\end{equation}
holds with probability at least $1- ne^{-\frac{\sqrt{n}\eta\lambda}{2d\sigma_1}}$. By combining the two properties together, we conclude that
$U_n(\cdot, \mu)$ is $(1-\eta)\lambda$-strongly convex for all $\mu\in B_{\mb W_2}(\delta_{\theta^\ast}, r)$, indicating that $nU_n(\cdot; \mu) - \log\pi_\theta(\cdot)$ is $\big[n(1-\eta)\lambda - \lambda_M(\nabla^2\log\pi_\theta)\big]$-strongly convex. Now since
\begin{align*}
    V_n(q_\theta\,|\,\mu) = \int_\Theta nU_n(\theta; \mu) -\log\pi_\theta(\theta)\,\dd q_\theta + \int_\Theta\log q_\theta\,\dd q_\theta,
\end{align*}
we can further conclude by Lemmas~\ref{lemma:entropy} and~\ref{lemma:potential} that functional $V_n(\cdot\,|\,\mu)$ is $\big[n(1-\eta)\lambda - \lambda_M(\nabla^2\log\pi_\theta)\big]$-strongly convex along generalized geodesics on $\ms P_2^r(\Theta)$ if $\mu\in B_{\mb W_2}(\delta_{\theta^\ast}, r)$.
By Corollary~\ref{coro: bound_square_expectation} we know
\begin{displaymath}
W_2(\widehat{q}_\theta, \delta_{\theta^\ast}) = \sqrt{\mb E_{\widehat{q}_\theta}\|\theta - \theta^\ast\|^2} \leq \frac{R_W}{2},
\end{displaymath}
where recall that $R_W$ is given in Theorem~\ref{thm: main_theorem} and satisfies $R_W\leq r/3$.
Therefore, it remains to show 
\begin{equation}\label{eqn: induction_hypothesis}
    W_2(\widehat{q}_\theta, q_\theta^{(k)}) \leq R_W,
\end{equation}
which implies 
\begin{displaymath}
W_2(\delta_{\theta^\ast}, q_\theta^{(k)}) \leq W_2(\widehat{q}_\theta, \delta_{\theta^\ast}) + W_2(\widehat{q}_\theta, q_\theta^{(k)}) \leq \frac{R_W}{2} + R_W \leq r.
\end{displaymath}
In the rest of the proof, we will use induction to prove~\eqref{eqn: induction_hypothesis} by showing that if $W_2(\widehat{q}_\theta, q_\theta^{(k)}) \leq R_W$ holds, then
\begin{align}
    &\quad W_2^2(\widehat{q}_\theta, q_\theta^{(k+1)})\notag \\
    &\leq \bigg(\frac{1 + \eta n\tau\lambda + (1+\eta)n\tau\gamma}{1 + (1-2\eta)n\tau\lambda - (1+\eta)n\tau\gamma - \lambda_M(\nabla^2\log\pi_\theta)}\bigg)\, W_2^2(\widehat{q}_\theta, q_\theta^{(k)})\label{eqn:induction_ineq}\\
    &\leq W_2^2(\widehat{q}_\theta, q_\theta^{(k)}) \leq R_W^2,\notag
\end{align}
which also implies the claimed bound in the theorem by repeatedly applying the above one-step contraction bound.

\smallskip

\noindent
Since according to the condition of the theorem on the initialization, 
\begin{align*}
    W_2(\wht q_\theta, q_\theta^{(0)}) \leq W_2(\wht q_\theta, \delta_{\theta^\ast}) + W_2(\delta_{\theta^\ast}, q_\theta^{(0)}) \leq \frac{R_W}{2} + \frac{R_W}{2} = R_W,
\end{align*}
we know that \eqref{eqn: induction_hypothesis} is true for $k=0$.

Now suppose statement~\eqref{eqn: induction_hypothesis} holds for integer $k\geq 1$, let us bound $W_2(\widehat{q}_\theta, q_\theta^{(k+1)})$ by proving~\eqref{eqn:induction_ineq}.
In fact, since according to our previous argument,~\eqref{eqn: induction_hypothesis} implies $V_n(\cdot\,|\,q_\theta^{(k)})$ to be 
$\big[n(1-\eta)\lambda - \lambda_M(\nabla^2\log\pi_\theta)\big]$-strongly convex along generalized geodesics, we may apply Theorem \ref{thm: implicit_WGF_conv} with $\m F = V_n(\cdot\,|\,q_\theta^{(k)})$, $\mu = q_\theta^{(k)}$, and $\pi =\widehat{q}_\theta$ to obtain
\begin{eqnarray}\label{eqn: main_inequality}
&\quad\,& \Big(1+(1-\eta)n\lambda\tau - \tau\lambda_M(\nabla^2\log\pi_\theta)\Big)\, W_2^2(q_{\theta}^{(k+1)}, \widehat{q}_\theta) \nonumber\\
&\leq& W_2^2(q_\theta^{(k)}, \widehat{q}_\theta) - 2\tau\,\big[V_n(q_\theta^{(k+1)}\,|\,q_\theta^{(k)}) - V_n(\widehat{q}_\theta\,|\,q_\theta^{(k)})\big] - W_2^2(q_\theta^{(k+1)}, q_\theta^{(k)}).
\end{eqnarray}
Unfortunately, $V_n(q_\theta^{(k+1)}\,|\,q_\theta^{(k)}) - V_n(\widehat{q}_\theta\,|\,q_\theta^{(k)})$ in the second term on the right hand side of the inequality is not necessarily non-negative since $\wht q_\theta$ is generally not the minimizer of $V_n(\cdot\,|\,q_\theta^{(k)})$. However, we know that $\wht q_\theta$ minimizes $V_n(\cdot\,|\,\wht q_\theta)$ according to the argument in Section~\ref{sec:fix_point_proof}.
This motivates us to consider a perturbation analysis by
substituting $V_n(q_\theta^{(k+1)}\,|\,q_\theta^{(k)}) - V_n(\widehat{q}_\theta\,|\,q_\theta^{(k)})$ with the non-negative quantity $V_n(q_\theta^{(k+1)}\,|\,\widehat{q}_\theta) - V_n(\widehat{q}_\theta\,|\,\widehat{q}_\theta)$ in~\eqref{eqn: main_inequality}, and properly analyzing the resulting difference due to the substitution. Specifically, the difference can be explicitly expressed as
\begin{align*}
    &\quad\, \Big|\big[V_n\big(q_\theta^{(k+1)}\,\big|\,q_\theta^{(k)}\big) - V_n\big(\widehat{q}_\theta\,\big|\,q_\theta^{(k)}\big)\big] - \big[V_n\big(q_\theta^{(k+1)}\,\big|\,\widehat{q}_\theta\big) - V_n\big(\widehat{q}_\theta\,\big|\,\widehat{q}_\theta\big)\big]\Big|\\
    &=\Big|\, n\int_\Theta U_n(\theta, q_\theta^{(k)}) - U_n(\theta, \widehat{q}_\theta)\,\dd(q_\theta^{(k+1)} - \widehat{q}_\theta)\, \Big|\\
    &= \bigg|\int_\Theta\sum_{i=1}^n\sum_{z=1}^K\log p(z\,|\,X_i,\theta)\big[\Phi(\widehat{q}_\theta, X_i)(z) - \Phi(q_\theta^{(k)}, X_i)(z)\big]\big(q_\theta^{(k+1)}(\theta) - \widehat{q}_\theta(\theta)\big)\,\dd\theta\bigg|.
\end{align*}
The following key lemma shows that this difference is related to a sample version of the missing data (Fisher) information matrix $\wht I_S(\theta^\ast)$. The analysis following this lemma illustrates that this term is generally of order $\gamma\,W_2(q^{(k)}_\theta,\,\wht q_\theta)\, W_2(q^{(k+1)},\, \wht q_\theta)$, where recall that $\gamma$ is defined as the matrix operator norm of $I_S(\theta^\ast)$ in the statement of the theorem. The proof of the lemma is quite long and technical, and is therefore postponed to Section~\ref{proof_lem:cross_produc}.
    
\begin{lemma}\label{lem:cross_produc}
We use the shorthand $W_k = W_2(\widehat{q}_\theta, q_\theta^{(k)})$ for $k\geq 0$, and $\Delta_q = \int_{\Theta} \theta\,\dd (q-\wht q_\theta) = \int_{\Theta} \big[t_{\wht q_\theta}^q(\theta) -\theta\big]\,\dd\wht q_\theta \in\mb R^d$ as the difference between the mean vectors under a generic distribution $q$ and the target MF approximation $\wht q_\theta$ over $\Theta$. Then
\begin{align*}
    &\bigg|\int_\Theta\sum_{i=1}^n\sum_{z=1}^K\log p(z\,|\,X_i,\theta)\big[\Phi(\widehat{q}_\theta, X_i)(z) - \Phi(q_\theta^{(k)}, X_i)(z)\big]\big(q_\theta^{(k+1)}(\theta) - \widehat{q}_\theta(\theta)\big)\,\dd\theta \\
    &\qquad\qquad\qquad\qquad - n\,\Big\langle\,\Delta_{q_\theta^{(k)}},\ \wht I_S(\theta^\ast)\,\Delta_{q_\theta^{(k+1)}}\,\Big\rangle\bigg| \leq R_1+R_2+R_3,
\end{align*}
where $\wht I_S(\theta^\ast)\in\mb R^{d\times d}$ is the sample missing data information matrix at $\theta^\ast$,
\begin{equation}\label{eqn: missing_information}
\wht I_S(\theta^\ast) = \frac{1}{n}\sum_{i=1}^n\sum_{z=1}^K p(z\,|\,X_i,\theta^\ast)\big[\nabla\log p(z\,|\,X_i,\theta^\ast)\big]\big[\nabla\log p(z\,|\,X_i,\theta^\ast)\big]^T,
\end{equation}
and the three higher-order remainder terms take the form as
\begin{align*}
R_1 &= nW_k^2W_{k+1}^2\bigg[\frac{3K^2+2K}{4n}\sum_{i=1}^nS_1(X_i)^3 + \frac{27K^3+24K^2+4K}{16n}\sum_{i=1}^n\lambda_1(X_i)^3 + \frac{K}{2}\Big(\frac{1}{n}\sum_{i=1}^n\lambda(X_1)^3\Big)^{\frac{2}{3}}\bigg]\\
&\quad+
nW_k^2W_{k+1}\bigg[\frac{21K^2+32K+12}{8n}\sum_{i=1}^n S_1(X_i)^3 + \frac{27K^3+42K^2+16K}{8n}\sum_{i=1}^n\lambda(X_i)^3\\
&\qquad\qquad\qquad\quad\qquad\qquad\qquad\qquad +
\frac{1}{2}\Big(\frac{1}{n}\sum_{i=1}^nS_1(X_i)^3\Big)^{\frac{2}{3}} + \frac{2K+1}{2}\Big(\frac{1}{n}\sum_{i=1}^n\lambda(X_i)^3\Big)^{\frac{2}{3}}\bigg],\\
R_2 &= nW_{k+1}^2W_k\bigg(\frac{2K^2+K}{2}\Big(\frac{1}{n}\sum_{i=1}^n\lambda(X_i)^3\Big)^{\frac{2}{3}} + \frac{K}{2}\Big(\frac{1}{n}\sum_{i=1}^nS_1(X_i)^3\Big)^{\frac{2}{3}}\bigg),\\
R_3 &= nW_{k+1}W_k\sqrt{\mb E_{\widehat{q}_\theta}\|\theta-\theta^\ast\|^2}\bigg[K(2K+3) \Big(\frac{1}{n}\sum_{i=1}^n\lambda(X_i)^3\Big)^{\frac{2}{3}}\\
&\qquad\qquad\qquad\qquad\qquad\qquad\quad+ 
(2K+1)\Big(\frac{1}{n}\sum_{i=1}^nS_1(X_i)^3\Big)^{\frac{2}{3}} + (\frac{K}{2}+1)\frac{1}{n}\sum_{i=1}^nS_2(X_i)^2\bigg].
\end{align*}
\end{lemma}
    
To further simplify some quantities in Lemma~\ref{lem:cross_produc}, we note that similar to the proof of Lemma~\ref{lem: ULLN_of_potential_function}, we may apply standard concentration inequalities to show that the following three inequalities
\begin{align*}
\frac{1}{n}\sum_{i=1}^n\lambda(X_i)^3 &\leq \mb E_{\theta^\ast}[\lambda(X)^3] + 1,\\
\frac{1}{n}\sum_{i=1}^nS_1(X_i)^3 &\leq \mb E_{\theta^\ast}[S_1(X)^3] + 1,\\
\frac{1}{n}\sum_{i=1}^nS_2(X_i)^2 &\leq \mb E_{\theta^\ast}[S_2(X)^2] + 1
\end{align*}
hold with at least probability $1-2e^{-Cn^\frac{1}{6}\sigma_2^{-1}} - 4e^{-Cn^\frac{1}{6}\sigma_3^{-1}}$. Let $\gamma_n = |\!|\!|\wht I_S(\theta^\ast)|\!|\!|_{\rm op}$. Also notice that by the Cauchy--Schwarz inequality, we can bound the leading term as
\begin{align*}
    \big|\big\langle\Delta_{q_\theta^{(k)}}, \wht I_S(\theta^\ast)\Delta_{q_\theta^{(k+1)}}\big\rangle\big| &\leq  \gamma_n W_2(q_\theta^{(k)}, \widehat{q}_\theta)W_2(q_\theta^{(k+1)}, \widehat{q}_\theta)\\
    &\leq  \Big(1+\frac{\eta}{2}\Big)\gamma W_2(q_\theta^{(k)}, \widehat{q}_\theta)W_2(q_\theta^{(k+1)}, \widehat{q}_\theta)
\end{align*}
with probability at least $1 - 2e^{3d - Cn\sigma_3^{-1}\gamma\eta/2}$ for $0\leq \eta\leq \frac{2\sigma_3}{\gamma}$. Here, the second line is due to the following lemma, whose proof is deferred to Section~\ref{proof_lem: matrix_conc}.
\begin{lemma}\label{lem: matrix_conc}
Define the population level of missing data information matrix as
\begin{align*}
    I_S(\theta^\ast) = \mb E_{\theta^\ast} \big[\wht I_S(\theta^\ast)\big] = \int_{\mb R^d} \sum_{z=1}^K p(z\,|\,x,\theta^\ast)\big[\nabla\log p(z\,|\,x,\theta^\ast)\big]\big[\nabla\log p(z\,|\,x,\theta^\ast)\big]^T p(x\,|\,\theta^\ast)\,\dd x.
\end{align*}
Under Assumption~\ref{assump: continuity_of_Hessian}, we have
\begin{align*}
    \mb P\bigg(\frac{1}{\sigma_3}\matnorm{\wht I_S(\theta^\ast) - I_S(\theta^\ast)} > t\bigg) \leq 2\,e^{3d - CN\min\{t^2, t\}},\quad t>0.
\end{align*}
\end{lemma}

\noindent Now combining all pieces above yields
\begin{align*}
&\quad\, \frac{1}{n}\Big|\big[V_n\big(q_\theta^{(k+1)}\,\big|\,q_\theta^{(k)}\big) - V_n\big(\widehat{q}_\theta\,\big|\,q_\theta^{(k)}\big)\big] - \big[V_n\big(q_\theta^{(k+1)}\,\big|\,\widehat{q}_\theta\big) - V_n\big(\widehat{q}_\theta\,\big|\,\widehat{q}_\theta\big)\big]\Big|\\
&\leq AW_k^2W_{k+1}^2 + BW_k^2W_{k+1} + CW_kW_{k+1}^2 + \Big(D\sqrt{\mb E_{\widehat{q}_\theta}\|\theta-\theta^\ast\|^2}+\big(1+\frac{\eta}{2}\big)\gamma\Big)W_kW_{k+1}.
\end{align*}
Since the sample size $n$ satisfies
\begin{displaymath}
n > \Big(\frac{4(\mb EG(X) + c_1 + c_3 + 2) + c_2}{c_2\eta^2\gamma^2/4}\Big)^{2},
\end{displaymath}
we have $D\sqrt{\mb E_{\widehat{q}_\theta}\|\theta-\theta^\ast\|^2} < \frac{\eta\gamma}{2}$ by Corollary~\ref{coro: bound_square_expectation}. 

\smallskip

By combining inequality~\eqref{eqn: main_inequality} and discussions above, and noticing $V_n(q_\theta^{(k+1)}\,|\,\widehat{q}_\theta) - V_n(\widehat{q}_\theta\,|\,\widehat{q}_\theta) \geq 0$ since $\wht q_\theta$ minimizes $V_n$, we obtain
\begin{align*}
&\quad\big(1+(1-\eta)n\lambda\tau - \tau\lambda_M(\nabla^2\log\pi_\theta)\big)\, W_{k+1}^2\\
&\leq W_k^2 +2\tau n\Big(AW_k^2W_{k+1}^2 + BW_k^2W_{k+1} + CW_kW_{k+1}^2 + \big(\frac{\eta\gamma}{2}+\big(1+\frac{\eta}{2}\big)\gamma\big)W_kW_{k+1}\Big)\\
&\leq W_k^2 +2\tau n\Big(AR_W^2W_{k+1}^2 + BR_WW_kW_{k+1} + CR_WW_{k+1}^2 + (1+\eta)\gamma W_kW_{k+1}\Big),
\end{align*}
where the last inequality is due to $W_k \leq R_W$ by our induction hypothesis. Therefore, rearranging the preceding inequality leads to
\begin{align*}
&\quad\,\Big[\big(1 + (1-\eta)n\lambda\tau - \tau\lambda_M(\nabla^2\log\pi_\theta)\big) - 2\tau n(AR_W^2 + CR_W)\Big]W_{k+1}^2 - W_k^2\\
&\leq 2\tau n\big(BR_W + (1+\eta)\gamma\big)W_kW_{k+1}
\leq \tau n\big(BR_W + (1+\eta)\gamma\big)(W_k^2 + W_{k+1}^2),
\end{align*}
which further implies
\begin{align*}
&\quad\Big[1 + (1-\eta)n\lambda\tau - \tau\lambda_M(\nabla^2\log\pi_\theta) - \tau n\Big(2AR_W^2 + 2CR_W + BR_W + \big(1+\eta)\gamma\Big)\Big] W_{k+1}^2\\ 
&\leq \Big[1 + \tau n\Big(BR_W + \big(1+\eta\big)\gamma\Big)\Big]W_k^2.
\end{align*}
By the definition of $R_W$, we have $AR_W^2<\frac{\eta\lambda}{8}$, $CR_W < \frac{\eta\lambda}{8}$, and $BR_W < \frac{\eta\lambda}{2}$, so
\begin{displaymath}
\Big[1 + (1-2\eta)n\tau\lambda - \tau\lambda_M(\nabla^2\log\pi_\theta) - (1+\eta)n\tau\gamma\Big]W_{k+1}^2 < \Big[1 + \big(1+\eta\big)n\tau\gamma + \tau n\eta\lambda\Big]W_k^2,
\end{displaymath}
or
\begin{displaymath}
\bigg(\frac{W_{k+1}}{W_k}\bigg)^2 < \frac{1 + \eta n\tau\lambda + (1+\eta)n\tau\gamma}{1 + (1-2\eta)n\tau\lambda - (1+\eta)n\tau\gamma - \tau\lambda_M(\nabla^2\log\pi_\theta)},
\end{displaymath}
which leads to the desired contraction bound~\eqref{eqn:induction_ineq}.
Furthermore, if $(1-3\eta)\lambda > (2+2\eta)\gamma + \frac{1}{n}\lambda_M(\nabla^2\log\pi_\theta)$, then we have $W_{k+1} \leq W_k \leq R_n$. By the induction, we proved~\eqref{eqn: induction_hypothesis}. Moreover, repeatedly applying the preceding display leads to
\begin{displaymath}
W_2^2(\widehat{q}_\theta, q_\theta^{(k)}) \leq \bigg(1 - \frac{(1-3\eta)\lambda - (2+2\eta)\gamma - \frac{1}{n}\lambda_M(\nabla^2\log\pi_\theta)}{(1-2\eta)\lambda - (1+\eta)\gamma - \frac{1}{n}\lambda_M(\nabla^2\log\pi_\theta) + \frac{1}{n\tau}}\bigg)^k W_2^2(\widehat{q}_\theta, q_\theta^{(0)}).
\end{displaymath}
Taking $\eta = \frac{\lambda-2\gamma}{2(3\lambda+2\gamma)}$ yields the desired result in the theorem.

\subsection{Proof of Theorem~\ref{thm: FA_approx}}\label{app: proof_FA}
For any $\rho$, let $T$ be the optimal map from $\rho_k^\tau$ to $\rho$.
Note that we have
\begin{align*}
\int (T_\#\rho)\log(T_\#\rho) = \int\rho\log\rho  -\int\log\lvert \det\nabla T\rvert\,\dd\rho.
\end{align*}
This equation is due to the change of measure formula, and more details can be found in~\cite{mokrov2021large}. Therefore, we have
\begin{align*}
&\m F_{\rm KL}(\rho_{k+1}^\tau) + \frac{1}{2\tau}W_2^2(\rho_{k+1}^\tau, \rho_k^\tau)\\
&= \int V\circ T_k^\tau\,\dd\rho_k^\tau + \int\rho_k^\tau\log\rho_k^\tau - \int \log\lvert\det \nabla T_k^\tau\rvert\,\dd\rho_k^\tau + \frac{1}{2\tau} \int\|T_k^\tau - \id\|^2\,\dd\rho_k^\tau\\
&\leq \int V\circ T\,\dd\rho_k^\tau + \int\rho_k^\tau\log\rho_k^\tau - \int \log\lvert\det \nabla T\rvert\,\dd\rho_k^\tau + \frac{1}{2\tau} \int\|T - \id\|^2\,\dd\rho_k^\tau\\
&= \m F_{\rm KL}(\rho) + \frac{1}{2\tau}W_2^2(\rho, \rho_k^\tau).
\end{align*}
Thus, we know $\rho_{k+1}^\tau = (T_k^\tau)_\#\rho_k^\tau$ minimizes~\eqref{eqn: KL_JKO}.
\section{More details and proofs about discretized Wasserstein gradient flow for MF inference}
In this appendix, we provide more details and proofs to some claims and results in Section~\ref{sec:MF-WGF} about the discretized Wasserstein gradient flow for implementing the MF approximation in Bayesian latent variable models. 

\subsection{Fix point characterization of MF-WGF}\label{sec:fix_point_proof}
We show that $\wht q_\theta$ is the unique fix point to the MF-WGF update equation, i.e.~any solution to
\begin{align*}
     \widehat{q}_\theta=\argmin_{\rho} V_n(\rho\,|\,\widehat{q}_\theta) + \frac{1}{2\tau} W_2^2(\rho,\widehat{q}_\theta),
\end{align*}
must satisfy equation~\eqref{eqn: distributional_equation} in the main paper, or
\begin{align*}
    \mu(\theta) = \frac{1}{Z_n(\mu)}\,\pi_\theta(\theta)\, e^{-n\,U_n(\theta;\, \mu)}, \quad\mx{with } Z_n(\mu) = \int_\Theta\pi_\theta(\theta)\, e^{-n\,U_n(\theta;\, \mu)}\,\dd\theta,
\end{align*}

To show this, notice that the first variation of $V_n(\cdot\,|\,\widehat{q}_\theta)$ at $\rho$ is
\begin{displaymath}
\frac{\delta V_n(\cdot\,|\,\widehat{q}_\theta)}{\delta\rho}(\rho) = n\,U_n(\theta,\, \widehat{q}_\theta) + \log \rho(\theta) - \log\pi_\theta(\theta) +C,
\end{displaymath}
where $C$ is any constant since the first variation is only uniquely determined up to a constant.
Since the first-order optimality condition of any distribution $\rho$ to be the local minimizer of $V_n(\cdot\,|\,\widehat{q}_\theta)$ is $\frac{\delta V_n(\cdot\,|\,\widehat{q}_\theta)}{\delta\rho}(\rho)$ being a constant a.e.~on the its support set $\{\rho>0\}$, it must satisfy
$\rho(\theta) \propto \pi_\theta(\theta)\, e^{-n\,U_n(\theta,\, \widehat{q}_\theta)}$. As a consequence, $\widehat{q}_\theta \propto \pi_\theta(\theta)\, e^{-n\,U_n(\theta,\, \widehat{q}_\theta)}$ is the unique global minimizer of $V_n(\cdot\,|\,\widehat{q}_\theta)$, i.e.~$\widehat{q}_\theta = \argmin_\mu V_n(\mu\,|\,\widehat{q}_\theta)$. Note that $\widehat{q}_\theta$ is also the unique minimizer of $W_2^2(\,\cdot\,, \widehat{q}_\theta)$; therefore $\widehat{q}_\theta$ is the unique minimizer to the objective functional in~\eqref{eqn: JKO_update_qtheta}, that is,
\begin{align*}
     \widehat{q}_\theta=\argmin_{\rho} V_n(\rho\,|\,\widehat{q}_\theta) + \frac{1}{2\tau} W_2^2(\rho,\widehat{q}_\theta),
\end{align*}
which proves the claim. It is worthwhile noticing that the fix point characterization of MF-WGF is identical to the characterization of the MF approximation $\wht q_\theta$ obtained in Lemma~\ref{lem:functional_Vn} using the optimality of $\wht q_\theta$ as the minimizer to the variational KL divergence objective functional.

\subsection{Proof of Lemma~\ref{lem: onestep_err}}\label{app:proof_lem:onestep}
Explicit expressions of the constants in the lemma are provided below:
\begin{align*}
    W_2^2(\rho_\tau^{\rm FP}, \rho_\tau^L) &\leq \frac{L^2e}{2}\Big(d+\mb E_{\rho}\|\nabla V(X)\|^2\Big)\,\tau^3,\\
    W_2^2(\rho_\tau^{\rm FP}, \rho_\tau) &\leq C\bigg(\int_{\mb R^d}\Big|\!\Big|\nabla\frac{\delta\m F_{\rm KL}}{\delta\rho}(\rho)\Big|\!\Big|^2\,\dd\rho\bigg)^{1/2}\\
    &\qquad\quad\cdot\Big(\int_{\mb R^d}\big(\|\nabla V\|^2 - 2\Delta V - 2\Delta\log\rho - \|\nabla\log\rho\|^2\big)^2\,\dd\rho\Big)^{1/2}\,\tau^3,\\
    W_2^2(\rho_\tau^{\rm ex}, \rho_\tau) &\leq C\bigg(\int_{\mb R^d}\bigg|\!\bigg|\nabla^2\log\rho(y)\nabla\frac{\delta\m F_{\rm KL}}{\delta\rho}(\rho)(y) + \nabla\textrm{tr}\Big(\nabla^2\frac{\delta\m F_{\rm KL}}{\delta\rho}(\rho)\Big)(y)\bigg|\!\bigg|^2\,\dd\rho\\
    &\qquad\quad  + \int_{\mb R^d}\Big|\!\Big| \nabla^2\frac{\delta\m F_{\rm KL}}{\delta\rho}(\rho)(y)\nabla\frac{\delta\m F_{\rm KL}}{\delta\rho}(\rho)(y)\big)\Big|\!\Big|^2\,\dd\rho\bigg)\, \tau^4,
\end{align*}
In particular, if the following quantities concerning the regularity of initial density $\rho$ are all bounded, 
\begin{equation}\label{eqn:bounded_quantities}
\begin{aligned}
    &\int_{\mb R^d}\|\nabla V\|^4\,\dd\rho, \quad \int_{\mb R^d}(\Delta V)^2\,\dd\rho, \quad \int_{\mb R^d}(\Delta\log\rho)^2\,\dd\rho, \quad \int_{\mb R^d}\|\nabla\log\rho\|^4\,\dd\rho,\\
    &\int_{\mb R^d}\bigg|\!\bigg|\nabla^2\log\rho(y)\nabla\frac{\delta\m F_{\rm KL}}{\delta\rho}(\rho)(y) + \nabla{\rm tr}\Big(\nabla^2\frac{\delta\m F_{\rm KL}}{\delta\rho}(\rho)\Big)(y)\bigg|\!\bigg|^2\,\dd\rho, \quad \mbox{and}\\
    &\int_{\mb R^d}\Big|\!\Big| \nabla^2\frac{\delta\m F_{\rm KL}}{\delta\rho}(\rho)(y)\nabla\frac{\delta\m F_{\rm KL}}{\delta\rho}(\rho)(y)\big)\Big|\!\Big|^2\,\dd\rho,
\end{aligned}
\end{equation}
then we have $W_2(\rho_\tau^{\rm ex}, \rho_\tau) \lesssim \tau^{3/2}$ and $W_2(\rho_\tau^L, \rho_\tau) \lesssim \tau^{3/2}$.

\smallskip
\noindent
For the proof, first, let us make the definition of metric slope as a generalization of the Fr\'{e}chet derivative defined on normed spaces to the Wasserstein space $W_2(\mb R^d)$ that only admits a distance metric but not a norm.

\begin{definition}[Section 10 in \cite{ambrosio2008gradient}]
For any functional $\m F$ on $\ms P_2^r(\mb R^d)$ and $\rho$ s.t. $\m F(\rho) < \infty$, the \emph{metric slope}
\begin{align*}
    |\partial\m F|(\rho) = \limsup_{\rho'\to\rho}\frac{\big(\m F(\rho) - \m F(\rho')\big)^+}{W_2(\rho,\rho')}
\end{align*}
is finite if and only if $\partial \m F(\rho)$ is not empty.
\end{definition}

\noindent \underline{\bf Bound of $W_2^2(\rho_\tau^{\rm FP}, \rho_\tau^{L})$}.
Consider the coupling
\begin{align*}
    \dd X_t &= -\nabla V(X_t)\,\dd t + \sqrt{2}\,\dd W_t\\
    \dd X_t' &= -\nabla V(X_0')\,\dd t + \sqrt{2}\,\dd W_t
\end{align*}
with initial values $X_0'= X_0\sim\rho_0 = \rho$ for $t\in[0, \tau]$. From the relationship between the Langevin SDE and the Fokker Planck equation, we have $X_\tau\sim \rho_\tau^{\rm FP}$. Moreover, by our construction, we have
\begin{displaymath}
X_\tau' = X_0' - \tau\nabla V(X_0') + \sqrt{2}W_\tau\sim \rho_{\tau}^L.
\end{displaymath}

By taking the difference, we get $\dd (X_t - X_t') = \big(\nabla V(X_0') - \nabla V(X_t)\big)\,\dd t$, which further implies 
\begin{align*}
\|X_\tau - X_\tau'\| &= \bigg|\!\bigg|\int_0^\tau \big(\nabla V(X_0') - \nabla V(X_t)\big)\,\dd t\bigg|\!\bigg|_2\\
&\leq \int_0^\tau \|\nabla V(X_0') - \nabla V(X_t)\|\,\dd t\leq L\int_0^\tau\|X_0' - X_t\|\,\dd t
\end{align*}
given that $\nabla V$ is $L$-Lipschitz. Therefore, we get
\begin{align*}
W_2^2(\rho_\tau^{\rm FP}, \rho_{\tau}^L) &\leq \mb E\|X_\tau - X_\tau'\|^2\\
&\leq \mb E\bigg(L\int_0^\tau\|X_0'-X_t\|\,\dd t\bigg)^2 \leq L^2\tau\int_0^\tau \mb E\|X_0'-X_t\|^2\,\dd t
\end{align*}
by applying the Cauchy--Schwarz inequality. 

Now, let us bound the expectation term $\mb E[\|X_0'-X_t\|^2]$. By Ito's formula, we have
\begin{displaymath} 
\dd \|X_t - X_0'\|^2 = \big[d - \big\langle 2(X_t - X_0'), \nabla V(X_t)\big\rangle\big]\,\dd t + \big\langle2\sqrt{2}(X_t-X_0'), \,\dd W_t\big\rangle.
\end{displaymath}
By the Cauchy--Schwarz inequality and the AM--GM inequality, we can further get
\begin{align*}
\mb E\|X_t - X_0'\|^2 &= dt - 2\int_0^t\mb E\big\langle X_s-X_0', \nabla V(X_s)\big\rangle\,\dd s\\
&= dt - 2\int_0^t\mb E\big\langle X_s - X_0', \nabla V(X_s) - \nabla V(X_0')\big\rangle\,\dd s - 2\int_0^t\mb E\big\langle X_s - X_0', \nabla V(X_0')\big\rangle\,\dd s\\
&\leq dt + 2L\int_0^t \mb E\|X_s-X_0'\|^2\,\dd s + 2\int_0^t\sqrt{\mb E\|X_s - X_0'\|^2}\cdot\sqrt{\mb E\|\nabla V(X_0')\|^2}\,\dd s\\
&\leq \Big(d + \mb E_{\rho_0}\|\nabla V(X)\|^2\Big)t + (2L+1)\int_0^t\mb E\|X_s-X_0'\|^2\,\dd s
\end{align*}
for all $t\in[0,\tau]$. By applying Gronwall's inequality to the above, we have
\begin{displaymath}
\mb E\|X_t - X_0'\|^2\leq \Big(d+\mb E_{\rho_0}\|\nabla V(X)\|^2\Big)t e^{(2L+1)t}.
\end{displaymath}
Since $0 \leq t\leq  \tau \leq (2L+1)^{-1}$, we finally reach
\begin{displaymath}
W_2^2(\rho_\tau^{\rm FP},\rho_{\tau}^L)\leq L^2\tau\int_0^\tau \Big(d+\mb E_{\rho_0}\|\nabla V(X)\|^2\Big)\,e\,t\,\dd t = \frac{L^2e}{2}\Big(d+\mb E_{\rho_0}\|\nabla V(X)\|^2\Big)\tau^3.
\end{displaymath}

\smallskip
\noindent \underline{\bf Bound of $W_2^2(\rho_\tau^{\rm FP}, \rho_\tau)$}. By (4.2.10) and Lemma 4.4.1 in \cite{ambrosio2008gradient}, we have
\begin{align*}
    W_2^2(\rho_\tau^{\rm FP}, \rho_\tau) \leq \frac{\tau^2}{2}\Big(|\partial\m F_{\rm KL}|^2(\rho) - |\partial\m F_{\rm KL}|^2(\rho_\tau)\Big).
\end{align*}
If we can show that the metric slope of the functional $|\partial\m F_{\rm KL}|^2(\rho)$ is finite, then 
\begin{align*}
    \frac{|\partial\m F_{\rm KL}|^2(\rho) - |\partial\m F_{\rm KL}|^2(\rho_\tau)}{W_2(\rho, \rho_\tau)}
    &\leq \frac{\big(|\partial\m F_{\rm KL}|^2(\rho) - |\partial\m F_{\rm KL}|^2(\rho_\tau)\big)^+}{W_2(\rho, \rho_\tau)}\\
    &\leq \limsup_{\rho'\to\rho}\frac{\big(|\partial\m F_{\rm KL}|^2(\rho) - |\partial\m F_{\rm KL}|^2(\rho')\big)^+}{W_2(\rho, \rho')}\\
    &= \big|\partial|\partial\m F_{\rm KL}|^2\big|(\rho).
\end{align*}
The above inequality further implies
\begin{align*}
    W_2^2(\rho_\tau^{\rm FP}, \rho_\tau) &\leq \frac{\tau^2W_2(\rho,\rho_\tau)}{2}\big|\partial|\partial\m F_{\rm KL}|^2\big|(\rho)\\
    &\leq C\bigg(\int_{\mb R^d}\Big|\!\Big|\nabla\frac{\delta\m F_{\rm KL}}{\delta\rho}(\rho)\Big|\!\Big|^2\,\dd\rho\bigg)^\frac{1}{2}\big|\partial|\partial\m F_{\rm KL}|^2\big|(\rho) \tau^3.
\end{align*}
To show the finiteness of the metric slope, or $\big|\partial|\partial\m F_{\rm KL}|^2\big|(\rho) < \infty$, we first notice that
\begin{align*}
    |\partial\m F_{\rm KL}|^2(\rho) &\stackrel{\ri}{=} \min\big\{\|\xi\|_{L^2(\rho)}^2: \xi\in\partial\m F_{\rm KL}(\rho)\big\}\stackrel{\rii}{=} \int_{\mb R^d}\|\nabla V + \nabla\log \rho\|^2\,\dd\rho.
\end{align*}
Here, step (i) is by Lemma 10.1.5 in \cite{ambrosio2008gradient}, and step (ii) is by Proposition 3.38 and Proposition 3.36 in \cite{ambrosio2013user} which characterize the uniqueness of subdifferential of $\m F_{\rm KL}$. If we can show that $|\partial\m F_{\rm KL}|$ is Fr\'{e}chet differentiable at $\rho$ relative to the $W_2$ metric, then by Lemma \ref{lem:FV_SD} we have $\big|\partial|\partial\m F_{\rm KL}|^2\big|(\rho) = \frac{\delta|\partial\m F_{\rm KL}|^2}{\delta\rho}(\rho)$. Therefore, we have
\begin{align*}
    |\partial\m F_{\rm KL}|^2(\rho) &= \int_{\mb R^d}\|\nabla V\|^2\,\dd\rho + 2\int_{\mb R^d}\big\langle\nabla V, \nabla\log\rho\big\rangle\,\dd\rho + \int_{\mb R^d}\|\nabla\log\rho\|^2\,\dd\rho\\
    &=: \m F_1(\rho) + \m F_2(\rho) + \m F_3(\rho).
\end{align*}
Here, $\m F_1$ is just the potential energy functional with $\frac{\delta\m F_1}{\delta\rho}(\rho) = \|\nabla V\|^2$. By the formula of integration by parts, we have
\begin{align*}
    \m F_2(\rho) = 2\int_{\mb R^d}\langle\nabla V, \nabla\rho\rangle\,\dd x = -2\int_{\mb R^d}\Delta V\,\dd\rho.
\end{align*}
Therefore $\frac{\delta\m F_2}{\delta\rho}(\rho) = -2\Delta V$. For any $\chi = \rho'-\rho$ such that $\int\,\dd\chi = 0$, we have
\begin{align*}
    \varepsilon^{-1}\big(\m F_3(\rho+\varepsilon\chi) - \m F_3(\rho)\big) &=
    \frac{1}{\varepsilon}\bigg(\int_{\mb R^d}\frac{\|\nabla\rho + \varepsilon\nabla\chi\|^2}{\rho+\varepsilon\nabla\chi}\,\dd x - \int_{\mb R^d}\frac{\|\nabla\rho\|^2}{\rho}\,\dd x\bigg)\\
    &= \frac{1}{\varepsilon}\int_{\mb R^d}\frac{\rho\|\nabla\rho+\varepsilon\nabla\chi\|^2 - (\rho+\varepsilon\chi)\|\nabla\rho\|^2}{\rho(\rho+\varepsilon\chi)}\,\dd x\\
    &= \int_{\mb R^d}\frac{2\rho\langle\nabla\rho, \nabla\chi\rangle - \chi\|\nabla\rho\|^2 + \rho\varepsilon\|\nabla\chi\|^2}{\rho(\rho+\varepsilon\chi)}\,\dd x\\
    &\to 2\int_{\mb R^d}\Big\langle\frac{\nabla\rho}{\rho}, \nabla\chi\Big\rangle\,\dd x - \int_{\mb R^d}\Big|\!\Big|\frac{\nabla\rho}{\rho}\Big|\!\Big|^2\,\dd \chi\\
    &= -\int_{\mb R^d}2\Delta\log\rho + \|\nabla\log\rho\|^2\,\dd\chi, \quad \mx{as } \varepsilon\to0.
\end{align*}
By definition, we have $\frac{\delta\m F_3}{\delta\rho}(\rho) = -2\Delta\log\rho - \|\nabla\log\rho\|^2$. So, we expect to have
\begin{align*}
    \frac{\delta|\partial\m F_{\rm KL}|^2}{\delta\rho}(\rho) &= \frac{\delta\m F_1}{\delta\rho}(\rho) + \frac{\delta\m F_2}{\delta\rho}(\rho) + \frac{\delta\m F_3}{\delta\rho}(\rho)\\
    &= \|\nabla V\|^2 - 2\Delta V - 2\Delta\log\rho  - \|\nabla\log\rho\|^2 \in\partial|\partial\m F_{\rm KL}|^2(\rho).
\end{align*}
Then, by applying again Lemma 10.1.5 in \cite{ambrosio2008gradient}, we obtain
\begin{align*}
    \big|\partial|\partial\m F_{\rm KL}|^2\big|(\rho) &\leq \bigg|\!\bigg|\frac{\delta|\partial\m F_{\rm KL}|^2}{\delta\rho}(\rho)\bigg|\!\bigg|_{L^2(\rho)}\\
    &= \Big(\int_{\mb R^d}\big(\|\nabla V\|^2 - 2\Delta V - 2\Delta\log\rho - \|\nabla\log\rho\|^2\big)^2\,\dd\rho\Big)^\frac{1}{2} < \infty.
\end{align*}
Finally, the Fr\'{e}chet differentiability of $|\partial\m F_{\rm KL}|^2(\rho)$ at $\rho$ is implied by the following identity,
\begin{align*}
    &\quad |\partial\m F_{\rm KL}|^2(\rho+\varepsilon\chi) - |\partial\m F_{\rm KL}|^2(\rho) - \varepsilon\int_{\mb R^d}\|\nabla V\|^2 - 2\Delta V - 2\Delta\log\rho - \|\nabla\log\rho\|^2\,\dd\chi\\
    &= 2\varepsilon^2\int_{\mb R^d}-\frac{\chi\langle\nabla\rho,\nabla\chi\rangle}{\rho(\rho+\varepsilon\chi)} + \frac{\chi^2\|\nabla\rho\|^2}{\rho^2(\rho+\varepsilon\chi)} + \frac{\|\nabla\chi\|^2}{\rho+\varepsilon\chi}\,\dd x= O(\varepsilon^2).
\end{align*}

\smallskip
\noindent \underline{\bf Bound of $W_2^2(\rho_\tau^{\rm ex}, \rho_\tau)$}. 
Recall that $T_{\rho_\tau}^\rho = \id + \tau\nabla\frac{\delta\m F_{\rm KL}}{\delta\rho}(\rho_\tau)$. By the fact that $\id = T_{\rho_\tau}^\rho\circ T_\rho^{\rho_\tau}$, we have $T_{\rho}^{\rho_\tau} = \id - \tau \nabla\frac{\delta\m F_{\rm KL}}{\delta\rho}(\rho_\tau)\circ T_\rho^{\rho_\tau}$. Notice that $(T_\rho^{\rho_\tau^{\rm ex}}, T_\rho^{\rho_\tau})_\#\rho \in \Pi(\rho_\tau^{\rm ex}, \rho_\tau)$ is a coupling between $\rho_\tau^{\rm ex}$ and $\rho_\tau$. Therefore, by definition of $W_2$ distance, we have
\begin{align*}
    W_2^2(\rho_\tau^{\rm ex}, \rho_\tau) &\leq \int_{\mb R^d\times \mb R^d}\|x-y\|^2\,\dd(T_\rho^{\rho_\tau^{\rm ex}}, T_\rho^{\rho_\tau})_\#\rho\\
    &= \int_{\mb R^d}\big|\!\big|T_\rho^{\rho_\tau^{\rm ex}}(x) - T_\rho^{\rho_\tau}(x)\big|\!\big|^2\,\dd\rho(x)= \tau^2\int_{\mb R^d}\Big|\!\Big|\nabla\frac{\delta\m F_{\rm KL}}{\delta\rho}(\rho) - \nabla\frac{\delta\m F_{\rm KL}}{\delta\rho}(\rho_\tau)\circ T_\rho^{\rho_\tau}\Big|\!\Big|^2\,\dd\rho.
\end{align*}
We can further bound the above by
\begin{align*}
    &\quad\,\Big|\!\Big|\nabla\frac{\delta\m F_{\rm KL}}{\delta\rho}(\rho) - \nabla\frac{\delta\m F_{\rm KL}}{\delta\rho}(\rho_\tau)\circ T_\rho^{\rho_\tau}\Big|\!\Big|^2\\
    &=\Big|\!\Big|\nabla\frac{\delta\m F_{\rm KL}}{\delta\rho}(\rho) - \nabla\frac{\delta\m F_{\rm KL}}{\delta\rho}(\rho_\tau) + \nabla\frac{\delta\m F_{\rm KL}}{\delta\rho}(\rho_\tau) - \nabla\frac{\delta\m F_{\rm KL}}{\delta\rho}(\rho_\tau)\circ T_\rho^{\rho_\tau}\Big|\!\Big|^2\\
    &\leq 2\Big|\!\Big|\nabla\frac{\delta\m F_{\rm KL}}{\delta\rho}(\rho) - \nabla\frac{\delta\m F_{\rm KL}}{\delta\rho}(\rho_\tau)\Big|\!\Big|^2 + 2\Big|\!\Big|\nabla\frac{\delta\m F_{\rm KL}}{\delta\rho}(\rho_\tau) - \nabla\frac{\delta\m F_{\rm KL}}{\delta\rho}(\rho_\tau)\circ T_\rho^{\rho_\tau}\Big|\!\Big|^2.
\end{align*}
The first term above can be reformulated as
\begin{align*}
    \nabla\frac{\delta\m F_{\rm KL}}{\delta\rho}(\rho_\tau) - \nabla\frac{\delta\m F_{\rm KL}}{\delta\rho}(\rho) &= \nabla\big[ V + \log\rho_\tau\big] - \nabla\big[ V + \log\rho\big] = \nabla\log\frac{\rho_\tau}{\rho}.
\end{align*}
Since $\rho_\tau = (T_\rho^{\rho_\tau})_\#\rho$, by applying a change of measure, we have
\begin{align*}
    \rho_\tau(y) = \rho\big(T_{\rho_\tau}^\rho(y)\big)\Big|\det\nabla T_{\rho_\tau}^\rho(y)\Big| = \rho\big(T_{\rho_\tau}^\rho(y)\big)\bigg|\det\Big(I_d + \tau\nabla^2\frac{\delta\m F_{\rm KL}}{\delta\rho}(\rho_\tau)\Big)\bigg|(y).
\end{align*}
Let $\lambda_1(\rho_\tau)$, \ldots, $\lambda_d(\rho_\tau)$ be the $d$ eigenvalues of the Hessian matrix $\nabla^2\frac{\delta\m F_{\rm KL}}{\delta\rho}(\rho_\tau)$. By the mean value theorem, there is some $\xi_y$ depending on $y$ such that
\begin{align*}
    \nabla\log\frac{\rho_\tau}{\rho} &= \nabla\log\rho\big(T_{\rho_\tau}^\rho(y)\big) - \nabla\log\rho(y) + \nabla\log\bigg|\det\Big(I_d + \tau\nabla^2\frac{\delta\m F_{\rm KL}}{\delta\rho}(\rho_\tau)\Big)\bigg|(y)\\
    &= \nabla^2\log\rho(\xi_y)\big(T_{\rho_\tau}^\rho(y) - y\big) + \nabla\log\prod_{i=1}^d\big(1+\tau\lambda_i(\rho_\tau)(y)\big)\\
    &= \nabla^2\log\rho(\xi_y)\big(T_{\rho_\tau}^\rho(y) - y\big) + \sum_{i=1}^d\nabla\log\big(1+\tau\lambda_i(\rho_\tau)(y)\big)\\
    &= \tau\nabla^2\log\rho(\xi_y)\nabla\frac{\delta\m F_{\rm KL}}{\delta\rho}(\rho_\tau)(y) + \tau\sum_{i=1}^d\frac{\nabla\lambda_i(\rho_\tau)(y)}{1+\tau\lambda_i(\rho_\tau)(y)}.
\end{align*}
For the second term, by applying the mean value theorem again, we can find some $\eta_y$ such that
\begin{align*}
    &\quad\,\Big|\!\Big|\nabla\frac{\delta\m F_{\rm KL}}{\delta\rho}(\rho_\tau)(y) - \nabla\frac{\delta\m F_{\rm KL}}{\delta\rho}(\rho_\tau)\circ T_\rho^{\rho_\tau}(y)\Big|\!\Big|^2\\
    &= \Big|\!\Big|\nabla\frac{\delta\m F_{\rm KL}}{\delta\rho}(\rho_\tau)(y) - \nabla\frac{\delta\m F_{\rm KL}}{\delta\rho}(\rho_\tau)\big(y - \tau\nabla\frac{\delta\m F_{\rm KL}}{\delta\rho}(\rho_\tau)\circ T_\rho^{\rho_\tau}(y)\big)\Big|\!\Big|^2\\
    &= \tau^2\Big|\!\Big| \nabla^2\frac{\delta_{\rm KL}}{\delta\rho}(\rho_\tau)(\eta_y)\nabla\frac{\delta\m F_{\rm KL}}{\delta\rho}(\rho_\tau)\big( T_\rho^{\rho_\tau}(y)\big)\Big|\!\Big|^2
\end{align*}
Combining all pieces above together yields
\begin{align*}
    W_2^2(\rho_\tau^{\rm ex}, \rho_\tau) &\leq 2\tau^4\int_{\mb R^d}\bigg|\!\bigg|\nabla^2\log\rho(\xi_y)\nabla\frac{\delta\m F_{\rm KL}}{\delta\rho}(\rho_\tau)(y) + \sum_{i=1}^d\frac{\nabla\lambda_i(\rho_\tau)(y)}{1+\tau\lambda_i(\rho_\tau)(y)}\bigg|\!\bigg|^2\,\dd\rho\\
    &\qquad + 2\tau^4\int_{\mb R^d}\Big|\!\Big| \nabla^2\frac{\delta\m F_{\rm KL}}{\delta\rho}(\rho_\tau)(\eta_y)\nabla\frac{\delta\m F_{\rm KL}}{\delta\rho}(\rho_\tau)\big( T_\rho^{\rho_\tau}(y)\big)\Big|\!\Big|^2\,\dd\rho\\
    &\leq C\tau^4\int_{\mb R^d}\bigg|\!\bigg|\nabla^2\log\rho(y)\nabla\frac{\delta\m F_{\rm KL}}{\delta\rho}(\rho)(y) + \nabla\textrm{tr}\Big(\nabla^2\frac{\delta\m F_{\rm KL}}{\delta\rho}(\rho)\Big)(y)\bigg|\!\bigg|^2\,\dd\rho\\
    &\qquad + C\tau^4\int_{\mb R^d}\Big|\!\Big| \nabla^2\frac{\delta\m F_{\rm KL}}{\delta\rho}(\rho)(y)\nabla\frac{\delta\m F_{\rm KL}}{\delta\rho}(\rho)(y)\big)\Big|\!\Big|^2\,\dd\rho,
\end{align*}
for all $\tau$ sufficiently small, which proved the desired bound. In the last inequality, we pass the limit $\tau\to0$ in the integration.

\subsection{Long term cumulative numerical error in particle approximation}\label{app:long_term}
Theorem~\ref{thm: discrete_langevin_vs_JKO} below shows that discretized Langevin dynamics well approximates the JKO-scheme under when the time horizon $[0,T]$ is finite. It is well known that under a fixed $T$ the numeric error of Euler--Maruyama method for discretizing SDE is $\m O(\tau^{1/2})$, where recall that $\tau$ is the step size and the constant depends on $T$. This implies that the Wasserstein distance between discretized Langevin dynamics and the solution of Fokker--Planck equation will be of order $\tau^{1/2}$. By Theorem 11.2.1 in \cite{ambrosio2008gradient}, we know that the Wasserstein distance between the JKO scheme and the solution from the accompanied Fokker--Planck equation is of order $\tau$. Therefore, the discretized Langevin dynamics approximates the JKO-scheme with numeric error of order $\tau^{1/2}$. However, the leading constant from such an analysis is usually exponentially large in $T$ due to the application of Gronwall's inequality. Our result below shows that if $V$ is strongly convex, then the dependence of the constant on $T$ can be polynomial, or more precisely, of order $\max\{T,\,T^2\}$. 
\begin{theorem}\label{thm: discrete_langevin_vs_JKO}
Let $V\in\m C^2(\mb R^d)$ be $L$-smooth and $\lambda$-strongly convex, i.e. $\lambda I_d\preceq \nabla^2V\preceq LI_d$. Define
$\m F(\rho) = \int V(x)\,\dd\rho(x) + \int\rho(x)\log\rho(x)\,\dd x$. For any time horizon $T>0$ and number of iteration $N$, we take step size $\tau = T/N$. For any $\rho_{\tau, L}^{(0)} = \rho_{\tau, J}^{(0)} = \rho_0\in\ms P_2^r(\mb R^d)$, we recursively define the JKO-scheme as
\begin{align*}
    \rho_{\tau, J}^{(k)} = \argmin_{\rho\in\ms P_2^r(\mb R^d)}\m F(\rho) + \frac{1}{2\tau}W_2^2(\rho_{\tau, J}^{(k-1)}),
\end{align*}
and a discretized Langevin dynamics scheme as
\begin{align*}
    X_k &= X_{k-1} - \tau\nabla V(X_{k-1}) + \sqrt{2\tau\eta_k}, \quad \rho_{\tau, L}^{(k)} = \m L(X_k),
\end{align*}
for $k=1, 2, \cdots, N$, where $\eta_k$ are i.i.d. standard normal random variables. If step size $\tau < \min\big\{\frac{\lambda}{2L^2}, \frac{\lambda^2}{16eL^4}, \frac{1}{2L+1}\big\}$, then we have $W_2^2(\rho_{\tau, J}^{(N)}, \rho_{\tau, L}^{(N)}) \leq C_1 \max\{T^2, T\}\,\tau^{1/2}$.
As an intermediate step in the proof, we also have the following cumulative error bound
\begin{align*}
    W_2\big(\rho_{\tau, L}^{(N)}, \rho_{N\tau}\big) \leq C\max\{T^2, T\}\,\tau^{1/2},
\end{align*}
where $\rho_t$ denotes the solution of the Fokker--Planck equation initialized at $\rho_0$. Here, both constants $C_1$ and $C$ only depend on $(\rho_0, V, d)$ and are independent of $(T,\tau)$. 
\end{theorem}
\begin{proof}[Proof of Theorem \ref{thm: discrete_langevin_vs_JKO}]
For any $\rho\in\ms P_2^r(\mb R^d)$, let $S_t(\rho)$ be the unique gradient flow of $\m F$ starting from $\rho$, i.e. $S_t(\rho)$ is the solution to Fokker--Planck equation (\ref{eqn: Fokker_Planck_equation}) satisfying $S_0(\rho) = \rho$. By Theorem 11.2.1 in \cite{ambrosio2008gradient}, we know $\{S_t: t\geq 0\}$ is a $\lambda$-contractive semigroup, i.e. 
\begin{align}\label{eqn: contraction_semigroup}
    W_2(S_t(\rho), S_t(\rho'))\leq e^{-\lambda t}W_2(\rho, \rho'), \quad\forall\,\rho,\rho'\in\ms P_2^r(\mb R^d).
\end{align}
Therefore, we have 
\begin{align}\label{eqn: discrete_langevin_Nstep_err}
\begin{aligned}
    &\quad\,W_2\big(\rho_{\tau, L}^{(N)}, S_{N\tau}(\rho_0)\big)\\ 
    &= W_2\big(S_0(\rho_{\tau, L}^{(N)}), S_{N\tau}(\rho_0)\big)\\
    &\stackrel{\ri}{\leq} \sum_{k=1}^N W_2\big(S_{(N-k)\tau}(\rho_{\tau, L}^{(k)}), S_{(N-k+1)\tau}(\rho_{\tau, L}^{(k-1)})\big)\\
    &\stackrel{\rii}{\leq} \sum_{k=1}^N e^{-\lambda(N-k)\tau}W_2\big(\rho_{\tau, L}^{(k)}, S_\tau(\rho_{\tau, L}^{(k-1)})\big)\\
    &\stackrel{\riii}{\leq} \sum_{k=1}^N e^{-\lambda(N-k)\tau}\cdot L\tau^\frac{3}{2}e^\frac{1}{2}\sqrt{d + \frac{1}{2}\mb E_{\rho_{\tau, L}^{(k-1)}}\|\nabla V(X)\|^2}\\
    &\leq \sum_{k=1}^N e^{-\lambda(N-k)\tau}\cdot L\tau^\frac{3}{2}e^\frac{1}{2}\Big(\sqrt{d} + \|\nabla V(0)\| + L\sqrt{\mb E_{\rho_{\tau, L}^{(k-1)}}\|X\|^2}\Big).
\end{aligned}
\end{align}
Here, step (i) is by the triangular inequality, step (ii) is by contraction property (\ref{eqn: contraction_semigroup}), step (iii) is by Lemma \ref{lem: onestep_err}, and the last inequality is due to the following inequality,
\begin{align}\label{eqn: nablaVx_to_X}
\begin{aligned}
    \mb E_{\rho_{\tau, L}^{(k-1)}}\|\nabla V(X)\|^2 &\leq 2\mb E_{\rho_{\tau, L}^{(k-1)}}\|\nabla V(X)-\nabla V(0)\|^2 + 2\|\nabla V(0)\|^2\\
    &\leq 2L^2\mb E_{\rho_{\tau, L}^{(k-1)}}\|X\|^2 + 2\|\nabla V(0)\|^2.
\end{aligned}
\end{align}
Now, let us estimate $\mb E_{\rho_{\tau, L}^{(k-1)}}\|X\|^2$. Let $\dd\rho_\infty \propto e^{-V(x)}\,\dd x$ be the stationary measure of Langevin dynamics, i.e. if we let $\dd Y_t = -\nabla V(Y_t)\,\dd t + \sqrt{2}\,\dd W_t$ and $Y_0\sim\rho_\infty$, then $Y_t\sim\rho_\infty$ for all $t\geq 0$. Let $\dd X_t' = -\nabla V(X_{k-1})\,\dd t + \sqrt{2}\,\dd W_t$ for $0\leq t\leq \tau$ with $X_0' = X_{k-1}$, then we can get $X_{\tau}' = X_k \sim\rho_{\tau, L}^{(k)}$. Notice that $\dd(Y_t - X_t') = -\big(\nabla V(Y_t) - \nabla V(X_{k-1})\big)\,\dd t$, which implies
\begin{align*}
    &\quad\,\|Y_{\tau}-X_\tau'\|^2\\ 
    &= \bigg|\!\bigg| (Y_0 - X_{k-1}) - \int_0^\tau \nabla V(Y_t) - \nabla V(X_{k-1})\,\dd t\bigg|\!\bigg|_2^2\\
    &=\|Y_0 - X_{k-1}\|^2 - 2\int_0^\tau\big\langle \nabla V(Y_0) - \nabla V(X_{k-1}), Y_0 - X_{k-1}\big\rangle\,\dd t\\
    &\qquad -2\int_0^\tau\big\langle\nabla V(Y_t) - \nabla V(Y_0), Y_0 - X_{k-1}\big\rangle\,\dd t + \bigg|\!\bigg|\int_0^\tau \nabla V(Y_t)-\nabla V(X_{k-1})\,\dd t\bigg|\!\bigg|_2^2\\
    &\leq \|Y_0 - X_{k-1}\|^2 - 2\lambda\int_0^\tau\|Y_0 - X_{k-1}\|^2\,\dd t\\
    &\qquad - 2\int_0^\tau\big\langle\nabla V(Y_t) - \nabla V(Y_0), Y_0 - X_{k-1}\big\rangle\,\dd t + L^2\tau\int_0^\tau\|Y_t - X_{k-1}\|^2\,\dd t
\end{align*}
By choosing $Y_0\sim\rho_\infty$ such that $\mb E\|Y_0 - X_{k-1}\|^2 = W_2^2\big(\rho_\infty, \rho_{\tau, L}^{(k-1)}\big)$, we further obtain
\begin{align*}
    W_2^2(\rho_\infty, \rho_{\tau, L}^{(k)}) &\leq \mb E\|Y_\tau - X_\tau'\|^2\\
    &\leq (1-2\tau\lambda)W_2^2(\rho_\infty, \rho_{\tau, L}^{(k-1)}) - 2\int_0^\tau\mb E\big\langle\nabla V(Y_t) - \nabla V(Y_0), Y_0 - X_{k-1}\big\rangle\,\dd t\\
    &\qquad + L^2\tau\int_0^\tau\mb E\|Y_t - X_{k-1}\|^2\,\dd t\\
    &\leq (1-2\tau\lambda)W_2^2(\rho_\infty, \rho_{\tau, L}^{(k-1)}) + L^2\tau\int_0^\tau\mb E\|Y_t - X_{k-1}\|^2\,\dd t\\
    &\qquad + \int_0^\tau \lambda\mb E\|Y_0 - X_{k-1}\|^2 + \frac{1}{\lambda}\mb E\|\nabla V(Y_t) - \nabla V(Y_0)\|^2\,\dd t\\
    &\leq (1-\tau\lambda)W_2^2(\rho_\infty, \rho_{\tau, L}^{(k-1)}) + L^2\tau\int_0^\tau\mb E\|Y_t - X_{k-1}\|^2\,\dd t + \frac{L^2}{\lambda}\int_0^\tau\mb E\|Y_t - Y_0\|^2\,\dd t.
\end{align*}
To analyze the last two terms on the right hand side in the above,  we notice that similar to the proof of Lemma \ref{lem: onestep_err}, we have
\begin{align*}
    \mb E\|Y_t - X_{k-1}\|^2 \leq \mb E\|Y_0 - X_{k-1}\|^2 + \big(d + \mb E_{\rho_{\tau, L}^{(k-1)}}\|\nabla V(X)\|^2\big) t + (2L + 1)\int_0^t\mb E\|Y_s - X_{k-1}\|^2\,\dd s.
\end{align*}
Then, applying Gronwall's inequality yields that for every $t\geq 0$,
\begin{align*}
    \mb E\|Y_t - X_{k-1}\|^2 \leq \Big(\mb E\|Y_0 - X_{k-1}\|^2 + \big(d + \mb E_{\rho_{\tau, L}^{(k-1)}}\|\nabla V(X)\|^2\big) t\Big)e^{(2L+1)t}.
\end{align*}
Taking $(2L+1)\tau < 1$, we get
\begin{align*}
    \int_0^\tau\mb E\|Y_t - X_{k-1}\|^2\,\dd t &\leq e\tau\mb E\|Y_0 - X_{k-1}\|^2 + \frac{e\tau^2}{2}\big(d + \mb E_{\rho_{\tau, L}^{(k-1)}}\|\nabla V(X)\|^2\big).
\end{align*}
A very same argument can be used to show 
\begin{align*}
    \mb E\|Y_t - Y_0\|^2 \leq \big(d + \mb E\|\nabla V(Y_0)\|^2\big)te^{(2L+1)t}.
\end{align*}

\smallskip
\noindent By putting pieces together, we obtain
\begin{align*}
    W_2^2(\rho_\infty, \rho_{\tau, L}^{(k)}) &\leq (1-\tau\lambda)W_2^2(\rho_\infty, \rho_{\tau, L}^{(k-1)})\\
    &\qquad + L^2\tau\Big(e\tau\mb E\|Y_0 - X_{k-1}\|^2 + \frac{e\tau^2}{2}\big(d + \mb E_{\rho_{\tau, L}^{(k-1)}}\|\nabla V(X)\|^2\big)\Big)\\
    &\qquad + \frac{L^2}{\lambda}\cdot \big(d + \mb E\|\nabla V(Y_0)\|^2\big)\frac{\tau^2}{2}\\
    &\leq (1 - \tau\lambda + L^2e\tau^2)W_2^2(\rho_\infty, \rho_{\tau, L}^{(k-1)}) + \frac{L^2\tau^2}{2\lambda}\big(d + \mb E\|\nabla V(Y_0)\|^2\big)\\
    &\qquad + \frac{eL^2\tau^3}{2}\big(d + 2L^2\mb E_{\rho_{\tau, L}^{(k-1)}}\|X\|^2 + 2\|\nabla V(0)\|^2\big)\Big),
\end{align*}
where the last inequality is due to \eqref{eqn: nablaVx_to_X}. Therefore, we further obtain
\begin{align}\label{eqn: W2dist_iter}
\begin{aligned}
    W_2(\rho_\infty, \rho_{\tau, L}^{(k)}) &\leq \sqrt{1-\tau\lambda + L^2e\tau^2}W_2(\rho_\infty, \rho_{\tau, L}^{(k-1)}) + \sqrt{eL^4\tau^3}\cdot\sqrt{\mb E_{\rho_{\tau, L}^{(k-1)}}\|X\|^2}\\
    &\qquad + \sqrt{\frac{eL^2\tau^3}{2}\big(d + 2\|\nabla V(0)\|^2\big) + \frac{L^2\tau^2}{2\lambda}\big(d+\mb E\|\nabla V(Y_0)\|^2\big)}
\end{aligned}
\end{align}
Notice that
\begin{align*}
    \mb E_{\rho_{\tau, L}^{(k-1)}}\|X\|^2 - \mb E_{\rho_\infty}\|X\|^2
    &= \int_{\mb R^d}\|x\|^2\dd\big(t_{\rho_\infty}^{\rho_{\tau, L}^{(k-1)}}\big)_\#\rho_\infty - \int_{\mb R^d}\|x\|^2\,\dd\rho_\infty\\
    &= \int_{\mb R^d}\big|\!\big|t_{\rho_\infty}^{\rho_{\tau, L}^{(k-1)}}(x)\big|\!\big|_2^2 - \|x\|^2\,\dd\rho_\infty\\
    &= \int_{\mb R^d}\big|\!\big|t_{\rho_\infty}^{\rho_{\tau, L}^{(k-1)}}(x) - x\big|\!\big|_2^2 + 2\big\langle x, t_{\rho_\infty}^{\rho_{\tau, L}^{(k-1)}}(x) - x\big\rangle\,\dd\rho_\infty\\
    &\leq W_2^2\big(\rho_\infty, \rho_{\tau, L}^{(k-1)}\big) + 2W_2\big(\rho_\infty, \rho_{\tau, L}^{(k-1)}\big)\sqrt{\mb E_{\rho_\infty}\|X\|^2}.
\end{align*}
This implies
\begin{align}\label{eqn: Ek2moment_iter}
    \sqrt{\mb E_{\rho_{\tau, L}^{(k)}}\|X\|^2} \leq W_2\big(\rho_\infty, \rho_{\tau, L}^{(k)}\big) + \sqrt{\mb E_{\rho_\infty}\|X\|^2}.
\end{align}
Choosing $\tau < \min\big\{\frac{\lambda}{2L^2}, \frac{\lambda^2}{16eL^4}\big\}$ so that $\alpha := \sqrt{1 - \tau\lambda + L^2e\tau^2} + \sqrt{eL^4\tau^3} < 1$, then by (\ref{eqn: W2dist_iter}) and (\ref{eqn: Ek2moment_iter}), we can get
\begin{align*}
    \sqrt{\mb E_{\rho_{\tau, L}^{(k)}}\|X\|^2} &\leq \alpha^kW_2(\rho_\infty,\rho_0) + \sqrt{\mb E_{\rho_\infty}\|X\|^2} + \frac{1-\alpha^k}{1-\alpha}\bigg(e^\frac{1}{2}L^2\tau^\frac{3}{2}\sqrt{\mb E_{\rho_\infty}\|X\|^2}\\
    &\qquad\qquad\quad + \sqrt{\frac{eL^2\tau^3}{2}\big(d + 2\|\nabla V(0)\|^2\big) + \frac{L^2\tau^2}{2\lambda}\big(d+\mb E\|\nabla V(Y_0)\|^2\big)}\bigg)\\
    &\leq C_1 + kC_2\tau.
\end{align*}
Here, for simplicity we may choose
\begin{align*}
    C_1 &= W_2(\rho_\infty,\rho_0) + \sqrt{\mb E_{\rho_\infty}\|X\|^2}\\
    C_2 &= e^\frac{1}{2}L^2\sqrt{\mb E_{\rho_\infty}\|X\|^2} + \sqrt{\frac{eL^2}{2}\big(d + 2\|\nabla V(0)\|^2\big) + \frac{L^2}{2\lambda}\big(d+\mb E\|\nabla V(Y_0)\|^2\big)}.
\end{align*}
Let $C_3 = \sqrt{d} + \|\nabla V(0)\| + LC_1$. Then by~\eqref{eqn: discrete_langevin_Nstep_err}, we finally reach
\begin{align*}
    W_2\big(\rho_{\tau, L}^{(N)}, S_{N\tau}(\rho_0)\big)
    &\leq \sum_{k=1}^N e^{-\lambda(N-k)\tau}L\tau^\frac{3}{2}e^\frac{1}{2}\Big(C_3 + (k-1)LC_2\tau\Big)\\
    &= C_3e^\frac{1}{2}L\tau^\frac{3}{2}\frac{1-e^{-N\lambda\tau}}{1-e^{-\lambda\tau}} + C_2e^\frac{1}{2}L^2\tau^\frac{5}{2}\frac{e^{-N\lambda\tau} - Ne^{-\lambda\tau} + (N-1)}{(1-e^{-\lambda\tau})^2}\\
    &\leq C_3e^\frac{1}{2}LN\tau^\frac{3}{2} + C_2e^\frac{1}{2}L^2N^2\tau^\frac{5}{2}= (C_3LT + C_2L^2T^2)e^\frac{1}{2}\tau^{1/2}.
\end{align*}
By Theorem 11.2.1 in \cite{ambrosio2008gradient}, we have $W_2(\rho_{\tau, J}^{(N)}, S_{N\tau}(\rho_0)) \leq |\partial \m F|(\rho_0)\tau$. Therefore, by combining these two bounds, we obtain $W_2\big(\rho_{\tau, L}^{(N)}, \rho_{\tau, J}^{(N)}\big) \leq C_1  \max\{T,\,T^2\}\,\tau^{1/2}$.
\end{proof}

\section{More details and proofs about examples}\label{app:pf_example}
In this appendix, we provide more details and proofs to the two examples, namely the repulsive Gaussian mixture model and the mixture of regression model, considered in Section~\ref{sec:thm_app} of the main paper.
For two density functions $f$ and $g$, we let
\begin{align*}
    D_H^2(f, g) = \frac{1}{2}\int_{\mb R^d}\big(\sqrt{f(x)} - \sqrt{g(x)}\big)^2\,\dd x
\end{align*}
denote the square of the Hellinger distance between the corresponding measures. It is well known that
\begin{equation}\label{eqn: Hellinger leq KL}
    D_H(f, g) \leq \sqrt{\frac{1}{2}D_{KL}(f\,\|\,g)}.
\end{equation}
\subsection{Bayesian Linear Regression and proof of Corollary \ref{coro: BLR}}
\begin{proof}[Proof of Corollary \ref{coro: BLR}]
\noindent\underline{Convexity and smoothness of $-\mb E_{\theta^\ast, \alpha^\ast}\nabla^2 p(X_i, y_i\,|\,\theta, \alpha)$.}
Note that
\begin{align*}
    -\log p(X_i, y_i\,|\,\theta,\alpha) = \frac{\alpha}{2}(y_i-X_i^T\theta)^2 - \frac{1}{2}\log\alpha + C(X_i, y_i)
\end{align*}
where $C(x_i, y_i)$ is a constant that only depends on $x_i$ and $y_i$. This implies
\begin{align*}
    -\nabla^2\log p(X_i, y_i\,|\,\theta,\alpha) =
    \begin{pmatrix}
        {\alpha} X_iX_i^T & X_i(X_i^T\theta-y_i)\\
        (X_i^T\theta-y_i)X_i^T & \frac{1}{2\alpha^2}
    \end{pmatrix},
\end{align*}
and its expectation is
\begin{align*}
    -\mb E_{\theta^\ast, \alpha^\ast}\nabla^2 p(X_i, y_i\,|\,\theta, \alpha) = 
    \begin{pmatrix}
        \alpha \Sigma & \Sigma(\theta - \theta^\ast)\\
        (\theta-\theta^\ast)^T\Sigma & \frac{1}{2\alpha^2}
    \end{pmatrix}.
\end{align*}
Let $(v^{\rm T}, u)^{\rm T}\in B^{d+1}(0, 1)$. Then, we have
\begin{align*}
    (v^{\rm T}, u)
    \begin{pmatrix}
        \alpha \Sigma & \Sigma(\theta - \theta^\ast)\\
        (\theta-\theta^\ast)^T\Sigma & \frac{1}{2\alpha^2}
    \end{pmatrix}
    \begin{pmatrix}
        v\\ u
    \end{pmatrix}
    &= \alpha v^{\rm T}\Sigma v + 2uv^{\rm T}\Sigma(\theta-\theta^\ast) + \frac{u^2}{2\alpha^2} =: f(u, v;\alpha).
\end{align*}
Note that
\begin{align*}
    f(u, v;\alpha) &\geq \alpha\lambda_1\|v\|^2 - 2\lambda_d\|\theta-\theta^\ast\||u|\|v\| + \frac{u^2}{2\alpha^2}\\
    &\geq \frac{(\alpha\lambda_1 + \frac{1}{2\alpha^2}) - \sqrt{(\alpha\lambda_1 - \frac{1}{2\alpha^2})^2 + 4\lambda_d^2R_\theta^2}}{2}\\
    &= \frac{\frac{\lambda_1}{\alpha} - 2\lambda_d^2R_\theta^2}{(\alpha\lambda_1 + \frac{1}{2\alpha^2}) + \sqrt{(\alpha\lambda_1 - \frac{1}{2\alpha^2})^2 + 4\lambda_d^2R_\theta^2}}\\
    &> \frac{\frac{\lambda_1}{2\alpha} - \lambda_d^2R_\theta^2}{\max\{\alpha\lambda_1, \frac{1}{2\alpha^2}\} + \lambda_dR_\theta} \geq \frac{\frac{\lambda_1}{2\alpha_{ub}} - \lambda_d^2R_\theta^2}{\max\{\alpha_{ub}\lambda_1, \frac{1}{2\alpha_{lb}^2}\} + \lambda_dR_\theta}\\
    &=: \tilde\lambda > 0.
\end{align*}
Similarly, we have
\begin{align*}
    f(u, v;\alpha) &\leq \alpha\lambda_d\|v\|^2 + 2\lambda_d\|\theta-\theta^\ast\||u|\|v\| + \frac{u^2}{2\alpha^2}\\
    &\leq \frac{(\alpha\lambda_d + \frac{1}{2\alpha^2}) + \sqrt{(\alpha\lambda_d - \frac{1}{2\alpha^2})^2 + 4\lambda_d^2R_\theta^2}}{2}\\
    &\leq \max\Big\{\alpha\lambda_d, \frac{1}{2\alpha^2}\Big\} + \lambda_dR_\theta\\
    &\leq \max\Big\{\alpha_{ub}\lambda_d, \frac{1}{2\alpha_{lb}^2}\Big\} + \lambda_dR_\theta\\
    &=: \tilde L.
\end{align*}
Therefore, $-\mb E_{\theta^\ast, \alpha^\ast}\nabla^2 p(X_i, y_i\,|\,\theta, \alpha)$ is strongly convex on $\Theta_\theta\times\Theta_\alpha$ when $0 < \alpha < \lambda_1/(2\lambda_d^2R_\theta^2)$ holds. This is true due to~\eqref{eqn: para_space, BLR}. 

\noindent\underline{Hessian statistical noise.} Let $v = (v_\theta^T, v_\alpha)^T$ is a $(d+1)$-dimensional unit vector. We have
\begin{align*}
\big\lvert v^T \nabla^2\log p(X_i, y_i\,|\,\theta,\alpha)v\big\rvert
&= \bigg\lvert\alpha(X_i^Tv_\theta)^2 + 2v_\theta^TX_i(X_i^T\theta - y_i)v_\alpha + \frac{v_\alpha^2}{2\alpha^2}\bigg\rvert\\
&= \bigg\lvert\alpha(X_i^Tv_\theta)^2 + 2v_\theta^TX_iX_i^T(\theta - \theta^\ast) v_\alpha + 2v_\theta^TX_i\varepsilon_iv_\alpha + \frac{v_\alpha^2}{2\alpha^2}\bigg\rvert\\
&\stackrel{\ri}{\leq} \alpha\|X_i\|^2\|v_\theta\|^2 + 2\lvert v_\alpha\rvert\|v_\theta\|\|X_i\|^2\|\theta-\theta^\ast\| + \|v_\theta\|^2\|X_i\|^2 + \varepsilon_i^2v_\alpha^2 + \frac{v_\alpha^2}{2\alpha^2}\\
&\stackrel{\rii}{\leq} \Big(1 + \alpha + \|\theta-\theta^\ast\|\Big)\|X_i\|^2 + \varepsilon_i^2 + \frac{1}{2\alpha^2}\\
&\stackrel{\riii}{\leq} \Big(1 + \alpha + (2\alpha_{ub})^{-\frac{1}{2}}\Big)\|X_i\|^2 + \varepsilon_i^2 + \frac{1}{2\alpha^2}.
\end{align*}
Here, (i) is by AM-GM inequality; (ii) is by $v_\alpha^2 + \|v_\theta\|^2 = 1$; (iii) is by~\eqref{eqn: para_space, BLR}. Since $X_i$ and $\varepsilon_i$ are sub-Gaussian, we know both $\|X_i\|^2$ and $\varepsilon_i^2$ are sub-exponential. Therefore, $|v^T\nabla^2\log p(X_i, y_i\,|\,\theta,\alpha)v|$ is sub-exponential with some parameter $\sigma_5$ due to the independence between $X_i$ and $\varepsilon_i$.

\noindent\underline{Hessian regularity.} Note that
\begin{align*}
&\quad\,\,\matnorm{\nabla^2\log p(X_i, y_i\,|\,\theta, \alpha) - \nabla^2\log p(X_i, y_i\,|\,\theta', \alpha')}\\
&= \matnorm{\begin{pmatrix}
    (\alpha - \alpha')X_iX_i^T & X_iX_i^T(\theta-\theta')\\
    (\theta-\theta')^TX_iX_i^T & \frac{1}{2\alpha^2} - \frac{1}{2(\alpha')^2}
\end{pmatrix}}\\
&= \sup_{v_\alpha^2 + \| v_\theta\|^2 = 1} v_\alpha^2\bigg\lvert\frac{1}{2\alpha^2} - \frac{1}{2(\alpha')^2}\bigg\rvert + \Big|2v_\theta^TX_iX_i^T(\theta-\theta')v_\alpha\Big| + \lvert \alpha-\alpha'\rvert (X_i^Tv_\theta)^2\\
&\leq \sup_{v_\alpha^2 + \| v_\theta\|^2 = 1}\bigg\lvert\frac{v_\alpha^2(\alpha'-\alpha)(\alpha'+\alpha)}{2(\alpha\alpha')^2}\bigg\rvert + 2\|v_\theta\|\lvert v_\alpha\rvert \|X_i\|^2\|\theta-\theta'\| + \lvert\alpha-\alpha'\rvert(X_i^Tv_\theta)^2\\
&\leq \sup_{v_\alpha^2 + \| v_\theta\|^2 = 1}\Big(\frac{v_\alpha^2\alpha_{ub}}{\alpha_{lb}^4} + \|X_i\|^2\|v_\theta\|^2\Big)\lvert\alpha-\alpha'\rvert + \|X_i\|^2\|\theta-\theta'\|\\
&\leq \Big(\frac{\alpha_{ub}}{\alpha^4_{lb}}+\|X_i\|^2\Big)|\alpha-\alpha'| + \|X_i\|^2\|\theta-\theta'\|\\
&\leq \sqrt{\Big(\frac{\alpha_{ub}}{\alpha^4_{lb}} + \|X_i\|^2\Big)^2 + \|X_i\|^4}\cdot\sqrt{\lvert\alpha-\alpha'\rvert + \|\theta-\theta'\|^2}.
\end{align*}
Therefore, we can take
\begin{align*}
J(X_i, y_i) = \sqrt{\Big(\frac{\alpha_{ub}}{\alpha^4_{lb}} + \|X_i\|^2\Big)^2 + \|X_i\|^4} \leq \frac{\alpha_{ub}}{\alpha_{lb}^4} + 2\|X_i\|^2,
\end{align*}
and thus
\begin{align*}
    J^\ast = \mb E_{\theta^\ast, \alpha^\ast} J(X_i, y_i) \leq 2d + \frac{\alpha_{ub}}{\alpha_{lb}^4} < \infty.
\end{align*}
Therefore, by Theorem \ref{thm: para_update_whp}, we have
\begin{align*}
W_2^2(q_{\theta}^{(k)}\otimes q_\alpha^{(k)}, \wht q_{\theta}\otimes\wht q_\alpha) \leq \Big(1 + \frac{\lambda_{lb}^2}{L_{ub}^2m}\Big)^{-k} W_2^2(q_{\theta}^{(0)}\otimes q_\alpha^{(0)}, \wht q_{\theta}\otimes\wht q_\alpha)
\end{align*}
with high probability, where 
\begin{align*}
    \lambda_{lb} &=  n\tilde{\lambda} - \lambda_M(\nabla^2\log\pi_\theta) - \sigma_5^2\sqrt{\frac{Cd\log n}{n}\cdot\max\Big\{\frac{\log\big(2d + \frac{\alpha_{ub}}{\alpha_{lb}^4}\big)}{\log d}, \log\frac{R_\theta\sigma_5}{\eta}, 1\Big\}}\\
    L_{ub} &= n\tilde L - \lambda_m(\nabla^2\log\pi_\theta) + \sigma_5^2\sqrt{\frac{Cd\log n}{n}\cdot\max\Big\{\frac{\log\big(2d + \frac{\alpha_{ub}}{\alpha_{lb}^4}\big)}{\log d}, \log\frac{R_\theta\sigma_5}{\eta}, 1\Big\}},
\end{align*}
with
\begin{align*}
\tilde\lambda &= \frac{\frac{\lambda_1}{2\alpha_{ub}} - \lambda_d^2R_\theta^2}{\max\{\alpha_{ub}\lambda_1, \frac{1}{2\alpha_{lb}^2}\} + \lambda_d R_\theta}
\quad\mx{and}\quad
\tilde L = \max\Big\{\alpha_{ub}\lambda_d, \frac{1}{2\alpha_{lb}^2}\Big\} + \lambda_dR_\theta.
\end{align*}
\end{proof}

\subsection{Repulsive Gaussian mixture model and proof of Corollary~\ref{coro: GMM}}

\begin{proof}[proof of Corollary~\ref{coro: GMM}]
First, let us collect some results useful to check all assumptions mentioned in Section \ref{sec: main results}. For simplicity, let $E_k = e^{-\frac{\|x-m_k\|^2}{2\beta^2}}$ and $p = \sum_{k=1}^Kw_kE_k$. It is easy to show that
\begin{align}\label{eqn: grad&hess_MoG}
\begin{aligned}
    \frac{\partial\log p(z\,|\,x,m)}{\partial m_k} &= \Big(\delta_{kz}-\frac{w_kE_k}{p}\Big)\frac{x-m_k}{\beta^2}\\ 
    \frac{\partial^2\log p(z\,|\,x,m)}{\partial m_k\partial m_{k'}} &= \frac{w_{k'}E_{k'}(w_{k}E_k - \delta_{kk'}p)}{p^2}\frac{x-m_k}{\beta^2}\Big(\frac{x-m_{k'}}{\beta^2}\Big)^T\\
    &\qquad\qquad\qquad\qquad\qquad\qquad\qquad- \delta_{kk'}(\delta_{kz}-\frac{w_kE_k}{p})\frac{I_d}{\beta^2}
\end{aligned}
\end{align}
Therefore, we know
\begin{align}\label{eqn: matnorm_MoG}
    \lambda(X) = \sup_{m\in\m M, k\in[K]}\matnorm{\nabla^2\log p(k\,|\,x,m)} \leq K\Big[\frac{2(\|X\|^2 + R^2)}{\beta^4}+1\Big].
\end{align}

\smallskip
\noindent\underline{\bf Concentration properties.} Let us start with Assumption \ref{assump: qudratic growth of KL}. By mean value theorem,
\begin{align*}
    D_{KL}\big(p(\cdot\,|\,m)\,\|\,p(\cdot\,|\,m')\big) \leq \frac{1}{2}\sup_{s,s'\in \m M}\bigg(\int_{\mb R^d}\matnorm{\nabla^2\log p(x\,|\,s)}p(x\,|\,s')\,\dd x\bigg)\|m'-m\|^2.
\end{align*}
Since $\m M$ is compact, $\log p(x\,|\,m)\in\m C^2(\m M)$, and $p(x\,|\,m')$ is a mixture of Gaussian indicating that the supreme of the above integration is always finite, we know $D_{KL}\big(p(\cdot\,|\,m)\,\|\,p(\cdot\,|\,m')\big) \lesssim \|m-m'\|^2$. Similarly, we have
\begin{align*}
    D_{KL}\big(p(\cdot\,|\,x, m)\,\|\,p(\cdot\,|\,x,m')\big) &\leq \frac{1}{2}\sup_{s,s'\in\m M}\sum_{z=1}^K\matnorm{\nabla^2\log p(z\,|\,x,s)} p(z\,|\,x,s')\cdot\|m-m'\|^2\\
    &\leq \frac{1}{2}\sup_{s\in\Theta, z\in[K]}\matnorm{\nabla^2\log p(z\,|\,x,s)}\cdot\|m-m'\|^2\\
    &\leq K\bigg[\frac{2\big(\|x\|^2+R^2\big)}{\beta^4}+1\bigg]\|m - m'\|^2.
\end{align*}
Since $X$ is sampled from mixture of Gaussian indicating that $X$ is sub-Gaussian, we know $\|X\|^2$ is sub-exponential. For Assumption \ref{assump: test condition}, notice that
\begin{align*}
    D_{H}\big(p(\cdot\,|\,m), p(\cdot\,|\,m')\big) \lesssim \sqrt{D_{KL}\big(p(\cdot\,|\,m)\,\|\,p(\cdot\,|\,m')\big)} \lesssim \|m-m'\|.
\end{align*}
Also, by Taylor expansion
\begin{align*}
    \frac{D_H^2\big(p(\cdot\,|\,m), p(\cdot\,|\,m^\ast)\big)}{\|m-m^\ast\|^2} &= \|m-m^\ast\|^{-2}\int_{\mb R^d}\Big(\sqrt{p(x\,|\,m)} - \sqrt{p(x\,|\,m^\ast)}\Big)^2\,\dd x\\
    &= \int_{\mb R^d}\Big\langle\frac{\nabla p(x\,|\,m^\ast)}{2\sqrt{p(x\,|\,m^\ast)}},\frac{m-m^\ast}{\|m-m^\ast\|}\Big\rangle^2\,\dd x + \zeta(m)\|m-m^\ast\|\\
    &= \int_{\mb R^d}\Big\langle\frac{\nabla p(x\,|\,m^\ast)}{2\sqrt{p(x\,|\,m^\ast)}}, \tilde{m}\Big\rangle^2\,\dd x + \zeta(m)\|m-m^\ast\|
\end{align*}
for some continuous functions $\zeta(m)$ on the compact space $\m M$, and a unit vector $\tilde{m}\in\m M-\{m^\ast\}$. Then, we know
\begin{align*}
    \frac{D_H^2\big(p(\cdot\,|\,m), p(\cdot\,|\,m^\ast)\big)}{\|m-m^\ast\|^2}
\end{align*}
is bound away from zero in the neighborhood of $\theta^\ast$. Also, we know the Hellinger distance between $p(\cdot\,|\,m)$ and $p(\cdot\,|\,m^\ast)$ is zero if and only if $m$ is a rearrangement of $m^\ast$. So, the KL divergence is not zero outside the neighborhood of $\theta^\ast$. Then, Assumption \ref{assump: test condition} can be derived by Example 7.1 in \cite{ghosal2000convergence}. 

\smallskip
\noindent\underline{\bf Regularity of log-likelihood function.} By (\ref{eqn: grad&hess_MoG}), we have
\begin{align*}
    S_2(X) = \sum_{z,k=1}^K\Big(\delta_{kz}-\frac{w_kE_k}{p}\Big)^2\beta^{-4}\|X-m_k^\ast\|^2 \leq K\beta^{-4}\sum_{k=1}^K\|X-m_k^\ast\|^2 < \infty.
\end{align*}
It is clear that both $\mb E_{m^\ast}S_1(X)$ and $\mb E_{m^\ast}S_2(X)$ are finite. Also, by Taylor expansion
\begin{align*}
    \mb E_{m^\ast}e^{\sigma_3^{-1}S_2(X)} &= 1 + \sum_{n=1}^\infty \frac{\sigma_3^{-n}}{n!}\mb E_{m^\ast}[S_2(X)^n]\leq 1+ \sum_{n=1}^\infty \Big(\frac{\mb E_{m^\ast}S_2(X)}{\sigma_3}\Big)^n= \Big(1 - \frac{\mb E_{m^\ast}S_2(X)}{\sigma_3}\Big)^{-1}.
\end{align*}
So, a sufficient condition for $\mb E_{m^\ast}e^{\sigma_3^{-1}|S_2(X)|} \leq 2$ is
\begin{align*}
    \mb E_{m^\ast}S_2(X) &= \beta^{-4}\mb E_{m^\ast}\sum_{z,k=1}^K\Big(\delta_{kz}-\frac{w_kE_k}{p}\Big)^2\|X-m_k^\ast\|^2\\
    &\leq \matnorm{I_S(m^\ast)}\\
    &\leq \frac{3(2+\frac{d_{\min}^2}{\beta^2})Ke^{-\frac{d_{\min}^2}{256\beta^2}}}{\beta^2w_{\min}}.
\end{align*}
See the proof of $\lambda > 2\gamma$ later for this bound. So, we can take
\begin{align*}
    \sigma_3 = \frac{6(2+\frac{d_{\min}^2}{\beta^2})Ke^{-\frac{d_{\min}^2}{256\beta^2}}}{\beta^2w_{\min}}.
\end{align*}
Notice that
\begin{align*}
    \nabla^2\log p(X,k\,|\,m) = -\beta^{-2}I_{Kd}.
\end{align*}
So, we can simply take $J_k(X) = 0$. Also, we have
\begin{align*}
    &\quad\,\mb E_{m^\ast}\exp\bigg\{\sigma_1^{-1}\Big|\sum_{k=1}^K p(k\,|\,X,\theta^\ast)v^T\nabla^2\log p(X,k\,|\,m)v\Big|\bigg\}\\
    &\leq 1 + \sum_{n=1}^\infty \frac{\sigma_1^{-n}}{n!}\bigg|\sum_{k=1}^Kp(k\,|\,X,\theta^\ast)\matnorm{\nabla^2\log p(X,k\,|\,m)}\bigg|^n\\
    &= 1 + \sum_{n=1}^\infty \frac{(\sigma_1\beta^2)^{-n}}{n!}\\
    &= e^{\frac{1}{\sigma_1\beta^2}}.
\end{align*}
Therefore, a sufficient condition for Assumption \ref{assump: continuity_of_Hessian}.3 is to take $\sigma_1 = 2\beta^{-2}$. Similarly, we can take
\begin{align*}
    \sigma_2 = 2K\bigg[\frac{2R^2 + 2d\beta^2 + 2\sum_{k=1}^K\|m_k^\ast\|^2}{\beta^4}+1\bigg],
\end{align*}
so that $\mb E_{\theta^\ast}\exp\{\sigma_2^{-1}|\lambda(X)|\} \leq 2$ due to (\ref{eqn: matnorm_MoG}).

\smallskip
\noindent\underline{\bf Convexity of $U(\cdot\,;\mu)$.} By definition, 
\begin{displaymath}
U(m;\mu) = \int_{\mb R^d}\sum_{z=1}^K\Phi(\mu, x)(z)\bigg[\frac{\|x-m_z\|^2}{2\beta^2} - \log w_z\bigg]\,p(\dd x\,|\,\theta^\ast) + \frac{d}{2}\log(2\pi\beta^2).
\end{displaymath}
Then, for all $\mu\in B_{\mb W_2}(\delta_{m^\ast}, r)$
\begin{align*}
    \frac{\partial^2 U(m;\mu)}{\partial m_k\partial m_{k'}} &= \frac{\delta_{kk'}I_d}{\beta^2}\int_{\mb R^d}\Phi(\mu, x)(k)\,p(\dd x\,|\,\theta^\ast)\\
    &\stackrel{\ri}{\succeq} \frac{\delta_{kk'}I_d}{\beta^2}\int_{\mb R^d}p(k\,|\,x, \theta^\ast) - S_1(x)W_2(\mu, \delta_{\theta^\ast}) - \frac{K\lambda(x)}{2}W_2^2(\mu, \delta_{\theta^\ast})\,p(\dd x\,|\,\theta^\ast)\\
    &= \frac{\delta_{kk'}I_d}{\beta^2}\Big(w_k - \mb E_{\theta^\ast}S_1(X) W_2(\mu, \delta_{\theta^\ast}) - \frac{K\mb E_{\theta^\ast}\lambda(X)}{2}W_2^2(\mu, \delta_{\theta^\ast})\Big)\\
    &\stackrel{\rii}{\succeq} \frac{\delta_{kk'}I_d}{\beta^2}\Big(w_k - r\sqrt{K\sigma_3} - \frac{r^2K\sigma_2}{2}\Big).
\end{align*}
Here, (i) is by Lemma \ref{lem: diff_Phimu_thetastar}; (ii) is by $\mb E_{\theta^\ast}S_1(X) \leq \sqrt{K\mb E_{\theta^\ast}S_2(X)} \leq \sqrt{K\sigma_3}$ and $\mb E_{\theta^\ast}\lambda(X) \leq \sigma_2$. So, we know $U(\cdot\,;\mu)$ is $\beta^{-2}\big(\min_k w_k - r\sqrt{K\sigma_3} - r^2K\sigma_2/2\big)$ strongly convex. Assumption \ref{assump: strong_convexity} holds.

\smallskip
\noindent\underline{\bf Verification of $\lambda > 2\gamma$.} Let's first give an upper bound of $\gamma = \matnorm{I_S(m^\ast)}$. We can assume that $\frac{\sum_k m_k^\ast}{K} = 0$. Otherwise, consider a shifted model with $\tilde{x} = x - \frac{\sum_k m_k^\ast}{K}$ and $\tilde{m}_l^\ast = m_l^\ast - \frac{\sum_km_k^\ast}{K}$ for $l\in[K]$. In this case, the missing data Fisher information matrix $I_S(m^\ast)$ remains the same. Recall that
\begin{align*}
    I_S(m^\ast) &= \int_{\mb R^d} \sum_{z=1}^K p(z\,|\,x,m^\ast)\big[\nabla\log p(z\,|\,x,m^\ast)\big]\big[\nabla\log p(z\,|\,x,m^\ast)\big]^T p(x\,|\,m^\ast)\,\dd x\\
    &= \mb E_{X, Z}\big[\nabla\log p(Z\,|\,X,m^\ast)\big]\big[\nabla\log p(Z\,|\,X,m^\ast)\big]^T\\
    &\stackrel{}{=}\mb E_Z \mb E_{X\,|\,Z}\begin{pmatrix}
    \big(\delta_{1Z} - \frac{w_1E_1}{p}\big)\frac{X-m_1}{\beta^2}\\
    \vdots\\
    \big(\delta_{KZ} - \frac{w_KE_K}{p}\big)\frac{X-m_K}{\beta^2}
    \end{pmatrix}
    \begin{pmatrix}
    \big(\delta_{1Z} - \frac{w_1E_1}{p}\big)\frac{X-m_1}{\beta^2}\\
    \vdots\\
    \big(\delta_{KZ} - \frac{w_KE_K}{p}\big)\frac{X-m_K}{\beta^2}
    \end{pmatrix}^T\\
    &= \mb E_Z \bigg(\mb E\Big[\Big(\delta_{iZ} - \frac{w_iE_i}{p}\Big)\Big(\delta_{jZ} - \frac{w_{j}E_{j}}{p}\Big)\Big(\frac{X-m_i}{\beta^2}\Big)\Big(\frac{X-m_{j}}{\beta^2}\Big)^T\,\Big|\,Z\Big]\bigg)_{1\leq i,j\leq K}.
\end{align*}
Since $X\sim\m N(m_Z, \beta^2I_d)$, let $X = m_Z + \beta\eta$ where $\eta\sim\m N(0, I_d)$. For $i=Z$, we have
\begin{align*}
    \delta_{iZ} - \frac{w_iE_i}{p} &= \frac{\sum_{k\neq Z}w_k\exp\Big\{-\frac{\|m_Z-m_k\|^2}{2\beta^2} + \frac{\eta^T(m_Z-m_k)}{\beta}\Big\}}{w_Z + \sum_{k\neq Z}w_k\exp\Big\{-\frac{\|m_Z-m_k\|^2}{2\beta^2} + \frac{\eta^T(m_Z-m_k)}{\beta}\Big\}}\\
    &\leq {\sum_{k\neq Z}\frac{w_k}{w_Z}\exp\Big\{-\frac{\|m_Z-m_k\|^2}{2\beta^2} + \frac{\|\eta\|\cdot\|m_Z-m_k\|}{\beta}\Big\}}\\
    &\stackrel{\ri}{\leq} \sum_{k\neq Z}\frac{w_k}{w_Z}\exp\Big\{-\frac{\|m_Z-m_k\|^2}{4\beta^2}\Big\}.
\end{align*}
Here, (i) holds when $\|\eta\| \leq \frac{d_{\min}}{4\beta} \leq \min_{k\neq Z}\frac{\|m_Z-m_k\|}{4\beta}$ where $d_{\min} = \min_{i\neq j}\|m_i - m_j\|$. Similarly, when $i\neq Z$ we have
\begin{align*}
    \Big|\delta_{i} - \frac{w_iE_i}{p}\Big| \leq \frac{w_i}{w_Z}\exp\Big\{-\frac{\|m_Z - m_i\|^2}{4\beta^2}\Big\} \leq \frac{w_i}{w_Z}\exp\Big\{-\frac{d_{\min}^2}{4\beta^2}\Big\}
\end{align*}
when $\|\eta\| \leq \frac{\|m_Z-m_i\|}{4\beta}$. Then, by definition
\begin{align*}
    \gamma &= \matnorm{I_S(m^\ast)}\\
    &= \sup_{v\in\mb S^{Kd-1}} \mb E_Z \sum_{i,j}\mb E\Big[\Big(\delta_{iZ} - \frac{w_iE_i}{p}\Big)\Big(\delta_{jZ} - \frac{w_{j}E_{j}}{p}\Big)\cdot v_i^T\Big(\frac{X-m_i}{\beta^2}\Big)\Big(\frac{X-m_{j}}{\beta^2}\Big)^Tv_{j}\,\Big|\,Z\Big]\\
    &= \sup_{v\in\mb S^{Kd-1}}\mb E_Z\mb E\Big[\Big(\sum_i\big(\delta_{iZ}-\frac{w_iE_i}{p}\big)\cdot\frac{v_i^T(X-m_i)}{\beta^2}\Big)^2\,\Big|\,Z\Big]\\
    &\leq K\sup_{v\in\mb S^{Kd-1}} \mb E_Z\mb E\bigg[\sum_{i\neq Z}\Big(\delta_{iZ}-\frac{w_iE_i}{p}\Big)^2\bigg(\frac{v_i^T(X-m_i)}{\beta^2}\bigg)^2I\Big\{\|\eta\| \leq \frac{\|m_Z-m_i\|}{4\beta}\Big\}\,\bigg|\,Z\bigg]\\
    &\qquad + K\sup_{v\in\mb S^{Kd-1}} \mb E_Z\mb E\bigg[\sum_{i\neq Z}\bigg(\frac{v_i^T(X-m_i)}{\beta^2}\bigg)^2I\Big\{\|\eta\| > \frac{\|m_Z-m_i\|}{4\beta}\Big\}\,\bigg|\,Z\bigg]\\
    &\qquad + K\sup_{v\in\mb S^{Kd-1}}\mb E_Z\mb E\bigg[\Big(1-\frac{w_ZE_Z}{p}\Big)^2\bigg(\frac{v_Z^T(X-m_Z)}{\beta^2}\bigg)^2I\Big\{\|\eta\| \leq \frac{d_{\min}}{4\beta}\Big\}\,\bigg|\,Z\bigg]\\
    &\qquad + K\sup_{v\in\mb S^{Kd-1}}\mb E_Z\mb E\bigg[\bigg(\frac{v_Z^T(X-m_Z)}{\beta^2}\bigg)^2I\Big\{\|\eta\| > \frac{d_{\min}}{4\beta}\Big\}\,\bigg|\,Z\bigg]\\
    &= I_1 + I_2 + I_3 + I_4.
\end{align*}
To bound the first term, we have
\begin{align*}
    I_1 &\leq K\sup_{v\in\mb S^{Kd-1}}\mb E_Z\mb E\bigg[\sum_{i\neq Z}\Big(\frac{w_i}{w_Z}\exp\Big\{-\frac{\|m_Z-m_i\|^2}{4\beta^2}\Big\}\Big)^2\bigg(\frac{v_i^T(X-m_i)^2}{\beta^2}\bigg)^2\,\bigg|\,Z\bigg]\\
    &=K\sup_{v\in\mb S^{Kd-1}}\mb E_Z\bigg[\sum_{i\neq Z}\Big(\frac{w_i}{w_Z}\exp\Big\{-\frac{\|m_Z-m_i\|^2}{4\beta^2}\Big\}\Big)^2\frac{\big(v_i^T(m_Z-m_i)\big)^2 + \beta^2\|v_i\|^2}{\beta^4}\bigg]\\
    &< \frac{Kw_{\max}}{\beta^2w_{\min}}\Big(1+\frac{d_{\min}^2}{\beta^2}\Big)e^{-\frac{d_{\min}^2}{2\beta^2}}.
\end{align*}
Here we let $w_{\min} = \min_i w_i$ and $w_{\max} = \max_i w_i$. To bound the second term, by Example 2.28 in \cite{wainwright2019high} we have
\begin{align*}
    \mb P\Big(\frac{\|\eta\|^2}{d} > \frac{\|m_Z-m_i\|^2}{16d\beta^2}\Big) \leq \exp\Big\{-\frac{d}{2}\Big(\frac{\|m_Z-m_i\|}{4\beta\sqrt{d}}-1\Big)^2\Big\}
\end{align*}
when $\frac{\|m_Z-m_i\|}{4\beta\sqrt{d}} \geq \frac{d_{\min}}{4\beta\sqrt{d}} \geq 1$. Therefore, by Cauchy--Schwarz inequality, we have
\begin{align*}
    I_2 &\leq K\sup_{v\in\mb S^{Kd-1}}\mb E_Z\sum_{i\neq Z}\sqrt{\mb E\Big[\Big(\frac{v_i^T(X-m_i)}{\beta^2}\Big)^4\,\Big|\,Z\Big]\cdot\mb P\Big(\|\eta\|>\frac{\|m_Z-m_i\|}{4\beta}\Big)}\\
    &\leq K\sup_{v\in\mb S^{Kd-1}}\mb E_Z\sum_{i\neq Z}\sqrt{3\beta^{-8}\big[\big(v_i^T(m_Z-m_i)\big)^2 + \beta^2\|v_i\|^2\big]^2\cdot\exp\Big\{-\frac{d}{2}\Big(\frac{\|m_Z-m_i\|}{4\sqrt{d}\beta}-1\Big)^2\Big\}}\\
    &\leq \sqrt{3}\beta^{-2}K\Big(\frac{d_{\min}^2}{\beta^2}+1\Big)\exp\Big\{-\frac{d_{\min}^2}{256\beta^2}\Big\}.
\end{align*}
The last inequality holds when $\frac{d_{\min}}{8\beta\sqrt{d}} \geq 1$. For the third term, notice that
\begin{align*}
    I_3 &\leq K\sup_{v\in\mb S^{Kd-1}}\mb E_Z\mb E\bigg[\bigg(\sum_{k\neq Z}\frac{w_k}{w_Z}\exp\Big\{-\frac{\|m_Z-m_k\|^2}{4\beta^2}\Big\}\bigg)^2\bigg(\frac{v_Z^T(X-m_Z)}{\beta^2}\bigg)^2\,\bigg|\,Z\bigg]\\
    &\leq K\sup_{v\in\mb S^{Kd-1}}\mb E_Z\bigg[\frac{\|V_Z\|^2}{\beta^2}\Big(\sum_{k\neq Z}\frac{w_k}{w_Z}\Big)^2\exp\Big\{-\frac{d_{\min}^2}{2\beta^2}\Big\}\bigg]\\
    &\leq \frac{K}{\beta^2w_{\min}}e^{-\frac{d_{\min}^2}{2\beta^2}}.
\end{align*}
The fourth term can be bounded as
\begin{align*}
    I_4 &\leq K\sup_{v\in\mb S^{Kd-1}}\mb E_Z\sqrt{\mb E\bigg[\bigg(\frac{v_Z^T(X-m_Z)}{\beta^2}\bigg)^4\,\bigg|\,Z\bigg]\mb P\Big(\|\eta\| > \frac{d_{\min}}{4\beta}\Big)} \\
    &\leq K\sup_{v\in\mb S_{Kd-1}}\mb E_Z\sqrt{\frac{3\|v_Z\|^4}{\beta^4}\exp\Big\{-\frac{d_{\min}^2}{128\beta^2}\Big\}}\\
    &\leq \frac{\sqrt{3}Kw_{\max}}{\beta^2}e^{-\frac{d_{\min}^2}{256\beta^2}}
\end{align*}
Therefore, we have
\begin{align*}
    \gamma &\leq I_1 + I_2 + I_3 + I_4 \leq \frac{3(2+\frac{d_{\min}^2}{\beta^2})Ke^{-\frac{d_{\min}^2}{256\beta^2}}}{\beta^2w_{\min}}.
\end{align*}
Recall that we have
\begin{align*}
    \lambda &\geq \beta^{-2}\Big(w_{\min} - r\sqrt{K\sigma_3} - \frac{r^2K\sigma_2}{2}\Big) > \frac{w_{\min}}{2\beta^2}.
\end{align*}
The last inequality requires
\begin{align}\label{eqn: GMM_r}
\begin{aligned}
    r \leq \frac{2\sqrt{K\sigma_3}}{K\sigma_2} + \sqrt{\frac{w_{\min}}{K\sigma_2}}
    &= \frac{\beta^{4}\sqrt{K\cdot \frac{6(2+\kappa_{\rm SNR}^2)Ke^{-\kappa_{\rm SNR}^2/256}}{\beta^2w_{\min}}}}{2K^2\big[{R^2 + d\beta^2 + \sum_{k=1}^K\|m_k^\ast\|^2}+\beta^{4}/2\big]}\\
    &\qquad\qquad+\sqrt{\frac{w_{\min}\beta^4}{4K^2\big[{R^2 + d\beta^2 + \sum_{k=1}^K\|m_k^\ast\|^2}+{\beta^4}/2\big]}}
\end{aligned}
\end{align}
This implies
\begin{align*}
    \frac{\lambda}{\gamma} > \frac{w_{\min}^2}{6K}\cdot\frac{e^{\frac{d_{\min}^2}{256\beta^2}}}{2+\frac{d_{\min}^2}{\beta^2}} > 2
\end{align*}
when the signal-to-noise ratio $\kappa_{\rm SNR}=\frac{d_{\min}}{\beta}$ is large enough.

We have shown all assumptions to apply Theorem \ref{thm: main_theorem} hold. When the sample size $n\to\infty$, the contraction number tends to
\begin{align*}
    1 - \frac{\big(\frac{w_{\min}^2}{6K}\cdot\frac{\exp\{\kappa_{\rm SNR}^2/256\}}{2+\kappa_{\rm SNR}^2}-2\big)\big(\frac{3w_{\min}^2}{6K}\cdot\frac{\exp\{\kappa_{\rm SNR}^2/256\}}{2+\kappa_{\rm SNR}^2}+2\big)}{4\big(\frac{w_{\min}^2}{6K}\cdot\frac{\exp\{\kappa_{\rm SNR}^2/256\}}{2+\kappa_{\rm SNR}^2}\big)^2 + \frac{w_{\min}^2}{6K}\cdot\frac{\exp\{\kappa_{\rm SNR}^2/256\}}{2+\kappa_{\rm SNR}^2} - 2}.
\end{align*}
It is decreasing w.r.t. $\kappa_{\rm SNR}$ and the limit is $\frac{1}{4}$. Also, by Cauchy--Schwarz inequality,
\begin{align*}
    \mb E_{m^\ast}[S_1(X)^3] &\leq K\sqrt{K}\mb E_{m^\ast}[S_2(X)^\frac{3}{2}]\\
    &\leq K^3\beta^{-6}\mb E_{m^\ast}\Big(\sum_{k=1}^K\|X-m_k^\ast\|^2\Big)^\frac{3}{2}\\
    &\leq K^\frac{7}{2}\beta^{-6}\mb E_{m^\ast}\sum_{k=1}^K\|X-m_k^\ast\|^3\\
    &\leq C'K^\frac{7}{2}\beta^{-6}\sum_{k=1}^K\Big(\sum_{i=1}^Kw_i(\|m_i^\ast - m_k^\ast\|^2 +d\beta^2)^2\Big)^\frac{3}{4}
\end{align*}
and
\begin{align*}
    \mb E_{m^\ast}[\lambda(X)^3] &\leq K^3\mb E_{m^\ast}\Big[\frac{2(\|X\|^2+R^2)}{\beta^4} + 1\Big]^3\\
    &\leq C'K^3\Big[\beta^{-12}\sum_{k=1}^Kw_i(d\beta^2+\|m_k^\ast\|^2)^3 + \beta^{-12}R^2+1\Big].
\end{align*}
The above analysis imply
\begin{align}\label{eqn: GMM_ABC}
\begin{aligned}
    A, B, C&\leq C'\Big[K^2\big(\mb E_{m^\ast}[S_1(X)^3]+1\big) + K^3\big(\mb E_{m^\ast}[\lambda(X)^3]+1\big)\Big]\\
    &\leq C'\bigg[K^2\Big[K^\frac{7}{2}\beta^{-6}\sum_{k=1}^K\Big(\sum_{i=1}^Kw_i(\|m_i^\ast - m_k^\ast\|^2 +d\beta^2)^2\Big)^\frac{3}{4}+1\Big]\\
    &\qquad\qquad+ K^3\Big[K^3\Big(\beta^{-12}\sum_{k=1}^Kw_i(d\beta^2+\|m_k^\ast\|^2)^3 + \beta^{-12}R^2+1\Big)+1\Big]\bigg]
\end{aligned}
\end{align}
So, The radius can be taken as
\begin{align}\label{eqn: GMM_RW}
    R_W = C'\min\bigg\{\sqrt{\frac{w_{\min}(\frac{w_{\min}^2}{6K}\cdot\frac{e^{\kappa_{\rm SNR}^2/256}}{2+\kappa_{\rm SNR}^2}-2)}{2\beta^2A(\frac{w_{\min}^2}{6K}\cdot\frac{e^{\kappa_{\rm SNR}^2/256}}{2+\kappa_{\rm SNR}^2}+3)}}, \frac{w_{\min}(\frac{w_{\min}^2}{6K}\cdot\frac{e^{\kappa_{\rm SNR}^2/256}}{2+\kappa_{\rm SNR}^2}-2)}{2\beta^2C(\frac{w_{\min}^2}{6K}\cdot\frac{e^{\kappa_{\rm SNR}^2/256}}{2+\kappa_{\rm SNR}^2}+3)},r\bigg\}
\end{align}
where $A, C$ is bounded as in (\ref{eqn: GMM_ABC}), and $r$ is bounded as in (\ref{eqn: GMM_r}). Here, $C'$ is a universal constant varying from line to line.
\end{proof}

Here, we will show that the repulsive prior (\ref{eqn: repulsive_prior}) satisfies Assumption \ref{assump: prior condition}. By mean value theorem, there exists $\xi\in\m M$ on the segment of $m$ and $m^\ast$, s.t.
\begin{align*}
    \int_{\mb R^d}\bigg(\log\frac{p(x\,|\,m)}{p(x\,|\,m^\ast)}\bigg)^2 p(x\,|\,m^\ast)\,\dd x
    &= \int_{\mb R^d}\big\langle \nabla\log p(x\,|\,m)\Big|_{m=\xi}, m-m^\ast\big\rangle^2 p(x\,|\,m^\ast)\,\dd x\\
    &\leq \|m-m^\ast\|^2\int_{\mb R^d}\big|\!\big|\nabla\log p(x\,|\,\xi)\big|\!\big|_2^2\, p(x\,|\,m^\ast)\,\dd x\\
    &\leq \|m-m^\ast\|^2\cdot\sup_{m\in\m M}\int_{\mb R^d}\big|\!\big|\nabla\log p(x\,|\,m)\big|\!\big|_2^2\, p(x\,|\,m^\ast)\,\dd x\\
    &\lesssim \|m-m^\ast\|^2.
\end{align*}
So, there is a universal constant $C$, such that $\{m\in\m M: \|m-m^\ast\|\leq C\varepsilon_n\}\subset B_n$. Let $V(\cdot)$ denote the volume in $\mb R^d$. Then, for some normalization constant $C'(g_0, \sigma^2)$, we have
\begin{align*}
    \Pi(B_n) &\geq\Pi\Big(\{m\in\m M: \|m-m^\ast\|\leq C\varepsilon_n\}\Big)\\
    &= C'(g_0, \sigma^2)\int_{B_{\mb R^d}(m^\ast, C\varepsilon_n)} \frac{\min_{1\leq i<j\leq K}\|m_i - m_j\|}{\min_{1\leq i<j\leq K}\|m_i - m_j\| + g_0} \cdot\prod_{k=1}^K e^{-\frac{\|m_i\|^2}{2\sigma^2}}\,\dd m\\
    &\stackrel{\ri}{\geq} \frac{C'(g_0, \sigma^2)d_{\min}}{d_{\min} + 2g_0}\int_{B_{\mb R^d}(m^\ast, C\varepsilon_n)}\prod_{k=1}^Ke^{-\frac{\|m_i\|^2}{2\sigma^2}}\,\dd\theta\\
    &\geq \frac{C'(g_0, \sigma^2)d_{\min}}{d_{\min} + 2g_0} \min_{m:\|m-m^\ast\|\leq C\varepsilon_n} e^{-\sum_{i=1}^K\frac{\|m_i\|^2}{2\sigma^2}} V\Big(\{m\in\m M: \|m-m^\ast\| \leq C\varepsilon_n\}\Big)\\
    &= \frac{C'(g_0, \sigma^2)d_{\min}}{d_{\min} + 2g_0} \min_{m:\|m-m^\ast\|\leq C\varepsilon_n} e^{-\sum_{i=1}^K\frac{\|m_i\|^2}{2\sigma^2}} V\Big(\{m\in\m M: \|m-m^\ast\| \leq 1\}\Big)(C\varepsilon_n)^d\\
    &\gtrsim \Big(\frac{\log n}{n}\Big)^\frac{d}{2}\\
    &= e^{-dn\varepsilon_n^2}.
\end{align*}

\subsection{Mixture of regression and proof of Corollary~\ref{coro: MR}}\label{app: MR_proof}

\begin{proof}[Proof of Corollary \ref{coro: MR}]
By definition, we have
\begin{align*}
    p(X, y\,|\,\theta) &= \frac{1}{(2\pi)^\frac{d}{2}}e^{-\frac{\|X\|^2}{2}}\cdot\frac{1}{2\sqrt{2\pi\beta^2}}\Big(e^{-\frac{(y-X^T\theta)^2}{2\beta^2}} + e^{-\frac{(y+X^T\theta)^2}{2\beta^2}}\Big)\\
    p(z\,|\,X,y,\theta) &= \frac{e^{-\frac{\|y-zX^T\theta\|^2}{2\beta^2}}}{e^{-\frac{\|y+X^T\theta\|^2}{2\beta^2}} + e^{-\frac{\|y-X^T\theta\|^2}{2\beta^2}}}.
\end{align*}
Therefore, we know
\begin{align}\label{eqn: grad&hess_MR}
\begin{aligned}
    \log p(z\,|\,X, y, \theta) &= \log p(z, y\,|\,X,\theta) - \log p(y\,|\,X,\theta)\\
    &= -\frac{(y-zX^T\theta)^2}{2\beta^2} - \log p(y\,|\,X,\theta) + \log\frac{1}{2\sqrt{2\pi\beta^2}}\\
    \nabla\log p(z\,|\,X,y,\theta) &= \frac{(y-zX^T\theta)zX}{\beta^2} - \nabla\log p(y\,|\,X,\theta)\\
    &= \frac{Xy}{\beta^2}\bigg(z - \frac{e^{-\frac{(y-X^T\theta)^2}{2\beta^2}} - e^{-\frac{(y+X^T\theta)^2}{2\beta^2}}}{e^{-\frac{(y-X^T\theta)^2}{2\beta^2}} + e^{-\frac{(y+X^T\theta)^2}{2\beta^2}}}\bigg)\\
    \nabla^2\log p(z\,|\,X, y, \theta) &= -\frac{XX^T}{\beta^2} - \nabla^2\log p(y\,|\,X,\theta)\\
    &= - \frac{4e^{-\frac{(y-X^T\theta)^2}{2\beta^2}}\cdot e^{-\frac{(y+X^T\theta)^2}{2\beta^2}}}{(e^{-\frac{(y-X^T\theta)^2}{2\beta^2}} + e^{-\frac{(y+X^T\theta)^2}{2\beta^2}})^2} \cdot \frac{y^2}{\beta^2}\cdot\frac{XX^T}{\beta^2},
\end{aligned}
\end{align}
which indicates that
\begin{align*}
    \lambda(X, y) = \sup_{\theta\in\Theta, z\in\{-1, 1\}}\matnorm{\nabla^2\log p(z\,|\,X,y,\theta)} \leq \frac{\|X\|^2y^2}{\beta^4}.
\end{align*}

\smallskip
\noindent\underline{\bf Concentration properties.} By mean value theorem
\begin{align*}
    \KL\big(p(\cdot\,|\,\theta_1)\,||\,p(\cdot\,|\,\theta_2)\big) \leq \frac{1}{2}\sup_{\theta,\theta'\in\Theta}\bigg(\int_{\mb R^d\times \mb R}\matnorm{\nabla^2\log p(X,y\,|\,\theta)}p(X,y\,|\,\theta')\,\dd X\dd y\bigg)\|\theta_1 - \theta_2\|^2.
\end{align*}
Since $\Theta$ is compact, $\log p(X, y\,|\,\theta)\in\m C^2(\Theta)$, and the integration is always finite, we know $\KL\big(p(\cdot\,|\,\theta_1)\,||\,p(\cdot\,|\,\theta_2)\big) \lesssim\|\theta_1-\theta_2\|^2$. Similarly, we have
\begin{align*}
    &\quad\,\KL\big(p(\cdot\,|\,X,y,\theta_1)\,||\,p(\cdot\,|\,X,y,\theta_2)\big)\\ 
    &\leq \frac{1}{2}\sup_{\theta,\theta'\in\Theta}\sum_{z\in\{-1, +1\}}\matnorm{\nabla^2\log p(z\,|\,X,y,\theta)}p(z\,|\,X,y,\theta')\cdot\|\theta_1-\theta_2\|^2\\
    &\leq \frac{1}{2}\sup_{\theta\in\Theta, z\in\{1, -1\}}\matnorm{\nabla^2\log p(z\,|\,X,y,\theta)}\cdot\|\theta_1-\theta_2\|^2\\
    &\leq \frac{\|X\|^2y^2}{2\beta^4}\|\theta_1-\theta_2\|^2.
\end{align*}
$\frac{\|X\|^2y^2}{2\beta^4}$ has finite $\psi_\frac{1}{4}$-norm $\sigma_4$ since both $X$ and $y$ are Gaussian distribution. Thus, Assumption \ref{assump: qudratic growth of KL} holds. Assumption \ref{assump: test condition} holds due to the same reasons as of Corollary \ref{coro: GMM}.

\smallskip
\noindent\underline{\bf Regularity of log-likelihood function.} By (\ref{eqn: grad&hess_MR}), we know
\begin{align*}
    S_2(X, y) \leq 8\beta^{-4}y^2\|X\|^2.
\end{align*}
This implies both $\mb E_{\theta^\ast}S_1(X, y)$ and $\mb E_{\theta^\ast}S_2(X, y)$ are finite, and $S_2(X, y)$ has finite $\psi_\frac{1}{4}$-norm $\sigma_3 (= \sigma_4/4)$. Notice that
\begin{align*}
    \nabla^2\log p(X, y,z\,|\,\theta) = -\frac{XX^T}{\beta^2}.
\end{align*}
So, $\nabla^2\log p(X, y, z\,|\,\theta_1) - \nabla^2\log p(X,y,z\,|\,\theta_2) = 0$, and we can take $J_z(X, y) = 0.$ For any $v\in B_{\mb R^d}(0, 1)$, we have
\begin{align*}
    \Big|\sum_{z} p(z\,|\,X,\theta^\ast)\langle v, \nabla^2\log p(X, y, z\,|\,\theta)v\rangle\Big| = \frac{\|X^Tv\|^2}{\beta^2} \leq \frac{\|X\|^2}{\beta^2}.
\end{align*}
It is sub-exponential with some parameter $\sigma_1$. Recall that $\lambda(X, y) \leq \frac{\|X\|^2y^2}{\beta^4}$ has finite $\psi_\frac{1}{4}$-norm $\sigma_2 = 2\sigma_4$. Therefore, Assumption \ref{assump: continuity_of_Hessian} holds.

\smallskip
\noindent\underline{\bf Convexity of $U(\cdot\,;\mu)$.} By definition,
\begin{align*}
    U(\theta; \mu) &= \int_{\mb R^d\times\mb R}\sum_{z}\Phi(\mu, X, y)(z)\bigg[\frac{\|X\|^2}{2} + \frac{1}{2\beta^2}\|y-zX^T\theta\|^2\bigg]\,p(X,y\,|\,\theta^\ast)\,\dd X\dd y\\
    &\qquad\qquad + \log 2\sqrt{2\pi\beta^2} + \frac{d}{2}\log(2\pi). 
\end{align*}
So, for all $\mu\in B_{\mb W_2}(\delta_{\theta^\ast}, r)$
\begin{align*}
    \frac{\partial^2 U(\theta, \mu)}{\partial\theta^2} &= \int_{\mb R^d\times \mb R} \frac{XX^T}{\beta^2}  p(X,y\,|\,\theta^\ast)\,\dd X\dd y = \beta^{-2}I_d.
\end{align*}
This implies $U(\cdot\,;\mu)$ is $\beta^{-2}$-strongly convex for all $\mu\in\ms P_r^2(\Theta)$. Assumption \ref{assump: strong_convexity} holds.

\smallskip
\noindent\underline{\bf Verification of $\lambda > 2\gamma$.} Recall that
\begin{align*}
    I_S(\theta^\ast) &= \mb E \big[\nabla\log p(Z\,|\,X,y,\theta^\ast)\big]\big[\nabla\log p(Z\,|\,X,y,\theta^\ast)\big]^T\\
    &= \mb E_{Z, X}\frac{XX^T}{\beta^4}\mb E\bigg[\bigg(Z - \frac{\exp\big\{-\frac{(y-X^T\theta^\ast)^2}{2\beta^2}\big\} - \exp\big\{-\frac{(y+X^T\theta^\ast)^2}{2\beta^2}\big\}}{\exp\big\{-\frac{(y-X^T\theta^\ast)^2}{2\beta^2}\big\} + \exp\big\{-\frac{(y+X^T\theta^\ast)^2}{2\beta^2}\big\}}\bigg)^2y^2\,\bigg|\,Z,X\bigg].
\end{align*}
Since $y\sim\m N(X^T\theta^\ast, \beta^2)$ when $Z=1$, we can write $y = X^T\theta^\ast + \beta\eta$ where $\eta\sim\m N(0, 1)$. We have
\begin{align*}
    0\leq \Big|Z - \frac{f_1 - f_{-1}}{f_1 + f_{-1}}\Big|
    &= \frac{2\exp\big\{-\frac{(2X^T\theta^\ast + \beta\eta)^2}{2\beta^2}\big\}}{\exp\big\{-\frac{\eta^2}{2}\big\} + \exp\big\{-\frac{(2X^T\theta^\ast + \beta\eta)^2}{2\beta^2}\big\}}\\
    &\leq 2\exp\Big\{\frac{\eta^2}{2} - \frac{(2X^T\theta^\ast + \beta\eta)^2}{2\beta^2}\Big\}\\
    &= 2\exp\Big\{-\Big(\frac{X^T\theta^\ast}{\beta}\Big)^2 - \frac{2X^T\theta^\ast}{\beta}\eta\Big\}\\
    &\leq 2\exp\Big\{-\frac{1}{2}\Big(\frac{X^T\theta^\ast}{\beta}\Big)^2\Big\}.
\end{align*}
Here, the last inequality holds when $|\eta|\leq \big|\frac{X^T\theta^\ast}{4\beta}\big|$. The same inequality holds for $Z=-1$. Then
\begin{align*}
    &\quad\,\mb E\bigg[\Big(Z - \frac{f_1 - f_{-1}}{f_1 + f_{-1}}\Big)^2y^2\,\bigg|\, Z, X\bigg]\\
    &= \mb E\bigg[\Big(Z - \frac{f_1 - f_{-1}}{f_1 + f_{-1}}\Big)^2y^2I\Big\{|\eta|\leq\Big|\frac{X^T\theta^\ast}{4\beta}\Big|\Big\}\,\bigg|\, Z, X\bigg]\\ 
    &\qquad\qquad\qquad+ \mb E\bigg[\Big(Z - \frac{f_1 - f_{-1}}{f_1 + f_{-1}}\Big)^2y^2I\Big\{|\eta|>\Big|\frac{X^T\theta^\ast}{4\beta}\Big|\Big\}\,\bigg|\, Z, X\bigg]\\
    &\leq \mb E\bigg[4\exp\Big\{-\Big(\frac{X^T\theta^\ast}{\beta}\Big)^2\Big\}y^2\,\bigg|\,Z, X\bigg] + \mb E\bigg[4y^2I\Big\{|\eta| > \Big|\frac{X^T\theta^\ast}{4\beta}\Big|\Big\}\,\bigg|\,Z,X\bigg]\\
    &\leq 4\exp\Big\{-\Big(\frac{X^T\theta^\ast}{\beta}\Big)^2\Big\}\big[(ZX^T\theta^\ast)^2 + \beta^2\big]
     + 4\sqrt{\mb E[y^4\,|\,Z,X]\mb P\bigg(|\eta| > \Big|\frac{X^T\theta^\ast}{4\beta}\Big|\,\bigg|\, Z,X\bigg)}\\
    &\leq 4\exp\Big\{-\Big(\frac{X^T\theta^\ast}{\beta}\Big)^2\Big\}\big[(X^T\theta^\ast)^2 + \beta^2\big] + 4\sqrt{4\big[\beta^2 + (X^T\theta^\ast)^2\big]^2\cdot2\exp\Big\{-\frac{(X^T\theta^\ast)^2}{32\beta^2}\Big\}}\\
    &\leq 16\exp\Big\{-\Big(\frac{X^T\theta^\ast}{8\beta}\Big)^2\Big\}\cdot\big[(X^T\theta^\ast)^2 + \beta^2\big].
\end{align*}
Therefore, we know
\begin{align*}
    I_S(\theta^\ast) &\preceq \mb E\bigg[\frac{XX^T}{\beta^4}\cdot 16\exp\Big\{-\Big(\frac{X^T\theta^\ast}{8\beta}\Big)^2\Big\}\cdot\big((X^T\theta^\ast)^2 + \beta^2\big)\bigg]\\
    &= 16\beta^{-2}\mb E\bigg[XX^T\Big[\Big(\frac{X^T\theta^\ast}{\beta}\Big)^2+1\Big]\exp\Big\{-\Big(\frac{X^T\theta^\ast}{8\beta}\Big)^2\Big\}\bigg].
\end{align*}
This yields the bound of $\gamma$ as
\begin{align*}
    \gamma &= \matnorm{I_S(\theta^\ast)}\\
    &= 16\beta^{-2}\sup_{v\in\mb S^{d-1}}\mb E\bigg[v^TXX^Tv\Big[\Big(\frac{X^T\theta^\ast}{\beta}\Big)^2+1\Big]\exp\Big\{-\Big(\frac{X^T\theta^\ast}{8\beta}\Big)^2\Big\}\bigg]\\
    &\leq 16\beta^{-2}\sup_{v\in\mb S^{d-1}}\sqrt{\mb E(X^Tv)^4}\sqrt{\mb E\Big[\Big(\frac{X^T\theta^\ast}{\beta}\Big)^2+1\Big]^2\exp\Big\{-\frac{1}{32}\Big(\frac{X^T\theta^\ast}{\beta}\Big)^2\Big\}}\\
    &\leq 16\beta^{-2}\sup_{v\in\mb S^{d-1}}\sqrt{3\|v\|^4}\sqrt{2\mb E\Big[\Big(\frac{X^T\theta^\ast}{\beta}\Big)^4+1\Big]\exp\Big\{-\frac{1}{32}\Big(\frac{X^T\theta^\ast}{\beta}\Big)^2\Big\}}\\
    &\leq 16\beta^{-2}\sqrt{3}\sqrt{2\bigg[\sqrt{\frac{16}{16+\kappa_{\rm SNR}^2}}\cdot3\Big(\frac{16\kappa_{\rm SNR}^2}{16+\kappa_{\rm SNR}^2}\Big)^2 + \frac{4}{\sqrt{16+\kappa_{\rm SNR}^2}}\bigg]}\\
    &< \frac{2174\beta^{-2}}{\sqrt[4]{16+\kappa_{\rm SNR}^2}}.
\end{align*}
Here $\kappa_{\rm SNR} = \frac{\|\theta^\ast\|}{\beta}$ is the signal-to-noise ratio (SNR). So, we have
\begin{align*}
    \kappa = \frac{\lambda}{\gamma} > \frac{\sqrt[4]{16+\kappa_{\rm SNR}^2}}{2174} > 2
\end{align*}
when $\kappa_{\rm SNR}$ is large enough. So, the contraction number tends to
\begin{align*}
    1 - \frac{\big(\frac{\sqrt[4]{16+\kappa_{\rm SNR}^2}}{2174}-2\big)(3\frac{\sqrt[4]{16+\kappa_{\rm SNR}^2}}{2174}+2)}{4\big(\frac{\sqrt[4]{16+\kappa_{\rm SNR}^2}}{2174}\big)^2 + \frac{\sqrt[4]{16+\kappa_{\rm SNR}^2}}{2174} - 2}
\end{align*}
as the sample size $n\to\infty$.

Recall that 
\begin{align*}
S_2(X, y) \leq 8\beta^{-4}y^2\|X\|^2
\end{align*}
and
\begin{align*}
    \lambda(X, y) \leq \beta^{-4}y^2\|X\|^2.
\end{align*}
So, by Cauchy--Schwarz's inequality,
\begin{align*}
    \mb E_{\theta^\ast}[S_1(X, y)^3] & \leq 2\sqrt{2}\mb E_{\theta^\ast}\big[\big(S_2(X, y)\big)^\frac{3}{2}\big] \leq C'd\beta^{-6}(\|\theta^\ast\|^2 + \beta^2)^\frac{3}{2}
\end{align*}
and
\begin{align*}
    \mb E_{\theta^\ast}[\lambda(X, y)^3] \leq C'd\beta^{-12}(\|\theta^\ast\|^2 + \beta^2)^3.
\end{align*}
for some constant $C' > 0$. Therefore, we have
\begin{align*}
    A, B, C &\leq C'\Big[\big[d\beta^{-3}(\kappa_{\rm SNR}^2+1)^\frac{3}{2}+1\big] + \big[d\beta^{-6}(\kappa_{\rm SNR}^2+\beta^2)^3+1\big]\Big]
\end{align*}
So, we can take the radius
\begin{align*}
    R_W &= \frac{C'\beta^{-2}(\frac{\sqrt[4]{16+\kappa_{\rm SNR}^2}}{2174}-2)}{\Big[\big[d\beta^{-3}(\kappa_{\rm SNR}^2+1)^\frac{3}{2}+1\big] + \big[d\beta^{-6}(\kappa_{\rm SNR}^2+1)^3+1\big]\Big](\frac{\sqrt[4]{16+\kappa_{\rm SNR}^2}}{2174}+3)}.
\end{align*}
Notice $C'$ may vary from line to line.
\end{proof}

In fact, the lower bound of $\kappa_{\rm SNR}$ derived above can be sharpened. Note that
\begin{align*}
\gamma &= \matnorm{I_S(\theta^\ast)}\\
&= \sup_{v\in\mb S^{d-1}}\mb E_{Z,X}v^T\frac{XX^T}{\beta^4}v\mb E\bigg[\bigg(Z - \frac{\exp\big\{-\frac{(y-X^T\theta^\ast)^2}{2\beta^2}\big\} - \exp\big\{-\frac{(y+X^T\theta^\ast)^2}{2\beta^2}\big\}}{\exp\big\{-\frac{(y-X^T\theta^\ast)^2}{2\beta^2}\big\} + \exp\big\{-\frac{(y+X^T\theta^\ast)^2}{2\beta^2}\big\}}\bigg)^2y^2\,\bigg|\,Z,X\bigg]\\
&\leq \sup_{v\in\mb S^{d-1}}\sqrt{\mb E_{Z,X}\frac{\|X^Tv\|^4}{\beta^8}}\sqrt{\mb E_{Z,X}\bigg\{E\bigg[\bigg(Z - \frac{\exp\big\{-\frac{(y-X^T\theta^\ast)^2}{2\beta^2}\big\} - \exp\big\{-\frac{(y+X^T\theta^\ast)^2}{2\beta^2}\big\}}{\exp\big\{-\frac{(y-X^T\theta^\ast)^2}{2\beta^2}\big\} + \exp\big\{-\frac{(y+X^T\theta^\ast)^2}{2\beta^2}\big\}}\bigg)^2y^2\,\bigg|\,Z,X\bigg]\bigg\}^2}\\
&= \frac{\sqrt{3}}{\beta^{4}}\sqrt{\mb E_{Z,X}\bigg\{E\bigg[\bigg(Z - \frac{\exp\big\{-\frac{(y-X^T\theta^\ast)^2}{2\beta^2}\big\} - \exp\big\{-\frac{(y+X^T\theta^\ast)^2}{2\beta^2}\big\}}{\exp\big\{-\frac{(y-X^T\theta^\ast)^2}{2\beta^2}\big\} + \exp\big\{-\frac{(y+X^T\theta^\ast)^2}{2\beta^2}\big\}}\bigg)^2y^2\,\bigg|\,Z,X\bigg]\bigg\}^2}.
\end{align*}
Let $f_z = \exp\{-\frac{(y-zX^T\theta^\ast)^2}{2\beta^2}\}$ for simplicity. Note that $\frac{X^T\theta^\ast}{\beta} \stackrel{d}{=} \kappa_{\rm SNR}^2\xi$ with $\xi\sim\m N(0, 1)$, where the notation $\stackrel{d}{=}$ means equal in distribution. Since $y\sim\m N(X^T\theta^\ast, \beta^2)$ when $Z=1$, we can write $y = X^T\theta^\ast + \beta\eta$ where $\eta\sim\m N(0, 1)$. So, we have
\begin{align*}
\Big|Z - \frac{f_1 - f_{-1}}{f_1 + f_{-1}}\Big|^2
&= \bigg(\frac{2\exp\big\{-\frac{(2X^T\theta^\ast + \beta\eta)^2}{2\beta^2}\big\}}{\exp\big\{-\frac{\eta^2}{2}\big\} + \exp\big\{-\frac{(2X^T\theta^\ast + \beta\eta)^2}{2\beta^2}\big\}}\bigg)^2 \stackrel{d}{=} \Big(\frac{2\exp\big\{-\frac{1}{2}(2\kappa_{\rm SNR}\xi + \eta)^2\big\}}{e^{-\eta^2/2}+\exp\big\{-\frac{1}{2}(2\kappa_{\rm SNR}\xi + \eta)^2\big\}}\Big)^2.
\end{align*}
Similarly, since $y\sim\m N(-X^T\theta^\ast, \beta^2)$ when $Z=-1$, we can write $y = -X^T\theta^\ast + \beta\eta$ where $\eta\sim\m N(0, 1)$. In this case, we have
\begin{align*}
\Big|Z - \frac{f_1 - f_{-1}}{f_1 + f_{-1}}\Big|^2
&= \bigg(\frac{2\exp\big\{-\frac{(-2X^T\theta^\ast + \beta\eta)^2}{2\beta^2}\big\}}{\exp\big\{-\frac{\eta^2}{2}\big\} + \exp\big\{-\frac{(-2X^T\theta^\ast + \beta\eta)^2}{2\beta^2}\big\}}\bigg)^2 \stackrel{d}{=} \Big(\frac{2\exp\big\{-\frac{1}{2}(2\kappa_{\rm SNR}\xi + \eta)^2\big\}}{e^{-\eta^2/2}+\exp\big\{-\frac{1}{2}(2\kappa_{\rm SNR}\xi + \eta)^2\big\}}\Big)^2.
\end{align*}
Therefore, we have
\begin{align*}
\gamma \leq \frac{\sqrt{3}}{\beta^2}\sqrt{\mb E_\xi\Big\{\mb E_\eta\Big(\frac{2\exp\big\{-\frac{1}{2}(2\kappa_{\rm SNR}\xi + \eta)^2\big\}}{e^{-\eta^2/2}+\exp\big\{-\frac{1}{2}(2\kappa_{\rm SNR}\xi + \eta)^2\big\}}\Big)^2(\eta + \kappa_{\rm SNR}\xi)^2\Big\}^2}.
\end{align*}
Recall that $\lambda = \beta^{-2}$. The requirement for $\lambda / \gamma > 2$ can be met when
\begin{align*}
\mb E_\xi\Big\{\mb E_\eta\Big(\frac{2\exp\big\{-\frac{1}{2}(2\kappa_{\rm SNR}\xi + \eta)^2\big\}}{e^{-\eta^2/2}+\exp\big\{-\frac{1}{2}(2\kappa_{\rm SNR}\xi + \eta)^2\big\}}\Big)^2(\eta + \kappa_{\rm SNR}\xi)^2\Big\}^2 < \frac{1}{12} \approx 0.083.
\end{align*}
The left-hand side is around 0.05 when $\kappa_{\rm SNR}=10$ by Monte Carlo integration.

\section{Technical results and proofs}\label{app:tech_results}
In this appendix, we collect details and proofs of all technical results used in the proofs of the main results.

\subsection{More technical results}
In this subsection, we list all other technical lemmas and their proofs used in the proofs in this supplementary material.

\begin{corollary}\label{coro: bound_square_expectation}
When (\ref{eqn: concentration_target_measure}) holds, and the sample size satisfies
\begin{displaymath}
\log n\geq \max\frac{8}{c_2(1+c_2^{-1})^{2}(\mb EG(X) + c_1 + c_3 + 5)^2}\cdot\max\Big\{1, \log\frac{2}{c_2}\Big\}
\end{displaymath}
we have
\begin{displaymath}
\mb E_{\widehat{q}_\theta}\|\theta-\theta^\ast\|^2
\leq \Big[(1+c_2^{-1})^2\big(\mb EG(X) + c_3 + c_1 + 5\big)^2 + 1\Big]\cdot\frac{\log n}{n}.
\end{displaymath}
Additionally, for any $0 < \varepsilon\leq 1$, if
\begin{displaymath}
n\geq \Big(\frac{(1+c_2^{-1})^2\big(\mb EG(X) + c_3 + c_1 + 5\big)^2 + 1}{\varepsilon}\Big)^{2},
\end{displaymath}
we have $\mb E_{\widehat{q}_\theta}\|\theta-\theta^\ast\|^2
\leq \varepsilon$.
\end{corollary}
\begin{proof}
By taking $f_{x,z}(\theta) = \|\theta - \theta^\ast\|^2$ in Lemma \ref{lem: approx_to_delta_mass}.
\end{proof}

\begin{lemma}\label{lem: approx_to_delta_mass}
Consider a function $f_{x,z}(\theta)$. Assume there are functions $a_1(x), \cdots, a_m(x)$ such that
\begin{displaymath}
|f_{x,z}(\theta) - f_{x, z}(\theta^\ast)| \leq a_1(x)\|\theta-\theta^\ast\| + \cdots + a_m(x)\|\theta-\theta^\ast\|^m,
\end{displaymath}
and the number of sample size satisfies
\begin{align}\label{cond: sample_size1}
\begin{aligned}
n &\geq \max\bigg\{\exp\Big\{\frac{4m}{c_2(1+c_2^{-1})^{2}(\mb EG(X) + c_1 + c_3 + 5)^2}\Big\},\\ 
&\qquad\qquad\qquad\qquad\qquad \max_{1\leq k\leq m}\Big(\frac{k\Gamma(\frac{k}{2})}{2\cdot(\frac{c_2}{2})^{\frac{k}{2}}}\Big)^{\frac{8}{c_2(1+c_2^{-1})^2(\mb EG(X) + c_3 + c_1 + 5)^2}}\bigg\}
\end{aligned}
\end{align}
When the property of concentration (\ref{eqn: concentration_target_measure}) holds, 
\begin{displaymath}
\bigg|\int_{\Theta}f_{x,z}(\theta)\,\dd \widehat{q}_{\theta}(\theta) - f_{x,z}(\theta^\ast)\bigg|\leq \sum_{k=1}^m\big|a_k(x)\big|\Big[(1+c_2^{-1})^k\big(\mb EG(X) + c_3 + c_1 + 5\big)^k + 1\Big]\cdot\Big(\frac{\log n}{n}\Big)^{\frac{k}{2}}.
\end{displaymath}
\end{lemma}
\begin{proof}
Taking $r_k = (1+c_2^{-1})\big(\mb E[G(X)] + c_3 + c_1 + 5\big)\,\sqrt{\frac{\log n}{n}} \geq \sqrt{\frac{4k}{c_2n}}$ for $1\leq k\leq m$, we have
\begin{align*}
&\quad\, \int_{\Theta}\|\theta - \theta^\ast\|^k\,\dd \widehat{q}_{\theta}\\
&= \int_0^{\infty}kt^{k-1}\mb P_{\widehat{q}_{\theta}}\big(\|\theta - \theta^\ast\| > t\big)\,\dd t\\
&= \int_{0}^{r_k}kt^{k-1}\mb P_{\widehat{q}_{\theta}}\big(\|\theta-\theta^\ast\| > t\big)\,\dd t + \int_{r_k}^{\infty}kt^{k-1}\mb P_{\widehat{q}_{\theta}}\big(\|\theta - \theta^\ast\|>t\big)\,\dd t\\
&\leq \int_{0}^{r_k}kt^{k-1}\,\dd t + \int_{r_k}^{\infty}e^{-\frac{c_2nt^2}{2}}kt^{k-1}\,\dd t\\
&= r_k^k + \frac{k}{2}\int_{r_k^2}^{\infty}e^{-\frac{c_2n}{2}s}s^{\frac{k}{2}-1}\,\dd s\\
&\stackrel{\ri}{\leq} r_k^k + \frac{k}{2}\cdot \frac{\Gamma(\frac{k}{2})}{(\frac{c_2n}{2})^{k/2}}\cdot\exp\Big\{-\frac{(\frac{c_2n}{2}r_k^2 - \frac{k}{2})^2}{2\cdot\frac{c_2n}{2}r_k^2}\Big\}\\
&= r_k^k + \dfrac{k\Gamma(\frac{k}{2})}{2\cdot(\frac{c_2n}{2})^{k/2}}\cdot\exp\Big\{-\frac{(c_2nr_k^2 - k)^2}{4c_2nr_k^2}\Big\}\\
&\stackrel{\rii}{\leq}  (1+c_2^{-1})^k\big(\mb EG(X) + c_3 + c_1 + 5\big)^k\Big(\frac{\log n}{n}\Big)^{\frac{k}{2}} + \frac{k\Gamma(\frac{k}{2})}{2\cdot(\frac{c_2n}{2})^{\frac{k}{2}}}\cdot\exp\Big\{-\frac{c_2nr_k^2}{8}\Big\}\\
&{\leq} (1+c_2^{-1})^k\big(\mb EG(X) + c_3 + c_1 + 5\big)^k\Big(\frac{\log n}{n}\Big)^{\frac{k}{2}}\\
&\qquad\qquad\qquad + \frac{k\Gamma(\frac{k}{2})}{2\cdot(\frac{c_2n}{2})^{\frac{k}{2}}}\cdot\exp\Big\{-\frac{c_2(1+c_2^{-1})^2\big(\mb EG(X) + c_3 + c_1 + 5\big)^2\log n}{8}\Big\}\\
&= (1+c_2^{-1})^k\big(\mb EG(X) + c_3 + c_1 + 5\big)^k\cdot\Big(\frac{\log n}{n}\Big)^{\frac{k}{2}} + \frac{k\Gamma(\frac{k}{2})}{2\cdot(\frac{c_2}{2})^{\frac{k}{2}}}\cdot n^{-\frac{c_2(1+c_2^{-1})^2(\mb EG(X) + c_3 + c_1 + 5)^2}{8}-\frac{k}{2}}\\
&\stackrel{\riii}{\leq} \Big[(1+c_2^{-1})^k\big(\mb EG(X) + c_3 + c_1 + 5\big)^k + 1\Big]\cdot\Big(\frac{\log n}{n}\Big)^{\frac{k}{2}}.
\end{align*}
Here, (i) is by lemma \ref{lem: tail_of_gamma}, (ii) is by $c_2nr_k^2 \geq 4k$, and (iii) is by the condition of sample size (\ref{cond: sample_size1}). From the discussions above, 
\begin{align*}
\bigg|\int_\Theta f_{x,z}(\theta)\,\dd\widehat{q}_\theta - f_{x,z}(\theta^\ast)\bigg|
&\leq \int_\Theta\big|f_{x,z}(\theta) - f_{x,z}(\theta^\ast)\big|\,\dd\widehat{q}_\theta\\
&\leq \sum_{k=1}^m\big|a_k(x)\big|\int_\Theta\|\theta-\theta^\ast\|^k\,\dd\widehat{q}_\theta\\
&\leq \sum_{k=1}^m\big|a_k(x)\big|\Big[(1+c_2^{-1})^k\big(\mb EG(X) + c_3 + c_1 + 5\big)^k + 1\Big]\cdot\Big(\frac{\log n}{n}\Big)^{\frac{k}{2}}.
\end{align*}
\end{proof}

\begin{lemma}\label{lem: diff_log_integration}
Under Assumption \ref{assump: continuity_of_Hessian}, for $\mu, \nu\in \ms P_2(\Theta)$ such that the optimal map $t_\mu^\nu$ exists, we have
\begin{displaymath}
\bigg|\int_\Theta\log p(k\,|\,x,\theta)\,\dd(\nu-\mu) - \int_\Theta\langle\nabla\log p(k\,|\,x,\theta), t_{\mu}^\nu(\theta) - \theta\rangle\,\dd\mu\bigg| \leq \frac{\lambda(x)}{2}W_2^2(\mu, \nu).
\end{displaymath}
for all $k\in[K]$ and $x\in\mb R^d$.
\end{lemma}
\begin{proof}
By Taylor's expansion,
\begin{eqnarray*}
&\quad\,&\int_\Theta\log p(k\,|\,x,\theta)\,\dd(\nu - \mu)\\
&=& \int_\Theta\log p(k\,|\,x,\theta)\,\dd (t_\mu^\nu)_\#\mu - \int_\Theta\log p(k\,|\,x,\theta)\,\dd\mu\\
&=& \int_\Theta\log p(k\,|\,x, t_\mu^\nu(\theta)) - \log p(k\,|\,x,\theta)\,\dd\mu\\
&=& \int_\Theta\langle\nabla\log p(k\,|\,x,\theta), t_\mu^\nu(\theta) - \theta\rangle + \frac{1}{2}\int_\Theta\big\langle t_\mu^\nu(\theta) - \theta, \nabla^2\log p(k\,|\,x,\theta')(t_\mu^\nu(\theta)-\theta)\big\rangle\,\dd\mu
\end{eqnarray*}
for some $\theta'$ on the segment of $\theta$ and $t_\mu^\nu(\theta)$. Therefore
\begin{eqnarray*}
&\quad\,&
\bigg|\int_\Theta\log p(k\,|\,x,\theta)\,\dd(\nu-\mu) - \int_\Theta\langle\nabla\log p(k\,|\,x,\theta), t_{\mu}^\nu(\theta) - \theta\rangle\,\dd\mu\bigg|\\
&\leq&
\frac{1}{2}\bigg|\int_\Theta\big\langle t_\mu^\nu(\theta) - \theta, \nabla^2\log p(k\,|\,x,\theta')(t_\mu^\nu(\theta)-\theta)\big\rangle\,\dd\mu\bigg|\\
&\leq&\frac{1}{2}\int_\Theta \matnorm{\nabla^2\log p(k\,|\,x,\theta')}\|t_\mu^\nu(\theta) - \theta\|^2\,\dd\mu\\
&\leq&\frac{\lambda(x)}{2}\int_\Theta\|t_\mu^\nu(\theta)-\theta\|^2\,\dd\mu\\
&=&\frac{\lambda(x)}{2}W_2^2(\mu, \nu).
\end{eqnarray*}
\end{proof}

\begin{lemma}\label{lem: bound_diff_Phi}
Let $\mu, \nu\in\ms P_2^r(\Theta)$. Under Assumption \ref{assump: continuity_of_Hessian}, we have
\begin{align*}
&\quad\bigg|\Phi(\nu, x)(z) - \Phi(\mu, x)(z) - \Phi(\mu, x)(z)\sum_{k=1}^K\Phi(\mu, x)(k)\cdot\int_{\Theta}\Big\langle\nabla \log\frac{p(z\,|\,x,\theta)}{p(k\,|\,x,\theta)}, t_{\mu}^{\nu}(\theta)-\theta\Big\rangle\,\dd\mu(\theta)\bigg|\\
&\leq \Phi(\mu, x)(z)\lambda(x)W_2^2(\mu, \nu) + \frac{3}{2}\bigg(\sum_{k=1}^K\bigg|\int_\Theta\langle\nabla\log p(k\,|\,x,\theta), t_\mu^\nu(\theta) - \theta\rangle\,\dd\mu\bigg| + \frac{K\lambda(x)}{2}W_2^2(\mu, \nu)\bigg)^2.
\end{align*}
\end{lemma}
\begin{proof}
Let $A_z(\mu) = \int_\Theta\log p(z\,|\,x,\theta)\,\dd\mu(\theta)$, $h_z:\mb R^K\to[0, 1]$ be the function
\begin{displaymath}
h_z(x_1,\cdots, x_K) = \frac{e^{x_z}}{e^{x_1} + \cdots + e^{x_K}}
\end{displaymath}
for all $z\in[K]$, and $A(\mu) = (A_1(\mu)\cdots, A_K(\mu))\in\mb R^K$. By Taylor's expansion, there is $\eta\in\mb R^K$ such that
\begin{eqnarray}
\Phi(\nu, x)(z) - \Phi(\mu, x)(z)
&=& h_z(A(\nu)) - h_z(A(\mu))\nonumber\\
&=& \sum_{k=1}^K \frac{\partial h_z}{\partial x_k}\bigg|_{A(\mu)} (A_k(\nu) - A_k(\mu))\nonumber\\ 
&\qquad& + \frac{1}{2}\sum_{k,l=1}^K\frac{\partial^2h_z}{\partial x_k\partial x_l}\bigg|_\eta(A_k(\mu) - A_k(\mu))(A_l(\nu) - A_l(\mu)).\label{eqn: Taylor_Az}
\end{eqnarray}
It is easy to check that $\frac{\partial h_z}{\partial x_k} = \delta_{zk}h_z - h_zh_k$, and
\begin{displaymath}
\frac{\partial^2 h_z}{\partial x_k\partial x_l} = \delta_{kz}\delta_{zl}h_z - \delta_{kz}h_zh_l - \delta_{kl}h_zh_k + 2h_zh_kh_l - \delta_{zl}h_kh_z,
\end{displaymath}
which can be bounded by $3$ since $0\leq h_z, h_k, h_l \leq 1$. Notice that
\begin{align*}
&\quad\,\sum_{k=1}^K\frac{\partial h_z}{\partial x_k}\bigg|_{A(\mu)}(A_k(\nu) - A_k(\mu))\\
&=\sum_{k=1}^K \big(\delta_{zk}\Phi(\mu, x)(z) - \Phi(\mu, x)(z)\Phi(\mu, x)(k)\big)(A_k(\nu) - A_k(\mu))\\
&= \Phi(\mu, x)(z)(A_z(\nu) - A_z(\mu)) - \Phi(\mu, x)(z)\sum_{k=1}^K\Phi(\mu, x)(k)(A_k(\nu) - A_k(\mu))\\
&= \Phi(\mu, x)(z)\sum_{k=1}^K\Phi(\mu, x)(k)\big[\big(A_z(\nu) - A_z(\mu)\big) - \big(A_k(\nu) - A_k(\mu)\big)\big].
\end{align*}
By lemma \ref{lem: diff_log_integration}, 
\begin{equation}
\bigg|A_k(\nu) - A_k(\mu) - \int_\Theta\langle\nabla\log p(k\,|\,x,\theta), t_\mu^\nu(\theta) - \theta\rangle\,\dd\mu(\theta)\bigg|\leq \frac{\lambda(x)}{2}W_2^2(\mu, \nu).\label{inequal: diff_of_Ak}
\end{equation}
So
\begin{displaymath}
\bigg|\big[\big(A_z(\nu) - A_z(\mu)\big) - \big(A_k(\nu) - A_k(\mu)\big)\big] - \int_\Theta\bigg\langle\nabla\log\frac{p(z\,|\,x,\theta)}{p(k\,|\,x,\theta)}, t_\mu^\nu(\theta) - \theta\bigg\rangle\,\dd\mu\bigg| \leq \lambda(x)W_2^2(\mu, \nu).
\end{displaymath}
Therefore, we get
\begin{align*}
&\quad\bigg|\Phi(\nu, x)(z) - \Phi(\mu, x)(z) - \Phi(\mu, x)(z)\sum_{k=1}^K\Phi(\mu, x)(k)\cdot\int_{\Theta}\Big\langle\nabla \log\frac{p(z\,|\,x,\theta)}{p(k\,|\,x,\theta)}, t_{\mu}^{\nu}(\theta)-\theta\Big\rangle\,\dd\mu(\theta)\bigg|\\
&\stackrel{\ri}{=} \bigg|\Phi(\mu, x)(z)\sum_{k=1}^K\Phi(\mu, x)(k)\bigg[\big(A_z(\nu) - A_z(\mu)\big) - \big(A_k(\nu) - A_k(\mu)\big) - \int_\Theta\Big\langle\nabla \log\frac{p(z\,|\,x,\theta)}{p(k\,|\,x,\theta)}, t_{\mu}^{\nu}(\theta)-\theta\Big\rangle\,\dd\mu\bigg]\\
&\qquad + \frac{1}{2}\sum_{k,l=1}^K\frac{\partial^2h_z}{\partial x_k\partial x_l}\Big|_\eta(A_k(\nu) - A_k(\mu))(A_l(\nu) - A_l(\mu))\bigg|\\
&\leq \Phi(\mu, x)(z)\sum_{k=1}^K\Phi(\mu, x)(k)\cdot\lambda(x)W_2^2(\mu,\nu) + \frac{3}{2}\sum_{k,l=1}^K\big|(A_k(\nu) - A_k(\mu))(A_l(\nu) - A_l(\mu))\big|\\
&= \Phi(\mu, x)(z)\lambda(x)W_2^2(\mu, \nu) + \frac{3}{2}\bigg(\sum_{k=1}^K\big|A_k(\nu) - A_k(\mu)\big|\bigg)^2\\
&\stackrel{\rii}{\leq} \Phi(\mu, x)(z)\lambda(x)W_2^2(\mu, \nu) + \frac{3}{2}\bigg(\sum_{k=1}^K\bigg|\int_\Theta\langle\nabla\log p(k\,|\,x,\theta), t_\mu^\nu(\theta) - \theta\rangle\,\dd\mu\bigg| + \frac{K\lambda(x)}{2}W_2^2(\mu, \nu)\bigg)^2.
\end{align*}
Here, (i) is by Taylor expansion (\ref{eqn: Taylor_Az}), and (ii) is by triangular inequality and (\ref{inequal: diff_of_Ak}).
\end{proof}

\begin{corollary}\label{coro: subdifferential_is_variation}
Let $\mu\in\ms P_2^r(\Theta)$, and assume
\begin{displaymath}
\sum_{k=1}^K\big\|\nabla\log p(k\,|\,x,\theta)\big\|_{L^2(\mu)} < \infty.
\end{displaymath}
Then, under Assumption \ref{assump: continuity_of_Hessian}, we have
\begin{equation}\label{eqn: grad_first_order_variation}
\nabla\frac{\delta\Phi}{\delta\mu}(\mu, x)(z) = \Phi(\mu, x)(z)\sum_{k=1}^K\Phi(\mu, x)(k)\nabla\log\frac{p(z\,|\,x,\theta)}{p(k\,|\,x,\theta)}
\end{equation}
is a subdifferential of $\Phi(\mu, x)(z)$.
\end{corollary}
\begin{proof}
We shall first prove (\ref{eqn: grad_first_order_variation}) holds. By chain rule of functional derivative (formula A.38 in \cite{engel2013density}),
\begin{align*}
\frac{\delta\Phi}{\delta\mu}(\mu, x)(z)
&= \frac{\delta h_z}{\delta \mu}\bigg|_{A(\mu)}\\
&= \sum_{k=1}^K \frac{\partial h_z}{\partial x_k}\bigg|_{A(\mu)}\cdot\frac{\delta A_k}{\delta\mu}\\
&= \sum_{k=1}^K \big(\delta_{kz}\Phi(\mu, x)(z) - \Phi(\mu, x)(z)\Phi(\mu, x)(k)\big) \cdot\log p(k\,|\,x,\theta)\\
&= \Phi(\mu, x)(z)\log p(z\,|\,x,\theta) - \Phi(\mu, x)(z)\sum_{k=1}^K\Phi(\mu, x)(k)\log p(k\,|\,x,\theta)\\
&= \Phi(\mu, x)(z)\sum_{k=1}^K\Phi(\mu, x)(k)\log\frac{p(z\,|\,x,\theta)}{p(k\,|\,x,\theta)}.
\end{align*}
Taking gradient w.r.t. $\theta$ on both sides, we have shown (\ref{eqn: grad_first_order_variation}). Next, we will prove that it is a subdifferential. By the definition of subdifferential (\ref{eqn: def_subdifferential}), we only need to show
\begin{displaymath}
\Phi(\nu, x)(z) - \Phi(\mu, x)(z) \geq \int_\Theta\bigg\langle\nabla\frac{\delta\Phi}{\delta\mu}(\mu, x)(z), t_\mu^\nu(\theta) - \theta\bigg\rangle\,\dd\mu + o(W_2(\mu, \nu))
\end{displaymath}
for all $\nu\in\ms P_2(\mb R^d)$. By lemma \ref{lem: bound_diff_Phi},
\begin{align*}
&\quad\, \Phi(\nu, x)(z) - \Phi(\mu, x)(z) - \int_\Theta\bigg\langle\nabla\frac{\delta\Phi}{\delta\mu}(\mu, x)(z), t_\mu^\nu(\theta) - \theta\bigg\rangle\,\dd\mu\\
&= \Phi(\nu, x)(z) - \Phi(\mu, x)(z) - \Phi(\mu, x)(z)\sum_{k=1}^K\Phi(\mu, x)(k)\cdot\int_{\Theta}\Big\langle\nabla \log\frac{p(z\,|\,x,\theta)}{p(k\,|\,x,\theta)}, t_{\mu}^{\nu}(\theta)-\theta\Big\rangle\,\dd\mu\\
&\geq - \Phi(\mu, x)(z)\lambda(x)W_2^2(\mu, \nu) - \frac{3}{2}\bigg(\sum_{k=1}^K\bigg|\int_\Theta\langle\nabla\log p(k\,|\,x,\theta), t_\mu^\nu(\theta) - \theta\rangle\,\dd\mu\bigg| + \frac{K\lambda(x)}{2}W_2^2(\mu, \nu)\bigg)^2\\
&\stackrel{\ri}{\geq} -\lambda(x)W_2^2(\mu, \nu) - \frac{3}{2}\bigg(W_2(\mu,\nu)\sum_{k=1}^K\big\|\nabla\log p(k\,|\,x,\theta)\big\|_{L^2(\mu)} + \frac{K\lambda(x)}{2}W_2^2(\mu, \nu)\bigg)^2\\
&= -W_2^2(\mu, \nu)\bigg[\lambda(x) + \frac{3}{2}\bigg(\sum_{k=1}^K\big\|\nabla\log p(k\,|\,x,\theta)\big\|_{L^2(\mu)} + \frac{K\lambda(x)}{2}W_2(\mu, \nu)\bigg)^2\bigg].
\end{align*}
Here, (i) is by Cauchy--Schwarz inequality and the fact that $\Phi(\mu, x) \leq 1$.
\end{proof}

\begin{corollary}\label{coro: bound_R1}
For any $q\in\ms P_2^r(\Theta)$
\begin{align*}
&\quad\bigg|\Phi(q, x)(z) - \Phi(\widehat{q}_\theta, x)(z) - \int_{\Theta}\Big\langle\nabla \frac{\delta\Phi}{\delta\mu}(\widehat{q}_\theta, x)(z), t_{\widehat{q}_\theta}^{q}(\theta)-\theta\Big\rangle\,\dd \widehat{q}_\theta\bigg|\\
&\leq W_2^2(\widehat{q}_\theta, q)\bigg[\frac{3}{2}\bigg(\sum_{k=1}^K\big\|\nabla \log p(k\,|\,x,\theta^\ast)\big\|_2 + K\lambda(x)\sqrt{\mb E_{\widehat{q}_\theta}\|\theta-\theta^\ast\|^2}+\frac{K\lambda(x)}{2}W_2(\widehat{q}_\theta, q)\bigg)^2 + \lambda(x)\bigg].
\end{align*}
\end{corollary}
\begin{proof}
Just notice that
\begin{align*}
&\quad\, \bigg|\Phi(q, x)(z) - \Phi(\widehat{q}_\theta, x)(z) - \int_{\Theta}\Big\langle\nabla \frac{\delta\Phi}{\delta\mu}(\widehat{q}_\theta, x)(z), t_{\widehat{q}_\theta}^{q}(\theta)-\theta\Big\rangle\,\dd \widehat{q}_\theta\bigg|\\
&\stackrel{\ri}{=}\bigg|\Phi(q, x)(z) - \Phi(\widehat{q}_\theta, x)(z) - \Phi(\widehat{q}_\theta, x)(z)\sum_{k=1}^K\Phi(\widehat{q}_\theta, x)(k)\int_\Theta\bigg\langle\nabla\log\frac{p(z\,|\,x,\theta)}{p(k\,|\,x,\theta)}, t_{\widehat{q}_\theta}^q(\theta) - \theta\bigg\rangle\,\dd\widehat{q}_\theta\bigg|\\
&\stackrel{\rii}{\leq} \Phi(\widehat{q}_\theta, x)(z)\lambda(x)W_2^2(q, \widehat{q}_\theta) + \frac{3}{2}\bigg(\sum_{k=1}^K\bigg|\int_\Theta\big\langle\nabla\log p(k\,|\,x,\theta), t_{\widehat{q}_\theta}^q(\theta)-\theta\big\rangle\,\dd\widehat{q}_\theta\bigg| + \frac{K\lambda(x)}{2}W_2^2(q, \widehat{q}_\theta)\bigg)^2\\
&\stackrel{\riii}{\leq} W_2^2(\widehat{q}_\theta, q)\bigg[\frac{3}{2}\bigg(\sum_{k=1}^K\big\|\nabla \log p(k\,|\,x,\theta^\ast)\big\|_2 + K\lambda(x)\sqrt{\mb E_{\widehat{q}_\theta}\|\theta-\theta^\ast\|^2}+\frac{K\lambda(x)}{2}W_2(\widehat{q}_\theta, q)\bigg)^2 + \lambda(x)\bigg]
\end{align*}
Here, (i) is by corollary \ref{coro: subdifferential_is_variation}, (ii) is by lemma \ref{lem: bound_diff_Phi}, and (iii) is by lemma \ref{lem: change_to_Deltaq} and the fact that $\Phi(\widehat{q}_\theta, x)(z)\leq 1$.
\end{proof}

\begin{lemma}\label{lem: change_to_Deltaq}
For any $q\in\ms P_2(\Theta)$, under Assumption \ref{assump: continuity_of_Hessian},
\begin{displaymath}
\bigg|\int_\Theta\big\langle\nabla\log p(k\,|\,x,\theta), t_{\widehat{q}_\theta}^q(\theta)-\theta\big\rangle\,\dd\widehat{q}_\theta - \big\langle\nabla\log p(k\,|\,x,\theta^\ast), \Delta_q\big\rangle\bigg| \leq \lambda(x)W_2(q, \widehat{q}_\theta)\sqrt{\mb E_{\widehat{q}_\theta}\|\theta - \theta^\ast\|^2}.
\end{displaymath}
Here, for simplicity let
\begin{displaymath}
\Delta_q = \int_\Theta\theta\,\dd(q - \widehat{q}_\theta) = \int_\Theta t_{\widehat{q}_\theta}^q(\theta) - \theta\,\dd\widehat{q}_\theta.
\end{displaymath}
As a result, we can show that
\begin{displaymath}
\bigg|\int_\Theta\big\langle\nabla\log p(k\,|\,x,\theta), t_{\widehat{q}_\theta}^q(\theta)-\theta\big\rangle\,\dd\widehat{q}_\theta\bigg| \leq M_{1k}(x)W_2(q, \widehat{q}_\theta),
\end{displaymath}
where
\begin{displaymath}
M_{1k}(x) := \lambda(x)\sqrt{\mb E_{\widehat{q}_\theta}\|\theta-\theta^\ast\|^2} + \|\nabla\log p(k\,|\,x,\theta^\ast)\|.
\end{displaymath}
\end{lemma}
\begin{proof}
Notice that by Taylor's expansion, there is some $\theta'\in\Theta$ such that
\begin{align*}
&\quad\, \int_\Theta\big\langle\nabla\log p(k\,|\,x,\theta), t_{\widehat{q}_\theta}^q(\theta) - \theta\big\rangle\,\dd\widehat{q}_\theta\\
&= \int_\Theta\big\langle\nabla\log p(k\,|\,x,\theta^\ast), t_{\widehat{q}_\theta}^q(\theta)-\theta\big\rangle\,\dd\widehat{q}_\theta + \int_\Theta\big\langle\nabla\log p(k\,|\,x,\theta) - \log p(k\,|\,x,\theta^\ast), t_{\widehat{q}_\theta}^q(\theta)-\theta\big\rangle\,\dd\widehat{q}_\theta\\
&= \bigg\langle\nabla\log p(k\,|\,x,\theta^\ast), \int_\Theta t_{\widehat{q}_\theta}^q(\theta) -\theta \,\dd\widehat{q}_\theta\bigg\rangle + \int_\Theta\big\langle\nabla^2\log p(k\,|\,x,\theta')(\theta - \theta^\ast), t_{\widehat{q}_\theta}^q(\theta)-\theta\big\rangle\,\dd\widehat{q}_\theta\\
&= \big\langle\nabla\log p(k\,|\,x,\theta^\ast), \Delta_q\big\rangle + \int_\Theta\big\langle\nabla^2\log p(k\,|\,x,\theta')(\theta-\theta^\ast), t_{\widehat{q}_\theta}^q(\theta)-\theta\big\rangle\,\dd\widehat{q}_\theta.
\end{align*}
Therefore,
\begin{align*}
&\quad\, \bigg|\int_\Theta\big\langle\nabla\log p(k\,|\,x,\theta), t_{\widehat{q}_\theta}^q(\theta)-\theta\big\rangle\,\dd\widehat{q}_\theta - \big\langle\nabla\log p(k\,|\,x,\theta^\ast), \Delta_q\big\rangle\bigg|\\
&= \bigg|\int_\Theta\big\langle\nabla^2\log p(k\,|\,x,\theta')(\theta-\theta^\ast), t_{\widehat{q}_\theta}^q(\theta)-\theta\big\rangle\,\dd\widehat{q}_\theta\bigg|\\
&\leq \int_\Theta \|\nabla^2\log p(k\,|\,x,\theta')(\theta-\theta^\ast)\|\|t_{\widehat{q}_\theta}^q(\theta)-\theta\|\,\dd\widehat{q}_\theta\\
&\leq \int_\Theta\matnorm{\nabla^2\log p(k\,|\,x,\theta')}\|\theta-\theta^\ast\|\|t_{\widehat{q}_\theta}^q(\theta)-\theta\|\,\dd q_{\widehat{q}_\theta}\\
&\leq \lambda(x)\int_\Theta\|\theta-\theta^\ast\|\|t_{\widehat{q}_\theta}^q(\theta)-\theta\|\,\dd \widehat{q}_\theta\\
&\leq \lambda(x)\bigg(\int_\Theta\|\theta-\theta^\ast\|^2\,\dd \widehat{q}_\theta\bigg)^{\frac{1}{2}}\bigg(\int_\Theta\|t_{\widehat{q}_\theta}^q(\theta)-\theta\|^2\,\dd \widehat{q}_\theta\bigg)^{\frac{1}{2}}\\
&= \lambda(x)W_2(q, \widehat{q}_\theta)\sqrt{\mb E_{\widehat{q}_\theta}\|\theta - \theta^\ast\|^2}.
\end{align*}
By Cauchy--Schwarz inequality, 
\begin{align*}
\big\langle\nabla\log p(k\,|\,x,\theta^\ast), \Delta_q\big\rangle
&\leq \|\nabla\log p(k\,|\,x,\theta^\ast)\|\|\Delta_q\|\\
&\leq \|\nabla\log p(k\,|\,x,\theta^\ast)\|\cdot W_2(q, \widehat{q}_\theta).
\end{align*}
By triangular inequality, we have proved that
\begin{displaymath}
\bigg|\int_\Theta\big\langle\nabla\log p(k\,|\,x,\theta), t_{\widehat{q}_\theta}^q(\theta)-\theta\big\rangle\,\dd\widehat{q}_\theta\bigg| \leq M_{1k}(x)W_2(q, \widehat{q}_\theta).
\end{displaymath}
\end{proof}

\begin{corollary}\label{coro: bound_diff_log_integration}
For any $q\in\ms P_2(\Theta)$, under Assumption \ref{assump: continuity_of_Hessian},
\begin{displaymath}
\bigg|\int_\Theta\log p(k\,|\,x,\theta)\,\dd(q-\widehat{q}_\theta)\bigg| \leq M_{1k}(x)W_2(q, \widehat{q}_\theta) + \frac{\lambda(x)}{2}W_2^2(q,\widehat{q}_\theta),
\end{displaymath}
where $M_{1k}(x)$ is defined in lemma \ref{lem: change_to_Deltaq}.
\begin{proof}
\begin{align*}
\bigg|\int_\Theta\log p(k\,|\,x,\theta)\,\dd(q-\widehat{q}_\theta)\bigg|
&\leq \bigg|\int_\Theta\langle\nabla\log p(k\,|\,x,\theta), t_{\widehat{q}_\theta}^q(\theta) - \theta\rangle\,\dd\widehat{q}_\theta\bigg| + \frac{\lambda(x)}{2}W_2^2(q,\widehat{q}_\theta)\\
&\leq M_{1k}(x)W_2(q, \widehat{q}_\theta) + \frac{\lambda(x)}{2}W_2^2(q,\widehat{q}_\theta).
\end{align*}
Here, the first step is by lemma \ref{lem: diff_log_integration}, and the second step is by lemma \ref{lem: change_to_Deltaq}.
\end{proof}
\end{corollary}

\begin{lemma}\label{lem: missing_data_info}
Recall that $\Delta_q = \int_\Theta\theta\,\dd(q - \widehat{q}_\theta)$. Let $\wht I_S(\theta^\ast)$ be the sample missing data information, i.e.
\begin{equation}
\wht I_S(\theta^\ast) = \frac{1}{n}\sum_{i=1}^n\sum_{z=1}^K p(z\,|\,X_i,\theta^\ast)\big[\nabla\log p(z\,|\,X_i,\theta^\ast)\big]\big[\nabla\log p(z\,|\,X_i,\theta^\ast)\big]^T.
\end{equation}
Then, under Assumption \ref{assump: continuity_of_Hessian} we have
\begin{align*}
&\quad\,\bigg|\big\langle\Delta_\mu, \wht I_S(\theta^\ast)\Delta_q\big\rangle - \frac{1}{n}\sum_{i=1}^n\sum_{z,l=1}^K \Phi(\widehat{q}_\theta, X_i)(z)\Phi(\widehat{q}_\theta, X_i)(l)\\
&\qquad\qquad\qquad\qquad\qquad\cdot 
\int_\Theta\Big\langle\nabla\log\frac{p(z\,|\,X_i,\theta)}{p(l\,|\,X_i,\theta)}, t_{\widehat{q}_\theta}^q(\theta) - \theta\Big\rangle\dd\widehat{q}_\theta(\theta)\\
&\qquad\qquad\qquad\qquad\qquad\cdot\int_\Theta\big\langle\nabla\log p(z\,|\,X_i,\theta), t_{\widehat{q}_\theta}^\mu(\theta)-\theta\big\rangle\,\dd\widehat{q}_\theta(\theta) \bigg|\\
&\leq 2KW_2(q, \widehat{q}_\theta)W_2(\mu, \widehat{q}_\theta)\cdot\frac{1}{n}\sum_{i=1}^n\bigg[2\sqrt{\mb E_{\widehat{q}_\theta}\|\theta-\theta^\ast\|^2}S_1(X_i)\lambda(X_i)\\
&\qquad\qquad\qquad\qquad\qquad\qquad\qquad+ 
K\mb E_{\widehat{q}_\theta}\|\theta-\theta^\ast\|^2\lambda(X_i)^2\\
&\qquad\qquad\qquad\qquad\qquad\qquad\qquad+
2S_2(X_i)\Big(S_1(x) \sqrt{\mb E_{\widehat{q}_\theta}\|\theta-\theta^\ast\|^2} + \frac{K\lambda(x)}{2}\cdot \mb E_{\widehat{q}_\theta}\|\theta-\theta^\ast\|^2\Big)\bigg].
\end{align*}
\end{lemma}
\begin{proof}
For simplicity, let
\begin{align*}
M_\mu^n(z)(x) &= \big\langle\nabla\log p(z\,|\,x,\theta^\ast), \Delta_\mu\big\rangle\\
R_\mu^n(z)(x) &= \int_\Theta\big\langle\nabla\log p(z\,|\,x,\theta), t_{\widehat{q}_\theta}^\mu(\theta)-\theta\big\rangle\,\dd\widehat{q}_\theta - M_\mu^n(z)(x).
\end{align*}
Similarly, let
\begin{align*}
M_q^n(z, l)(x) &= \Big\langle\nabla\log\frac{p(z\,|\,x,\theta^\ast)}{p(l\,|\,x,\theta^\ast)}, \Delta_q\Big\rangle\\
R_q^n(z, l)(x) &= \int_\Theta\Big\langle\nabla\log\frac{p(z\,|\,x,\theta)}{p(l\,|\,x,\theta)}, t_{\widehat{q}_\theta}^q(\theta)-\theta\Big\rangle\,\dd\widehat{q}_\theta - M_q^n(z,l)(x).
\end{align*}
Also, write
\begin{align*}
R^n_z(x) &= \Phi(\widehat{q}_\theta, x)(z) - p(z\,|\,x,\theta^\ast)
\end{align*}
for $z = 1, 2, \cdots, K$. So
\begin{align*}
&\quad\,
\bigg|\frac{1}{n}\sum_{i=1}^n\sum_{z,l=1}^K \Phi(\widehat{q}_\theta, X_i)(z)\Phi(\widehat{q}_\theta, X_i)(l) \cdot 
\int_\Theta\Big\langle\nabla\log\frac{p(z\,|\,X_i,\theta)}{p(l\,|\,X_i,\theta)}, t_{\widehat{q}_\theta}^q(\theta) - \theta\Big\rangle\dd\widehat{q}_\theta(\theta)\\
&\qquad\qquad\qquad\qquad\qquad\qquad\qquad\qquad\qquad\cdot\int_\Theta\big\langle\nabla\log p(z\,|\,X_i,\theta), t_{\widehat{q}_\theta}^\mu(\theta)-\theta\big\rangle\,\dd\widehat{q}_\theta(\theta)\\
&\qquad\,
- \frac{1}{n}\sum_{i=1}^n\sum_{z,l=1}^K p(z\,|\,X_i,\theta^\ast)\cdot p(l\,|\,X_i,\theta^\ast)\cdot \Big\langle\nabla\log\frac{p(z\,|\,x,\theta^\ast)}{p(l\,|\,x,\theta^\ast)}, \Delta_q\Big\rangle\\ 
&\qquad\qquad\qquad\qquad\qquad\qquad\qquad\qquad\qquad\qquad\qquad
\cdot\big\langle\nabla\log p(z\,|\,x,\theta^\ast), \Delta_\mu\big\rangle\bigg|\\
&= \bigg|\frac{1}{n}\sum_{i=1}^n\sum_{z,l=1}^K \big(R_z^n(X_i) + p(z\,|\,X_i, \theta^\ast)\big)\big(R_l^n(X_i) + p(l\,|\,X_i, \theta^\ast)\big)\\
&\qquad\qquad\qquad\qquad\qquad
\cdot \big(M_q^n(z,l)(X_i) + R_q^n(z,l)(X_i)\big)\big(M_\mu^n(z)(X_i) + R_\mu^n(z)(X_i)\big)\\
&\qquad
- \frac{1}{n}\sum_{i=1}^n\sum_{z,l=1}^K p(z\,|\,X_i,\theta^\ast)\cdot p(l\,|\,X_i,\theta^\ast)\cdot M_q^n(z,l)(X_i) M_\mu^n(z)(X_i)\bigg|\\
&= \bigg|\frac{1}{n}\sum_{i=1}^n\sum_{z,l=1}^K R_q^n(z,l)(X_i)\big(M_\mu^n(z)(X_i) + R_\mu^n(z)(X_i)\big)\Phi(\widehat{q}_\theta, X_i)(l)\Phi(\widehat{q}_\theta, X_i)(z)\\
&\qquad + \frac{1}{n}\sum_{i=1}^n\sum_{z,l=1}^KM_q^n(z,l)(X_i)R_\mu^n(z)(X_i)\Phi(\widehat{q}_\theta, X_i)(z)\Phi(\widehat{q}_\theta, X_i)(l)\\
&\qquad + \frac{1}{n}\sum_{i=1}^n\sum_{z,l=1}^KM_q^n(z,l)(X_i)M_\mu^n(z)(X_i)R_z^n(X_i)\Phi(\widehat{q}_\theta, X_i)(l)\\
&\qquad + \frac{1}{n}\sum_{i=1}^n\sum_{z,l=1}^KM_q^n(z,l)(X_i)M_\mu^n(z)(X_i)p(z\,|\,X_i, \theta^\ast)R_l^n(X_i) \bigg|\\
&=: \bigg|\frac{1}{n}(I_1 + I_2 + I_3 + I_4)\bigg|.
\end{align*}
We will bound $I_1, \cdots, I_4$ separately. For remainder terms, by Lemma \ref{lem: change_to_Deltaq},
\begin{displaymath}
\big|R_\mu^n(z)(x)\big| \leq \lambda(x)W_2(\mu, \widehat{q}_\theta)\sqrt{\mb E_{\widehat{q}_\theta}\|\theta-\theta^\ast\|^2},
\end{displaymath}
and
\begin{displaymath}
\big|R_q^n(z, l)(x)\big| \leq 2\lambda(x)W_2(q,\widehat{q}_\theta)\sqrt{\mb E_{\widehat{q}_\theta}\|\theta-\theta^\ast\|^2}.
\end{displaymath}
By lemma \ref{lem: diff_Phimu_thetastar}, 
\begin{align*}
|R_z^n(x)| 
&= \big|\Phi(\widehat{q}_\theta, x)(z) - \Phi(\delta_{\theta^\ast}, x)(z)\big|\\
&\leq S_1(x)W_2(\widehat{q}_\theta, \delta_{\theta^\ast}) + \frac{K\lambda(x)}{2}\cdot W_2^2(\widehat{q}_\theta, \delta_{\theta^\ast})\\
&= S_1(x) \sqrt{\mb E_{\widehat{q}_\theta}\|\theta-\theta^\ast\|^2} + \frac{K\lambda(x)}{2}\cdot \mb E_{\widehat{q}_\theta}\|\theta-\theta^\ast\|^2.
\end{align*}
For leading terms, by Cauchy--Schwarz inequality,
\begin{displaymath}
\big|M_\mu^n(z)(x)\big| = \big|\big\langle\nabla\log p(z\,|\,x,\theta^\ast), \Delta_\mu\big\rangle\big| \leq \|\nabla\log p(z\,|\,x,\theta^\ast)\|\cdot W_2(\mu, \widehat{q}_\theta),
\end{displaymath}
and
\begin{displaymath}
\big|M_q^n(z, l)(x)\big| \leq \Big(\|\nabla\log p(z\,|\,x,\theta^\ast)\| + \|\nabla\log p(l\,|\,x,\theta^\ast)\|\Big)\cdot W_2(\mu, \widehat{q}_\theta).
\end{displaymath}
By lemma \ref{lem: change_to_Deltaq}
\begin{displaymath}
\big|M_\mu^n(z)(x) + R_\mu^n(z)(x)\big| = \bigg|\int_\Theta\big\langle\nabla\log p(z\,|\,x,\theta), t_{\widehat{q}_\theta}^\mu(\theta)-\theta\big\rangle\,\dd\widehat{q}_\theta\bigg| \leq M_{1z}(x)W_2(\mu, \widehat{q}_\theta).
\end{displaymath}
Therefore, we can derive the bounds
\begin{align*}
    |I_1| &\leq \sum_{i=1}^n\sum_{z,l=1}^K 2\lambda(X_i)W_2(q, \widehat{q}_\theta)\sqrt{\mb E_{\widehat{q}_\theta}\|\theta-\theta^\ast\|^2} \cdot M_{1z}(X_i)W_2(\mu, \widehat{q}_\theta)\cdot1\cdot 1\\
    &= 2KW_2(q, \widehat{q}_\theta)W_2(\mu, \widehat{q}_\theta)\!\!\sum_{i=1}^n\!\Big(\!K\lambda(X_i)^2\mb E_{\widehat{q}_\theta}\|\theta-\theta^\ast\|^2 \!+\! S_1(X_i)\lambda(X_i)\sqrt{E_{\widehat{q}_\theta}\|\theta-\theta^\ast\|^2}\Big);\\
    |I_2| &\leq \sum_{i=1}^n\sum_{z,l=1}^K \Big(\|\nabla\log p(z\,|\,X_i,\theta^\ast)\| + \|\nabla\log p(l\,|\,X_i,\theta^\ast)\|\Big)W_2(\mu, \widehat{q}_\theta)\\
    &\qquad\qquad\qquad\qquad
    \cdot \lambda(X_i)W_2(\mu, \widehat{q}_\theta)\sqrt{\mb E_{\widehat{q}_\theta}\|\theta-\theta^\ast\|^2}\cdot 1\cdot 1\\
    &= 2KW_2(q, \widehat{q}_\theta)W_2(\mu, \widehat{q}_\theta)\sqrt{\mb E_{\widehat{q}_\theta}\|\theta-\theta^\ast\|^2}\sum_{i=1}^n\lambda(X_i)S_1(X_i);\\
    |I_3| &\leq \sum_{i=1}^n\sum_{z,l=1}^K \Big(\|\nabla\log p(z\,|\,X_i,\theta^\ast)\| + \|\nabla\log p(l\,|\,X_i,\theta^\ast)\|\Big)W_2(\mu, \widehat{q}_\theta)\\
    &\qquad\qquad\qquad\qquad
    \cdot \|\nabla\log p(z\,|\,X_i,\theta^\ast)\| W_2(\mu, \widehat{q}_\theta) \cdot \Big(S_1(X_i) \sqrt{\mb E_{\widehat{q}_\theta}\|\theta-\theta^\ast\|^2} + \frac{K\lambda(X_i)}{2}\cdot \mb E_{\widehat{q}_\theta}\|\theta-\theta^\ast\|^2\Big)\cdot 1\\
    &= W_2(q, \widehat{q}_\theta)W_2(\mu, \widehat{q}_\theta)\sum_{i=1}^n\Big(S_1(X_i) \sqrt{\mb E_{\widehat{q}_\theta}\|\theta-\theta^\ast\|^2} + \frac{K\lambda(X_i)}{2}\cdot \mb E_{\widehat{q}_\theta}\|\theta-\theta^\ast\|^2\Big)\Big(KS_2(X_i) + S_1(X_i)^2\Big)\\
    &\leq 2KW_2(q, \widehat{q}_\theta)W_2(\mu, \widehat{q}_\theta)\sum_{i=1}^nS_2(X_i)\Big(S_1(X_i) \sqrt{\mb E_{\widehat{q}_\theta}\|\theta-\theta^\ast\|^2} + \frac{K\lambda(X_i)}{2}\cdot \mb E_{\widehat{q}_\theta}\|\theta-\theta^\ast\|^2\Big);\\
    |I_4| &\leq \sum_{i=1}^n\sum_{z,l=1}^K \Big(\|\nabla\log p(z\,|\,X_i,\theta^\ast)\| + \|\nabla\log p(l\,|\,X_i,\theta^\ast)\|\Big)W_2(\mu, \widehat{q}_\theta)\\
    &\qquad\qquad\qquad\qquad
    \cdot \|\nabla\log p(z\,|\,X_i,\theta^\ast)\| W_2(\mu, \widehat{q}_\theta) \cdot 1 \cdot \Big(S_1(X_i) \sqrt{\mb E_{\widehat{q}_\theta}\|\theta-\theta^\ast\|^2} + \frac{K\lambda(X_i)}{2}\cdot \mb E_{\widehat{q}_\theta}\|\theta-\theta^\ast\|^2\Big)\\
    &= W_2(q, \widehat{q}_\theta)W_2(\mu, \widehat{q}_\theta)\sum_{i=1}^n\Big(S_1(X_i) \sqrt{\mb E_{\widehat{q}_\theta}\|\theta-\theta^\ast\|^2} + \frac{K\lambda(X_i)}{2}\cdot \mb E_{\widehat{q}_\theta}\|\theta-\theta^\ast\|^2\Big)\Big(KS_2(X_i) + S_1(X_i)^2\Big)\\
    &\leq 2KW_2(q, \widehat{q}_\theta)W_2(\mu, \widehat{q}_\theta)\sum_{i=1}^nS_2(X_i)\Big(S_1(X_i) \sqrt{\mb E_{\widehat{q}_\theta}\|\theta-\theta^\ast\|^2} + \frac{K\lambda(X_i)}{2}\cdot \mb E_{\widehat{q}_\theta}\|\theta-\theta^\ast\|^2\Big).
\end{align*}
Notice that
\begin{align*}
&\quad\,
\frac{1}{n}\sum_{i=1}^n\sum_{z,l=1}^K p(z\,|\,X_i,\theta^\ast)\cdot p(l\,|\,X_i,\theta^\ast)\cdot \big\langle\nabla\log p(z\,|\,X_i,\theta^\ast), \Delta_q\big\rangle
\cdot\big\langle\nabla\log p(z\,|\,X_i,\theta^\ast), \Delta_\mu\big\rangle
\\
&= \Delta_q^T \frac{1}{n}\sum_{i=1}^n \Big(\sum_{z=1}^K p(z\,|\,X_i,\theta^\ast)\nabla\log p(z\,|\,X_i,\theta^\ast)\Big)\Big(\sum_{l=1}^K p(l\,|\,X_i,\theta^\ast)\nabla\log p(l\,|\,X_i,\theta^\ast)\Big)^T
\Delta_\mu\\
&= \Delta_q^T \frac{1}{n}\sum_{i=1}^n \Big(\sum_{z=1}^K \nabla p(z\,|\,X_i,\theta^\ast)\Big)\Big(\sum_{l=1}^K \nabla p(l\,|\,X_i,\theta^\ast)\Big)^T \Delta_\mu\\
&= 0.
\end{align*}
So
\begin{align*}
&\quad\,
\frac{1}{n}\sum_{i=1}^n\sum_{z,l=1}^K p(z\,|\,X_i,\theta^\ast)\, p(l\,|\,X_i,\theta^\ast)\cdot \Big\langle\nabla\log\frac{p(z\,|\,X_i,\theta^\ast)}{p(l\,|\,X_i,\theta^\ast)}, \Delta_q\Big\rangle
\cdot\big\langle\nabla\log p(z\,|\,X_i,\theta^\ast), \Delta_\mu\big\rangle\\
&= \frac{1}{n}\sum_{i=1}^n\sum_{z,l=1}^K p(z\,|\,X_i,\theta^\ast)\, p(l\,|\,X_i,\theta^\ast)\cdot \big\langle\nabla\log p(z\,|\,X_i,\theta^\ast), \Delta_q\big\rangle
\cdot\big\langle\nabla\log p(z\,|\,X_i,\theta^\ast), \Delta_\mu\big\rangle\\
&= \frac{1}{n}\sum_{i=1}^n\sum_{z=1}^K p(z\,|\,X_i,\theta^\ast)\Delta_q^T\big[\nabla\log p(z\,|\,X_i,\theta^\ast)\big]\big[\nabla\log p(z\,|\,X_i,\theta^\ast)\big]^T\Delta_\mu\\
&= \big\langle\Delta_\mu, \wht I_S(\theta^\ast)\Delta_q\big\rangle. 
\end{align*}
As a conclusion, we have shown that
\begin{align*}
&\quad\,\bigg|\big\langle\Delta_\mu, \wht I_S(\theta^\ast)\Delta_q\big\rangle - \frac{1}{n}\sum_{i=1}^n\sum_{z,l=1}^K \Phi(\widehat{q}_\theta, X_i)(z)\Phi(\widehat{q}_\theta, X_i)(l)\\
&\qquad\qquad\qquad\qquad\qquad\cdot 
\int_\Theta\Big\langle\nabla\theta\log\frac{p(z\,|\,X_i,\theta)}{p(l\,|\,X_i,\theta)}, t_{\widehat{q}_\theta}^q(\theta) - \theta\Big\rangle\dd\widehat{q}_\theta(\theta)\\
&\qquad\qquad\qquad\qquad\qquad\cdot\int_\Theta\big\langle\nabla\log p(z\,|\,X_i,\theta), t_{\widehat{q}_\theta}^\mu(\theta)-\theta\big\rangle\,\dd\widehat{q}_\theta(\theta) \bigg|\\
&\leq 2KW_2(q, \widehat{q}_\theta)W_2(\mu, \widehat{q}_\theta)\cdot\frac{1}{n}\sum_{i=1}^n\bigg[2\sqrt{\mb E_{\widehat{q}_\theta}\|\theta-\theta^\ast\|^2}S_1(X_i)\lambda(X_i)\\
&\qquad\qquad\qquad\qquad\qquad\qquad\qquad+ 
K\mb E_{\widehat{q}_\theta}\|\theta-\theta^\ast\|^2\lambda(X_i)^2\\
&\qquad\qquad\qquad\qquad\qquad\qquad\qquad+
2S_2(X_i)\Big(S_1(X_i) \sqrt{\mb E_{\widehat{q}_\theta}\|\theta-\theta^\ast\|^2} + \frac{K\lambda(X_i)}{2}\cdot \mb E_{\widehat{q}_\theta}\|\theta-\theta^\ast\|^2\Big)\bigg].
\end{align*}
\end{proof}

\subsection{Proof of Lemma~\ref{lem: functional_Wn}}\label{app: proof_C.1}
Just note that
\begin{align*}
&\quad\,\KL(\rho_1\otimes\cdots\otimes\rho_m\,\|\,\pi_n)\\
&= \int_\Theta\log\frac{\rho_1(\theta_1)\cdots\rho_m(\theta_m)}{C(\pi_\theta, X^n)\pi_\theta(\theta)\prod_{i=1}^np(X_i\,|\,\theta)}\,\dd\rho_1(\theta_1)\cdots\dd\rho_m(\theta_m)\\
&= -\int_\Theta \sum_{i=1}^n\log p(X_i\,|\,\theta)\,\dd\rho_1(\theta_1)\cdots\dd\rho_m(\theta_m) + \KL(\rho_1\otimes\cdots\otimes\rho_m\,\|\,\Pi_\theta) + C(\pi_\theta, X^n)\\
&= \int_\Theta \sum_{i=1}^n \log\frac{p(X_i\,|\,\theta^\ast)}{p(X_i\,|\,\theta)}\,\dd\rho_1(\theta_1)\cdots\dd\rho_m(\theta_m) + \KL(\rho_1\otimes\cdots\otimes\rho_m\,|\,\Pi_\theta) + C(\pi_\theta, X^n)\\
&= \widetilde W_n(\rho_1, \cdots, \rho_m) + C(\pi_\theta, X^n).
\end{align*}
Here, $C(\pi_\theta, X^n)$ is a constant that does not depend on $\rho_1, \cdots, \rho_m$ varying from line to line. By definition, we have $(\wht q_1, \cdots, \wht q_m) = \argmin\widetilde W_n(\rho_1, \cdots, \rho_m)$.

To show the second part, applying the first-order optimality condition of $\widetilde W_n$ with respect to $\wht q_j$ yields that
\begin{align*}
-\int_{\Theta_{-j}} \sum_{i=1}^n \log p(X_i\,|\,\theta) + \log \pi_\theta(\theta)\,\dd\wht q_{-j}(\theta_{-j}) + \log\wht q_j(\theta_j)
\end{align*}
is a constant a.e. over the support $\{\wht q_j > 0\}$. This implies the desired result~\eqref{eqn: FOC_qj}.

\subsection{Proof of Lemma~\ref{lem: prob_antilde}}\label{app: proof_C.2}
Note that
\begin{align*}
-\mb E_{\theta^\ast} \int_{\widetilde\Theta}\sum_{i=1}^n \log\frac{p(X_i\,|\,\theta)}{p(X_i\,|\,\theta^\ast)}\,\dd\widetilde Q 
&= n\int_{\widetilde\Theta}\!\int_{\m X}\log\frac{p(x\,|\,\theta^\ast)}{p(x\,|\,\theta)}p(x\,|\,\theta)\,\dd x\dd\widetilde Q\\
&= n\int_{\widetilde \Theta}\KL\big(p(\cdot\,|\,\theta^\ast)\,\|\,p(\cdot\,|\,\theta)\big)\,\dd\widetilde Q\\
&\leq \tilde c_4n\varepsilon_n^2 \widetilde Q(\widetilde\Theta).
\end{align*}
Here, the last inequality is due to Assumption~\ref{assump: prior condition}. There, we obtain that
\begin{align*}
\mb P_{\theta^\ast}(\widetilde{\m A}_n)
&= \mb P_{\theta^\ast}\bigg(\frac{1}{\widetilde Q(\widetilde\Theta)}\int_{\widetilde \Theta}\sum_{i=1}^n\log\frac{p(X_i\,|\,\theta)}{p(X_i\,|\,\theta^\ast)}\,\dd\widetilde Q \leq -(\tilde c_4 + 1)n\varepsilon_n^2\bigg)\\
&\leq \mb P_{\theta^\ast}\bigg(\frac{1}{\widetilde Q(\widetilde\Theta)}\int_{\widetilde \Theta}\sum_{i=1}^n\log\frac{p(X_i\,|\,\theta)}{p(X_i\,|\,\theta^\ast)}\,\dd\widetilde Q - \mb E_{\theta^\ast}\frac{1}{\widetilde Q(\widetilde\Theta)}\int_{\widetilde \Theta}\sum_{i=1}^n\log\frac{p(X_i\,|\,\theta)}{p(X_i\,|\,\theta^\ast)}\,\dd\widetilde Q \leq -n\varepsilon_n^2\bigg)\\
&\stackrel{\ri}{\leq} \frac{1}{n^2\varepsilon_n^4}\mb V\bigg(\frac{1}{\widetilde Q(\widetilde\Theta)}\int_{\widetilde \Theta}\sum_{i=1}^n\log\frac{p(X_i\,|\,\theta)}{p(X_i\,|\,\theta^\ast)}\,\dd\widetilde Q\bigg)
= \frac{1}{n\varepsilon_n^4}\mb V\bigg(\frac{1}{\widetilde Q(\widetilde\Theta)}\int_{\widetilde \Theta}\log\frac{p(X\,|\,\theta)}{p(X\,|\,\theta^\ast)}\,\dd\widetilde Q\bigg)\\
&\leq \frac{1}{n\varepsilon_n^4}\cdot\frac{1}{\widetilde Q(\widetilde\Theta)}\int_{\widetilde\Theta}\!\int_{\m X}\Big(\log\frac{p(x\,|\,\theta^\ast)}{p(x\,|\,\theta)}\Big)^2 p(x\,|\,\theta^\ast)\,\dd x\dd\widetilde Q\\
&\stackrel{\rii}{\leq} \frac{\tilde c_4}{n\varepsilon_n^2}.
\end{align*}
Here, (i) is by Chebyshev's inequality; (ii) is due to Assumption~\ref{assump: prior condition}.

\subsection{Proof of Lemma~\ref{lem:functional_Vn}}\label{sec:proof_lem:functional_Vn}

We first show that
\begin{equation}\label{eqn: is_minimizer}
    \big(\,\widehat{q}_\theta(\cdot), \,\Phi(\widehat{q}_\theta, x)(\cdot)\,\big) = \argmin_{\rho, F_x} W_n(\rho, F_x).
\end{equation}
Note that for a fixed $\rho$, we can equivalently write the minimization problem of $\argmin_{F_x}W_n(\rho, F_x)$ by adding or ignoring some additive terms independent of $F_x$ as
\begin{eqnarray*}
\argmin_{F_x}W_n(\rho, F_x)
&=& \argmin_{F_x} \int_\Theta\sum_{i=1}^n\sum_{z=1}^K\log \frac{F_{X_i}(z)}{p(z\,|\,X_i,\theta)}F_{X_i}(z)\,\rho(\dd\theta)\\
&=& \argmin_{F_x} \sum_{i=1}^n\sum_{z=1}^KF_{X_i}(z)\log F_{X_i}(z) - \int_\Theta F_{X_i}(z)\log p(z\,|\,X_i,\theta)\,\rho(\dd\theta)\\
&=& \argmin_{F_x} \sum_{i=1}^n\sum_{z=1}^KF_{X_i}(z)\log\frac{F_{X_i}(z)}{\exp\{\int_\Theta\log p(z\,|\,X_i,\theta)\,\rho(\dd\theta)\}}\\
&=& \argmin_{F_x} \sum_{i=1}^n\sum_{z=1}^K F_{X_i}(z)\log\frac{F_{X_i}(z)}{\Phi(\rho, X_i)(z)}\\
&=& \argmin_{F_x} \sum_{i=1}^n D_{KL}(F_{X_i}\,\|\,\Phi(\rho, X_i)(\cdot)).
\end{eqnarray*}
This implies that given $\rho$, the minimizer of $W_n(\rho, F_x)$ is $F_x(z) = \Phi(\rho, x)(z)$. By plugging-in $F_x$ with $\Phi(\rho, x)(z)$ back into $W_n(\rho, F_x)$, it is straightforward to verify that the resulting $\min_{F_x}W_n(\rho, F_x)$ is up to a $q_\theta$-independent constant the same as the $W_n(\rho)$ functional defined in~\eqref{eqn:profile_KL}, or the (profile) objective functional $q_\theta\mapsto \min_{q_{Z^n}}D_{KL}\big(q_{\theta}\otimes q_{Z^n}\,\big\|\,\pi_n\big)$ after $q_{Z^n}$ being maxed out. As a consequence, $\wht q_\theta$ minimizes $\min_{F_x}W_n(\rho, F_x)$. Putting pieces together, we proved the first part of the lemma, that is, $\big(\wht q_\theta,\, \Phi(\wht q_\theta, \,x)\big)$ is a minimizer to functional $W_n(\rho, F_x)$.

To show the second part, notice that since $\big(\wht q_\theta,\, \Phi(\wht q_\theta, \,x)\big)$ jointly minimizes the functional $W_n(\rho, F_x)$, $\rho=\wht q_\theta$ should minimize the functional $W_n\big(\rho,\,\Phi(\wht q_\theta, \,x)\big)$ when $F_x$ is replaced by its optimum. Since, up to a $\rho$ independent constant, functional $W_n\big(\rho,\,F_x\big)$ for any $F_x$ is equivalent to 
\begin{align*}
    -\int_\Theta\sum_{i=1}^n\sum_{z=1}^KF_{X_i}(z)\log p(z,X_i\,|\,\theta)\,\rho(\dd\theta) + D_{KL}(\rho\,\|\,\pi_\theta).
\end{align*}
By applying the first order optimality condition, its first-order variation
\begin{equation*}
    -\sum_{i=1}^n\sum_{z=1}^KF_{X_i}(z)\log p(z,X_i\,|\,\theta) + 1 + \log\rho(\theta) - \log\pi_\theta(\theta)
\end{equation*}
should be a constant at its minimizer $\rho$ (proposition 7.20 in \cite{santambrogio2015optimal}) a.e.~over the support $\{\rho>0\}$, implying that
\begin{eqnarray*}
\rho(\theta)
\ \propto\  \pi_\theta(\theta)\, \exp\Big\{\sum_{i=1}^n\sum_{z=1}^KF_{X_i}(z)\log p(z,X_i\,|\,\theta)\Big\}, \quad \theta\in\Theta.
\end{eqnarray*}
In particular, the second part of the lemma can be proved by replacing $\rho$ with $\wht q_\theta$ and $F_x$ with $\Phi(\wht q_\theta,x)$ in the preceding display.

\subsection{Proof of Lemma~\ref{lemma:A_n_bound}}\label{sec:proof_lemma:A_n_bound}
Let $\tilde{\pi}_n = \frac{\pi_\theta 1_{\widetilde\Theta}}{\Pi(\widetilde\Theta)}$ be the prior density restricted to set $\widetilde\Theta\subset\Theta$. By Assumption~\ref{assump: qudratic growth of KL}, we have
\begin{align*}
&-\mb E_{\theta^\ast}\Big[\int_{\widetilde\Theta}\big[\ell_n(\theta) - \ell_n(\theta^\ast)\big]\,\tilde{\pi}_\theta(\dd \theta)\Big]
= n\int_{\widetilde\Theta}\int_{\m X}\log\frac{p(x\,|\,\theta^\ast)}{p(x\,|\,\theta)}\,p(\dd x\,|\,\theta^\ast)\,\tilde{\pi}_\theta(\dd\theta)\\
&=n\int_{\widetilde\Theta}D_{KL}\big[p(\cdot\,|\,\theta^\ast)\,\big\|\,p(\cdot\,|\,\theta)\big]\,\tilde{\pi}_\theta(\dd\theta)
\leq c_4\,n\varepsilon_n^2,
\end{align*}
where the last step is due to the definition of $\widetilde\Theta$ in Assumption~\ref{assump: prior condition}.
Therefore, by applying Chebyshev's inequality, we obtain
\begin{align*}
&\mb P_{\theta^\ast}(\m A_n)
= \mb P_{\theta^\ast}\bigg(\int_{\widetilde\Theta} \big[\ell_n(\theta) - \ell_n(\theta^\ast)\big]\,\tilde{\pi}_\theta(\dd\theta)\leq -(c_4+1)n\varepsilon_n^2\bigg)\\
&\leq \mb P_{\theta^\ast}\bigg(\int_{\widetilde\Theta}\big[\ell_n(\theta) - \ell_n(\theta^\ast)\big]\,\tilde{\pi}_\theta(\dd\theta)- \mb E_{\theta^\ast}\Big[\int_{\widetilde\Theta}\big[\ell_n(\theta) - \ell_n(\theta^\ast)\big]\,\tilde{\pi}_\theta(\dd \theta)\Big]
\leq -(c_4+1)n\varepsilon_n^2 + c_4n\varepsilon_n^2\bigg)\\
&\leq \mb P_{\theta^\ast}\bigg(\, \bigg|\int_{\widetilde\Theta}\big[\ell_n(\theta) - \ell_n(\theta^\ast)\big]\,\tilde{\pi}_\theta(\dd\theta)- \mb E_{\theta^\ast}\Big[\int_{\widetilde\Theta}\big[\ell_n(\theta) - \ell_n(\theta^\ast)\big]\,\tilde{\pi}_\theta(\dd\theta)\Big]\ \bigg| \geq n\varepsilon_n^2\bigg)\\
&\leq \frac{1}{n^2\varepsilon_n^4}\, {\rm Var}_{\theta^\ast}\bigg(\int_{\widetilde\Theta}\big[\ell_n(\theta)-\ell_n(\theta^\ast)\big]\,\tilde{\pi}_\theta(\dd\theta)\bigg)= \frac{1}{n\varepsilon_n^4}\, {\rm Var}_{\theta^\ast}\bigg(\int_{\widetilde\Theta}\log\frac{p(X\,|\,\theta)}{p(X\,|\,\theta^\ast)}\,\tilde{\pi}_\theta(\dd \theta)\bigg)\\
&\stackrel{\ri}{\leq} \frac{1}{n\varepsilon_n^4}\int_{\m X}\bigg(\int_{\widetilde\Theta}\log\frac{p(x\,|\,\theta)}{p(x\,|\,\theta^\ast)}\,\tilde{\pi}_\theta(\dd \theta)\bigg)^2\,p(\dd x\,|\,\theta^\ast)\\
&\leq \frac{1}{n\varepsilon_n^4}\int_{\m X}\int_{\widetilde\Theta}\bigg(\log\frac{p(x\,|\,\theta)}{p(x\,|\,\theta^\ast)}\bigg)^2\,\tilde{\pi}_\theta(\dd \theta)\,p(\dd x\,|\,\theta^\ast) \stackrel{\rii}{\leq} \frac{c_4}{n\varepsilon_n^2}.
\end{align*}
Here, step (i) is by the inequality ${\rm Var}(Z) \leq \mb E[Z^2]$ for any random variable $Z$, and step (ii) by the definition of $\widetilde\Theta$.
\subsection{Proof of Lemma \ref{lem: para_update}}
Recall that $\widetilde U(\theta) = nU_n(\theta) - \log\pi_\theta(\theta)$, and denote
\begin{align*}
    \widetilde U_j^{(k)}(\theta_j) = \int_{\Theta_{-j}} \widetilde U(\theta_j, \theta_{-j})\,\dd q_{-j}^{(t)}(\theta_{-j}), \quad\mbox{and}\quad \widetilde U_{j}^\ast(\theta_j) = \int_{\Theta_j}\widetilde U(\theta_j, \theta_{-j})\,\dd \wht q_{-j}(\theta_{-j}).
\end{align*} 
Since $\widetilde U$ is $\tilde\lambda_n$-strongly convex, functional
\begin{align*}
    \m F(q) = \int_\Theta \widetilde U\,\dd q + \sum_{j=1}^m \int_{\Theta_j}\log q_j\,\dd q_j
\end{align*}
is $\tilde\lambda_n$-strongly convex along generalized geodesics.

Recall that $T_\mu^\nu$ is the optimal map from $\mu$ to $\nu$ for any regular measures $\mu, \nu\in\ms P_2^r$. Notice that
\begin{align*}
    \sum_{j=1}^m\Big\|T_{q_j^{(k)}}^{q_j^{(k+1)}} - T_{q_j^{(k)}}^{\wht q_j}\Big\|^2_{L^2(q_j^{(k)};\Theta_j)}
    &= \sum_{j=1}^m\Big\|T_{q_j^{(k)}}^{\wht q_j} - \id\Big\|^2_{L^2(q_j^{(k)};\Theta_j)} - \sum_{j=1}^m\Big\|T_{q_j^{(k)}}^{q_j^{(k+1)}}-\id\Big\|^2_{L^2(q_j^{(k)};\Theta_j)} + 2J\\
    &= W_2^2(q^{(t)}, \wht q) - W_2^2(q^{(t)}, q^{(t+1)}) + 2J,
\end{align*}
where we let
\begin{align*}
    J &= \sum_{j=1}^m\Big\langle T_{q_j^{(k)}}^{q_j^{(k+1)}}-T_{q_j^{(k)}}^{\wht q_j}, T_{q_j^{(k)}}^{q_j^{(k+1)}}-\id\Big\rangle_{L^2(q_j^{(k)}; \Theta_j)}
    = \sum_{j=1}^m\Big\langle\id - T_{q_j^{(k)}}^{\wht q_j}\circ T_{q_j^{(k+1)}}^{q_j^{(k)}}, \id - T_{q_j^{(k+1)}}^{q_j^{(k)}}\Big\rangle_{L^2(q_j^{(k+1)};\Theta_j)}\\
    &\stackrel{\ri}{=} \sum_{j=1}^m\Big\langle\id - T_{q_j^{(k)}}^{\wht q_j}\circ T_{q_j^{(k+1)}}^{q_j^{(k)}}, -\tau\Big(\nabla\log q_j^{(k+1)} + \nabla\log \widetilde U_j^{(k)}\Big)\Big\rangle_{L^2(q_j^{(k+1)},\Theta_j)}\\
    &= J_1 + J_2,
\end{align*}
and
\begin{align*}
    J_1 &= \sum_{j=1}^m\Big\langle\id - T_{q_j^{(k)}}^{\wht q_j}\circ T_{q_j^{(k+1)}}^{q_j^{(k)}}, -\tau\Big(\nabla\log q_j^{(k+1)} + \nabla\log \widetilde U_j^{(k+1)}\Big)\Big\rangle_{L^2(q_j^{(k+1)},\Theta_j)}\\
    J_2 &= \sum_{j=1}^m\Big\langle\id - T_{q_j^{(k)}}^{\wht q_j}\circ T_{q_j^{(k+1)}}^{q_j^{(k)}}, \tau\Big(\nabla\widetilde U_j^{(k+1)} - \nabla\widetilde U_j^{(k)}\Big)\Big\rangle_{L^2(q_j^{(k+1)};\Theta_j)}.
\end{align*}
Here, (i) is due to the first-order optimality condition in terms of the first variation.
Then, the goal is to bound $J_1$ and $J_2$ separately. 

To bound $J_1$, notice that $\xi^{(t+1)} = \nabla \widetilde U + \nabla\log q^{(t+1)}$ is a strong subgradient of $\m F$ at $q^{(t+1)}$. Applying Lemma \ref{Lem:stong_conx_generalized} with $\mu^1 = q^{(t)}$, $\mu^2 = q^{(t+1)}$, and $\mu^3 = \wht q$ yields
\begin{align*}
    \int_{\Theta}\Big\langle\xi^{(t+1)}, T_{q^{(t)}}^{\wht q}\circ T_{q^{(t+1)}}^{q^{(t)}} - \id\Big\rangle\,\dd q^{(t+1)} \leq \m F(\wht q) - \m F(q^{(t+1)}) - \frac{\tilde\lambda_n}{2}W_2^2(q^{(t+1)}, \wht q).
\end{align*}
Notice we have
\begin{align*}
    \int_{\Theta}\Big\langle\xi^{(t+1)}, T_{q^{(t)}}^{\wht q}\circ T_{q^{(t+1)}}^{q^{(t)}} - \id\Big\rangle\,\dd q^{(t+1)}
    &= \sum_{j=1}^m\int_\Theta\Big\langle\nabla_j\widetilde U+\nabla\log q_j^{(k+1)}, T_{q_j^{(k)}}^{\wht q_j}\circ T_{q_j^{(k+1)}}^{q_j^{(k)}} - \id\Big\rangle\,\dd q^{(t+1)}\\
    &= \sum_{j=1}^m\int_\Theta\Big\langle\nabla\widetilde U_j^{(k+1)}+\nabla\log q_j^{(k+1)}, T_{q_j^{(k)}}^{\wht q_j}\circ T_{q_j^{(k+1)}}^{q_j^{(k)}} - \id\Big\rangle\,\dd q_j^{(k+1)}\\
    &= \tau^{-1}J_1,
\end{align*}
and by convexity of $\m F$ we have
\begin{align*}
    \m F(q^{(t+1)}) - \m F(\wht q)
    &\geq \int_{\Theta}\Big\langle \nabla\frac{\delta\m F}{\delta q}(\wht q), T_{\wht q}^{q^{(t+1)}}-\id\Big\rangle\,\dd \wht q + \frac{\tilde\lambda_n}{2}W_2^2(\wht q, q^{(t+1)})\\
    &= \sum_{j=1}^m\int_{\Theta}\Big\langle \nabla_j\widetilde U + \nabla\log \wht q_j, T_{\wht q_j}^{q_j^{(k+1)}}-\id\Big\rangle\,\dd \wht q + \frac{\tilde\lambda_n}{2}W_2^2(\wht q, q^{(t+1)})\\
    &= \sum_{j=1}^m\int_{\Theta_j}\Big\langle \nabla\widetilde U_j^\ast + \nabla\log \wht q_j, T_{\wht q_j}^{q_j^{(k+1)}}-\id\Big\rangle\,\dd \wht q_j + \frac{\tilde\lambda_n}{2}W_2^2(\wht q, q^{(t+1)})\\
    &= \frac{\tilde\lambda_n}{2}W_2^2(\wht q, q^{(t+1)}).
\end{align*}
Here, the last equality is because $\wht q_j\in\ms P_2^r(\Theta_j)$ is the stationary point of $\m F(\cdot\otimes \wht q_{-j})$, which implies $\nabla\widetilde U_j^\ast + \nabla\log\wht q_j = 0$. Combining all pieces above yields
\begin{align}\label{eqn: para_boundJ1}
    \tau^{-1}J_1 \leq -\tilde\lambda_nW_2^2(q^{(t+1)}, \wht q).
\end{align}

To bound $J_2$, by Cauchy--Schwarz inequality we have
\begin{align*}
    J_2 &\leq \tau\sum_{j=1}^m \Big\|\id - T_{q_j^{(k)}}^{\wht q_j}\circ T_{q_j^{(k+1)}}^{q_j^{(k)}}\Big\|_{L^2(q_j^{(k+1)};\Theta_j)}\cdot\Big\|\nabla\widetilde U_j^{(k+1)} - \nabla\widetilde U_j^{(k)}\Big\|_{L^2(q_j^{(k+1)};\Theta_j)}\\
    &\stackrel{\ri}{\leq} \tilde L_n\tau W_2(q^{(t+1)}, q^{(t)})\sum_{j=1}^m \Big\|\id - T_{q_j^{(k)}}^{\wht q_j}\circ T_{q_j^{(k+1)}}^{q_j^{(k)}}\Big\|_{L^2(q_j^{(k+1)};\Theta_j)}\\
    &\stackrel{\rii}{\leq} \frac{W_2^2(q^{(t+1)}, q^{(t)})}{2} + \frac{\tilde L_n^2\tau^2m}{2}\sum_{j=1}^m\Big\|\id - T_{q_j^{(k)}}^{\wht q_j}\circ T_{q_j^{(k+1)}}^{q_j^{(k)}}\Big\|_{L^2(q_j^{(k+1)};\Theta_j)}^2\\
    &= \frac{W_2^2(q^{(t+1)}, q^{(t)})}{2} + \frac{\tilde L_n^2\tau^2m}{2}\sum_{j=1}^m\Big\|T_{q_j^{(k)}}^{q_j^{(k+1)}} - T_{q_j^{(k)}}^{\wht q_j}\Big\|^2_{L^2(q_j^{(k)};\Theta_j)}.
\end{align*}
Here, (i) is by the fact that
\begin{align*}
    \Big\|\nabla\widetilde U_j^{(k+1)} - \nabla\widetilde U_j^{(k)}\Big\|^2
    &= \bigg\|\int_{\Theta_{-j}}\nabla_j\widetilde U(\theta_j, \theta_{-j})\,\dd(q^{(t+1)}_{-j} - q^{(t)}_{-j})\bigg\|^2\\
    &=\bigg\|\int_{\Theta_{-j}}\nabla_j\widetilde U\Big(\theta_j, T_{q_{-j}^{(t)}}^{q_{-j}^{(t+1)}}(\theta_{-j})\Big) - \nabla_j \widetilde U(\theta_j, \theta_{-j})\,\dd q_{-j}^{(t)}\bigg\|^2\\
    &\leq \int_{\Theta_{-j}}\Big\|\nabla_j \widetilde U\Big(\theta_j, T_{q_{-j}^{(t)}}^{q_{-j}^{(t+1)}}(\theta_{-j})\Big) - \nabla_j \widetilde U(\theta_j, \theta_{-j})\Big\|^2\,\dd q_{-j}^{(t)}\\
    &\leq \tilde L_n^2\int_{\Theta_{-j}} \Big\|T_{q_{-j}^{(t)}}^{q_{-j}^{(t+1)}}(\theta_{-j}) - \theta_{-j}\Big\|^2\,\dd q_{-j}^{(t)}\\
    &= \tilde L_n^2W_2^2(q_{-j}^{(t)}, q_{-j}^{(t+1)}) \leq \tilde L_n^2W_2^2(q^{(t)}, q^{(t+1)}).
\end{align*}
(ii) is by AM-GM inequality and Cauchy--Schwarz inequality.

Therefore, we have shown that
\begin{align*}
    \sum_{j=1}^m\Big\|T_{q_j^{(k)}}^{q_j^{(k+1)}} - T_{q_j^{(k)}}^{\wht q_j}\Big\|^2_{L^2(q_j^{(k)};\Theta_j)} \leq W_2^2(q^{(t)}, \wht q)  - 2\tau\tilde\lambda_n W_2^2(q^{(t+1)}, \wht q) + \tilde L_n^2\tau^2m\sum_{j=1}^m\Big\|T_{q_j^{(k)}}^{q_j^{(k+1)}} - T_{q_j^{(k)}}^{\wht q_j}\Big\|^2_{L^2(q_j^{(k)};\Theta_j)},
\end{align*}
which implies
\begin{align*}
    (1-\tilde L_n^2\tau^2m) W_2^2(q^{(t+1)}, \wht q) \leq (1-\tilde L_n^2\tau^2m)\sum_{j=1}^m\Big\|T_{q_j^{(k)}}^{q_j^{(k+1)}} - T_{q_j^{(k)}}^{\wht q_j}\Big\|^2_{L^2(q_j^{(k)};\Theta_j)} \leq W_2^2(q^{(t)}, \wht q) - 2\tau\tilde\lambda_n W_2^2(q^{(t+1)}, \wht q)
\end{align*}
since $1-\tilde L_n^2\tau^2m \geq 0$, i.e.
\begin{align*}
    (1+2\tau\tilde\lambda_n - \tilde L_n^2\tau^2m)W_2^2(q^{(t+1)}, \wht q) \leq W_2^2(q^{(t)}, \wht q).
\end{align*}
This implies the desiring result.

\subsection{Proof of Lemma \ref{lem: ULLN_of_potential_function}}

A similar lemma without the distributional dependence on $\mu$ is proved in~\cite{mei2018landscape} for studying the optimization landscape of empirical risk minimization over parameter $\theta\in\Theta$. Our proof strategy of the lemma is to discretize $\Theta$ based on an $\varepsilon$-covering of $\Theta$. Since the diameter of $\Theta$ is bounded by $R>0$, the metric entropy of $\Theta$ is bounded by $d\log\frac{3R}{\varepsilon}$. After the discretization, for any fixed $\theta$ in the $\varepsilon$-covering of $\Theta$ and those $\mu$ close to $\delta_{\theta^\ast}$, we may approximate $\big[\nabla^2U_n(\theta;\,\mu) - \nabla^2U(\theta;\,\mu)\big]$ by $\big[\nabla^2U_n(\theta, \delta_{\theta^\ast}) - \nabla^2U(\theta, \delta_{\theta^\ast})\big]$ By Lemma \ref{lem: bound_opnorm_by_metric_entropy} in Appendix~\ref{proof_lem:B_t_bound}, the operator norm $\matnorm{\nabla^2U_n(\theta, \delta_{\theta^\ast}) - \nabla^2U(\theta, \delta_{\theta^\ast})}$ at a given $\theta$ can be bounded with high probability by using the variational characterization of the operator norm and discretizing the unit ball therein. Finally, we can uniformly control the operator norm $\matnorm{\nabla^2U_n(\theta;\,\mu) - \nabla^2U(\theta;\,\mu)}$ over all $\theta\in\Theta$ and $\mu$ in a $W_2$ neighborhood of $\delta_{\theta^\ast}$ by optimally balancing between the discretization error and approximation error.


Since $U_n$ in equation~\eqref{eq:U_n_exp} depends on the $\Phi$ function that defines the updating formula for $q_{Z^n}$, we need a lemma about some perturbation bound of $\Phi$, whose proof is deferred to Section~\ref{app:proof_diff_Phimu_thetastar}. We list this lemma here since it will also be used in other proofs in this supplement.
\begin{lemma}\label{lem: diff_Phimu_thetastar}
Under assumption \ref{assump: continuity_of_Hessian}, we have
\begin{displaymath}
|\Phi(\mu, x)(z) - \Phi(\delta_{\theta^\ast}, x)(z)| \leq \sum_{k=1}^K\|\nabla\log p(k\,|\,x,\theta^\ast)\|\cdot W_2(\mu, \delta_{\theta^\ast}) + \frac{K\lambda(x)}{2}\cdot W_2^2(\mu, \delta_{\theta^\ast}).
\end{displaymath}
\end{lemma}

\noindent Now let us return to the proof of Lemma~\ref{lem: ULLN_of_potential_function}.
Let $N_\varepsilon$ be the $\varepsilon$-covering number of $\Theta \subset B^d(0, R)$. Let $\{\theta_1, \cdots, \theta_{N_\varepsilon}\}$ be such an $N_\varepsilon$-covering, and let $j(\theta) := \argmin_{j\in[N_\varepsilon]}\|\theta - \theta_j\|$ denote the index corresponding to the closet point in the covering to any $\theta\in\Theta$ . By the triangular inequality, we have
\begin{align*}
    \sup_{\substack{\theta\in\Theta\\ \mu:W_2(\mu, \delta_{\theta^\ast})\leq r}}&\, \matnorm{\nabla^2 U_n(\theta;\,\mu) - \nabla^2 U(\theta;\,\mu)}
    \leq \sup_{\substack{\theta\in\Theta\\ \mu:W_2(\mu;\, \delta_{\theta^\ast})\leq r}}\matnorm{\nabla^2 U_n(\theta;\,\mu) - \nabla^2 U_n(\theta_{j(\theta)};\, \mu)}\\
    &\quad + \sup_{\substack{\theta\in\Theta\\ \mu:W_2(\mu;\, \delta_{\theta^\ast})\leq r}}\matnorm{\nabla^2 U_n(\theta_{j(\theta)};\, \mu) - \nabla^2 U(\theta_{j(\theta)};\, \mu)}\\
    &\qquad\quad + \sup_{\substack{\theta\in\Theta\\ \mu:W_2(\mu, \delta_{\theta^\ast})\leq r}}\matnorm{\nabla^2 U(\theta_{j(\theta)};\, \mu) - \nabla^2 U(\theta;\,\mu)}.
\end{align*}
Therefore, for every $t > 0$, we have the following decomposition,
\begin{displaymath}
\mb P\bigg(\sup_{\substack{\theta\in\Theta\\ \mu:W_2(\mu, \delta_{\theta^\ast})\leq r}} \matnorm{\nabla^2 U_n(\theta;\,\mu) - \nabla^2 U(\theta;\,\mu)} > t\bigg) \leq \mb P(A_t) + \mb P(B_t) + \mb P(C_t),
\end{displaymath}
where
\begin{align*}
A_t &= \bigg\{\sup_{\substack{\theta\in\Theta\\ \mu: W_2(\mu, \delta_{\theta^\ast})\leq r}}\matnorm{\nabla^2U_n(\theta;\,\mu) - \nabla^2U_n(\theta_{j(\theta)}, \mu)} > \frac{t}{3}\bigg\},\\
B_t &= \bigg\{\sup_{\substack{\theta\in\Theta\\ \mu: W_2(\mu, \delta_{\theta^\ast})\leq r}}\matnorm{\nabla^2U_n(\theta_{j(\theta)}, \mu) - \nabla^2U(\theta_{j(\theta)}, \mu)} > \frac{t}{3}\bigg\},\\
C_t &= \bigg\{\sup_{\substack{\theta\in\Theta\\ \mu: W_2(\mu, \delta_{\theta^\ast})\leq r}}\matnorm{\nabla^2U(\theta_{j(\theta)}, \mu) - \nabla^2U(\theta;\,\mu)} > \frac{t}{3}\bigg\}.
\end{align*}
To bound $\mb P(A_t)$, we notice that by definition
\begin{align*}
&\quad\, \sup_{\substack{\theta\in\Theta\\ \mu: W_2(\mu, \delta_{\theta^\ast})\leq r}}\matnorm{\nabla^2U_n(\theta;\,\mu) - \nabla^2U_n(\theta_{j(\theta)}, \mu)}\\
&= \sup_{\substack{\theta\in\Theta\\ \mu: W_2(\mu, \delta_{\theta^\ast})\leq r}}\matnorm{\frac{1}{n}\sum_{i=1}^n\sum_{z=1}^K\Big[\nabla^2\log p(X_i, z\,|\,\theta) - \nabla^2\log p(X_i, z\,|\,\theta_{j(\theta)})\Big]\Phi(\mu, X_i)(z)}\\
&\leq \sup_{\substack{\theta\in\Theta\\ \mu: W_2(\mu, \delta_{\theta^\ast})\leq r}} \bigg\{\frac{1}{n}\sum_{i=1}^n\sum_{z=1}^K\matnorm{\nabla^2\log p(X_i, z\,|\,\theta) - \nabla^2\log p(X_i, z\,|\,\theta_{j(\theta)})}\Phi(\mu, X_i)(z)\bigg\}\\
&\leq \sup_{\theta\in\Theta}\bigg\{\frac{1}{n}\sum_{i=1}^n\sum_{z=1}^K\matnorm{\nabla^2\log p(X_i, z\,|\,\theta) - \nabla^2\log p(X_i,z\,|\,\theta_{j(\theta)})}\bigg\}.
\end{align*}
Therefore, we can apply the Markov inequality to obtain
\begin{align*}
\mb P(A_t)
&\leq \mb P\bigg(\sup_{\theta\in\Theta} \bigg\{\frac{1}{n}\sum_{i=1}^n\sum_{z=1}^K\matnorm{\nabla^2\log p(X_i,z\,|\,\theta)-\nabla^2\log p(X_i,z\,|\,\theta_{j(\theta)})}\bigg\} > \frac{t}{3}\bigg)\\
&\leq \frac{3}{t}\, \mb E\bigg[\sup_{\theta\in\Theta}\bigg\{\frac{1}{n}\sum_{i=1}^n\sum_{z=1}^K\matnorm{\nabla^2\log p(X_i,z\,|\,\theta) - \nabla^2\log p(X_i,z\,|\,\theta_{j(\theta)})}\bigg\}\bigg]\\
&\leq \frac{3}{t}\,\mb E \bigg[\frac{1}{n}\sum_{i=1}^n\sum_{z=1}^K\sup_{\theta\in\Theta}\,\matnorm{\nabla^2\log p(X_i,z\,|\,\theta) - \nabla^2\log p(X_i,z\,|\,\theta_{j(\theta)})}\bigg]\\
&= \frac{3}{t}\,\mb E\bigg[\sum_{z=1}^K\sup_{\theta\in\Theta}\,\matnorm{\nabla^2\log p(X,z\,|\,\theta) - \nabla^2\log p(X,z\,|\,\theta_{j(\theta)})}\bigg]\\
&\stackrel{\ri}{\leq} \frac{3}{t}\mb E \bigg[\sum_{z=1}^K \varepsilon J_z(X)\bigg] = \frac{3\varepsilon}{t}J_\ast.
\end{align*}
Here, step (i) follows from
\begin{align*}
&\quad\, \sup_{\theta\in\Theta}\matnorm{\nabla^2\log p(X,z\,|\,\theta) - \nabla^2\log p(X,z\,|\,\theta_{j(\theta)})}\\
&=\sup_{\theta\in\Theta}\frac{\matnorm{\nabla^2\log p(X,z\,|\,\theta) - \nabla^2\log p(X,z\,|\,\theta_{j(\theta)})}}{\|\theta - \theta_{j(\theta)}\|}\cdot\|\theta - \theta_{j(\theta)}\|\\
&\leq \sup_{\theta\in\Theta}\frac{\matnorm{\nabla^2\log p(X,z\,|\,\theta) - \nabla^2\log p(X,z\,|\,\theta_{j(\theta)})}}{\|\theta - \theta_{j(\theta)}\|}\cdot \varepsilon \leq \varepsilon J_z(X),
\end{align*}
where the last step is due to Assumption~\ref{assump: continuity_of_Hessian}.
\smallskip

\noindent Similarly, to control the probability of the deterministic event $C_t$, we note
\begin{align*}
&\quad\, \sup_{\substack{\theta\in\Theta\\ \mu: W_2(\mu, \delta_{\theta^\ast})\leq r}} \matnorm{\nabla^2 U(\theta_{j(\theta)}, \mu) - \nabla^2 U(\theta;\,\mu)}\\
&= \sup_{\substack{\theta\in\Theta\\ \mu: W_2(\mu, \delta_{\theta^\ast})\leq r}}\matnorm{\int_{\m X}\sum_{z=1}^K\Big[\nabla^2\log p(x, z\,|\,\theta) - \nabla^2\log p(x, z\,|\,\theta_{j(\theta)})\Big]\Phi(\mu, x)(z)\,p(\dd x\,|\,\theta^\ast)}\\
&\leq \sup_{\substack{\theta\in\Theta\\ \mu: W_2(\mu, \delta_{\theta^\ast})\leq r}} \bigg\{\int_{\m X}\sum_{z=1}^K\matnorm{\nabla^2\log p(x, z\,|\,\theta) - \nabla^2\log p(x,z\,|\,\theta_{j(\theta)})}\Phi(\mu, x)(z)\,p(\dd x\,|\,\theta^\ast)\bigg\}\\
&\leq \ \sup_{\theta\in\Theta} \bigg\{\int_{\m X}\sum_{z=1}^K\matnorm{\nabla^2\log p(x,z\,|\,\theta) - \nabla^2\log p(x,z\,|\,\theta_{j(\theta)})}\,p(\dd x\,|\,\theta^\ast)\bigg\}\\
&\leq\  \int_{\m X}\sum_{z=1}^K \, \sup_{\theta\in\Theta}\bigg\{\frac{\matnorm{\nabla^2\log p(x,z\,|\,\theta) - \nabla^2\log p(x,z\,|\,\theta_{j(\theta)})}}{\|\theta - \theta_{j(\theta)}\|}\bigg\} \cdot\varepsilon\,p(\dd x\,|\,\theta^\ast)\\
&\leq \ \varepsilon\,\int_{\mb R^d}\sum_{z=1}^K J_z(x)\,p(\dd x\,|\,\theta^\ast)=\varepsilon J_\ast.
\end{align*}
Therefore, if we take $t > 3\varepsilon J_\ast$, then $\mb P(C_t) = 0$. 
\smallskip

\noindent Finally, let us consider the most difficult term $\mb P(B_t)$. Since its proof is quite long, we summarize the result in the following lemma and defer its proof to Section~\ref{proof_lem:B_t_bound}. The proof will utilize Lemma~\ref{lem: diff_Phimu_thetastar} and a standard discretization technique for random matrix concentration.
\begin{lemma}\label{lem:B_t_bound}
Under Assumption~\ref{assump: continuity_of_Hessian}, we have that for any $t>0$,
\begin{align*}
    \mb P(B_t)\ \leq  \ 2e^{-Cn^\frac{1}{6}\sigma_2^{-1}} + 2e^{-Cn^\frac{1}{6}\sigma_3^{-1}} + 2 \exp\bigg\{d\log\frac{36R}{\varepsilon} -\frac{Cn^\frac{1}{2}t}{\sigma_1}\bigg\}.
\end{align*}
\end{lemma}

\noindent Putting all pieces together, we have shown that
\begin{align*}
&\quad\, \mb P\bigg(\sup_{\substack{\theta\in\Theta\\ \mu:W_2(\mu, \delta_{\theta^\ast})\leq r}} \matnorm{\nabla^2 U_n(\theta;\,\mu) - \nabla^2 U(\theta;\,\mu)} > t\bigg)\\
&\leq \mb P(A_t) + \mb P(B_t) + \mb P(C_t)\\
&\leq \frac{3\varepsilon}{t}J_\ast + \Big(2e^{-Cn^\frac{1}{6}\sigma_2^{-1}} + 2e^{-Cn^\frac{1}{6}\sigma_3^{-1}}\Big) + 2 \exp\bigg\{d\log\frac{36R}{\varepsilon} -\frac{Cn^\frac{1}{2}t}{\sigma_1}\bigg\}.
\end{align*}
For any $\eta\in(0, 1)$ and $n > \big(C\log\frac{6}{\eta}\cdot\max\{\sigma_2, \sigma_3\}\big)^6$, we have $2e^{-Cn^\frac{1}{6}\sigma_e^{-1}} + 2e^{-Cn^\frac{1}{6}\sigma_e^{-1}} < \frac{\eta}{3}$. Therefore, if
\begin{displaymath}
t > \max\bigg\{\frac{9\varepsilon J_\ast}{\eta}, \frac{\sigma_1}{\sqrt{n}}\Big(\log\frac{6}{\eta} + d\log\frac{36R}{\varepsilon}\Big)\bigg\},
\end{displaymath}
then
\begin{align}\label{eqn:hessian_con}
   \mb P\bigg(\sup_{\substack{\theta\in\Theta\\ \mu:W_2(\mu, \delta_{\theta^\ast})\leq r}} \matnorm{\nabla^2 U_n(\theta;\,\mu) - \nabla^2 U(\theta;\,\mu)} > t\bigg) \leq \eta. 
\end{align}
With $\varepsilon = \frac{\eta \sigma_1}{9ndJ_\ast}$, or $\frac{9\varepsilon J_\ast}{\eta} = \frac{\sigma_1}{nd}$, since
\begin{displaymath}
n \geq \max\bigg\{\frac{324RdJ_\ast}{\sigma_1}, 6\bigg\},
\end{displaymath}
we can take
\begin{align*}
    t:\,= \frac{2d\sigma_1\log\frac{n}{\eta}}{\sqrt{n}} >\frac{\sigma_1}{\sqrt{n}}\Big(d\log\frac{36R}{\varepsilon} + \log\frac{6}{\eta}\Big)
\end{align*}
in inequality~\eqref{eqn:hessian_con} to obtain
\begin{displaymath}
P\bigg(\sup_{\substack{\theta\in\Theta\\ \mu:W_2(\mu, \delta_{\theta^\ast})\leq r}} \matnorm{\nabla^2 U_n(\theta;\,\mu) - \nabla^2 U(\theta;\,\mu)} > \frac{2d\sigma_1\log\frac{n}{\eta}}{\sqrt{n}}\bigg) \leq \eta.
\end{displaymath}

\subsection{Proof of Lemma~\ref{lem:cross_produc}}\label{proof_lem:cross_produc}
We will prove the following series of inequalities,
\begin{align*}
   &\bigg|\int_\Theta\sum_{i=1}^n\sum_{z=1}^K\log p(z\,|\,X_i,\theta)\big[\Phi(\widehat{q}_\theta, X_i)(z) - \Phi(q_\theta^{(k)}, X_i)(z)\big]\big(q_\theta^{(k+1)}(\theta) - \widehat{q}_\theta(\theta)\big)\,\dd\theta\bigg|\\ &\stackrel{\textrm{step} 1}{\leq} \bigg|\int_\Theta\sum_{i=1}^n\sum_{z=1}^K\log p(z\,|\,X_i,\theta)\cdot\int_\Theta\Big\langle\nabla\frac{\delta\Phi}{\delta\mu}(\widehat{q}_\theta, X_i)(z), t_{\widehat{q}_\theta}^{q_\theta^{(k)}}(\theta)-\theta\Big\rangle\,\dd\widehat{q}_\theta\\
    &\qquad\qquad\qquad\qquad\qquad\qquad\qquad\qquad\qquad\qquad\cdot\big(q_\theta^{(k+1)}(\theta) - \widehat{q}_\theta(\theta)\big)\,\dd\theta\bigg| + R_1\\
    &\stackrel{\textrm{step} 2}{\leq} \bigg|\sum_{i=1}^n\sum_{z=1}^K \int_\Theta\big\langle\nabla\log p(z\,|\,X_i,\theta), t_{\widehat{q}_\theta}^{q_\theta^{(k+1)}}(\theta)-\theta\big\rangle\,\dd\widehat{q}_\theta\\
    &\qquad\qquad\qquad\qquad\qquad\qquad\cdot\int_\Theta\Big\langle\nabla\frac{\delta\Phi}{\delta\mu}(\widehat{q}_\theta, X_i)(z), t_{\widehat{q}_\theta}^{q_\theta^{(k)}}(\theta)-\theta\Big\rangle\,\dd\widehat{q}_\theta\bigg| + R_2 + R_1\\
    &\stackrel{\textrm{step} 3}{\leq} n\big|\big\langle\Delta_{q_\theta^{(k)}}, \wht I_S(\theta^\ast)\Delta_{q_\theta^{(k+1)}}\big\rangle\big| + R_3 + R_2 + R_1,
\end{align*}
where steps 1, 2 and 3 will be elaborated as follows.

\smallskip
\noindent {\bf Proof of step 1.}
For simplicity, let
\begin{displaymath}
R(q, x)(z) := \Phi(q, x)(z) - \Phi(\widehat{q}_\theta, x)(z) - \int_\Theta\Big\langle\nabla\frac{\delta\Phi}{\delta\mu}(\widehat{q}_\theta, x)(z), t_{\widehat{q}_\theta}^q(\theta) -\theta\Big\rangle\,\dd\widehat{q}_\theta. 
\end{displaymath}
By the triangular inequality, we have
\begin{align*}
&\quad\, \Big|\big[V_n\big(q_\theta^{(k+1)}\,\big|\,q_\theta^{(k)}\big) - V_n\big(\widehat{q}_\theta\,\big|\,q_\theta^{(k)}\big)\big] - \big[V_n\big(q_\theta^{(k+1)}\,\big|\,\widehat{q}_\theta\big) - V_n\big(\widehat{q}_\theta\,\big|\,\widehat{q}_\theta\big)\big]\Big|\\
&= \bigg|\int_\Theta\sum_{i=1}^n\sum_{z=1}^K\log p(z\,|\,X_i,\theta)\big[\Phi(\widehat{q}_\theta, X_i)(z) - \Phi(q_\theta^{(k)}, X_i)(z)\big]\big(q_\theta^{(k+1)}(\theta) - \widehat{q}_\theta(\theta)\big)\,\dd\theta\bigg|\\
&\leq \bigg|\int_\Theta\sum_{i=1}^n\sum_{z=1}^K\log p(z\,|\,X_i,\theta)\cdot\int_\Theta\Big\langle\nabla\frac{\delta\Phi}{\delta\mu}(\widehat{q}_\theta, X_i)(z), t_{\widehat{q}_\theta}^{q_\theta^{(k)}}(\theta)-\theta\Big\rangle\,\dd\widehat{q}_\theta\cdot\big(q_\theta^{(k+1)}(\theta) - \widehat{q}_\theta(\theta)\big)\,\dd\theta\bigg|\\
&\quad+ \bigg|\int_\Theta\sum_{i=1}^n\sum_{z=1}^K\log p(z\,|\,X_i,\theta)R(q_\theta^{(k)}, X_i)(z)\big(q_\theta^{(k+1)}(\theta) - \widehat{q}_\theta(\theta)\big)\,\dd\theta\bigg|
\end{align*}
The remainder term from above can be further bounded by
\begin{align*}
&\quad\,\bigg|\int_\Theta\sum_{i=1}^n\sum_{z=1}^K\log p(z\,|\,X_i,\theta)R(q_\theta^{(k)}, X_i)(z)\big(q_\theta^{(k+1)}(\theta) - \widehat{q}_\theta(\theta)\big)\,\dd\theta\bigg|\\
&= \bigg|\sum_{i=1}^n\sum_{z=1}^KR(q_\theta^{(k)}, X_i)(z)\int_\Theta\log p(z\,|\,X_i,\theta)\dd(q_\theta^{(k+1)} - \widehat{q}_\theta)\bigg|\\
&\stackrel{\ri}{\leq} \sum_{i=1}^n\sum_{z=1}^K W_2^2(\widehat{q}_\theta, q_\theta^{(k)})\!\bigg[\frac{3}{2}\!\bigg(\sum_{l=1}^K\big\|\nabla \log p(l\,|\,X_i,\theta^\ast)\big\|_2\! +\!\! K\lambda(X_i)\sqrt{\mb E_{\widehat{q}_\theta}\|\theta-\theta^\ast\|^2}\!+\!\!\frac{K\lambda(X_i)}{2}\!W_2(\widehat{q}_\theta, q_\theta^{(k)})\!\!\bigg)^2 \!\!\!\!+\! \lambda(X_i)\bigg]\\
&\qquad\qquad\cdot \Big[\Big(\lambda(X_i)\sqrt{\mb E_{\widehat{q}_\theta}\|\theta-\theta^\ast\|^2} + \big\|\nabla\log p(z\,|\,X_i,\theta^\ast)\big\|_2\Big)W_2(\widehat{q}_\theta, q_\theta^{(k+1)}) + \frac{\lambda(X_i)}{2}W_2^2(\widehat{q}_\theta, q_\theta^{(k+1)})\Big]\\
&= \sum_{i=1}^n W_2^2(\widehat{q}_\theta, q_\theta^{(k)})\bigg[\frac{3}{2}\bigg(\sum_{l=1}^K\big\|\nabla \log p(l\,|\,X_i,\theta^\ast)\big\|_2 + \!K\lambda(X_i)\sqrt{\mb E_{\widehat{q}_\theta}\|\theta-\theta^\ast\|^2}+\frac{K\lambda(X_i)}{2}W_2(\widehat{q}_\theta, q_\theta^{(k)})\bigg)^2 \!\!\!+\! \lambda(X_i)\bigg]\\
&\qquad\qquad\cdot \Big[\Big(K\lambda(X_i)\sqrt{\mb E_{\widehat{q}_\theta}\|\theta-\theta^\ast\|^2} + \sum_{z=1}^K\big\|\nabla\log p(z\,|\,X_i,\theta^\ast)\big\|_2\Big)W_2(\widehat{q}_\theta, q_\theta^{(k+1)}) + \frac{K\lambda(X_i)}{2}W_2^2(\widehat{q}_\theta, q_\theta^{(k+1)})\Big]\\
&\stackrel{\rii}{\leq} W_2^2(\widehat{q}_\theta, q_\theta^{(k)})W_2(\widehat{q}_\theta, q_\theta^{(k+1)})\sum_{i=1}^n \bigg[\frac{3}{2}\Big(S_1(X_i)+K\lambda(X_i) + \frac{K\lambda(X_i)}{2}\Big)^2 + \lambda(X_i)\bigg]\\
&\qquad\qquad\qquad\qquad\qquad\qquad\qquad\qquad\qquad \cdot\bigg(K\lambda(X_i) + S_1(X_i) + \frac{K\lambda(X_i)}{2}W_2(\widehat{q}_\theta, q_\theta^{(k+1)})\bigg)\\
&\leq W_2^2(\widehat{q}_\theta, q_\theta^{(k)})W_2^2(\widehat{q}_\theta, q_\theta^{(k+1)})\sum_{i=1}^n\bigg[\frac{3K^2+2K}{4}S_1(X_i)^3 + \frac{27K^3+24K^2+4K}{16}\lambda_1(X_i)^3 + \frac{K}{2}\lambda(X_1)^2\bigg]\\
&\quad+
W_2^2(\widehat{q}_\theta, q_\theta^{(k)})W_2(\widehat{q}_\theta, q_\theta^{(k+1)})\sum_{i=1}^n\bigg[
\frac{21K^2+32K+12}{8} S_1(X_i)^3 + \frac{27K^3+42K^2+16K}{8}\lambda(X_i)^3\\
&\qquad\qquad\qquad\qquad\qquad\qquad\qquad\qquad\qquad\qquad\qquad\qquad\qquad\quad +
\frac{1}{2}S_1(X_i)^2 + \frac{2K+1}{2}\lambda(X_i)^2\bigg]\\
&\leq nW_2^2(\widehat{q}_\theta, q_\theta^{(k)})W_2^2(\widehat{q}_\theta, q_\theta^{(k+1)})\bigg[\frac{3K^2+2K}{4n}\sum_{i=1}^nS_1(X_i)^3\\
&\qquad\qquad\qquad\quad\qquad\qquad\qquad\qquad+ \frac{27K^3+24K^2+4K}{16n}\sum_{i=1}^n\lambda_1(X_i)^3\\ 
&\qquad\qquad\qquad\quad\qquad\qquad\qquad\qquad+ \frac{K}{2}\Big(\frac{1}{n}\sum_{i=1}^n\lambda(X_1)^3\Big)^{\frac{2}{3}}\bigg]\\
&\quad+
nW_2^2(\widehat{q}_\theta, q_\theta^{(k)})W_2(\widehat{q}_\theta, q_\theta^{(k+1)})\bigg[
\frac{21K^2+32K+12}{8n}\sum_{i=1}^n S_1(X_i)^3\\ 
&\qquad\qquad\qquad\quad\qquad\qquad\qquad\qquad+ \frac{27K^3+42K^2+16K}{8n}\sum_{i=1}^n\lambda(X_i)^3\\
&\qquad\qquad\qquad\quad\qquad\qquad\qquad\qquad +
\frac{1}{2}\Big(\frac{1}{n}\sum_{i=1}^nS_1(X_i)^3\Big)^{\frac{2}{3}} + \frac{2K+1}{2}\Big(\frac{1}{n}\sum_{i=1}^n\lambda(X_i)^3\Big)^{\frac{2}{3}}\bigg].
\end{align*}
Here, in step (i) we use Corollary \ref{coro: bound_R1} to bound $R(q_\theta^{(k)}, X_i)(z)$ and Corollary \ref{coro: bound_diff_log_integration} to bound the integration. In step (ii), we used the induction hypothesis $W_2(\widehat{q}_\theta, q_\theta^{(k)}) < R_W \leq 1$ and the fact that $\mb E_{\widehat{q}_\theta}[\|\theta-\theta^\ast\|^2] \leq 1$ by Corollary \ref{coro: bound_square_expectation}.

\smallskip
\noindent {\bf Proof of step 2.}
By applying Taylor's expansion in a similar way as in the proof of Lemma \ref{lem: diff_log_integration}, we obtain
\begin{align*}
&\quad\, \bigg|\int_\Theta\sum_{i=1}^n\sum_{z=1}^K\log p(z\,|\,X_i,\theta)\cdot\int_\Theta\Big\langle\nabla\frac{\delta\Phi}{\delta\mu}(\widehat{q}_\theta, X_i)(z), t_{\widehat{q}_\theta}^{q_\theta^{(k)}}(\theta)-\theta\Big\rangle\,\dd\widehat{q}_\theta\cdot\big(q_\theta^{(k+1)}(\theta) - \widehat{q}_\theta(\theta)\big)\,\dd\theta\bigg|\\
&= \bigg|\sum_{i=1}^n\sum_{z=1}^K \int_\Theta\log p(z\,|\,X_i,\theta)\,\dd(q_\theta^{(k+1)}-\widehat{q}_\theta)\cdot \int_\Theta\Big\langle\nabla\frac{\delta\Phi}{\delta\mu}(\widehat{q}_\theta, X_i)(z), t_{\widehat{q}_\theta}^{q_\theta^{(k)}}(\theta)-\theta\Big\rangle\,\dd\widehat{q}_\theta\bigg|\\
&\leq\bigg|\sum_{i=1}^n\sum_{z=1}^K \frac{1}{2}\int_\Theta\big\langle t_{\widehat{q}_\theta}^{q_\theta^{(k+1)}}(\theta)-\theta, \nabla^2\log p(z\,|\,X_i,\theta')\big(t_{\widehat{q}_\theta}^{q_\theta^{(k+1)}}(\theta)-\theta\big)\big\rangle\,\dd\widehat{q}_\theta \cdot\int_\Theta\Big\langle\nabla\frac{\delta\Phi}{\delta\mu}(\widehat{q}_\theta, X_i)(z), t_{\widehat{q}_\theta}^{q_\theta^{(k)}}(\theta)-\theta\Big\rangle\,\dd\widehat{q}_\theta\bigg|\\
&\quad + \bigg|\sum_{i=1}^n\sum_{z=1}^K \int_\Theta\big\langle\nabla\log p(z\,|\,X_i,\theta), t_{\widehat{q}_\theta}^{q_\theta^{(k+1)}}(\theta)-\theta\big\rangle\,\dd\widehat{q}_\theta\cdot\int_\Theta\Big\langle\nabla\frac{\delta\Phi}{\delta\mu}(\widehat{q}_\theta, X_i)(z), t_{\widehat{q}_\theta}^{q_\theta^{(k)}}(\theta)-\theta\Big\rangle\,\dd\widehat{q}_\theta\bigg|.
\end{align*}
This time we bound the remainder term as
\begin{align*}
&\quad\, \bigg|\sum_{i=1}^n\sum_{z=1}^K \frac{1}{2}\int_\Theta\big\langle t_{\widehat{q}_\theta}^{q_\theta^{(k+1)}}(\theta)-\theta, \nabla^2\log p(z\,|\,X_i,\theta')\big(t_{\widehat{q}_\theta}^{q_\theta^{(k+1)}}(\theta)-\theta\big)\big\rangle\,\dd\widehat{q}_\theta \cdot\int_\Theta\Big\langle\nabla\frac{\delta\Phi}{\delta\mu}(\widehat{q}_\theta, X_i)(z), t_{\widehat{q}_\theta}^{q_\theta^{(k)}}(\theta)-\theta\Big\rangle\,\dd\widehat{q}_\theta\bigg|\\
&\stackrel{\ri}{\leq} \frac{1}{2}\sum_{i=1}^n\sum_{z=1}^K \lambda(X_i)W_2^2(\widehat{q}_\theta, q_\theta^{(k+1)})\cdot \Phi(\widehat{q}_\theta, X_i)(z)\sum_{l=1}^K\Phi(\widehat{q}_\theta, X_i)(l)\bigg|\int_\Theta\Big\langle\nabla\log\frac{p(z\,|\,X_i,\theta)}{p(k\,|\,X_i,\theta)}, t_{\widehat{q}_\theta}^{q_\theta^{(k)}}(\theta)-\theta\Big\rangle\,\dd\widehat{q}_\theta\bigg|\\
&\stackrel{\rii}{\leq} \sum_{i=1}^n\lambda(X_i)W_2^2(\widehat{q}_\theta, q_\theta^{(k+1)})\cdot K\sum_{l=1}^K \bigg|\int_\Theta\big\langle\nabla\log p(l\,|\,X_i,\theta), t_{\widehat{q}_\theta}^{q_\theta^{(k)}}(\theta)-\theta\big\rangle\,\dd\widehat{q}_\theta\bigg|\\
&\stackrel{\riii}{\leq} \sum_{i=1}^n\lambda(X_i)W_2^2(\widehat{q}_\theta, q_\theta^{(k+1)})\cdot K\sum_{l=1}^K W_2(q_\theta^{(k)}, \widehat{q}_\theta)\Big[\lambda(X_i)\sqrt{\mb E_{\widehat{q}_\theta}\|\theta-\theta^\ast\|^2} + \|\nabla\log p(l\,|\,X_i,\theta^\ast)\|\Big]\\
&\stackrel{(\textrm{iv})}{\leq} W_2^2(\widehat{q}_\theta, q_\theta^{(k+1)})W_2(\widehat{q}_\theta, q_\theta^{(k)})\cdot \sum_{i=1}^n\Big(\frac{2K^2+K}{2}\lambda(X_i)^2 + \frac{K}{2}S_1(X_i)^2\Big)\\
&\leq n\,W_2^2(\widehat{q}_\theta, q_\theta^{(k+1)})\, W_2(\widehat{q}_\theta, q_\theta^{(k)})\bigg(\frac{2K^2+K}{2}\Big(\frac{1}{n}\sum_{i=1}^n\lambda(X_i)^3\Big)^{\frac{2}{3}} + \frac{K}{2}\Big(\frac{1}{n}\sum_{i=1}^nS_1(X_i)^3\Big)^{\frac{2}{3}}\bigg).
\end{align*}
Here, step (i) is derived by the expression of $\nabla\frac{\delta\Phi}{\delta\mu}$ in Corollary \ref{coro: subdifferential_is_variation}; step (ii) is obtained by the triangular inequality and the fact that $\Phi(\widehat{q}_\theta, X_i)(\cdot)\leq 1$; step (iii) is by lemma \ref{lem: change_to_Deltaq}; step (iv) is by AM-GM inequality and the fact that $\mb E_{\widehat{q}_\theta}[\|\theta-\theta^\ast\|^2]\leq 1$.

\smallskip
\noindent {\bf Proof of step 3.}
Notice that
\begin{align*}
&\quad\, \bigg|\sum_{i=1}^n\sum_{z=1}^K \int_\Theta\big\langle\nabla\log p(z\,|\,X_i,\theta), t_{\widehat{q}_\theta}^{q_\theta^{(k+1)}}(\theta)-\theta\big\rangle\,\dd\widehat{q}_\theta\cdot\int_\Theta\Big\langle\nabla\frac{\delta\Phi}{\delta\mu}(\widehat{q}_\theta, X_i)(z), t_{\widehat{q}_\theta}^{q_\theta^{(k)}}(\theta)-\theta\Big\rangle\,\dd\widehat{q}_\theta\bigg|\\
&\stackrel{\ri}{=} \bigg|\sum_{i=1}^n\sum_{z=1}^K \int_\Theta\big\langle\nabla\log p(z\,|\,X_i,\theta), t_{\widehat{q}_\theta}^{q_\theta^{(k+1)}}(\theta)-\theta\big\rangle\,\dd\widehat{q}_\theta\\
&\qquad\qquad\qquad\cdot \Phi(\widehat{q}_\theta, X_i)(z)\sum_{l=1}^K \Phi(\widehat{q}_\theta, X_i)(l)\int_\Theta\Big\langle\nabla\log\frac{p(z\,|\,X_i,\theta)}{p(l\,|\,X_i,\theta)}, t_{\widehat{q}_\theta}^{q_\theta^{(k)}}(\theta)-\theta\Big\rangle\,\dd\widehat{q}_\theta\bigg|\\
&\stackrel{\rii}{\leq} 2KW_2(q, \widehat{q}_\theta)W_2(\mu, \widehat{q}_\theta)\cdot\frac{1}{n}\sum_{i=1}^n\bigg[2\sqrt{\mb E_{\widehat{q}_\theta}\|\theta-\theta^\ast\|^2}S_1(X_i)\lambda(X_i)\\
&\qquad\qquad\qquad\qquad\qquad\qquad\qquad+ 
K\mb E_{\widehat{q}_\theta}\|\theta-\theta^\ast\|^2\lambda(X_i)^2\\
&\qquad\qquad\qquad\qquad\qquad\qquad\qquad+
2S_2(X_i)\Big(S_1(X_i) \sqrt{\mb E_{\widehat{q}_\theta}\|\theta-\theta^\ast\|^2} + \frac{K\lambda(X_i)}{2}\cdot \mb E_{\widehat{q}_\theta}\|\theta-\theta^\ast\|^2\Big)\bigg]\\
&\quad + n\big|\big\langle\Delta_{q_\theta^{(k)}}, \wht I_S(\theta^\ast)\Delta_{q_\theta^{(k+1)}}\big\rangle\big|\\
&\leq W_2(q_\theta^{(k+1)}, \widehat{q}_\theta)W_2(q_\theta^{(k)}, \widehat{q}_\theta)\sqrt{\mb E_{\widehat{q}_\theta}\|\theta-\theta^\ast\|^2}\\
&\qquad\qquad\qquad\qquad
\cdot\sum_{i=1}^n\bigg[K(2K+3)\lambda(X_i)^2+(2K+1)S_1(X_i)^2+(\frac{K}{2}+1)S_2(X_i)^2\bigg]\\
&\quad + n\big|\big\langle\Delta_{q_\theta^{(k)}}, \wht I_S(\theta^\ast)\Delta_{q_\theta^{(k+1)}}\big\rangle\big|\\
&\leq nW_2(q_\theta^{(k+1)}, \widehat{q}_\theta)W_2(q_\theta^{(k)}, \widehat{q}_\theta)\sqrt{\mb E_{\widehat{q}_\theta}\|\theta-\theta^\ast\|^2}\\
&\qquad
\cdot\bigg[K(2K+3) \Big(\frac{1}{n}\sum_{i=1}^n\lambda(X_i)^3\Big)^{\frac{2}{3}} + 
(2K+1)\Big(\frac{1}{n}\sum_{i=1}^nS_1(X_i)^3\Big)^{\frac{2}{3}} + (\frac{K}{2}+1)\frac{1}{n}\sum_{i=1}^nS_2(X_i)^2\bigg]\\
&\quad+
n\big|\big\langle\Delta_{q_\theta^{(k)}}, \wht I_S(\theta^\ast)\Delta_{q_\theta^{(k+1)}}\big\rangle\big|,
\end{align*}
where $\Delta_q$ is defined in Lemma \ref{lem: change_to_Deltaq}, step (i) follows by Corollary \ref{coro: subdifferential_is_variation}, and step (ii) is due to Lemma \ref{lem: missing_data_info} and the fact that $\mb E_{\widehat{q}_\theta}[\|\theta-\theta^\ast\|^2]\leq 1$.

\subsection{Proof of Lemma~\ref{lem: matrix_conc}}\label{proof_lem: matrix_conc}
Let $V_{1/4}$ be a $(1/4)$-covering of $B^d(0, 1)$ and we know $\log|V_{1/4}| \leq d\log 12$. By Lemma~\ref{lem: bound_opnorm_by_metric_entropy} we know
\begin{align*}
    \matnorm{\wht I_S(\theta^\ast) - I_S(\theta^\ast)} \leq 2\sup_{v\in V_\delta}\big|\langle v, (\wht I_S(\theta^\ast) - I_S(\theta^\ast))v\rangle\big|.
\end{align*}
So, we can bound
\begin{align*}
    \mb P\bigg(\matnorm{\wht I_S(\theta^\ast) - I_S(\theta^\ast)} > t\bigg) &\leq \mb P\bigg(\max_{v\in V_{1/4}}\big|\langle v^T\big(\wht I_S(\theta^\ast) - I_S(\theta^\ast)\big)v\big| > \frac{t}{2}\bigg)\\
    &\leq \sum_{v\in V_{1/4}}\mb P\Big(\big|v^T\big(\wht I_S(\theta^\ast) - I_S(\theta^\ast)\big)v\big| > \frac{t}{2}\Big)\\
    &\leq e^{d\log12}\cdot\sup_{v\in V_{1/4}}\mb P\Big(\big|v^T\big(\wht I_S(\theta^\ast) - I_S(\theta^\ast)\big)v\big| > \frac{t}{2}\Big).
\end{align*}
Notice that
\begin{align*}
    v^T\wht I_S(\theta^\ast) v &= \frac{1}{n}\sum_{i=1}^n\sum_{z=1}^K p(z\,|\, X_i, \theta^\ast)\Big(v^T\nabla\log p(z\,|\,X_i,\theta^\ast)\Big)^2
\end{align*}
is the sample average of $n$ i.i.d.~sub-exponential random variables. Moreover, for any $v\in B^d(0, 1)$, we have
\begin{align*}
    \bigg|\!\bigg|\sum_{z=1}^k p(z\,|\, X_i, \theta^\ast)\Big(v^T\nabla\log p(z\,|\,X_i,\theta^\ast)\Big)^2\bigg|\!\bigg|_{\psi_1} &\leq \bigg|\!\bigg|\sum_{z=1}^K \|v\|^2\|\nabla\log p(z\,|\,X_i,\theta^\ast)\|^2\bigg|\!\bigg|_{\psi_1}\\
    &\leq \big|\!\big|S_2(X_i)\big|\!\big|_{\psi_1} =\sigma_3 < \infty
\end{align*}
by Assumption~\ref{assump: continuity_of_Hessian}. Therefore, by Bernstein's inequality (Theorem 2.8.1 in \cite{vershynin2018high}), there exists a constant $C>0$ such that
\begin{align*}
    \mb P\Big(\big|v^T\big(\wht I_S(\theta^\ast) - I_S(\theta^\ast)\big)v\big| > \frac{t}{2}\Big) \leq 2\exp\bigg\{-Cn\min\Big(\frac{t^2}{\sigma_3^2}, \frac{t}{\sigma_3}\Big)\bigg\},\quad t>0.
\end{align*}
Since $d\log12 < 3d$, by combining all pieces above we get
\begin{align*}
    \mb P\bigg(\frac{1}{\sigma_3}\matnorm{\wht I_S(\theta^\ast) - I_S(\theta^\ast)} > t\bigg) \leq 2e^{3d - Cn\min\{t^2, t\}}.
\end{align*}

\subsection{Proof of Lemma~\ref{lem: diff_Phimu_thetastar}}\label{app:proof_diff_Phimu_thetastar}
For simplicity, we use the shorthand $A_z(\mu) = \int_\Theta\log p(z\,|\,x,\theta)\,\dd\mu(\theta)$, and let $h_z: \mb R^K\to [0, 1]$ be the function defined as
\begin{displaymath}
h_z(x_1, \cdots, x_K) = \frac{e^{x_z}}{e^{x_1} + \cdots + e^{x_K}}\quad \mbox{for all}\ \  z\in[K],
\end{displaymath}
and $A(\mu) = \big(A_1(\mu), \cdots, A_K(\mu)\big)\in\mb R^K$. Under these notations, we have
\begin{displaymath}
\Phi(\mu, x)(z) = \frac{\exp\{\int_\Theta\log p(z\,|\,x,\theta)\,\dd\mu(\theta)\}}{\sum_{z=1}^K \exp\{\int_\Theta\log p(z\,|\,x,\theta)\,\dd\mu(\theta)\}} = h_z(A(\mu)).
\end{displaymath}
By the mean value theorem, there is some $\xi\in\mb R^K$ such that
\begin{align*}
&\quad\, \big|\Phi(\mu, x)(z) - \Phi(\delta_{\theta^\ast}, x)(z)\big|\\
&= \big|h_z(A(\mu)) - h_z(A(\delta_{\theta^\ast}))\big|\\
&= \bigg|\sum_{k=1}^K\frac{\partial h_z}{\partial x_k}(\xi)\cdot(A_k(\mu) - A_k(\delta_{\theta^\ast}))\bigg|\\
&\stackrel{\ri}{\leq} \sum_{k=1}^K|A_k(\mu) - A_k(\delta_{\theta^\ast})|\\
&= \sum_{k=1}^K\bigg|\int_\Theta\log p(k\,|\,x,\theta)\,\dd\mu(\theta) - \log p(k\,|\,x,\theta^\ast)\bigg|\\
&\leq \sum_{k=1}^K \int_\Theta |\log p(k\,|\,x,\theta) - \log p(k\,|\,x,\theta^\ast)|\,\dd\mu(\theta)\\
&\stackrel{\rii}{\leq} \sum_{k=1}^K \int_\Theta \|\nabla\log p(k\,|\,x,\theta^\ast)\|\cdot\|\theta - \theta^\ast\| + \frac{\lambda(x)}{2}\|\theta-\theta^\ast\|^2\,\dd\mu(\theta)\\
&\stackrel{\riii}{\leq} \sum_{k=1}^K\|\nabla\log p(k\,|\,x,\theta^\ast)\|\cdot W_2(\mu, \delta_{\theta^\ast}) + \frac{K\lambda(x)}{2}\cdot W_2^2(\mu, \delta_{\theta^\ast}).
\end{align*}
Here, step (i) is due to the fact
\begin{displaymath}
\frac{\partial h_z}{\partial x_k} = \delta_{kz}h_z - h_zh_k \in [-1, 1],
\end{displaymath}
where $\delta_{zk}$ denotes the Kronecker function; step (ii) is because
\begin{align*}
&\quad\, \big|\log p(k\,|\,x,\theta) - \log p(k\,|\,x,\theta^\ast)\big|\\
&= \Big|\big\langle\nabla\log p(k\,|\,x,\theta^\ast), \theta - \theta^\ast\big\rangle + \frac{1}{2}\big\langle\nabla^2\log p(k\,|\,x,\theta')(\theta - \theta^\ast), \theta - \theta^\ast\big\rangle\Big|\\
&\leq \|\nabla\log p(k\,|\,x,\theta^\ast)\|\cdot\|\theta - \theta^\ast\| + \frac{\lambda(x)}{2}\|\theta-\theta^\ast\|^2
\end{align*}
for some $\theta'\in\mb R^d$ by applying the mean value theorem again; step (iii) is by Cauchy--Schwarz inequality, and the fact that
\begin{align*}
\int_\Theta \|\theta - \theta^\ast\|^2\,\dd\mu(\theta) = W_2^2(\mu, \delta_{\theta^\ast}).
\end{align*}

\subsection{Proof of Lemma~\ref{lem:B_t_bound}}\label{proof_lem:B_t_bound}
We will use the following lemma, which provides an upper bound of matrix operator norm by discretizing the unit sphere, and can be used to study the concentration property of sum of i.i.d. random matrices. The proof of the lemma can be found in \cite{vershynin2010introduction}.
\begin{lemma}\label{lem: bound_opnorm_by_metric_entropy}
Let $M \in\mb R^{d\times d}$ be a symmetric matrix, and $V_\delta$ be an $\delta$-covering of $B^d(0,1)$, then
\begin{displaymath}
\matnorm{M} \leq \frac{1}{1-2\delta}\sup_{v\in V_\delta}|\langle v, Mv\rangle|.
\end{displaymath}
\end{lemma}

By applying the triangular inequality, we obtain
\begin{align*}
\sup_{\substack{\theta\in\Theta\\ \mu: W_2(\mu, \delta_{\theta^\ast})\leq r}} &\, \matnorm{\nabla^2 U_n(\theta_{j(\theta)}, \mu) - \nabla^2 U(\theta_{j(\theta)}, \mu)}
= \sup_{\substack{j\in[N_\varepsilon]\\ \mu: W_2(\mu, \delta_{\theta^\ast}) \leq r}} \matnorm{\nabla^2 U_n(\theta_j, \mu) - \nabla^2 U(\theta_j, \mu)}\\
&\leq \sup_{\substack{j\in[N_\varepsilon]\\ \mu: W_2(\mu, \delta_{\theta^\ast}) \leq r}} \matnorm{\nabla^2 U_n(\theta_j ,\mu) - \nabla^2 U_n(\theta_j, \delta_{\theta^\ast})}\\
&\quad\qquad + \sup_{\substack{j\in[N_\varepsilon]\\ \mu: W_2(\mu, \delta_{\theta^\ast}) \leq r}} \matnorm{\nabla^2 U_n(\theta_j, \delta_{\theta^\ast}) - \nabla^2 U(\theta_j, \delta_{\theta^\ast})}\\
&\qquad \qquad+ \sup_{\substack{j\in[N_\varepsilon]\\ \mu: W_2(\mu, \delta_{\theta^\ast}) \leq r}} \matnorm{\nabla^2 U(\theta_j, \mu) - \nabla^2 U(\theta_j, \delta_{\theta^\ast})}.
\end{align*}
Therefore
\begin{align*}
\mb P(B_t) &= \mb P\bigg(\sup_{\substack{\theta\in\Theta\\ \mu: W_2(\mu, \delta_{\theta^\ast})\leq r}}\matnorm{\nabla^2U_n(\theta_{j(\theta)}, \mu) - \nabla^2U(\theta_{j(\theta)}, \mu)} > \frac{t}{3}\bigg)\\
&\leq \mb P\bigg(\sup_{\substack{j\in[N_\varepsilon]\\ \mu: W_2(\mu, \delta_{\theta^\ast}) \leq r}} \matnorm{\nabla^2 U_n(\theta_j ,\mu) - \nabla^2 U_n(\theta_j, \delta_{\theta^\ast})} > \frac{t}{9}\bigg)\\
&\quad + \mb P\bigg(\sup_{\substack{j\in[N_\varepsilon]\\ \mu: W_2(\mu, \delta_{\theta^\ast}) \leq r}} \matnorm{\nabla^2 U_n(\theta_j, \delta_{\theta^\ast}) - \nabla^2 U(\theta_j, \delta_{\theta^\ast})} > \frac{t}{9}\bigg)\\
&\quad + \mb P\bigg(\sup_{\substack{j\in[N_\varepsilon]\\ \mu: W_2(\mu, \delta_{\theta^\ast}) \leq r}} \matnorm{\nabla^2 U(\theta_j, \mu) - \nabla^2 U(\theta_j, \delta_{\theta^\ast})} > \frac{t}{9}\bigg)\\
&=: \mb P(B^1_t) + P(B^2_t) + P(B^3_t). 
\end{align*}
By definition, the quantity inside event $B_t^1$ can be bounded as
\begin{align*}
&\quad\, \sup_{\substack{j\in[N_\varepsilon]\\ \mu: W_2(\mu, \delta_{\theta^\ast}) \leq r}} \matnorm{\nabla^2 U_n(\theta_j, \mu) - \nabla^2 U_n(\theta_j, \delta_{\theta^\ast})}\\
&= \sup_{\substack{j\in[N_\varepsilon]\\ \mu: W_2(\mu, \delta_{\theta^\ast}) \leq r}} \matnorm{\frac{1}{n}\sum_{i=1}^n\sum_{z=1}^K\nabla^2\log p(X_i,z\,|\,\theta_j)\big[\Phi(\mu, X_i)(z) - \Phi(\delta_{\theta^\ast}, X_i)(z)\big]}\\
&= \sup_{\substack{j\in[N_\varepsilon]\\ \mu: W_2(\mu, \delta_{\theta^\ast}) \leq r}} \matnorm{\frac{1}{n}\sum_{i=1}^n\sum_{z=1}^K\nabla^2\log p(z\,|\,X_i,\theta_j)\big[\Phi(\mu, X_i)(z) - \Phi(\delta_{\theta^\ast}, X_i)(z)\big]}\\
&\leq \sup_{\substack{j\in[N_\varepsilon]\\ \mu: W_2(\mu, \delta_{\theta^\ast}) \leq r}} \frac{1}{n}\sum_{i=1}^n\sum_{z=1}^K\matnorm{\nabla^2\log p(z\,|\,X_i, \theta_j)}\cdot |\Phi(\mu, X_i)(z) - \Phi(\delta_{\theta^\ast}, X_i)(z)|\\
&\leq \sup_{\mu: W_2(\mu, \delta_{\theta^\ast})\leq r} \frac{1}{n}\sum_{i=1}^n\sum_{z=1}^K\lambda(X_i)|\Phi(\mu, X_i)(z) - \Phi(\delta_{\theta^\ast}, X_i)(z)|\\
&\stackrel{\ri}{\leq} \sup_{\mu: W_2(\mu, \delta_{\theta^\ast})\leq r}\frac{1}{n}\sum_{i=1}^n\sum_{z=1}^K\lambda(X_i)\cdot\bigg(\sum_{k=1}^K\|\nabla\log p(k\,|\,X_i, \theta^\ast)\|\cdot W_2(\mu, \delta_{\theta^\ast}) + \frac{K\lambda(X_i)}{2}\cdot W_2^2(\mu, \delta_{\theta^\ast})\bigg)\\
&\leq \frac{Kr}{n}\sum_{i=1}^n\sum_{k=1}^K\lambda(X_i)\|\nabla\log p(k\,|\,X_i,\theta^\ast)\| + \frac{K^2r^2}{2n}\sum_{i=1}^n\lambda(X_i)^2\\
&\leq \frac{Kr(Kr+1)}{2n}\sum_{i=1}^n\lambda(X_i)^2 + \frac{Kr}{2n}\sum_{i=1}^n\Big(\sum_{k=1}^K\|\nabla\log p(k\,|\,X_i,\theta^\ast)\|\Big)^2\\
&\leq \frac{Kr(Kr+1)}{2}\Big(\frac{1}{n}\sum_{i=1}^n\lambda(X_i)^3\Big)^\frac{2}{3} + \frac{Kr}{2}\Big(\frac{1}{n}\sum_{i=1}^nS_1(X_i)^3\Big)^\frac{2}{3}\\
&\stackrel{\rii}{\leq} \frac{(Kr+1)^2}{2}\bigg[\Big(\mb E_{\theta^\ast}[\lambda(X)^3] + 1\Big)^\frac{2}{3} + \Big(\mb E_{\theta^\ast}[S_1(X)^3] + 1\Big)^\frac{2}{3}\bigg]
\end{align*}
with probability at least $1- 2e^{-Cn^\frac{1}{6}\sigma_2^{-1}} - 2e^{-Cn^\frac{1}{6}\sigma_3^{-1}}$. Here step (i) is by applying lemma \ref{lem: diff_Phimu_thetastar}; step (ii) is because $\|\lambda(X)\|_{\psi_1} = \sigma_2$ and $\|S_1(X)\|_{\psi_1} < \|S_2(X)\|_{\psi_1} = \sigma_3$ yield $\|\lambda(X)^3\|_{\psi_{1/3}} = \sigma_2^3$ and $\|S_1(X)\|_{\psi_{1/3}} \leq \sigma_3^3$, so that step (ii) follows by applying Lemma \ref{lem: tail_iidsum} with $\alpha = \frac{1}{3}$. 
\smallskip

\noindent Similarly, we can show that the quantity inside the deterministic event $B_t^3$ can be bounded as
\begin{align*}
\sup_{\substack{j\in[N_\varepsilon]\\ \mu: W_2(\mu, \delta_{\theta^\ast}) \leq r}} \matnorm{\nabla^2 U(\theta_j, \mu) - \nabla^2 U(\theta_j, \delta_{\theta^\ast})}
&\leq \frac{(Kr+1)^2}{2}\bigg[\Big(\mb E_{\theta^\ast}[\lambda(X)^3]\Big)^\frac{2}{3} + \Big(\mb E_{\theta^\ast}[S_1(X)^3]\Big)^\frac{2}{3}\bigg].
\end{align*}
By combining the two preceding displays together, we obtain that if
\begin{align*}
    r\leq \frac{1}{K}\bigg(\sqrt{\frac{2t}{9\big[\big(\mb E_{\theta^\ast}[\lambda(X)^3] + 1\big)^\frac{2}{3} + \big(\mb E_{\theta^\ast}[S_1(X)^3]+1\big)^\frac{2}{3}\big]}} - 1\bigg)
\end{align*}
then
\begin{displaymath}
\mb P(B_t^1) \leq 2e^{-Cn^\frac{1}{6}\sigma_2^{-1}} + 2e^{-Cn^\frac{1}{6}\sigma_3^{-1}}, \quad\mbox{and}\quad \mb P(B_t^3)=0.
\end{displaymath}
\smallskip

\noindent
Lastly, let us bound $\mb P(B_t^2)$, which requires matrix concentration and a uniform control of the difference over the $\varepsilon$-net $\{\theta_j\}_{j=1}^{N_\varepsilon}$. Let $V_{1/4}$ be a $(1/4)$-covering of $B^d(0, 1)$. By applying the definition of $U_n$, $U$, and Lemma~\ref{lem: bound_opnorm_by_metric_entropy}, we obtain
\begin{align*}
&\quad\,\,
\sup_{\substack{j\in[N_\varepsilon]\\ \mu: W_2(\mu, \delta_{\theta^\ast})\leq r}}\matnorm{\nabla^2U(\theta_j,\delta_{\theta^\ast}) - \nabla^2U_n(\theta_j, \delta_{\theta^\ast})}\\
&= \sup_{j\in[N_\varepsilon]}\matnorm{\nabla^2 U_n(\theta_j,\delta_{\theta^\ast}) - \nabla^2U(\theta_j, \delta_{\theta^\ast})}\\
&= \sup_{j\in[N_\varepsilon]}\matnorm{\frac{1}{n}\sum_{i=1}^n\bigg[\sum_{z=1}^Kp(z\,|\,X_i,\theta^\ast)\nabla^2\log p(X_i,z\,|\,\theta_j) - \mb E\sum_{z=1}^Kp(z\,|\,X_i,\theta^\ast)\nabla^2\log p(X_i,z\,|\,\theta_j)\bigg]}\\
&\stackrel{\ri}{\leq} 2 \sup_{\substack{j\in[N_\varepsilon]\\ v\in V_{1/4}}}\bigg|\bigg\langle v, \frac{1}{n}\sum_{i=1}^n\bigg[\sum_{z=1}^Kp(z\,|\,X_i,\theta^\ast)\nabla^2\log p(X_i,z\,|\,\theta_j) - \mb E\sum_{z=1}^Kp(z\,|\,X_i,\theta^\ast)\nabla^2\log p(X_i,z\,|\,\theta_j)\bigg]v\bigg\rangle\bigg|\\
&= 2\sup_{\substack{j\in[N_\varepsilon]\\ v\in V_{1/4}}}\bigg|\frac{1}{n}\sum_{i=1}^n\sum_{z=1}^K \big\langle v, Q_z(X_i,\theta_j)v\big\rangle\bigg|,
\end{align*}
where we have used in the following shorthand in the last line,
\begin{displaymath}
Q_z(X_i, \theta_j) = p(z\,|\,X_i,\theta^\ast)\nabla^2\log p(X_i,z\,|\,\theta_j) - \mb E p(z\,|\,X_i,\theta^\ast)\nabla^2\log p(X_i,z\,|\,\theta_j).
\end{displaymath}
Therefore, we can bound $\mb P(B_t^2)$ by a union bound argument as
\begin{align*}
\mb P(B_t^2) 
&\leq \mb P\bigg(\sup_{\substack{j\in[N_\varepsilon]\\ v\in V_{1/4}}}\bigg|\frac{1}{n}\sum_{i=1}^n\sum_{z=1}^K \big\langle v, Q_z(X_i,\theta_j)v\big\rangle\bigg| > \frac{t}{18}\bigg)\\
&\leq \sum_{j\in[N_\varepsilon]}\sum_{v\in V_{1/4}}\mb P\bigg(\bigg|\frac{1}{n}\sum_{i=1}^n\sum_{z=1}^K \big\langle v, Q_z(X_i,\theta_j)v\big\rangle\bigg| > \frac{t}{18}\bigg)\\
&\leq N_\varepsilon\big|V_{1/4}\big|\sup_{\substack{j\in[N_\varepsilon]\\ v\in V_{1/4}}}\mb P\bigg(\bigg|\frac{1}{n}\sum_{i=1}^n\sum_{z=1}^K \big\langle v, Q_z(X_i,\theta_j)v\big\rangle\bigg| > \frac{t}{18}\bigg)\\
&\stackrel{\ri}{\leq} e^{d\log\frac{3R}{\varepsilon}}\cdot e^{d\log12}\cdot\sup_{\substack{j\in[N_\varepsilon]\\ v\in V_{1/4}}}\mb P\bigg(\bigg|\frac{1}{n}\sum_{i=1}^n\sum_{z=1}^K \big\langle v, Q_z(X_i,\theta_j)v\big\rangle\bigg| > \frac{t}{18}\bigg)\\
&\stackrel{\rii}{\leq} e^{d\log\frac{3R}{\varepsilon}}\cdot e^{d\log12}\cdot2\exp\Big\{-\frac{Cn^\frac{1}{2}t}{\sigma_1}\Big\}\\
&= 2 \exp\bigg\{d\log\frac{36R}{\varepsilon} -\frac{Cn^\frac{1}{2}t}{\sigma_1}\bigg\},
\end{align*}
where in step (i) we used $\log N_\varepsilon \leq d\log \frac{3R}{\varepsilon}$ and $\log|V_\delta| \leq d\log\frac{3}{\delta}$ (see \cite{vershynin2010introduction} for a proof), and step (ii) is by applying Lemma~\ref{lem: tail_iidsum} with $\alpha=1$.

\section{Concentration inequalities and Orlicz norm}\label{Appendix: conc}
In this appendix, we briefly review the some commonly used concentration inequalities and the notion of Orlicz norm of a random variable.

For any $\alpha\geq 0$, define function $\psi_\alpha(x)=e^{x^\alpha}-1$ for $x\geq 0$. It is easy to verify that $\psi_\alpha$ is convex when $\alpha\geq 1$ and non-convex otherwise. For $0<\alpha < 1$, we can consider a convex modification $\tilde{\psi}_\alpha$ of $\psi_\alpha$, defined as
\begin{align}\label{eqn: convex_psi}
\begin{aligned}
\tilde{\psi}_\alpha(x) = 
\begin{cases}        
\psi_\alpha(x) & x\geq x_\alpha\\ 
\frac{\psi_\alpha(x_\alpha)}{x_\alpha}x & 0\leq x\leq x_\alpha
\end{cases}
\end{aligned}
\end{align}
for some sufficiently large $x_\alpha > 0$ (e.g. we can take $x_\alpha = \sqrt[\alpha]{-\log\alpha}$), so that $\tilde{\psi}_\alpha$ is a convex function on $\mathbb{R}_{\geq 0}$. For a random variable $X$, define its Orlicz norm with respect to $\psi_\alpha$ and $\tilde \psi_\alpha$, respectively, as
\begin{align*}
    \|X\|_{\psi_\alpha} = \inf\{t: \mb E\psi_\alpha(|X|/t) \leq 1\}, \quad \mbox{and}\quad \|X\|_{\tilde{\psi}_\alpha} = \inf\{t: \mb E\tilde{\psi}_\alpha(|X|/t) \leq 1\}.
\end{align*}
It can be shown that $\|\cdot\|_{\psi_\alpha}$ and $\|\cdot\|_{\tilde{\psi}_\alpha}$ are equivalent norms for $\alpha\in(0, 1)$, i.e. there exists a constant $C_\alpha > 1$ depending only on $\alpha$ such that for any random variable $\xi$,
$$
C_\alpha^{-1}\|\xi\|_{\tilde{\psi}_\alpha} \leq \|\xi\|_{\psi_\alpha}\leq C_\alpha\|\xi\|_{\tilde{\psi}_\alpha}.
$$
See Lemma C.2 in \cite{chen2019randomized} for a proof. For $\alpha\in(0,1)$, the triangular inequality does not hold for $\psi_\alpha$, which is why we introduce the modification. When $\alpha<1$, we only have 
$$
\Big|\!\Big|\sum_{i=1}^nX_i\Big|\!\Big|_{\psi_\alpha}^\alpha \leq \sum_{i=1}^n\|X_i\|_{\psi_\alpha}^\alpha.
$$
The Orlicz-norm with respect to $\psi_\alpha$ for each $\alpha>0$ characterizes the tail probability of a random variable by the following inequality,
\begin{align}\label{eqn: tail_orlicz_norm}
    \mb P(|X| > t) \leq 2e^{-\frac{t^\alpha}{\|X\|_{\psi_\alpha}^\alpha}}.
\end{align}
For a sum of i.i.d.~random variables with finite $\psi_\alpha$-norm with $\alpha\in(0, 1]$, we have the following concentration property. The result for $\alpha > 1$ can be shown through a similar argument. The only difference is to substitute $\big|\!\big|\max_{1\leq i\leq n}|X_i|\big|\!\big|_{\tilde{\psi}_\alpha}$ by $\big(\sum_i\|X_i\|_{\psi_\alpha}^{\alpha'}\big)^{1/\alpha'}$ in the proof, where $1/\alpha + 1/\alpha' = 1$.
\begin{lemma}\label{lem: tail_iidsum}
For any positive integer $n$ such that $\sqrt{n} \geq \sqrt[\alpha]{\log(n+1)}$ and $n \geq e^{x_\alpha^\alpha}-1$, where $x_\alpha$ is defined in (\ref{eqn: convex_psi}), we have
\begin{align*}
    \mathbb{P}\bigg(\Big|\frac{1}{n}\sum_{i=1}^n\big(X_i - \mathbb{E}X_i\big)\Big| \geq t\bigg) \leq 2 \exp\Big\{-\frac{A_\alpha n^\frac{\alpha}{2}t^\alpha}{\max_{i}\|X_i\|_{\psi_\alpha}^\alpha}\Big\}
\end{align*}
for some constant $A_\alpha > 0$ only depending on $\alpha\in(0, 1]$. Moreover, we have $A_1 = 4\sqrt{n}\sqrt[4]{\pi}$.
\end{lemma}
\begin{proof}
By Theorem 6.21 in \cite{ledoux1991probability}, there is a constant $K_\alpha$ such that
$$
\Big|\!\Big|\sum_{i=1}^nX_i - \mathbb{E}X_i\Big|\!\Big|_{\tilde{\psi}_\alpha} \leq K_\alpha\bigg(\mathbb{E}\Big|\sum_{i=1}^n X_i - \mathbb{E}X_i\Big| + \big|\!\big|\max_{1\leq i\leq n}|X_i|\big|\!\big|_{\tilde{\psi}_\alpha}\bigg).
$$
To bound the first term, using symmetrization argument yields
$$
\mathbb{E}\Big|\sum_{i=1}^n X_i - \mathbb{E}X_i\Big| \leq 2\mathbb{E}_{\varepsilon, X}\Big|\sum_{i=1}^n\varepsilon_iX_i\Big| \leq 2\sqrt{\sum_{i=1}^n\mathbb{E}X_i^2}.
$$
Here, $\varepsilon_i$ are i.i.d. Rademacher random variables. The first inequality is by a standard symmetrization argument, and the second inequality is by Cauchy-Schwarz’s inequality. Recall that, we have
$$
\|X\|_p^p \leq 2\Gamma\Big(\frac{\alpha}{p}+1\Big)\|X\|_{\psi_\alpha}^p
$$
for any $p, \alpha > 0$ and random variable $X$. Take $p=2$ and we have 
$$
\mathbb{E}X_i^2 \leq 2\Gamma\Big(\frac{\alpha}{2}+1\Big) \|X_i\|_{\psi_\alpha}^2 \leq 2\Gamma\Big(\frac{\alpha}{2}+1\Big)C_\alpha^2\|X_i\|_{\tilde{\psi}_\alpha}^2.
$$
To bound the second term, by Lemma 8.2 in \cite{kosorok2008introduction}, there is a constant $B_\alpha$ s.t.
$$
\big\|\max_{1\leq i\leq n}|X_i|\big\|_{\tilde{\psi}_\alpha} \leq B_\alpha\tilde{\psi}_{\alpha}^{-1}(n)\max_{1\leq i\leq n}\|X_i\|_{\tilde{\psi}_\alpha}.
$$
Thus, when $n>e^{x_\alpha^\alpha} - 1$ we have
\begin{align*}
\Big|\!\Big|\sum_{i=1}^nX_i - \mathbb{E}X_i\Big|\!\Big|_{\tilde{\psi}_\alpha} &\leq K_\alpha\bigg(\sqrt{8\Gamma\big(\frac{\alpha}{2} + 1\big)C_\alpha^2}\cdot\sqrt{n} + B_\alpha\tilde{\psi}_\alpha^{-1}(n)\bigg)\max_{1\leq i\leq n}\|X_i\|_{\tilde{\psi}_\alpha}\\ &\leq K_\alpha\bigg(\sqrt{8\Gamma\big(\frac{\alpha}{2} + 1\big)C_\alpha^2}\cdot\sqrt{n} + B_\alpha\sqrt[\alpha]{\log(n+1)}\bigg)\max_{1\leq i\leq n}\|X_i\|_{\tilde{\psi}_\alpha}\\&\leq K_\alpha'\max\big\{\sqrt{n}, \sqrt[\alpha]{\log(n+1)}\}\max_{1\leq i\leq n}\big\|X_i\|_{\tilde{\psi}_\alpha}.
\end{align*}
This implies
$$
\Big\|\sum_{i=1}^nX_i - \mathbb{E}X_i\Big\|_{\psi_\alpha} \leq C_\alpha^2K_\alpha'\max\big\{\sqrt{n}, \sqrt[\alpha]{\log(n+1)}\big\}\max_{1\leq i\leq n}\|X_i\|_{\psi_\alpha}
$$
and for $n$ large enough such that $\sqrt{n} \geq \sqrt[\alpha]{\log(n+1)}$, we derive the desiring result by applying (\ref{eqn: tail_orlicz_norm}).
\end{proof}
In the proof of main theorems, we use this result to derive the deviation inequality of i.i.d. sum of high-order moments of $S_1(X)$, $S_2(X)$, and $\lambda(X)$. Since $S_2(X)$ and $\lambda(X)$ have finite $\psi_1$-norm by Assumption \ref{assump: continuity_of_Hessian}, their high-order moments will have finite $\psi_\alpha$-norm for $0<\alpha<1$. This Lemma also allow us to weaken the assumption of finite $\psi_1$-norm to any $\psi_\alpha$-norm for some $\alpha > 0$.

Next two lemma provide an explicit tail bound of Gamma distribution. We use them to control the difference $|\mb E_{\wht q_\theta}f(\theta) - f(\theta^\ast)|$ when $f$ has polynomial growth and $\wht q_\theta$ concentrates around $\theta^\ast$.

\begin{lemma}\label{lem: tale_of_subgamma}
Let $X$ be a random variable with mean $\mu$ such that its cumulant generating function satisfies
\begin{displaymath}
\log\mb Ee^{\lambda(X-\mu)} \leq \frac{v\lambda^2}{2(1-a\lambda)}
\end{displaymath}
for some positive constant $v, a$ and $\lambda < \frac{1}{a}$. Then we have
\begin{displaymath}
\mb P(X > \mu + t) \leq e^{-\frac{t^2}{2(v+at)}}, \quad t > 0.
\end{displaymath}
\end{lemma}
\begin{proof}
It is easy to see that
\begin{align*}
\log\mb P(X-\mu > t) \leq -\lambda t + \frac{v\lambda^2}{2(1-a\lambda)}
\end{align*}
for any $0 < \lambda < \frac{1}{a}$. By letting $\lambda = \frac{t}{v+at}$, we have
\begin{align*}
-\lambda t + \frac{v\lambda^2}{2(1-a\lambda)} = -\frac{t^2}{v + at} + \frac{v\big(\frac{t}{v+at}\big)^2}{2 - 2a\cdot\frac{t}{v+at}} = -\frac{t^2}{2(v+at)}.
\end{align*}
\end{proof}

\begin{lemma}\label{lem: tail_of_gamma}
If $X\sim Ga(\alpha, \beta)$, then
\begin{displaymath}
\mb P\Big(X - \frac{\alpha}{\beta} > t\Big)\leq e^{-\frac{\beta^2t^2}{2(\alpha + \beta t)}}.
\end{displaymath}
Moreover, for $r > \frac{\alpha}{\beta}$, we have
\begin{displaymath}
\int_r^{\infty}e^{-\beta x}x^{\alpha - 1}\,\dd x \leq \frac{\Gamma(\alpha)}{\beta^{\alpha}}e^{-\frac{(\beta r - \alpha)^2}{2\beta r}}.
\end{displaymath}
\end{lemma}
\begin{proof}
Notice that
\begin{align*}
\log\mb Ee^{\lambda(X - \frac{\alpha}{\beta})}
&= - \frac{\lambda\alpha}{\beta} - \alpha\log\Big(1 - \frac{\lambda}{\beta}\Big)\\
&= \alpha\Big[-\log(1 - \frac{\lambda}{\beta}) - \frac{\lambda}{\beta}\Big]\\
&\leq \alpha\cdot\frac{\lambda^2/\beta^2}{2(1-\lambda/\beta)}\\
&= \frac{\frac{\alpha}{\beta^2}\cdot\lambda^2}{2(1-\frac{1}{\beta}\cdot\lambda)}.
\end{align*}
Here, we use the inequality $-\log(1-x) - x \leq \frac{x^2}{2(1-x)}$. By lemma \ref{lem: tale_of_subgamma}, we have
\begin{displaymath}
\mb P\Big(X - \frac{\alpha}{\beta} > t\Big) \leq \exp\Big\{-\frac{t^2}{2(\frac{\alpha}{\beta^2} + \frac{t}{\beta})}\Big\} = e^{-\frac{\beta^2t^2}{2(\alpha + \beta t)}}.
\end{displaymath}
With this inequality in hand, we can see that
\begin{align*}
\int_{r}^{\infty}e^{-\beta x}x^{\alpha - 1}\,\dd x
&= \frac{\Gamma(\alpha)}{\beta^{\alpha}}\cdot \int_{r}^{\infty}\frac{\beta^{\alpha}}{\Gamma(\alpha)}e^{-\beta x}x^{\alpha - 1}\,\dd x\\
&= \frac{\Gamma(\alpha)}{\beta^{\alpha}}\mb P\bigg(Ga(\alpha, \beta) - \frac{\alpha}{\beta} > r - \frac{\alpha}{\beta}\bigg)\\
&\leq \frac{\Gamma(\alpha)}{\beta^{\alpha}}\cdot\exp\bigg\{\frac{\beta^2(r-\frac{\alpha}{\beta})^2}{2(\alpha + \beta(r-\frac{\alpha}{\beta}))}\bigg\}\\
&= \frac{\Gamma(\alpha)}{\beta^{\alpha}}\cdot e^{-\frac{(\beta r - \alpha)^2}{2\beta r}}.
\end{align*}
\end{proof}

\section{other materials}
\subsection{Particle approximation} \label{sec:particle_approx}
Although the minimization movement scheme is mathematically appealing and leads to exponential convergence when the objective functional is convex along generalized geodesics (c.f.~Section~\ref{sec:con_one_step}), numerically computing it may require extra efforts as~\eqref{eqn: JKO_update_qtheta} generally does not admit an explicit solution, similar to the implicit Euler scheme for approximating gradient flows in the Euclidean space. In this subsection, we discuss two particle approximation methods for numerically realizing the minimization movement scheme. To simplify the notation, we consider the following one-step scheme with step size $\tau>0$,
\begin{align}\label{eqn:imp_scheme}
    \rho_\tau =\argmin_{\nu\in\PX} \m F_{\rm KL}(\nu) + \frac{1}{2\tau} W_2^2(\nu,\rho),
\end{align}
for a generic KL divergence functional~\eqref{eqn: functional_of_interest} with potential $V$ that includes the sample energy functional $V_n(\cdot\,|\,q_{\theta}^{(k)})$ with sample potential $U_n(\cdot,\,q_{\theta}^{(k)})$ in MF-WGF as a special case.

The first approach is to instead use the following explicit scheme with step size $\tau$:
\begin{align}\label{eqn:exp_scheme}
    \rho_{\tau}^{\rm ex} = \Big[\,\id - \tau \underbrace{\nabla \frac{\delta \m F_{\rm KL}}{\delta \rho}(\rho)}_{-\mx{\small velocity}}\Big]_\# \rho,
\end{align}
which corresponds to applying the usual one-step explicit Euler scheme to solve the ODE~\eqref{eqn:flow_ODE} that defines the particle flow.
In comparison, by using the first order optimality condition of~\eqref{eqn:imp_scheme}, the implicit scheme can also be equivalently written as
\begin{align*}
    \rho = \Big[\,\id + \tau\nabla\frac{\delta\m F_{\rm KL}}{\delta\rho}(\rho_\tau)\Big]_\#\rho_\tau,
\end{align*}
which requires solving for $\rho_\tau$ in a distributional equation, and corresponds to the one-step implicit Euler scheme for solving the flow ODE~\eqref{eqn:flow_ODE}.

Notice that when step size $\tau$ is small, we can approximate the inverse of the optimal transport map $T_{\rho_\tau}^\rho = \id + \tau\nabla\frac{\delta\m F_{\rm KL}}{\delta\rho}(\rho_\tau)$ as $\id - \tau\nabla\frac{\delta\m F_{\rm KL}}{\delta\rho}(\rho_\tau)$; therefore, we have $\rho_\tau \approx \big[\,\id-\tau\nabla\frac{\delta\m F_{\rm KL}}{\delta\rho}(\rho_\tau)\big]_\#\rho\approx \big[\,\id-\tau\nabla\frac{\delta\m F_{\rm KL}}{\delta\rho}(\rho)\big]_\#\rho$ by applying the inverse of the optimal transport map to both sides of the preceding implicit scheme formula. As a consequence, the two schemes are first-order equivalent (relative to the step size $\tau$).
The same heuristics can also be used to motivate a theoretical analysis for bounding the difference between the two schemes. In practice, one can realize the explicit scheme~\eqref{eqn:exp_scheme} via particle approximation: if one has a collection of $N$ particles $\{\theta_{\ell}\}_{\ell=1}^N$ approximately sampling from $\rho$ and an estimator $\wht T$ of subdifferential $\nabla \frac{\delta \m F_{\rm KL}}{\delta \rho}(\rho)$ based on the particles (e.g.~plug-in estimator with a kernel density estimator of $\rho$),
then the transformed particles $\big\{\theta_{\ell} - \tau\, \wht T(\theta_{\ell})\big\}_{\ell=1}^N$ approximately form a sample from $\rho_{\tau}^{\rm ex}$. Unfortunately, since the subdifferential usually depends on $\rho$ (or its higher-order derivatives) in a complicated manner, e.g.~$\nabla \frac{\delta \m F_{\rm KL}}{\delta \rho}(\rho) = \nabla V + \nabla\log \rho$, the plug-in estimator may suffer from low accuracy.

For the KL divergence functional, one may use the relation between the Fokker-Planck equation~\eqref{eqn: Fokker_Planck_equation} and the Langevin SDE~\eqref{eqn: Langevin_dynamics} to motivate a second approach. Although this approach is not applicable to functionals beyond $\m F_{\rm KL}$, it does not require explicit estimation of any density function and therefore is more accurate. More precisely, we will use the one-step discretization of the Langevin SDE with step size $\tau$, that is, the distribution of $Y$ obtained by 
\begin{align}\label{eqn:SDE_one_step}
    Y= X - \tau \nabla V(X) + \sqrt{2\tau} \,\eta, \quad\mx{with }\ X\sim \rho \ \ \mx{and}\ \ \eta\sim N(0,I_d), 
\end{align}
to approximate the solution $\rho_\tau$ in~\eqref{eqn:imp_scheme}, which can also be easily implemented via particle approximation. The Langevin SDE approximation~\eqref{eqn:SDE_one_step} to $\rho_\tau$ can also be written as
\begin{align*}
    \rho_{\tau}^L=\big\{\,[\,\id - \tau \nabla V]_\# \rho\,\big\} \ast N(0,2\tau I_d),
\end{align*}
where $\ast$ denotes the (distribution) convolution operator,
can also be interpreted as from the operator splitting technique in optimization~\cite{,parikh2014proximal} and numerical PDE~\cite{macnamara2016operator}. To see this, the velocity field (as an operator) $v=-\nabla \frac{\delta \m F_{\rm KL}}{\delta \rho}(\rho)=-\nabla V -\nabla \log \rho$ in the ODE~\eqref{eqn:flow_ODE} characterizing the flow of each particle can be split into the sum of two simpler operators: $-\nabla V$ can be realized by a one-step gradient descent; and $-\nabla \log \rho$ can be realized by injecting one-step pure diffusion (i.e.~Brownian motion).
We choose the second approach based on SDE or operator splitting as our default method for realizing the minimization movement scheme in MF-WGF. For a general funcitonal $\m F$, we may also apply the same operator splitting trick if its subdifferential can be written as a sum of easy-handling operators. Algorithm~\ref{algo: MFWGF_particle} summarizes the full implementation of MF-WGF using particle approximation.
\begin{algorithm}[ht]
\caption{MF-WGF via particle approximation}
\label{algo: MFWGF_particle}
\KwData{Prior distribution $\pi_\theta$, data set $\{X_i\}_{i=1}^n$, number of latent labels $K$, number of particles $B$, number of iterations $T$, and step size $\tau$}
\KwResult{Estimation of mean-field variational approximation $\wht Q_\theta$}
Initialize particles $\theta_1^{(0)}, \ldots, \theta_B^{(0)}$ by sampling i.i.d.~from some initial distribution $q_\theta^{(0)}$ over $\Theta\subseteq\mb R^d$\;
\For{$t\leftarrow 0$ \KwTo $T-1$}{
    \tcp*[h]{Updating $q_{Z_i}^{(t)}$.}
    
    \For{$k\leftarrow 1$ \KwTo $K$}{
        \For{$i\leftarrow 1$ \KwTo $n$}{
            Compute $\mb E_{q_\theta^{(t)}}\log p(k\,|\,X_i,\theta) := B^{-1}\sum_{b=1}^B\log p(k\,|\,X_i,\theta_b^{(t)})$ via Monte Carlo approximation\;
            Update $\Phi(q_\theta^{(t)}, X_i)(k) := \frac{\exp\big\{\mb E_{q_\theta^{(t)}}\log p(k\,|\,X_i,\theta)\big\}}{\sum_{z=1}^K \exp\big\{\mb E_{q_\theta^{(t)}}\log p(z\,|\,X_i,\theta)\big\}}$
        }
    }
    \tcp*[h]{Compute the drift term in the Langevin SDE.}
    Compute $\nabla V(\theta) := -\sum_{i=1}^n\sum_{k=1}^K\nabla\log p(X_i, z\,|\,\theta)\, \Phi\big(q_\theta^{(t)}, X_i\big)(k) - \nabla\log\pi_\theta(\theta)$\;
    \For{$b\leftarrow 1$ \KwTo $B$}{
    \tcp*[h]{Updating particles whose empirical measure forms $q_{\theta}^{(t)}$.}
    
        Sample $\eta_b\sim \m N(0, I_d)$\;
        Update $\theta_b^{(t+1)} = \theta_b^{(t)} - \tau\nabla V(\theta_b^{(t)}) + \sqrt{2\tau}\eta_b$\;
    }
}
\tcp*[h]{Output the empirical measure of particles.}

Output $\wht {Q}_\theta(A) = B^{-1}\sum_{b=1}^B \delta_{\theta_b^{(T)}}$ as the empirical measure of particles $\big\{\theta_b^{(T)}\big\}_{b=1}^B$.
\end{algorithm}

The following lemma provides an error estimate by using the two numerical schemes to approximate $\rho_\tau$ from the JKO scheme as in~\eqref{eqn:imp_scheme}. More theoretical analysis about the long term cumulative error of using the discretized Langevin SDE method for approximating the JKO scheme can be found in Theorem \ref{thm: discrete_langevin_vs_JKO} in Appendix~\ref{app:long_term}.

\begin{lemma}[One-step numerical error]\label{lem: onestep_err}
Let $\rho_{\tau}^{\rm FP}$ be the solution of Fokker--Planck equation (\ref{eqn: Fokker_Planck_equation}) at time $\tau$ with initial density $\rho_0 = \rho\in\ms P_2^r(\mb R^d)$, and $\rho_\tau^{\rm ex}$, $\rho_\tau$, $\rho_\tau^L$ the three one-step schemes described earlier, namely, one-step explicit scheme, JKO scheme, and discretized Langevin SDE scheme. If $\nabla V$ is $L$-Lipschitz, 
then the following numerical error estimates hold
\begin{align*}
    W_2^2(\rho_\tau^{\rm FP}, \rho_\tau^L) \leq C_1\tau^3, \quad W_2^2(\rho_\tau^{\rm FP}, \rho_\tau) \leq C_2\tau^3, \quad\mbox{and}\quad W_2^2(\rho_\tau^{\rm ex}, \rho_\tau) \leq C_3\tau^4
\end{align*}
for sufficiently small $\tau$,
where $C_1, C_2$, and $C_3$ are constants depending on $\rho$ and $V$, whose concrete forms are provided in Appendix~\ref{app:proof_lem:onestep}.
As a result, if these constants are all bounded, 
then we have $W_2(\rho_\tau^{\rm ex}, \rho_\tau) \lesssim \tau^{3/2}$ and $W_2(\rho_\tau^L, \rho_\tau) \lesssim \tau^{3/2}$.
In particular, if the third-order derivatives of $V$ and $\log \rho$ are bounded by some constant $B > 0$, then all these constants $C_i$, $i=1,2,3$, can be bounded by $\max\{d^2B^4, d^2B^2, 4d^3B^2(1+B)^2, 16d^3B^4\}$.
\end{lemma}

The one step numerical error of order $\m O(\tau^{3/2})$ from the lemma implies a cumulative error over a given time period $T$ to be of order $\m O(T \tau^{1/2})$ by aggregating the errors from $N=T/\tau$ steps if all the quantities in~\eqref{eqn:bounded_quantities} in Appendix~\ref{app:proof_lem:onestep} remain bounded across the intermediate iterates, which unfortunately cannot be proved without making extra restrictive assumptions in general. For example, for the discretized Langevin SDE, without any extra assumptions on potential $V$, applying a classical analysis of the Euler--Maruyama method based on Gronwall's inequality leads to a cumulative error $O(e^T\tau^{1/2})$ that grows exponentially fast in $T$.
However, if we assume the potential $V$ to be $\lambda$-strongly convex for some $\lambda >0$, then a careful analysis leads to a cumulative error of order $\m O(\max\{T^2, \,T\}\,\tau^\frac{1}{2})$, which implies the $\m O(\tau^{3/2})$ bound on $W_2(\rho_\tau^{\rm FP}, \rho_\tau^L)$ in Lemma~\ref{lem: onestep_err} as a special case by taking $T=\tau$.
A similar remark applies to the cumulative error analysis of the JKO-scheme for approximating the solution of the Fokker--Planck equation with an autonomous potential $V$. The key step in the analysis is to derive a discrete evolutionary variational inequality for charaterizing the JKO-scheme (see Theorem~\ref{thm: discrete_langevin_vs_JKO} for further details and~\cite{ambrosio2008gradient} for results regarding general metric spaces). It is worthwhile mentioning that the strongly convexity assumption on $V$ is not needed in deriving the $\m O(\tau^{3/2})$ one-step error bound; however, it implies the long term stability (exponential convergence) of the dynamical system, leading to a better control on the cumulative error.



\end{document}